\documentclass[11pt,reqno]{amsart}

\usepackage[T1]{fontenc}
\usepackage{lmodern}
\usepackage{microtype}
\usepackage{amsmath,amssymb,mathtools}
\usepackage{amsthm}
\usepackage{aliascnt}
\usepackage{booktabs}
\usepackage{tabularx}
\usepackage{enumitem}
\usepackage{xcolor}
\usepackage{tikz}
\usetikzlibrary{arrows.meta,positioning,fit}
\usepackage[colorlinks=true,linkcolor=blue!55!black,
  citecolor=green!35!black,urlcolor=blue!60!black]{hyperref}

\allowdisplaybreaks

\newtheorem{theorem}{Theorem}[section]
\newaliascnt{proposition}{theorem}
\newtheorem{proposition}[proposition]{Proposition}
\aliascntresetthe{proposition}
\newaliascnt{lemma}{theorem}
\newtheorem{lemma}[lemma]{Lemma}
\aliascntresetthe{lemma}

\newaliascnt{corollary}{theorem}
\newtheorem{corollary}[corollary]{Corollary}
\aliascntresetthe{corollary}
\newaliascnt{claim}{theorem}

\aliascntresetthe{claim}
\theoremstyle{definition}
\newaliascnt{definition}{theorem}
\newtheorem{definition}[definition]{Definition}
\aliascntresetthe{definition}
\newaliascnt{example}{theorem}

\aliascntresetthe{example}
\theoremstyle{remark}
\newaliascnt{remark}{theorem}
\newtheorem{remark}[remark]{Remark}
\aliascntresetthe{remark}
\usepackage[nameinlink,noabbrev]{cleveref}

\newcommand{\R}{\mathbb R}
\newcommand{\T}{\mathbb T}
\newcommand{\B}{\mathbb B}
\newcommand{\W}{\mathbb W}
\newcommand{\Scal}{\mathbb S}
\newcommand{\Q}{\mathcal Q}
\newcommand{\Ccal}{\mathcal C}
\newcommand{\1}{\mathbf 1}
\newcommand{\dist}{\operatorname{dist}}
\newcommand{\diam}{\operatorname{diam}}
\newcommand{\conv}{\operatorname{conv}}
\newcommand{\ang}{\operatorname{angle}}

\newcommand{\labmass}{M_{\mathrm{lab}}}

\newcommand{\eps}{\varepsilon}
\newcommand{\lesslog}{\lesssim_{\log}}
\newcommand{\gtrlog}{\gtrsim_{\log}}

\title[Cellular maximal-density factoring]
{Cellular Maximal-Density Factoring:\\
 Joint Shadings, Stable Refinements, and a Conditional Kakeya Application}
\author{Zhixu Hua}
\address{Changkong College, Nanjing University of Aeronautics and
Astronautics, Nanjing, Jiangsu, China}
\email{hzx001@nuaa.edu.cn}

\author{Xiufan Yang}
\address{School of Science, Nanjing University of Posts and
Telecommunications, Nanjing, Jiangsu, China}
\email{B25100020@njupt.edu.cn}



\date{August 27, 2026}
\hypersetup{
  pdftitle={Cellular Maximal-Density Factoring: Joint Shadings, Stable Refinements, and a Conditional Kakeya Application},
  pdfauthor={Zhixu Hua, Xiufan Yang},
  pdfsubject={A cellular factoring method with a conditional application to the three-dimensional Kakeya reduction},
  pdfkeywords={Kakeya conjecture, maximal-density factoring, incidence geometry, shaded tubes}
}

\begin{document}

\begin{abstract}
Maximal-density factoring must often make one parent object carry several
forms of information at once: child mass, two multiplicity levels, a density
estimate, and the ability to survive later geometric refinements.  These
properties are not stable under arbitrary deletion.  We introduce a cellular
factoring method that places positive child mass on a weighted bipartite graph
of labelled parents and half-open spatial cells.  After regularizing edge
weights and cell degrees, one edge set defines both the child shading and the
parent shading.  Compatibility and pointwise multiplicity bounds are then
exact, while any later selection of whole parent--cell incidences lifts
quantitatively back to the children.  The combinatorial core does not require
Euclidean geometry: it extends to finite measurable atomic partitions with
comparable atom measures and admits an explicit finite-iteration loss.

We combine this joint construction with a weighted planar incidence estimate,
a resolution-limited cellwise angular selection, a labelled slab estimate,
and a global synchronization of tubelet and fine-tube mass.  The resulting
rectangular parent estimate has shading exponent \(2-\sigma\), which is sharp
uniformly over the class considered here.  With the stated GWZ factoring data
and corresponding \(K_F(\beta)\) or \(K_{KT}(\beta)\), we obtain the two
small-middle interfaces used in the reduction.  For the
very-non-sticky branch, assuming both \(K_{KT}(\beta)\) and \(K_F(\beta)\),
the stated GWZ Lemma~9.1 hypotheses together with the source-uniformity and
refinement interface (I1)--(I4) yield the terminal gain
for the represented fine-tube input; this conditional application supplies
the corresponding input to the surrounding reduction, while the
sharp-convex-parent problem lies outside its scope.
\end{abstract}

\maketitle
\tableofcontents

\subsection*{Notation and conventions}
For a finite labelled family \(\mathcal F\) of measurable carriers of
positive finite volume, \(|\mathcal F|\) denotes its cardinality; two labels
may have the same geometric carrier.  A
\emph{shading} is an assignment of a measurable set \(Y(F)\subset F\) to
each label \(F\in\mathcal F\).  We use
\[
 \begin{aligned}
 U(\mathcal F,Y)&:=\bigcup_{F\in\mathcal F}Y(F),&
 M(Y)&:=\sum_{F\in\mathcal F}|Y(F)|,\\
 D(\mathcal F)&:=\sum_{F\in\mathcal F}|F|,&
 \lambda(\mathcal F,Y)&:=\frac{M(Y)}{D(\mathcal F)},&
 \mu(\mathcal F,Y)&:=\frac{M(Y)}{|U(\mathcal F,Y)|}.
 \end{aligned}
 \tag{1.1}
\]
The last two ratios are set to zero when their denominators vanish.  Thus
\(M\), \(D\), and \(\mu\) count coincident carriers with multiplicity, while
\(U\) is their physical union.  For a measurable set \(E\), let
\(N_a(E)=\{x:\dist(x,E)<a\}\).  Let \(\mathfrak K_n\) denote the class of
measurable convex sets \(K\subset\R^n\) with
\(0<|K|<\infty\).  For \(K\in\mathfrak K_n\), define
\[
 \Delta(\mathcal F,K)
 :=\frac{\sum_{F\in\mathcal F:\,F\subset K}|F|}{|K|},
 \qquad
 \Delta_{\max}(\mathcal F)
 :=\sup_{K\in\mathfrak K_n}\Delta(\mathcal F,K).
 \tag{1.2}
\]
If every member of \(\mathcal F\) lies in
\(K_0\in\mathfrak K_n\), define
\[
 C_F(\mathcal F,K_0)
 :=\sup_{\substack{K\in\mathfrak K_n\\K\subset K_0}}
   \frac{\Delta(\mathcal F,K)}{\Delta(\mathcal F,K_0)}.
 \tag{1.3}
\]
We set \(\Delta_{\max}(\varnothing)=0\), and set
\(C_F(\mathcal F,K_0)=0\) when
\(\Delta(\mathcal F,K_0)=0\) (in particular for the empty family).  Thus
when one ambient body \(K_0\) has been fixed by the surrounding statement,
we abbreviate \(C_F(\mathcal F):=C_F(\mathcal F,K_0)\); no ambient body is
suppressed when more than one is in play.  Thus
neither a null or lower-dimensional convex set nor an infinite-volume set is
ever an admissible test body, and no ratio in (1.2)--(1.3) is undefined.
All sums in (1.1)--(1.3) are over labels.  We call the closed
\(c_{\rm cap}w\)-neighbourhood of a carrier segment of length exactly one a
\emph{source-standard tube of width \(w\)}; both endpoint caps are included.  The unit-scale
tubes and their cross-scale ancestors in the direct source interfaces of
Sections~9--10 are source-standard.  A local \(w\times w\times r\) tubelet,
on the other hand, may have carrier length in
\([c_{\rm len}r,C_{\rm len}r]\); this is the only place where ``comparable
length'' is used.  The four constants
\(c_{\rm cap},C_{\rm cap},c_{\rm len},C_{\rm len}\) are fixed once and for
all.  Every containment statement includes the caps.  A \emph{fixed
enlargement} means a dilation by a constant depending only on dimension and
these four structural constants.  It may change the cap constant, but it does
not change or merge the labelled carrier.  Thus every use of \(O(1)\),
``comparable'', or a fixed enlargement in a tube, plank, or slab interface has
the same declared structural dependence.
We write \(A\lesssim_{\log}B\) when
\(A\le C[\log(2/\delta)]^{C}B\) for a structural exponent, and write
\(A=\delta^{o(1)}B\) only when a displayed uniform function
\(\omega(\delta)\to0\) gives
\(\delta^{\omega(\delta)}B\le A\le\delta^{-\omega(\delta)}B\).
For an explicitly named secondary scale \(s\in\{\alpha,\rho,r_1\}\), the
notation \(s^{o(1)}\) has the identical meaning with one displayed uniform
function \(\omega_s(s)\to0\); the polynomial scale windows in the surrounding
statement are used whenever it is converted to a power of \(\delta\).
For results that absorb source/refinement losses into terminal powers---in
particular \cref{prop:s6A,prop:s6B} and the terminal results in Section~10---a
\emph{loss profile}
\(\mathfrak p\) consists of the fixed polynomial-complexity, bounded-fibre,
and structural constants together with fixed moduli controlling every
source/refinement occurrence of \(o(1)\), including \(\omega_0\),
\(\omega_{\rm br}\), and the derived envelopes \(\omega_{\rm ref}\) and
\(\omega_{\rm src}\).  The
profile is fixed before any cutoff \(\delta_0\) is chosen.  Thus
\(\delta_0=\delta_0(\beta,\zeta,\mathfrak p)\) may depend on these fixed
convergence rates and constants, while the positive gain exponents below
remain functions only of the displayed geometric parameters.  Every fixed
\(\delta\)-indexed input class with common subpower bounds determines such a
profile by taking the finite maximum of its named moduli.  No cutoff uniform
over all possible rates \(\omega(\delta)\to0\) is asserted.
A \emph{child point}
always means a point of the currently displayed child union
\(U(\mathcal T,Y)\); an unrestricted ambient point is called simply a point.
We denote a structural dilation of a carrier by \(C_{\rm car}F\) and a
structural enlargement of an ambient ball by \(C_{\rm amb}B\).  Thus source
carrier separation is never conflated with overlap of tagged ambient balls.
When all tubes in a displayed family are congruent, \(|T|\) denotes their
common geometric volume; \(|\T|\) always denotes labelled cardinality.
Whenever an ambient ball is
attached to a label, the construction is performed first in the tagged
disjoint union (3.4b); physical recombination occurs only through
\cref{lem:tag-recombine}.

\section{Introduction}
\label{sec:introduction}

The central question of this paper is one of simultaneous control.  Suppose
short shaded tubes are organized through larger convex parents.  Can the
parent output retain the child mass, regularize both the inner and outer
multiplicities, satisfy the required density estimate, and remain stable
under the geometric selections that follow?  In the streamlined
three-dimensional Kakeya reduction of Guth--Wang--Zahl~\cite{GWZ2026}, built
on the convex-set framework of Wang--Zahl~
\cite{WZ2025,WZSticky2026,WZAssouad2025}, all four demands occur in the same
factorization.  Treating them separately loses the compatibility needed at
the next scale.

The difficulty is not merely that a refinement may lose mass.  Parent mass
and child mass live on different objects, and the usual conclusions are not
monotone under arbitrary deletion.  Nearby child support does not determine
the multiplicity of a clipped parent neighbourhood at a different point;
after a further child selection, constant multiplicity and exact induced
compatibility may both fail.  The finite configurations in
\cref{sec:obstructions} isolate these mechanisms.  Their role is to identify
the structure that a stable factorization must retain.

Our organizing principle is to make the incidence, rather than either
shading alone, the indivisible object.  Fix a half-open grid of side
\(h=c_0a\), where \(a\) is the shortest parent scale.  A labelled parent
\(B\) and a cell \(Q\) are joined precisely when the children assigned to
\(B\) carry positive shaded mass in \(Q\).  We regularize first the positive
edge weights and then the cell degrees.  The surviving edges define the
child shading by restriction and the parent shading by whole cells.  Thus
mass retention, inner multiplicity, parent multiplicity, compatibility, and
the pointwise product bound all come from one finite weighted graph.

The same representation makes subsequent refinement quantitative.  An
admissible parent refinement is a subcollection of complete parent--cell
edges.  Equal cell volumes convert retention of labelled parent mass into
retention of edge count, while comparable positive edge weights convert that
edge count back into child mass.  In particular, the typical-angle descent
from~\cite{GWZ2026} may be run once per cell, with deterministic tie-breaking,
so that its output is again an edge subset and requires only one lift.  At
the angular resolution \(a/b\), parallel labels are distinguished through
the resolution-floored angular separation
\[
 \ang_{a/b}(B,B')=\max\{a/b,\ang(TB,TB')\}.
\]
The lower cutoff records the geometric resolution of the parent family.

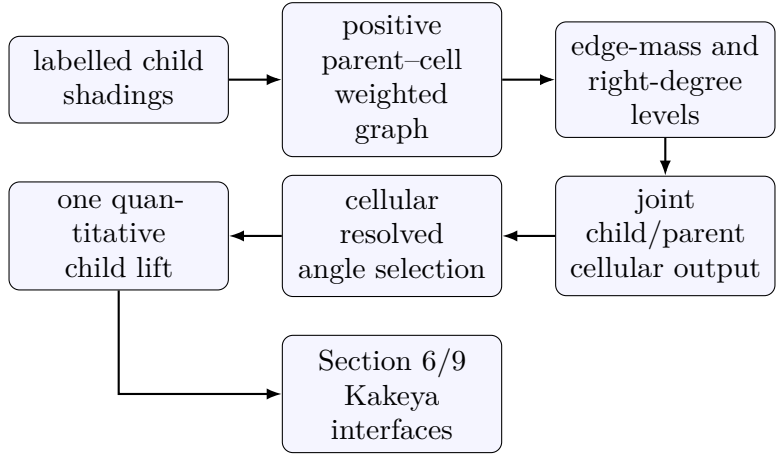
\begin{figure}[t]
\centering
\begin{tikzpicture}[
  node distance=5mm and 7mm,
  box/.style={draw,rounded corners,align=center,inner sep=5pt,
              text width=2.55cm,fill=blue!4},
  arr/.style={-{Latex[length=2mm]},thick}
]
\node[box] (child) {labelled child\\shadings};
\node[box,right=of child] (graph) {positive parent--cell\\weighted graph};
\node[box,right=of graph] (reg) {edge-mass and\\right-degree levels};
\node[box,below=of reg] (joint) {joint child/parent\\cellular output};
\node[box,left=of joint] (angle) {cellular resolved\\angle selection};
\node[box,left=of angle] (lift) {one quantitative\\child lift};
\node[box,below=of angle] (int) {Section 6/9\\Kakeya interfaces};
\draw[arr] (child)--(graph);
\draw[arr] (graph)--(reg);
\draw[arr] (reg)--(joint);
\draw[arr] (joint)--(angle);
\draw[arr] (angle)--(lift);
\draw[arr] (lift) |- (int);
\end{tikzpicture}
\caption{The cellular mechanism.  After regularization, every admissible
selection acts on complete positive parent--cell incidences.}
\label{fig:pipeline}
\end{figure}

Two further ideas connect the cellular theorem to the multiscale
application.  A tubelet-level refinement need not preserve mass at the level
of the represented fine tubes, so we choose one global fine-multiplicity
level after the coarse tubelet levels are fixed.  This makes cellular coarse
mass quantitatively equivalent to fine incidence mass.  The angular
selection may also lower parent degree by \(A_{\rm ang}^{-K}\); to meet the
next lemma's output-degree hypothesis, we split at the buffered
threshold
\[
 H_{\rm ang}:=4A_{\rm ang}^K,
 \qquad \alpha:=a/r_1.
\]
When \(d_0\ge H_{\rm ang}\alpha^{-\eta_{\rm pl}}\), the high branch enters
the plank-reduction lemma; the complementary range
\(d_0<H_{\rm ang}\alpha^{-\eta_{\rm pl}}\) falls within a direct low-degree
estimate.  The buffer is what prevents an uncovered middle band.

The density exponent comes from a separate analytic layer.  The planar
incidence theorem of~\cite{DOV2022}, transferred through a weighted
layer-cake argument, accommodates repeated projected parameters and
measurable shadings.  It yields the rectangular shading power
\(\lambda^{2-\sigma}\); the ordinary Frostman endpoint is supplied by the
quadratic C\'ordoba method~\cite{Cordoba1977}.  Our contribution is the joint
cellular interface in which these estimates remain compatible with mass
lifting, together with the two small-middle applications and the conditional
very-non-sticky closure under (I1)--(I4).

\subsection*{Main results and architecture}

The method has an abstract core and a Euclidean realization.  The atomic
theorem,
\cref{thm:abstract-cellular}, applies to an arbitrary finite measurable
partition and uses neither a metric nor convexity.  Its density-equipped
Euclidean companion, \cref{thm:mainfact}, starts from labelled children in
comparable rectangular parents of polynomially bounded complexity.  Both
theorems produce compatible child and parent shadings from a single positive
edge set.  Schematically,
\[
\begin{aligned}
 M(Y_1)&\ge \mathfrak L^{-1}M(Y),\\
 \mu_{\rm child}&\sim\mu_{\rm in},\qquad
 \mu_{\rm parent}\sim d_0,\\
 \mu_{\rm total}&\lesssim\mu_{\rm in}d_0.
\end{aligned}
\]
Every permitted later parent refinement then lifts quantitatively to the
children.  Under the biased complete-child hypothesis, the Euclidean output
also satisfies
\[
 \lambda(\widetilde\B',Z)
 \gtrsim C_0^{-(2-\sigma)}(a/b)^\eps
 \kappa_D^{-(2-\sigma)}\mathfrak L^{-(2-\sigma)}
 \lambda(\T,Y)^{2-\sigma}.
\]
The Kakeya application is Corollary~\ref{thm:section9}.  It assumes
\(K_{KT}(\beta)\), \(K_F(\beta)\), the uniform bounds in (10.33), and the
conditional source interface (I1)--(I4) of
\cref{prop:source-extraction}.  Once tubelet and outer-ball parameters have
been synchronized locally and across whole tags, these hypotheses close the
very-non-sticky branch on the represented original fine-tube input.  The
cellular theorems themselves are unconditional; only this multiscale
application depends on (I1)--(I4).

Sections~2--8 develop the reusable theorem.  Section~2 derives the necessary
stability requirements from finite obstructions.  Sections~3--5 build the
positive parent--cell graph, regularize its two multiplicity scales, prove
whole-incidence lifting, and isolate the atomic core.  Sections~6--7 add the
rectangular density estimate and the cellwise angular protocol; Section~8
assembles them.  Sections~9--10 then apply this machinery to the two
small-middle factorizations and to the very-non-sticky reduction.  Section~11
proves uniform sharpness of the density exponent, and Section~12 identifies
the sharp-convex boundary of the method.  Appendices A--D contain the weighted
planar transfer, the labelled slab estimate, the multiset filling lemma, and
the quantitative source map.

\subsection*{Relation to previous work and scope}

Grid cellularization, dyadic graph regularization, the DOV exponent, and the
ordinary typical-angle descent are established ingredients.  The point here
is their organization around a positive parent--cell edge set whose child and
parent outputs survive the same downstream selections.  Cohen's
notes~\cite{Cohen2026} develop a related but different
submodular/global decomposition, while Guth's survey~\cite{Guth2026Survey}
places the surrounding reduction in a broader context.  Both serve as
comparison sources for the present construction.

Zeraoulia Rafik's note~\cite{Zeraoulia2026} contains a broader set of
source-level bookkeeping observations and an independent C\'ordoba-type proof
of the shaded slab estimate.  The genuine overlap is the slab \(L^2\)
calculation in \cref{app:slab}.  The positive parent--cell construction,
whole-incidence lifting, and tagged synchronization are developed independently
here.

The local note~\cite{Zeraoulia2026Local} identifies that an unrestricted
\(C_F\) condition is stronger than the all-convex definition (1.3) supports.
Accordingly, the application uses the restricted source slab-test
condition required by GWZ Lemma~6.1; Proposition~9.7 and its conclusion are
unchanged.

The rectangular scope is structural.  A comparable outer box preserves the
labelled counting conditions needed in the WZ application, but its union lower
bound need not restrict uniformly to an arbitrary sharp convex parent.  We
therefore keep rectangular parents and verify the fixed-dilation downstream
lemmas in that category.

\section{Obstructions to non-cellular refinement}
\label{sec:obstructions}

Before constructing a stable factorization, we isolate two tests that any
candidate must pass.  The first asks whether nearby child support determines
the multiplicity of a clipped parent shading.  The second asks whether
compatibility and multiplicity survive a refinement made after the parent
conclusions have been established.  These are independent failures of
monotonicity, and the cellular construction will answer them by different
features: exact cell ownership for the first, and indivisible incidences for
the second.  For the first test, let
\[
 \mathbb V'=\bigsqcup_{W\in\mathbb W'}\mathbb V'_W
\]
and let \(Y\) be a child shading.  The induced parent shading in
\cite{GWZ2026} has the form
\[
 Z(W)=W\cap N_a\!\left(U(\mathbb V'_W,Y')\right).
\tag{2.1}
\]
Suppose one controls, on child points \(x\), the nearby-support quantity
\[
 \widetilde\mu(x)=
 \#\{W:B(x,a)\cap U(\mathbb V'_W,Y')\ne\varnothing\}.
\tag{2.2}
\]
The point of the example is that (2.2) and the multiplicity of (2.1) are
different functions: the former sees proximity to child support, while the
latter also sees membership in the individual parent.

\begin{proposition}[Nearby support does not regularize the induced parent]
\label{prop:F1}
For every \(M\ge4\), there are \(M\) congruent labelled rectangular parents,
one shaded child in each parent, and a common level on which
\(\widetilde\mu=M\), while the multiplicity of the induced parent shading
assumes both values \(1\) and \(M\).
\end{proposition}

\begin{proof}
Fix \(a>0\), put \(d=a/(4M)\), and let \(c_j=jd\) for \(0\le j<M\).
Choose congruent labelled parents
\[
 W_j=[c_j-a/2,c_j+a/2]\times[-b/2,b/2]\times[-r/2,r/2],
\qquad a\le b\le r.
\]
Inside \(W_j\), take the congruent convex child
\[
 V_j=[c_j-a/4,c_j+a/4]\times[-b/4,b/4]\times[-r/4,r/4]
\]
and shade it only on
\[
 S_j=B((c_j,0,0),d/10).
\]
The \(S_j\) are pairwise disjoint.  At every child point, the inner and outer
multiplicities equal one, while all \(M\) support balls lie within distance
\(a\).  Hence \(\widetilde\mu=M\) on the complete child union.

The singleton family \(\{V_j\}\) is \(8\)-Frostman in \(W_j\): a convex test
body either contains no child or has volume at least \(|V_j|=|W_j|/8\).
The exact induced shading is \(Z_j=W_j\cap N_a(S_j)\).  The point
\[
 y_{\rm core}=((c_0+c_{M-1})/2,0,0)
\]
belongs to all \(Z_j\).
In contrast,
\[
 y_{\rm edge}=(c_0-a/2+d/4,0,0)
\]
belongs to \(Z_0\) and lies strictly to the left of every \(W_j\), \(j\ge1\).
Thus the induced multiplicity is \(M\) at the core and \(1\) at the edge.
\end{proof}

\begin{remark}
The mismatch comes from clipping by \(W\).  Nearby child support gives
\(x\in N_a(U(\mathbb V'_W,Y'))\), whereas parent membership at \(y\) also
requires \(y\in W\).  An average multiplicity estimate cannot replace the
pointwise conclusion needed later.
\end{remark}

\begin{proposition}[A final child refinement is not automatically harmless]
\label{prop:F2}
Each of the following monotonicity assertions fails for a finite labelled
configuration:
\begin{enumerate}[label=\textup{(N\arabic*)},leftmargin=2.8em]
\item constant inner multiplicity survives an arbitrary child deletion;
\item the induced parent shading is unchanged after such a deletion;
\item constant parent multiplicity survives when the induced shadings are
recomputed;
\item the pointwise inner--outer multiplicity product survives; and
\item retaining a fixed fraction of labelled parent volume retains a fixed
fraction of child shading mass.
\end{enumerate}
In particular, none of these properties can be transferred through a final
``mass-preserving refinement'' without an additional structural hypothesis.
\end{proposition}

\begin{proof}
All examples use only two measurable pieces.

For inner multiplicity, take two identical child shadings on a set \(A\).
Delete the second child on a positive-measure subset \(A_0\subset A\).
The new inner multiplicity is one on \(A_0\) and two on \(A\setminus A_0\),
which proves (N1).  If the two children have one common parent whose shading
is all of \(A\), the outer multiplicity stays one and the inner--outer
product takes the values one and two.  This also proves (N4).

For exact parent shading, let one child shading have separated pieces
\(A_1\cup A_2\).  Keeping only \(A_1\) changes
\[
 W\cap N_a(A_1\cup A_2)
 \quad\text{to}\quad
 W\cap N_a(A_1).
\]
Retaining the old parent shading instead breaks the exact induced identity,
which proves (N2).  Here is a literal two-parent incidence table that also
proves (N3).  Take disjoint \(a\)-neighbourhoods of \(A_1,A_2\), and take two
labelled parents \(W_1,W_2\) containing both neighbourhoods.  Initially give
each parent one child shading \(A_1\cup A_2\).  Refine only the
\(W_1\)-child to \(A_1\), leave the \(W_2\)-child unchanged, and recompute
the induced parent shadings.  On the two positive-measure atoms the complete
before/after table is
\[
\begin{array}{c|cccc|cccc}
 &Y^{\rm in}_{W_1}&Y^{\rm in}_{W_2}&Z^{\rm in}(W_1)&Z^{\rm in}(W_2)
 &Y^{\rm out}_{W_1}&Y^{\rm out}_{W_2}&Z^{\rm out}(W_1)&Z^{\rm out}(W_2)\\ \hline
 A_1&1&1&1&1&1&1&1&1\\
 A_2&1&1&1&1&0&1&0&1
\end{array}
\tag{2.3a}
\]
where a \(Z\)-entry records membership in the exact recomputed set
\(W_j\cap N_a(Y_{W_j})\).  Thus the induced shading of \(W_1\) genuinely
changes, and labelled parent multiplicity changes from two on both atoms to
two on \(A_1\) and one on \(A_2\).  No informal transfer between different
pieces is being used.

Finally, give two parent pieces equal volume but child masses \(\eps\) and
\(1\).  A parent selection retaining the first piece keeps half the labelled
parent mass and only \(\eps/(1+\eps)\) of child mass.  This tends to zero.
This proves (N5).  A joint construction is therefore necessary.
\end{proof}

\begin{proposition}[Failure of parentwise restoration after angular refinement]
\label{prop:source-near}
Even the following finite matched-output interface does not imply
compatibility after the whole old child shading of each retained parent label
is restored: exact child--parent incidence inclusion, constant inner
multiplicity in each labelled parent, constant total parent multiplicity,
common labelled parent-mass and local-union-volume levels, and the pointwise
angle-class rule followed by a global labelled-parent-mass pigeonhole used in
\cite[Proposition~5.1, p.~12 and Lemma~6.11]{GWZ2026}.  These are precisely
the displayed properties used at the tangential lift.  Thus the implication
between source equations (95)--(96) and (100)--(101) is not derivable from
that displayed interface alone.
\end{proposition}

\begin{proof}
Fix \(0<\vartheta<2^{-8}\), put
\(x_A=(0,-1/4,0)\), \(x_B=(0,1/4,0)\), and choose
\(0<r_0\ll\delta\).  Let \(A=B(x_A,r_0)\) and
\(B_*=B(x_B,r_0)\).  Take three labelled rectangular parents whose labelled
tangent planes are
\[
 P_1=\operatorname{span}(e_2,e_3),\quad
 P_2=\operatorname{span}(\cos(16\vartheta)e_2+
                  \sin(16\vartheta)e_1,e_3),
\]
\[
 P_3=\operatorname{span}(\cos(\vartheta)e_2+
                  \sin(\vartheta)e_1,e_3).
\tag{2.4a}
\]
The first parent is long enough in its \(e_2\)-direction to meet both balls;
the second and third are translated to meet \(A\) and \(B_*\), respectively.
Thus the separation of the two centres is parallel to the long
\(e_2\)-axis of the first parent.  This coordinate correction changes
neither the three planes in (2.4a), nor the incidence table below, nor any
of the pairwise angle values.
Choose all three shortest widths much larger than \(r_0\).  Vertical
\(\delta\)-tubes through the two centres then contain the indicated balls and
are contained in the corresponding parents.  Use one labelled child of
\(W_1\) shaded by \(A\), a second shaded by \(B_*\), one child of \(W_2\)
shaded by \(A\), and one child of \(W_3\) shaded by \(B_*\).  Empty portions
of the geometric tubes play no role.

On the two positive-measure atoms the complete incidence table is
\[
\begin{array}{c|ccc|ccc}
 &Y_{W_1}&Y_{W_2}&Y_{W_3}&Z(W_1)&Z(W_2)&Z(W_3)\\ \hline
 A   &1&1&0&1&1&0\\
 B_* &1&0&1&1&0&1
\end{array}
\tag{2.4b}
\]
where each entry means that the corresponding shading equals the entire
atom.  Thus the inner multiplicity in every labelled parent is one, the
parent multiplicity is exactly two on both atoms, the labelled parent masses
of the two atoms agree, and parent-scale local union volumes agree.  These
are the matched-output properties used at the later lift.  If a literal
neighbourhood model is desired, first replace \(A,B_*\) by still smaller
concentric balls and take the displayed sets as measurable subshadings of the
induced clipped neighbourhoods; the table is unchanged.

At \(A\) the only parent pair has angle \(16\vartheta\), whereas at \(B_*\)
the only pair has angle \(\vartheta\).  These lie in disjoint dyadic classes.
The legal global angle-class pigeonhole that keeps the
\(16\vartheta\)-class therefore retains precisely the two \(A\)-incidences,
half of the labelled parent mass, and constant parent multiplicity two.  It
deletes the \(B_*\)-part of \(Z(W_1)\).  Restoring the whole old child shading
for the retained label \(W_1\) restores its \(B_*\)-child, although
\(B_*\cap Z_{\rm ang}(W_1)=\varnothing\).  Hence the asserted incidence
inclusion fails on all of \(B_*\).

The construction has a precise logical scope: the displayed implication is
not a consequence of the listed interface alone.  It makes no assertion
about Proposition~5.1, the typical-angle lemma, or Main Lemma~2 of
\cite{GWZ2026}.  What it identifies is the missing invariant.  In
\cref{thm:lift}, parent and child shadings are restricted on the same cells,
so restoration is replaced by a quantitative lift along surviving edges.
\end{proof}

The three propositions point to one common design rule.  A parent shading
should not be reconstructed after a child selection, and a child shading
should not be restored after a parent selection.  Both must instead be read
from a shared incidence object.  Sections~3--5 build that object and prove
the exact stability statements suggested by these examples.

\section{Cellular parent shadings and tagged geometry}
\label{sec:cellular}

The shared object has four layers.  A child belongs to a sharp parent; the
parent is assigned a permanent fixed enlargement; space is partitioned into
half-open cells; and positive child mass turns a parent--cell pair into a
weighted edge.  The parent and child shadings will both be recovered from
those edges.  Half-open cells make boundary ownership unique, while tagged
copies prevent local constructions in different ambient balls from being
identified before their multiplicities have been compared.

Fix \(0<c_0\le10^{-3}\), set \(h=c_0a\), and let
\[
 \Q_h=\{h(k+[0,1)^3):k\in\mathbb Z^3\}
\]
be the global half-open grid.  Every point belongs to exactly one cell.
For \(E\subset\R^3\), define
\[
 \Ccal_h(E)=\bigcup_{\substack{Q\in\Q_h\\Q\cap E\ne\varnothing}}Q.
\tag{3.1}
\]

\begin{lemma}[Cellular neighbourhood comparison]
\label{lem:CELL1}
For every \(E\subset\R^3\),
\[
 E\subset\Ccal_h(E)\subset N_{\sqrt3h}(E).
\tag{3.2}
\]
If \(h\asymp a\), then
\[
 |N_a(E)|\lesssim|\Ccal_h(E)|\lesssim|N_{Ca}(E)|,
\tag{3.3}
\]
with constants depending only on the comparison between \(h\) and \(a\).
\end{lemma}

\begin{proof}
The unique cell containing \(x\in E\) appears in (3.1).  If \(y\) lies
in an active cell, it is within that cell's diameter \(\sqrt3h\) of a point
of \(E\), proving (3.2).

For the nontrivial inequality in (3.3), let
\(\mathcal A=\{Q:Q\cap E\ne\varnothing\}\).  Then
\[
 N_a(E)\subset\bigcup_{Q\in\mathcal A}N_{a+\sqrt3h}(Q).
\]
Each last neighbourhood meets only \(O(1)\) grid cells when \(h\asymp a\).
Summing first over an increasing finite exhaustion
\(\mathcal A_1\subset\mathcal A_2\subset\cdots\) gives the same bound with
a constant independent of the exhaustion.  The corresponding unions increase
to the full union; continuity from below (equivalently, monotone convergence
applied to their indicators) therefore gives the result.  In every later
application \(\mathcal A\) is finite, but this last sentence makes the stated
arbitrary-set version literal.
\end{proof}

\begin{lemma}[Parent boundary]
\label{lem:CELL2}
Let \(B\) be an arbitrarily oriented \(a\times b\times r\) rectangular box,
\(a\le b\le r\), and let \(E\subset B\).  There is a fixed enlargement
\(\widetilde B\) with side lengths comparable to \(a,b,r\) such that
\[
 \Ccal_h(E)\subset\widetilde B,
 \qquad |\widetilde B|\asymp|B|.
\tag{3.4}
\]
\end{lemma}

\begin{proof}
If \(e_1,e_2,e_3\) are the side directions and \(x_B\) is the centre, take
\[
 \widetilde B=
 \left\{x_B+\sum_{j=1}^3s_je_j:
 |s_j|\le\tfrac12(a,b,r)_j+\sqrt3h\right\}.
\]
Every active cell point is within \(\sqrt3h\) of \(B\), and
\(h=c_0a\le c_0b\le c_0r\).  This proves both assertions.
\end{proof}

\begin{lemma}[Convex parallel bodies]
\label{lem:convex-parallel-body}
Let \(K\subset\R^n\) be convex and suppose \(B(z,r)\subset K\), where
\(r>0\).  For every \(\lambda\ge0\),
\[
 K+B(0,\lambda r)
 \subset z+(1+\lambda)(K-z).
\tag{3.4p}
\]
Consequently
\[
 |K+B(0,\lambda r)|\le(1+\lambda)^n|K|.
\tag{3.4q}
\]
Both conclusions remain valid with closed or open balls and are insensitive
to boundary points.
\end{lemma}

\begin{proof}
The case \(\lambda=0\) is immediate.  If \(x\in K\) and
\(|y|\le\lambda r\), then \(z+y/\lambda\in B(z,r)\subset K\).  Convexity
gives
\[
 \frac{x-z}{1+\lambda}
 +\frac{\lambda}{1+\lambda}\frac{y}{\lambda}
 \in K-z.
\]
Multiplication by \(1+\lambda\) proves (3.4p), and the determinant of the
homothety gives (3.4q).  The open-ball case follows by approximation.  This
is the elementary parallel-body inclusion used below; for standard convex
body conventions and Minkowski addition see \cite[Chs.~1 and 3]{Schneider2014}.
\end{proof}

To make later angular statements unambiguous, every labelled rectangular
parent is henceforth supplied with an ordered
orthonormal side frame \((e_a(B),e_b(B),e_r(B))\), corresponding to side
lengths \(a\le b\le r\).  The frame is part of the label when side lengths
coincide.  We write
\[
 TB:=\operatorname{span}\{e_b(B),e_r(B)\},
\tag{3.4a}
\]
and use the following unoriented plane metric.  If \(P,P'\) are two-planes
with arbitrarily chosen unit normals \(n_P,n_{P'}\), set
\[
 d_{\rm pl}(P,P')
 :=\arccos|\langle n_P,n_{P'}\rangle|\in[0,\pi/2].
\tag{3.4a$'$}
\]
This is independent of the signs of the normals and is the standard metric
on the projective plane \(\mathbb{RP}^2\).  In particular it satisfies the
triangle inequality.  We use its capped version
\[
 \ang(P,P'):=\min\{1,d_{\rm pl}(P,P')\}.
\tag{3.4a$''$}
\]
The map \(t\mapsto\min\{1,t\}\) is increasing and subadditive, so
\(\ang\) is again a metric; it is comparable to \(d_{\rm pl}\) and takes
values in \([0,1]\).  All plane diameters and plane caps below, including
those in Appendix~C, use this metric.  Thus \(TB\) in
(7.1) is defined even for repeated geometric boxes or equal side lengths.

\begin{remark}[Labels and ambient balls]
Coincident geometric boxes with different labels remain different parents
and different summands.  A fixed enlargement may leave an ambient ball that
the original parent touches; every application below correspondingly works
in a fixed dilation of that ball.  We never restrict the enlarged conclusion
back to a sharp parent boundary.
\end{remark}

The next convention prevents the carrier from changing type during a
refinement.  A parent label used below is the triple
\[
 \mathbf B=(B^\sharp,\widetilde B,\mathfrak e_B),
\tag{3.4d}
\]
where \(B^\sharp\) is the sharp box to which the children were assigned,
\(\widetilde B\) is the fixed enlargement in \cref{lem:CELL2}, and
\(\mathfrak e_B\) is the ordered frame in (3.4a).  We continue to write \(B\)
for the label, but a parent shading \(Z(B)\) is, by definition, a measurable
subset of the permanent carrier \(\widetilde B\), never a shading clipped to
\(B^\sharp\).  Conditions concerning child containment or complete-child
density refer to \(B^\sharp\); parent volumes in the cellular output refer to
\(\widetilde B\).  Since their volumes are comparable, this convention costs
only the fixed constant in (3.4), while preventing a later change of type.

\begin{lemma}[Tagged-cell recombination]
\label{lem:tag-recombine}
Let \(\mathcal A\) be a finite set of tags.  For each \(A\in\mathcal A\),
fix a bounded measurable support region \(A^+\subset\mathbb R^3\), and assume
that the family \((A^+)_{A\in\mathcal A}\) has overlap at most
\(\kappa_{\rm amb}\).  Work first in the disjoint union
\[
 X^{\rm tag}=\bigsqcup_{A\in\mathcal A}\{A\}\times\mathbb R^3.
\tag{3.4b}
\]
Give each tagged component its own full copy \((A,Q)\) of every half-open
grid cell \(Q\); only the finitely many cells meeting \(A^+\) can be active.
Then every tagged point belongs to exactly one tagged cell, so all degree and
product identities of \cref{thm:joint} are literal in \(X^{\rm tag}\).  If
\(E_A\subset A^+\) are measurable and \(\pi(A,x)=x\),
then
\[
 |\pi(\bigsqcup_A\{A\}\times E_A)|
 \le \sum_A|E_A|
 \le \kappa_{\rm amb}
       |\pi(\bigsqcup_A\{A\}\times E_A)|.
\tag{3.4c}
\]
Moreover, if a tagged pointwise multiplicity is at most \(M\) in every
component, its recombined physical multiplicity is at most
\(\kappa_{\rm amb}M\).  Thus projection back to physical space costs only
the displayed overlap constant; it never turns a tagged identity into an
unrecorded identity in physical space.
\end{lemma}

\begin{proof}
Uniqueness of \((A,Q)\) follows from uniqueness of \(Q\) inside each tagged
copy.  The first inequality in (3.4c) is subadditivity.  For the second,
write \(N(x)=\#\{A:x\in E_A\}\).  Then
\[
 \sum_A|E_A|=\int_{\cup_AE_A}N(x)\,dx
 \le\kappa_{\rm amb}|\cup_AE_A|.
\]
The multiplicity assertion is the same calculation pointwise: at most
\(\kappa_{\rm amb}\) tagged components contribute at one physical point.
\end{proof}

\subsection{The positive parent--cell graph}

We now pass from geometry to a finite weighted graph.  Let \(\B\) be a
finite nonempty labelled family of permanent rectangular
parents as in (3.4d), and let
\[
 \T=\bigsqcup_{B\in\B}\T_B
\]
be a finite indexed family.  Every child \(T\in\T_B\) has the unique labelled
parent \(B\), is contained in \(B^\sharp\), and has a measurable shading
\(Y(T)\subset T\).  Write
\[
 M(Y)=\sum_B\sum_{T\in\T_B}|Y(T)|,
 \qquad
 f_B^Y(x)=\sum_{T\in\T_B}\1_{Y(T)}(x).
\tag{3.5}
\]
Assume throughout Sections~3--5 that \(0<M(Y)<\infty\).  In particular
\(N_{\max}\ge1\), the selected inner mass is positive, and the active
positive-edge graph below is nonempty.
We also fix the global denominator and shading density notation
\[
 D(\T)=\sum_B\sum_{T\in\T_B}|T|,
 \qquad
 \lambda(\T,Y)=\frac{M(Y)}{D(\T)}.
\tag{3.5a}
\]
After the common inner-level selection of \cref{sec:joint}, define
\[
 \omega_{BQ}
 =\int_Qf_B^{Y_0}(x)\,dx
 =\sum_{T\in\T_B}|Y_0(T)\cap Q|.
\tag{3.6}
\]
The active graph is
\[
 E^+=\{(B,Q):\omega_{BQ}>0\}.
\tag{3.7}
\]
Only positive weights define edges.  This exclusion is structural: a
positive-measure shading in one cell together with an isolated point in
another would otherwise create a geometric touching edge carrying no child
mass, and such an edge could survive into the parent output without anything
to lift.

\section{Positive-incidence regularization}
\label{sec:joint}

The geometric input is now a finite weighted bipartite graph.  Its left
vertices are labelled parents, its right vertices are cells, and its edge
weights are child masses.  We select a single subgraph in three steps: one
common inner-multiplicity level, one positive edge-weight level, and one
saturated right-degree level.  Saturation means that every edge at a retained
right vertex is kept, so the degree selected before the output is still the
degree of the output.  All families remain indexed; coincident labels are
never merged.  Let
\[
 N_{\max}=\max_B|\T_B|.
\]

\subsection{The common inner level}

For \(0\le k\le\lceil\log_2N_{\max}\rceil\), set
\[
 X_{B,k}=\{x:2^k\le f_B^Y(x)<2^{k+1}\}.
\]
Since
\[
 M(Y)=\sum_k\sum_B\int_{X_{B,k}}f_B^Y,
\]
one value \(k_0\) carries at least the reciprocal of
\[
 L_{\rm in}=1+\lceil\log_2N_{\max}\rceil
\tag{4.1}
\]
of the mass.  Put \(\mu_{\rm in}=2^{k_0}\) and
\[
 Y_0(T)=Y(T)\cap X_{B,k_0},\qquad T\in\T_B.
\tag{4.2}
\]
At a fixed pair \((B,x)\), (4.2) keeps all or none of the
\(B\)-child incidences.  Hence
\[
 M(Y_0)\ge L_{\rm in}^{-1}M(Y),
\qquad
 \mu_{\rm in}\le f_B^{Y_0}(x)<2\mu_{\rm in}
\tag{4.3}
\]
on every nonempty retained inner union.

\subsection{Edge weights and right degrees}

\begin{lemma}[Positive edge-weight regularization]
\label{lem:REG1}
Let \(E\) be a finite nonempty edge set with weights \(\omega_e>0\), and
write \(\Omega=\sum_E\omega_e\), \(m=|E|\).  There are
\(E_1\subset E\) and \(\omega_0>0\) such that
\[
 \omega_0\le\omega_e<2\omega_0\quad(e\in E_1)
\tag{4.4}
\]
and
\[
 \sum_{e\in E_1}\omega_e
 \ge\frac{\Omega}{2(1+\lceil\log_2(2m)\rceil)}.
\tag{4.5}
\]
\end{lemma}

\begin{proof}
Delete weights below \(\Omega/(2m)\); their sum is below \(\Omega/2\).
The remaining ratio between the largest and smallest weights is at most
\(2m\).  The heaviest dyadic class proves both assertions.
\end{proof}

Apply \cref{lem:REG1} to the active graph \(E^+\) in (3.7).  For a
right cell incident to \(E_1\), write
\[
 d_1(Q)=\#\{B:(B,Q)\in E_1\}.
\]
Select a dyadic value \(d_0\) whose cell class carries the greatest
\(E_1\)-weight, and retain \emph{all} \(E_1\)-edges at those cells:
\[
 E_2=\{(B,Q)\in E_1:d_0\le d_1(Q)<2d_0\}.
\tag{4.6}
\]
Because every \(E_1\)-weight is comparable, the loss in (4.6) is at
most
\[
 L_{\rm deg}=2(1+\lceil\log_2|E_1|\rceil).
\tag{4.7}
\]
The saturation in (4.6) is important: the degree inside \(E_2\) is
still \(d_1(Q)\).

Define
\[
 L_{J1}:=L_{\rm in},\qquad
 L_{R1}:=2(1+\lceil\log_2(2|E^+|)\rceil),\qquad
 L_{R2}:=L_{\rm deg},
\tag{4.8a}
\]
and
\[
 L_{\rm reg}:=L_{J1}L_{R1}L_{R2},\qquad
 \mathfrak L:=L_{\rm reg}.
\tag{4.8}
\]
Thus (J1) means the whole-incidence inner multiplicity level, (R1) the
positive edge-weight level of \cref{lem:REG1}, and (R2) the saturated
right-degree level (4.6).  Their inputs, retained factors, and outputs are
exactly (4.1)--(4.7); the symbols are not placeholders for further
refinements.
Then
\[
 \sum_{(B,Q)\in E_2}\omega_{BQ}
 \ge \mathfrak L^{-1}M(Y).
\tag{4.9}
\]

\subsection{Reading both shadings from one graph}

Put
\[
 \B'=\{B:\text{there exists }Q\text{ with }(B,Q)\in E_2\}.
\tag{4.9a}
\]
For \(B\in\B'\) put
\[
 \Q_B=\{Q:(B,Q)\in E_2\},
\]
and define
\[
 Y_1(T)=Y_0(T)\cap\bigcup_{Q\in\Q_B}Q,\quad T\in\T_B,
\qquad
 Z(B)=\bigcup_{Q\in\Q_B}Q.
\tag{4.10}
\]
For \(B\notin\B'\), set \(\Q_B=\varnothing\), \(Z(B)=\varnothing\), and
\(Y_1(T)=\varnothing\) for \(T\in\T_B\).  Thus all output shadings are total
functions on the original labelled families.
Every selected edge has positive weight, so
\[
 \boxed{Z(B)=\Ccal_h(U(\T_B,Y_1)).}
\tag{4.11}
\]
This identity is the first payoff of the construction: the parent shading is
exactly the cellularization of its retained child union, with no separate
metric reconstruction.

\begin{theorem}[Cellular Joint Factoring]
\label{thm:joint}
The output in (4.10) has the following properties.
\begin{enumerate}[label=\textup{(P\arabic*)},leftmargin=2.8em]
\item \(M(Y_1)\ge\mathfrak L^{-1}M(Y)\).
\item If \(x\in Y_1(T)\) and \(T\in\T_B\), then \(x\in Z(B)\).
\item On \(U(\T_B,Y_1)\),
\[
 \mu_{\rm in}\le\sum_{T\in\T_B}\1_{Y_1(T)}(x)<2\mu_{\rm in}.
\]
\item On \(U(\B,Z)\),
\[
 d_0\le\sum_B\1_{Z(B)}(x)<2d_0.
\]
\item At every child point,
\[
 \sum_{T\in\T}\1_{Y_1(T)}(x)<4\mu_{\rm in}d_0.
\]
\end{enumerate}
Moreover \(Z(B)\) is contained in a labelled fixed enlargement
\(\widetilde B\) with the same side-length scales.
\end{theorem}

\begin{proof}
Disjointness of half-open cells and (3.6) give
\[
 M(Y_1)=\sum_{(B,Q)\in E_2}\omega_{BQ},
\]
so (P1) is (4.9).  The unique cell of a point in \(Y_1(T)\) is selected
for the parent \(B\), proving (P2).

For (P3), an edge \((B,Q)\) retains every \(B\)-child incidence inside
\(Q\).  Thus \(f_B^{Y_1}(x)=f_B^{Y_0}(x)\) on the final inner union, and
(4.3) applies.  For (P4), the unique cell \(Q(x)\) of a parent point
satisfies
\[
 \sum_B\1_{Z(B)}(x)=\#\{B:(B,Q(x))\in E_2\},
\]
which lies in \([d_0,2d_0)\) by (4.6).  Finally,
\[
\begin{aligned}
 \sum_{T\in\T}\1_{Y_1(T)}(x)
 &=\sum_{B:x\in U(\T_B,Y_1)}
       \sum_{T\in\T_B}\1_{Y_1(T)}(x)\\
 &<2\mu_{\rm in}\#\{B:x\in Z(B)\}
 <4\mu_{\rm in}d_0.
\end{aligned}
\]
This proves (P5).  The enlargement statement is \cref{lem:CELL2}.
\end{proof}

\begin{corollary}[Average multiplicity product]
\label{cor:product}
Let \(\B'\) be the retained parent family and put
\[
 \mu_{\rm child}:=\max_{B\in\B'}\mu(\T_B,Y_1).
\]
Then
\[
 \mu(\T,Y)
 \lesssim \mathfrak L\,
 \mu(\B',Z)\,\mu_{\rm child},
\tag{4.12}
\]
where \(\mu(\mathcal F,Y)=M(Y)/|U(\mathcal F,Y)|\).
\end{corollary}

\begin{proof}
Every retained parent \(B\in\B'\) is incident to at least one edge of
\(E_2\).  That edge has positive weight, so
\(U(\T_B,Y_1)\ne\varnothing\); in particular the local average
multiplicity used below has a nonzero denominator.
Integrating (P3)--(P5) gives, uniformly for \(B\in\B'\),
\(\mu(\T_B,Y_1)\sim\mu_{\rm in}\), while
\(\mu(\B',Z)\sim d_0\), and
\(\mu(\T,Y_1)\lesssim\mu_{\rm in}d_0\).
Since \(M(Y_1)\ge\mathfrak L^{-1}M(Y)\) and
\(U(\T,Y_1)\subset U(\T,Y)\),
\(\mu(\T,Y)\le\mathfrak L\mu(\T,Y_1)\).
\end{proof}

\section{Refinement by whole parent--cell incidences}
\label{sec:lift}

Joint factoring would be of limited use if it were destroyed by the next
geometric selection.  The relevant stability statement is therefore a
lifting theorem.  A later step may discard parent incidences, but it must do
so edge by edge.  Equal cell volumes turn retained labelled parent mass into
retained edge count, and the regularized positive weights turn retained edge
count into retained child mass.  No geometric hypothesis beyond this atomic
structure is needed.

\begin{theorem}[Whole-incidence lifting]
\label{thm:lift}
Let \(0<c\le1\), let \(E^\star\subset E_2\), and define
\[
 Z^\star(B)=\bigcup_{Q:(B,Q)\in E^\star}Q.
\tag{5.1}
\]
Suppose the labelled parent mass satisfies
\[
 \sum_B|Z^\star(B)|\ge c\sum_B|Z(B)|.
\tag{5.2}
\]
For \(T\in\T_B\), set
\[
 Y^\star(T)=Y_1(T)\cap
 \bigcup_{Q:(B,Q)\in E^\star}Q.
\tag{5.3}
\]
Then
\[
 M(Y^\star)\ge\frac c2M(Y_1).
\tag{5.4}
\]
The common inner multiplicity, exact cellular compatibility, and the
original pointwise product upper bound survive.
\end{theorem}

\begin{proof}
Each labelled parent--cell edge contributes exactly \(h^3\) to the parent
sum, so (5.2) implies \(|E^\star|\ge c|E_2|\).  Also
\[
 M(Y^\star)=\sum_{e\in E^\star}\omega_e
 \ge\omega_0|E^\star|,
\qquad
 M(Y_1)<2\omega_0|E_2|,
\]
which proves (5.4).  A selected edge keeps every child incidence of its
parent in that cell.  Therefore
\[
 Z^\star(B)=\Ccal_h(U(\T_B,Y^\star))
\]
and the inner multiplicity is unchanged on the new union.  If
\(d^\star(Q)\) is the new parent degree, then
\[
 \sum_T\1_{Y^\star(T)}(x)
 \le2\mu_{\rm in}d^\star(Q(x))
 \le4\mu_{\rm in}d_0.
\]
\end{proof}

\begin{corollary}[Restoring a constant refined parent degree]
\label{cor:rer}
Choose a dyadic right-degree class of \(E^\star\) having maximal edge count,
keep every edge at each cell in that class, let \(Y^{\star\star}\) be the
resulting child shading, and put
\[
 L_\star:=1+\lceil\log_2|E^\star|\rceil.
\]
Then the output again has constant parent degree and the complete comparison
chain is
\[
 M(Y^{\star\star})
 \ge \frac{1}{2L_\star}M(Y^\star)
 \ge \frac{c}{4L_\star}M(Y_1)
 \ge \frac{c}{4L_\star\mathfrak L}M(Y).
\tag{5.5}
\]
Here \(Y\) is the input to \cref{thm:joint}, \(Y_1\) is its output,
\(Y^\star\) is the immediate output of \cref{thm:lift}, and \(c\) is
exactly the parent-mass fraction in (5.2), measured relative to \(Z\).
Thus the lift supplies \(c/2\), the subsequent degree regularization supplies
\((2L_\star)^{-1}\), and the earlier joint construction supplies
\(\mathfrak L^{-1}\); none of these baselines is implicit.
\end{corollary}

\begin{proof}
Pigeonhole the positive values of
\(d^\star(Q)=\#\{B:(B,Q)\in E^\star\}\).  The selected degree class retains
the reciprocal of the \(L_\star\) classes in edge count.  Edge weights remain
in \([\omega_0,2\omega_0)\), so this converts to
\(M(Y^{\star\star})\ge(2L_\star)^{-1}M(Y^\star)\).  Combine
\cref{thm:lift} with (P1) of \cref{thm:joint} to obtain the remaining two
inequalities in (5.5).
\end{proof}

\begin{remark}[Three refinement types]
Deleting whole parent labels, deleting whole right cells, and deleting
individual parent--cell incidences are all instances of
\(E^\star\subset E_2\).  In the third case one usually applies
\cref{cor:rer}.  No operation may delete only selected child labels inside a
retained \((B,Q)\) edge.
\end{remark}

\begin{proposition}[The cellular restriction is sharp]
\label{prop:noncell}
A positive-volume measurable parent refinement that cuts through cells need
not retain any child mass.
\end{proposition}

\begin{proof}
In one cell \(Q\), support all child mass on a measurable half \(Q_+\).
Retain the disjoint half \(Q_-\) as the parent refinement.  It keeps half the
parent volume and zero child mass.
\end{proof}

\begin{remark}[Iteration budget]
For \(m\) successive whole-incidence lifts with parent fractions \(c_j\),
\[
 M(Y^{(m)})\ge2^{-m}\prod_{j=1}^mc_j\,M(Y^{(0)}),
\tag{5.5a}
\]
before any degree re-regularizations.  Thus a fixed number of controlled
lifts is safe.  Arbitrarily many harmless lifts are not: retaining half
the edges at each stage leaves \(2^{-m}\) of the mass.
\end{remark}

\subsection{The geometry-free core}

The proof above only used Euclidean cubes to obtain equal-volume atoms and to
reconstruct a geometric carrier.  The regularization and lifting themselves
see only a finite measurable partition, positive edge weights, and comparable
atom measures.  The following theorem records this geometry-free core and
makes the finite-iteration loss explicit.

\begin{theorem}[Atomic cellular factoring and lifting]
\label{thm:abstract-cellular}
Let \((X,\mathfrak m)\) be a measure space.  Let \(X_0\subset X\) be
measurable, let \(\mathcal Q\) be a finite measurable partition of
\(X_0\), and let
\(\mathcal B\) be a finite nonempty labelled parent set.  For each
\(B\in\mathcal B\), let \(\mathcal T_B\) be a finite labelled child set with measurable shadings
\(Y(T)\subset X_0\), and assume
\[
 0<M(Y):=\sum_B\sum_{T\in\mathcal T_B}\mathfrak m(Y(T))<\infty,
 \qquad N_{\max}:=\max_B|\mathcal T_B|<\infty.
\tag{5.6}
\]
Repeated geometric realizations, if any, are counted by labels.  Put
\[
 f_B(x)=\sum_{T\in\mathcal T_B}\1_{Y(T)}(x).
\]
Then there are
\(k_0\in\{0,\ldots,\lceil\log_2N_{\max}\rceil\}\), the common inner level
\(\mu_{\rm in}=2^{k_0}\), and the restricted shading
\[
 X_{B,k_0}:=\{x:2^{k_0}\le f_B(x)<2^{k_0+1}\},
 \qquad Y_0(T):=Y(T)\cap X_{B,k_0}\quad(T\in\mathcal T_B).
\tag{5.6a}
\]
There are also a positive edge set
\(E_2\subset\mathcal B\times\mathcal Q\), and numbers
\(\omega_0,d_0>0\) with the following properties.  If
\[
 \begin{aligned}
 Y_1(T)&=Y_0(T)\cap\bigcup_{Q:(B,Q)\in E_2}Q
       &&(T\in\mathcal T_B),\\
 Z(B)&=\bigcup_{Q:(B,Q)\in E_2}Q,
 \end{aligned}
\tag{5.7}
\]
then
\[
 M(Y_1)\ge L_{\rm at}^{-1}M(Y),
 \qquad
 L_{\rm at}\le
 4(1+\lceil\log_2N_{\max}\rceil)
 (1+\lceil\log_2(2e_+)\rceil)
 (1+\lceil\log_2e_+\rceil),
\tag{5.8}
\]
where \(e_+\) is the number of positive parent--atom incidences after the
inner-level selection.  Moreover:
\begin{enumerate}[label=\textup{(A\arabic*)},leftmargin=2.8em]
\item \(Y_1(T)\subset Z(B)\) for \(T\in\mathcal T_B\), and

\[
 Z(B)=\bigcup\{Q:\mathfrak m(Q\cap
       U(\mathcal T_B,Y_1))>0\};
\tag{5.9}
\]
\item on \(U(\mathcal T_B,Y_1)\), the \(B\)-child multiplicity lies in
\([\mu_{\rm in},2\mu_{\rm in})\);
\item on \(\bigcup_BZ(B)\), the labelled parent multiplicity lies in
\([d_0,2d_0)\);
\item everywhere, the total child multiplicity is below
\(4\mu_{\rm in}d_0\).
\end{enumerate}

Assume additionally that every atom occurring in \(E_2\) has measure in
\([q_0,\kappa_Qq_0]\), with \(q_0>0\).  Let
\[
 E^{(m)}\subset\cdots\subset E^{(1)}\subset E^{(0)}=E_2
\]
be whole-incidence refinements, and define
\[
 Y^{(j)}(T):=Y_0(T)\cap
   \bigcup_{Q:(B,Q)\in E^{(j)}}Q,
 \qquad
 Z^{(j)}(B):=\bigcup_{Q:(B,Q)\in E^{(j)}}Q.
\tag{5.9a}
\]
Suppose
\[
 \sum_B\mathfrak m(Z^{(j)}(B))
 \ge c_j\sum_B\mathfrak m(Z^{(j-1)}(B)),
 \qquad 0<c_j\le1.
\tag{5.10}
\]
Then
\[
 M(Y^{(m)})\ge
 L_{\rm at}^{-1}\prod_{j=1}^m\frac{c_j}{2\kappa_Q}\,M(Y).
\tag{5.11}
\]
At every stage of (5.9a), the common inner multiplicity, and the original total
multiplicity upper bound remain valid.  A constant refined parent degree can
be restored by one saturated dyadic right-degree selection, exactly as in
\cref{cor:rer}.  No topology, metric, convexity, or equal-volume assumption is
used beyond the stated comparability of atom measures in the lifting part.
\end{theorem}

\begin{proof}
Because \(\mathcal Q\) partitions \(X_0\) and every input shading is
contained in \(X_0\), the atoms cover all of the mass in (5.6).  In
particular, after the inner-level selection,
\(\sum_{B,Q}\omega_{BQ}\) is exactly the retained labelled mass rather than
the mass of only an unspecified part of \(X\).
Select the level \(k_0\) of \(f_B\) simultaneously over all parents, keeping at
least \((1+\lceil\log_2N_{\max}\rceil)^{-1}\) of \(M(Y)\).  For the restricted
shading \(Y_0\) in (5.6a), define
\[
 \omega_{BQ}:=\sum_{T\in\mathcal T_B}
       \mathfrak m(Y_0(T)\cap Q),
 \qquad E^+:=\{(B,Q):\omega_{BQ}>0\}.
\]
Apply \cref{lem:REG1} to these positive weights and then retain a saturated
dyadic right-degree class.  The calculation (4.4)--(4.9) is purely a weighted
bipartite-graph calculation and gives (5.8), with
\(\omega_e\in[\omega_0,2\omega_0)\) on \(E_2\).  Disjointness of the atoms
replaces disjointness of half-open cubes in (4.10).  Keeping every child
incidence of \(B\) in a retained atom proves (5.9) and preserves the selected
inner level.  The right-degree class proves (A3), and summing at most
\(2\mu_{\rm in}\) children over fewer than \(2d_0\) retained parents proves
(A4).

For the lifting assertion, (5.10) and atom comparability imply
\[
 |E^{(j)}|\ge(c_j/\kappa_Q)|E^{(j-1)}|.
\]
Since every surviving edge still has its original weight in
\([\omega_0,2\omega_0)\),
\[
 M(Y^{(j)})\ge\omega_0|E^{(j)}|
 \ge\frac{c_j}{2\kappa_Q}M(Y^{(j-1)}).
\]
Iteration and (5.8) give (5.11).  All remaining assertions follow because an
edge is always kept or deleted together with every child incidence it
represents.
\end{proof}

\section{Rectangular density estimates from planar incidences}
\label{sec:density}

The graph mechanism controls survival and compatibility, not the size of the
parent union.  The latter is an analytic question.  We obtain the required
rectangular density exponent by projecting the children to a planar line
parameter space, applying a weighted form of the incidence estimate in
\cite{DOV2022}, and reconstructing volume inside the parent.  Appendix~A
proves the weighted formulation needed for repeated parameters and arbitrary
measurable shadings.

\subsection{Weighted planar input}

Let \(1<t\le2\), set \(q=3-t\), and let \(0<h\le1\) and \(L\ge1\).  In a bounded
slope--intercept chart, suppose
\[
 \nu(B(p,\rho))\le L\rho^t,\qquad h\le\rho\le1,
\tag{6.1}
\]
for an atomic probability measure
\(\nu=\sum_i\alpha_i\delta_{p_i}\).  Repetitions among the \(p_i\) are
allowed.  If \(Y_i\) is an arbitrary measurable shading of the corresponding
\(h\)-tube and
\[
 \Lambda=h^{-1}\sum_i\alpha_i|Y_i|,
\]
then the weighted layer-cake corollary of DOV is
\[
 \left|\bigcup_iY_i\right|
 \gtrsim_{\eps,t}L^{-q}h^\eps\Lambda^q.
\tag{6.2}
\]
The proof of (6.2) in \cref{app:weighted} uses two layer-cake reductions:
the first removes atomic weights and the second removes spatial occupancy
weights.  The logarithmic costs are absorbed into \(h^\eps\).  In
particular, neither a minimum atom mass nor a bound on projection-fibre size
is part of the input.

\subsection{From planar parameters back to a rectangular parent}

Let \(B\) be a labelled \(a\times b\times r\) parent,
\(\delta\le a\le b\le r\), containing \(N\ge1\) labelled
\(\delta\times\delta\times r\) children.  Suppose that for every convex
\(K\subset B\),
\[
 \frac{\#\{T:T\subset K\}}{N}
 \le C_0\left(\frac{|K|}{|B|}\right)^{1+\sigma},
\qquad 0<\sigma\le1.
\tag{6.3}
\]
Write \(s=a/b\), \(q=2-\sigma\), and let
\[
 \lambda_B=\frac{\sum_{T\in\T_B}|Y(T)|}
                  {\sum_{T\in\T_B}|T|}.
\]

\begin{theorem}[Rectangular biased shading bound]
\label{thm:rect}
For every \(\eps>0\),
\[
 |B\cap N_a(U(\T_B,Y))|
 \gtrsim_{\eps,\sigma}
 C_0^{-q}s^\eps\lambda_B^q|B|.
\tag{6.4}
\]
For every \(\eta>0\), before epsilon absorption the scale factor may be
written (with an implicit constant also depending on \(\eta\))
\[
 C_0^{-q}[\log(2/s)]^{-q}s^{q\eta}\lambda_B^q.
\tag{6.5}
\]
\end{theorem}

The proof has five geometric bridges.  We choose finitely many normalized
line charts, convert parameter balls to convex slabs, transfer the
complete-child hypothesis to a Frostman estimate on parameters, thicken the
projected shading without losing boundary mass, and finally reconstruct
three-dimensional volume inside the parent.  The intervening capped-carrier
lemmas ensure that the same charts remain valid across the coarse ancestry
used later.  Once the structural tube and parent conventions are fixed, all
constants are uniform.

\begin{lemma}[Finite line-parameter charts]
\label{lem:finitecharts}
After a rigid motion, write the side coordinates of \(B\) as
\((x_1,x_2,x_3)\), with side lengths \(a\le b\le r\).  Let
\(0<w\le a\), and let the children be labelled
\(w\times w\times r\) tubes contained in \(B\).  They are the disjoint union
of at most three labelled subfamilies.  On each nonempty
subfamily there are distinct indices \(j,k,\ell\) such that the carrier is a
bounded graph over \(x_j\).  The capped carrier length is in
\([c_{\rm len}r,C_{\rm len}r]\), with fixed structural constants, and the
graph contains a relevant \(x_j\)-interval of length at least \(c r\) and at
most \(Cr\).  If a cap meets a parent face, this interval is chosen
one-sided from that cap.  More precisely, the relevant interval is the
\(x_j\)-projection of a carrier subsegment fixed inside one structural
enlargement of \(B\); the enlargement and the constants \(c,C\) are the
same for all three charts and all labels.  Every resulting affine graph line
is extended to one common fixed bounded normalized \(u\)-interval; the
physical carrier and its shading use only their relevant subinterval.
Moreover,
\[
 s_j\asymp r,\qquad s_k\asymp b,\qquad s_\ell\asymp a.
\]
Under
\[
 S_{jk}(x_k,x_j)=\left(x_k/s_k,x_j/s_j\right)=:(v,u),
\]
the projected carrier has equation \(v=\beta_T+m_Tu\), with
\(p(T)=(\beta_T,m_T)\) in one fixed bounded slope--intercept chart.  Its
normalized projected solid lies in a fixed-constant \(h_0\)-tube about that
line, where
\[
 h_0=w/s_k,\qquad h=a/s_k\asymp a/b,\qquad h_0\le h\le1.
\]
\end{lemma}

\begin{proof}
Let \(d(T)=\sum_i d_i e_i\) be the carrier direction and assign \(T\) to the
least index \(j\) for which \(|d_j|=\max_i|d_i|\).  Then
\(|d_j|\ge3^{-1/2}\).  By the capped-tube convention, the carrier has length
comparable to \(r\) and lies in the tube; hence a segment of length at least
\(cr\) lies in a fixed enlargement of \(B\).  If a cap meets a face of
\(B\), take this segment one-sided from that endcap.  Projecting its two
endpoints to the \(i\)-th
side gives
\[
 cr|d_i|\le C s_i+Cw\le C's_i,
\tag{6.3c}
\]
because \(w\le a\le s_i\).  Hence \(s_j\ge c'r\), while \(s_j\le r\), so
\(s_j\asymp r\).  If \(j=3\), choose \(k=2,\ell=1\).  If \(j=2\), then
\(b=s_j\asymp r\), and choose \(k=3,\ell=1\).  If \(j=1\), then
\(a=s_j\asymp r\), hence all three side lengths are comparable and either
remaining order has the asserted sizes.  This proves uniformly that
\(s_j\asymp r,s_k\asymp b,s_\ell\asymp a\); it is the coordinate
permutation omitted by the informal dominant-direction argument.

Since \(|d_j|\ge3^{-1/2}\), the carrier is a graph over the full relevant
\(x_j\)-interval.  The normalized slope
\[
 m_T=\frac{s_jd_k}{s_kd_j}
\]
is bounded by (6.3c), and the normalized intercept is bounded because the
carrier meets \(B\).  Normalized endpoint ordinates and
\((\beta_T,m_T)\) differ by a fixed invertible linear map.  The orthogonal
solid radius contributes \(O(w/s_k)\) after projection.  A cap is contained
in an \(O(w)\)-ball about its carrier endpoint and contributes the same
amount, so neither changes the chart nor its width.  Labels are never merged.  The
original fine-child application is \(w=\delta\); the coarse application
below is \(w=\rho\).
\end{proof}

\begin{lemma}[A parameter ball lifts to one convex slab]
\label{lem:paramslab}
Fix a width \(w\) and a chart from \cref{lem:finitecharts}.  For
\(h\le u\le1\) and every parameter ball \(D(p_0,u)\), there is a convex set
\(K_u\subset B\)
which contains every complete child in the chart whose parameter lies in
that ball, and
\[
 \frac{|K_u|}{|B|}\lesssim u.
\]
\end{lemma}

\begin{proof}
Let \(v=L_0(u)\) represent \(p_0\).  In physical coordinates take
\[
 K_u=B\cap
 \left\{x:\left|x_k-s_kL_0(x_j/s_j)\right|
             \le C(us_k+w)\right\}.
\]
This is the intersection of a box with two affine half-spaces.  Bounded
slope--intercept distance controls the extended graph on the common normalized
interval fixed in \cref{lem:finitecharts}, hence also on every relevant
carrier subinterval; \(Cw\) includes the solid radius and caps.  Thus
\[
 |K_u|\lesssim s_\ell s_j(us_k+w).
\]
Because \(w/s_k\le h\le u\) and
\(|B|\asymp s_\ell s_js_k\), the result follows.  The restriction is one
slab, so the relative volume is \(O(u)\), not \(O(u^2)\).
\end{proof}

\begin{lemma}[Weighted marginal Frostman estimate]
\label{lem:marginalfrostman}
Let a nonempty chart contain \(N_c=\vartheta N\) labelled children, counted
with repetitions, and put
\[
 \nu_c=\frac1{N_c}\sum_{T\ {\rm in\ the\ chart}}\delta_{p(T)}.
\]
Then, for \(h\le\rho\le1\),
\[
 \nu_c(D(p,\rho))
 \lesssim C_0\vartheta^{-1}\rho^{1+\sigma}.
\]
If instead the parent family obeys the ordinary complete-child condition
\[
 \frac{\#\{T:T\subset K\}}{N}
 \le C_F\frac{|K|}{|B|}
 \quad(K\subset B\ \hbox{convex}),
\tag{6.3a}
\]
then the same argument gives
\[
 \nu_c(D(p,\rho))
 \lesssim C_F\vartheta^{-1}\rho.
\tag{6.3b}
\]
No bound on the number of labels with the same parameter is required.
\end{lemma}

\begin{proof}
Every label counted on the left is a complete child of \(K_\rho\).  Hence
(6.3) and \cref{lem:paramslab} give
\[
 \nu_c(D(p,\rho))
 \le \frac{N}{N_c}C_0
       \left(\frac{|K_\rho|}{|B|}\right)^{1+\sigma}
 \lesssim C_0\vartheta^{-1}\rho^{1+\sigma}.
\]
This estimates total atomic weight and is unchanged by repeated parameters
or unbounded projection fibres.  Under (6.3a), replace the power
\(1+\sigma\) in the same calculation by one to obtain (6.3b).
\end{proof}

\begin{lemma}[Clipped one-dimensional thickening]
\label{lem:clipped1d}
Let \(I\) be an interval with \(c_I\le |I|\le C_I\),
\(0<h_0\le h\le1\), and let \(E\subset I\) be measurable with
\(\diam(E)\le C_0h_0\).  Fix \(0<c<c_I/4\).  Then
\[
 |(E+[-ch,ch])\cap I|
 \ge c_1(c,C_0)\frac h{h_0}|E|.
\tag{6.3d}
\]
The assertion remains true for almost every section of a Borel measurable
planar representative.
\end{lemma}

\begin{proof}
The statement is trivial when \(|E|=0\).  Otherwise choose a density point
\(x\in E\).  Since \(ch\le c<c_I/4\le |I|/4\), at least one of the two
one-sided intervals of length \(ch\) from \(x\) lies in \(I\).  Thus the left
side is at least \(ch\).  On the other hand
\(|E|\le\diam(E)\le C_0h_0\), so
\(ch\ge(c/C_0)(h/h_0)|E|\).  Choose a Borel representative of a planar
measurable set.  Fubini gives measurable sections for almost every transverse
parameter; applying the one-dimensional estimate to those sections and then
Cavalieri proves the final assertion.  Completion of the measure handles a
general Lebesgue-measurable representative.
\end{proof}

\begin{lemma}[Projection of an arbitrary measurable shading]
\label{lem:projectedmass}
For a child \(T\) in a fixed chart, let
\[
 A_T=S_{jk}(\pi_{jk}Y(T)),
\]
where \(\pi_{jk}\) forgets \(x_\ell\).  There is a measurable shading
\(Z_T\) of a fixed-constant \(h\)-tube about \(p(T)\), contained in the
normalized parent projection \(R_{jk}:=S_{jk}(\pi_{jk}B)\), such that
\[
 \frac{|Z_T|}{h}\gtrsim\frac{|Y(T)|}{|T|}.
\]
Consequently, if a chart carries at least \(J^{-1}\) of the total shaded
mass and \(N_c=\vartheta N\), then
\[
 h^{-1}N_c^{-1}\sum_T|Z_T|
 \gtrsim \frac{\lambda_B}{J\vartheta}.
\]
\end{lemma}

\begin{proof}
The normalized parent projection \(R_{jk}\) is a fixed-comparable rectangle,
so its transverse sections have one structural length
\(c_I\le |I|\le C_I\).  Fix once and for all a structural
\(c\in(0,c_I/4)\), sufficiently small for
the fibre reconstruction in \cref{lem:fibrereconstruct}; every \(Z_T\) below
uses this same \(c\).
If \(|Y(T)|=0\), take \(Z_T=\varnothing\).  Otherwise, after replacing the
shading by a Borel representative modulo a null set, inner regularity gives a
compact \(K_T\subset Y(T)\) with \(|K_T|\ge\tfrac12|Y(T)|\).  We perform the
construction on \(K_T\).  Its projection is compact, its transverse
Minkowski sum is compact, and its intersection with the parent rectangle is
Borel; hence the resulting \(Z_T\) is measurable.  The factor \(1/2\) is
absorbed in the structural implicit constant.  Since \(|d_j|\gtrsim1\), every
fibre of \(\pi_{jk}\) meets \(T\) in an \(x_\ell\)-interval of length
\(O(\delta)\).  Fubini and the planar Jacobian give
\[
 |A_T|\gtrsim\frac{|Y(T)|}{\delta s_js_k}.
\]
Every \(u\)-slice of \(A_T\) has diameter \(O(h_0)\).  Define the
\emph{parent-clipped} transverse thickening
\[
 Z_T=\bigl(A_T+[-ch,ch]e_v\bigr)\cap R_{jk}.
\]
By \cref{lem:clipped1d}, including slices meeting either transverse face,
and \(h/h_0=a/\delta\),
\[
 |Z_T|\gtrsim(h/h_0)|A_T|,
 \qquad
 \frac{|Z_T|}{h}
 \gtrsim\frac{|Y(T)|}{\delta^2s_j}
 \asymp\frac{|Y(T)|}{|T|}.
\]
Summing proves the second assertion.  The clipping is essential for the
in-parent fibre reconstruction below.
\end{proof}

\begin{lemma}[Two-point comparison for capped carriers]
\label{lem:capped-two-point}
Let \(T=N_{c_T w}(I_T)\) and \(S=N_{c_S\rho}(I_S)\) be capped solid
tubes whose carrier lengths lie in fixed intervals comparable to \(r\).
Assume \(T\subset S\) and \(0<w\le\rho\le c_{\rm sep}r\), where
\(c_{\rm sep}>0\) is sufficiently small in terms of the fixed tube
constants.  Then there are \(x_1,x_2\in I_T\) and
\(y_1,y_2\in I_S\) such that
\[
 |x_1-x_2|\ge cr,\qquad |x_i-y_i|\le C\rho,\qquad
 |y_1-y_2|\ge c'r.
\tag{6.6C}
\]
Consequently the unoriented carrier directions differ by at most
\(C\rho/r\), and a point of \(I_T\) has distance \(O(\rho)\) from
\(I_S\).  The assertion includes containment through either endcap.
\end{lemma}

\begin{proof}
Choose two carrier points of \(I_T\), one-sided if necessary, separated by
a fixed fraction of its length.  Since every carrier point of \(T\) belongs
to the solid capped tube \(S\), for each \(x_i\) there is a nearest
\(y_i\in I_S\) with \(|x_i-y_i|\le C\rho\).  The triangle inequality gives
\[
 |y_1-y_2|
 \ge |x_1-x_2|-|x_1-y_1|-|x_2-y_2|
 \ge cr-2C\rho\ge c'r.
\]
Thus \(y_2-y_1\) is a nonzero multiple of the coarse carrier direction,
whereas \(x_2-x_1\) is a nonzero multiple of the fine direction.  Comparing
the two normalized secants gives the \(O(\rho/r)\) angular error.  If a
nearest point is a carrier endpoint, the same calculation is one-sided and
unchanged; no continuation beyond a cap is used.
\end{proof}

\begin{lemma}[Hausdorff control for standardized capped ancestors]
\label{lem:standard-capped-ancestry}
Let
\[
 T=N_{c_Tw}(I),\qquad S=N_{c_S\rho}(J),
 \qquad 0<w\le\rho\le L,
\]
where the two carrier segments have the same length
\(|I|=|J|=L\), and suppose \(T\subset S\).  Then
\[
 d_{\rm H}(I,J)\le C\rho,
 \qquad S\subset N_{C\rho}(T).
\tag{6.6D}
\]
The constant depends only on the fixed cap constants.  In particular the
conclusion applies to the source-standard fine/coarse ancestry in
\cref{prop:s6B}; it is not asserted for arbitrary comparable-length local
tubelets.
\end{lemma}

\begin{proof}
Let \(p_-,p_+\) be the endpoints of \(I\).  Containment of the carrier in the
solid \(S\) gives points \(q_-,q_+\in J\) with
\(|p_\pm-q_\pm|\le C_0\rho\).  If
\(\rho\le L/(4C_0)\), then
\[
 |q_+-q_-|\ge L-2C_0\rho.
\]
Writing \(J=\{q_0+tu:0\le t\le L\}\), the two parameters of \(q_-\) and
\(q_+\) therefore differ by at least \(L-2C_0\rho\).  After possibly reversing
the order on \(J\), each endpoint of \(J\) is within \(2C_0\rho\) of the
corresponding \(q_\pm\), and hence within \(3C_0\rho\) of the corresponding
endpoint of \(I\).  Linear interpolation between the paired endpoints gives
both inclusions \(I\subset N_{C\rho}(J)\) and
\(J\subset N_{C\rho}(I)\).  If \(\rho>L/(4C_0)\), choose any point of \(I\)
and a nearest point of \(J\); the diameter bound \(L\lesssim\rho\) gives the
same two inclusions with a larger structural constant.  Finally every point
of \(S\) is within \(c_S\rho\) of \(J\), while \(I\subset T\), which proves
the second assertion.  Endpoint caps are already included in the two
neighbourhoods, so no extension of either carrier is used.
\end{proof}

\begin{lemma}[Proxy-chart inheritance]
\label{lem:proxy-chart}
There is an absolute constant
\(0<c_{\rm dom}\le c_{\rm sep}\), determined only by the finite charts in
\cref{lem:finitecharts} and the fixed cap/length constants in the carrier
convention, with the following property.  Let
\(\delta\le\rho\le a\le b\le r\), assume
\[
 \rho/r\le c_{\rm dom},
\tag{6.6c$'''$}
\]
 and let a fine
 \(\delta\times\delta\times r\) tube \(T\) be contained in a coarse
\(\rho\times\rho\times r\) tube \(S\), and \(S\) belongs to one of the
charts in \cref{lem:finitecharts}, then, after a fixed enlargement of that
chart, \(T\) is a graph in the same dominant coordinate and
\[
 |p(T)-p(S)|\lesssim \rho/b\le a/b.
\tag{6.6c$''$}
\]
Consequently every projected measurable shading of \(T\) constructed in
\cref{lem:projectedmass} is contained in a fixed enlargement of the
\(a/b\)-tube belonging to the proxy atom \(p(S)\).
\end{lemma}

\begin{proof}
\Cref{lem:capped-two-point} gives directional separation
\(O(\rho/r)\) and transverse offset \(O(\rho)\), including both endcaps.
In the normalized coordinates of
\cref{lem:finitecharts}, the slope error is
\(O((r/b)(\rho/r))=O(\rho/b)\), and the intercept error is also
\(O(\rho/b)\).  The coarse dominant component is bounded below by a fixed
chart constant.  Choose \(c_{\rm dom}\) so that the directional error
\(C\rho/r\) is at most half that lower bound.  The fine dominant component
is then bounded below by half the same constant, so the same graph
coordinate applies.  Finally
\(\rho/b\le a/b\), and the projected fine shading already lies in an
\(O(a/b)\)-thickening of its fine carrier.  The triangle inequality proves
the last assertion.
\end{proof}

\begin{lemma}[Fibre reconstruction inside the parent]
\label{lem:fibrereconstruct}
For the projected shadings in \cref{lem:projectedmass},
\[
 |B\cap N_a(\bigcup_TY(T))|
 \gtrsim |B|\,\left|\bigcup_TZ_T\right|,
\]
where planar measure is in normalized coordinates.  The constant is uniform
over all charts and over shadings meeting a face of \(B\).
\end{lemma}

\begin{proof}
If \((v',u)\in Z_T\), then for some
\((x_\ell,x_k,x_j)\in Y(T)\), the corresponding physical observed
coordinate satisfies \(|x_k'-x_k|\le ca\).  The clipping keeps \(x_k'\) in
the parent projection.  Inside the \(x_\ell\)-side of \(B\), the interval
\[
 \{x_\ell':|x_\ell'-x_\ell|\le ca\}
\]
has length at least \(ca\), even at a face.  Choosing \(c\) small puts every
\((x_\ell',x_k',x_j)\) in \(B\cap N_a(Y(T))\).  Therefore Fubini and
\(|\det S_{jk}^{-1}|=s_js_k\) yield
\[
 |B\cap N_a(\bigcup_TY(T))|
 \gtrsim a s_js_k\left|\bigcup_TZ_T\right|
 \asymp |B|\left|\bigcup_TZ_T\right|.
\]
\end{proof}

\begin{proof}[Proof of \cref{thm:rect}]
The calculation now follows the projection--incidence--reconstruction chain.
If \(\lambda_B=0\), the right side of (6.4) vanishes and the conclusion is
immediate.  Hence assume \(\lambda_B>0\).
Choose a chart carrying at least \(J^{-1}\) of the total shaded mass, where
\(J\le3\), and write \(N_c=\vartheta N\).
\Cref{lem:marginalfrostman,lem:projectedmass} give Frostman constant
\(L\lesssim C_0/\vartheta\) and planar density at least
\(c\lambda_B/(J\vartheta)\).  Apply (6.2) with
\(t=1+\sigma\), \(q=2-\sigma\), and \(h\asymp a/b\).  The chart fraction
cancels exactly:
\[
 \left|\bigcup_TZ_T\right|
 \gtrsim_{\eps,\sigma}
 (C_0/\vartheta)^{-q}h^\eps
 (\lambda_B/(J\vartheta))^q
 \gtrsim C_0^{-q}h^\eps\lambda_B^q.
\]
\Cref{lem:fibrereconstruct} now gives (6.4).  In the unabsorbed planar
estimate the right side is
\[
 C_0^{-q}[\log(2/h)]^{-q}h^{q\eta}\lambda_B^q.
\]
This is (6.5), since \(h\asymp s=a/b\).  Taking
\(\eta=\eps/(2q)\) and using
\[
 [\log(2/s)]^{-q}s^{\eps/2}
 \gtrsim_{\eps,q}s^\eps
\]
recovers (6.4).  The compact approximation is removed by letting its mass
loss tend to zero.  All counts remain labelled, so repeated children and
repeated projected parameters are included.
\end{proof}

\begin{corollary}[Ordinary rectangular endpoint]
\label{cor:ordinaryrect}
Assume (6.3a).  Then
\[
 |B\cap N_a(U(\T_B,Y))|
 \gtrsim C_F^{-1}[\log(2/s)]^{-1}\lambda_B^2|B|.
\tag{6.6}
\]
This statement allows labelled repetitions, coincident projected parameters,
and arbitrary measurable shadings.
\end{corollary}

\begin{proof}
Repeat the chart selection, projection, parent clipping, and fibre
reconstruction in the proof of \cref{thm:rect}.  The chart marginal now
satisfies (6.3b).  Apply the weighted labelled C\'ordoba estimate
\cref{thm:cordobaappendix} with
\(L\lesssim C_F\vartheta^{-1}\) and planar density
\(\Lambda\gtrsim\lambda_B/(J\vartheta)\).  Its lower bound is
\[
 (C_F\vartheta^{-1})^{-1}[\log(2/h)]^{-1}
   (\lambda_B/(J\vartheta))^2
 =C_F^{-1}[\log(2/h)]^{-1}
   \lambda_B^2(J^2\vartheta)^{-1}.
\]
Since \(0<\vartheta\le1\) and \(J\le3\),
\((J^2\vartheta)^{-1}\ge J^{-2}\gtrsim1\).  Notice that the chart was
selected by shaded mass, so no lower bound for its fraction of labels is
being used.
The parent-clipped fibre reconstruction contributes only the fixed chart
Jacobian, and \(h\asymp s\).  This proves (6.6).  In particular, the endpoint
is not obtained by formally setting \(t=1\) in DOV.
\end{proof}

\begin{lemma}[Coarse/thickened-child C\'ordoba endpoint]
\label{lem:coarse-cordoba}
Let \(\delta\le\rho\le a\le b\le r\), and let \(B\) be an
\(a\times b\times r\) labelled rectangular parent.  Suppose a nonempty
labelled family \(\Scal_B\) of \(\rho\times\rho\times r\) coarse tubes is
contained in \(B\), and a nonempty fine family \(\T_B\) is given, with each
fine child \(T\in\T_B\) assigned to one
\(S(T)\in\Scal_B\) with \(T\subset S(T)\).
The fine and coarse carriers use one common capped length normalization at
scale \(r\): their carrier lengths lie in the same fixed interval
\([c_{\rm len}r,C_{\rm len}r]\), with the same endpoint convention.  Thus
all ancestry comparisons below use one declared length window rather than
unrelated ``comparable length'' constants.  Suppose there are
\(m\ge1\) and \(R_{\rm br}\ge1\) such that every coarse label \(S\) has
\(n_S\) assigned fine labels with
\[
 m\le n_S\le R_{\rm br}m.
\tag{6.6c$'$}
\]
Assume
\[
 \frac{\#\{S\in\Scal_B:S\subset K\}}{|\Scal_B|}
 \le C_{F,\rho}\frac{|K|}{|B|}
 \qquad(K\subset B\ \hbox{convex}).
\tag{6.6c}
\]
Then arbitrary measurable fine shadings, repeated fine labels, and coincident
fine carriers satisfy
\[
 |B\cap N_a(U(\T_B,Y))|
 \gtrsim C_{F,\rho}^{-1}R_{\rm br}^{-3}
          [\log(2b/a)]^{-1}
          \lambda_B^2|B|.
\tag{6.6d}
\]
The conclusion asserts no fine-scale all-convex Frostman condition.
\end{lemma}

\begin{proof}
Put \(h=a/b\), and fix a sufficiently small absolute \(h_*>0\).  If
\(h\ge h_*\), partition the fine labels themselves by
\cref{lem:finitecharts} with \(w=\delta\), and choose a chart carrying at
least one third of the fine shaded mass.  For the normalized equal-weight
measure on that chart,
\(\nu(D(p,u))\le1\le h_*^{-1}u\) for \(h\le u\le1\).
The weighted labelled C\'ordoba endpoint
\cref{thm:cordobaappendix} and fine fibre reconstruction give
\(|B\cap N_a(U)|\gtrsim\lambda_B^2|B|\).  Since
\(C_{F,\rho},R_{\rm br}\ge1\) and \(\log(2b/a)=O_{h_*}(1)\), this is stronger
than (6.6d).

Assume henceforth that \(h<h_*\), where at the outset we also choose
\(h_*\le c_{\rm dom}\).  Since \(\rho\le a\) and \(b\le r\),
\[
 \rho/r\le a/b=h<h_*\le c_{\rm dom}.
\]
Thus \cref{lem:proxy-chart} is applicable.  Partition the coarse labels into
the at most three graph charts of \cref{lem:finitecharts} with \(w=\rho\), assigning
all labelled descendants of a coarse label to its chart.  Choose a chart
carrying at least one third of the total fine shaded mass.  If it contains
\(N_c\) coarse labels and \(N_f=\sum_Sn_S\) fine labels, put equal mass
\(1/N_f\) on each descendant label, but place that atom at the parameter
\(p(S(T))\) of its coarse proxy.

For \(h\le u\le1\), the parameter-slab construction
\cref{lem:paramslab} with transverse padding \(C(ub+\rho)\) contains every
coarse proxy whose parameter lies in a \(u\)-ball.  Its relative volume is
\(O(u)\), because \(\rho\le a=hb\le ub\).  Applying (6.6c), and then dividing
by the coarse chart fraction, proves the one-dimensional Frostman estimate
for the repeated proxy measure with constant
\[
 L\lesssim C_{F,\rho}R_{\rm br}\vartheta^{-1},
\tag{6.6e}
\]
where \(\vartheta=N_c/|\Scal_B|\).  Indeed a coarse label has proxy mass
\(n_S/N_f\le R_{\rm br}/N_c\).  This is the only use of branching
comparability.

By \cref{lem:proxy-chart}, the projected shading \(Z_T\) constructed in
\cref{lem:projectedmass} is a shading of a fixed enlargement of the
\(h\)-tube belonging to the proxy atom \(p(S(T))\).  This verifies the chart
and carrier inheritance for every assigned fine label; no fine shading mass
is merged or replaced.  Since \(N_f\le R_{\rm br}mN_c\), whereas the full
family has at least \(m|\Scal_B|\) descendants, selecting the chart by shaded
mass gives planar density
\[
 \Lambda\gtrsim\lambda_B/(R_{\rm br}\vartheta).
\tag{6.6f}
\]

Apply the weighted labelled C\'ordoba endpoint
\cref{thm:cordobaappendix} to these repeated proxy atoms.  Combining
(6.6e)--(6.6f) gives
\[
 L^{-1}\Lambda^2
 \gtrsim C_{F,\rho}^{-1}R_{\rm br}^{-3}
           \vartheta^{-1}\lambda_B^2
 \ge C_{F,\rho}^{-1}R_{\rm br}^{-3}\lambda_B^2.
\]
Finally,
\cref{lem:fibrereconstruct} applies to the original fine projected shadings,
so parent clipping and boundary faces cost only a fixed constant.  This gives
(6.6d).
\end{proof}

\subsection{Density after joint regularization}

The rectangular estimate is parentwise, whereas the joint graph retains mass
globally across parents.  The next averaging lemma is the bridge: comparable
child denominators convert the global retained-mass statement into a lower
bound for the mean surviving density, after which Jensen supplies the parent
union estimate.

\begin{lemma}[Denominator-comparability averaging]
\label{lem:denominator-average}
Let \(\B\) be a finite nonempty parent family, let \(D_B>0\), let
\(m_B\ge0\) with \(\sum_Bm_B>0\), and suppose
\[
 \frac{\max_{B\in\B}D_B}{\min_{B\in\B}D_B}\le\kappa_D.
\tag{6.6a}
\]
Put \(\lambda=(\sum_Bm_B)/(\sum_BD_B)\).  If
\(\B'\subset\B\) and \(0\le m'_B\le m_B\) satisfy
\[
 \sum_{B\in\B'}m'_B\ge\mathfrak L^{-1}\sum_{B\in\B}m_B,
\]
then
\[
 \frac1{|\B'|}\sum_{B\in\B'}\frac{m'_B}{D_B}
 \ge \kappa_D^{-1}\mathfrak L^{-1}\lambda.
\tag{6.6b}
\]
\end{lemma}

\begin{proof}
Write \(D_- =\min_BD_B\) and \(D_+=\max_BD_B\).  Since
\(|\B'|\le|\B|\),
\[
 \frac1{|\B'|}\sum_{B\in\B'}\frac{m'_B}{D_B}
 \ge \frac{\mathfrak L^{-1}\sum_Bm_B}{|\B|D_+}
 =\mathfrak L^{-1}\lambda
   \frac{\sum_BD_B}{|\B|D_+}
 \ge\mathfrak L^{-1}\lambda\frac{D_-}{D_+}.
\]
This is (6.6b).  Notice that the symmetric normalization
\(\kappa_D^{-1}D\le D_B\le\kappa_DD\) would give only
\(D_+/D_-\le\kappa_D^2\); it is not used here.
\end{proof}

Apply \cref{thm:rect} to the original labelled child family, assigning empty
shadings to discarded labels.  In the ordinary alternative use
\cref{cor:ordinaryrect}; in the coarse/thickened-child alternative use
\cref{lem:coarse-cordoba}.  By \cref{lem:CELL1} and (4.11), in the biased
case
\[
 |Z(B)|\gtrsim
 C_0^{-q}s^\eps\lambda_B^q|B|.
\tag{6.7}
\]
Suppose the original child denominators
\[
 D_B=\sum_{T\in\T_B}|T|
\]
satisfy the ratio condition (6.6a).  Apply
\cref{lem:denominator-average} with the P1 masses; this gives
\[
 \frac1{|\B'|}\sum_{B\in\B'}\lambda_B
 \ge \kappa_D^{-1}\mathfrak L^{-1}
       \lambda(\T,Y).
\tag{6.8}
\]
Jensen now proves the global estimate.

\begin{corollary}[Dimension-sensitive parent density]
\label{cor:density}
For \(0<\sigma\le1\),
\[
 \lambda(\widetilde\B',Z)
 \gtrsim
 C_0^{-(2-\sigma)}s^\eps
 \kappa_D^{-(2-\sigma)}\mathfrak L^{-(2-\sigma)}
 \lambda(\T,Y)^{2-\sigma}.
\tag{6.9}
\]
In the ordinary Frostman endpoint, the corresponding bound is
\[
 \lambda(\widetilde\B',Z)
 \gtrsim
 C_F^{-1}[\log(2/s)]^{-1}
 \kappa_D^{-2}\mathfrak L^{-2}\lambda(\T,Y)^2.
\tag{6.10}
\]
In the coarse/thickened-child alternative, replace \(C_F\) in (6.10) by
\(C_{F,\rho}R_{\rm br}^3\).
\end{corollary}

\begin{proof}
Parent volumes are comparable and
\(|\widetilde B|\asymp|B|\).  Apply Jensen to (6.7) using
(6.8).  For (6.10), apply the quadratic local input
(6.6), or (6.6d) in the coarse/thickened-child alternative; every mass loss
is then squared, while branching comparability contributes the explicit
\(R_{\rm br}^{-3}\) from that lemma.
\end{proof}

\section{Cellwise angular refinement}
\label{sec:angle}

Angular refinement must preserve the object type created in Sections~3--5.
The ordinary typical-angle descent is pointwise, but the lifting theorem acts
on whole parent--cell incidences.  We therefore freeze the incident parent
family on each half-open cell, run the descent once on that finite family,
and use deterministic tie-breaking.  The result is a genuine edge subset,
not a collection of unrelated pointwise choices.

Set
\[
 \alpha=a/r,\qquad q_{\rm res}=a/b.
\]
For parent labels \(B,B'\), define
\[
 \ang_{q_{\rm res}}(B,B')
 =\max\{q_{\rm res},\ang(TB,TB')\}
\tag{7.1}
\]
and for a nonempty labelled family \(\mathcal F\),
\[
 M_q(\mathcal F)=
 \max_{B,B'\in\mathcal F}\ang_{q_{\rm res}}(B,B').
\tag{7.2}
\]
Thus a singleton has resolved diameter \(q_{\rm res}\).

Put
\[
 L_0=\log\alpha^{-1},\qquad L_+=\max\{1,L_0\},
\]
and take
\[
 A_{\rm ang}=e^{L_+^{1/2}},
\qquad B_{\rm ang}=e^{L_+^{3/4}}.
\tag{7.3}
\]
Thus \(A_{\rm ang},B_{\rm ang}>1\) on the full range \(0<\alpha\le1\), including
the endpoint \(\alpha=1\) covered by the definition of \(L_+\).
At a retained right cell \(Q\), start from
\(\mathcal F_0(Q)=\{B:(B,Q)\in E_2\}\).  Recursively, while the indicated
collection is nonempty, define
\[
 \mathcal F_{j+1}(Q)=\operatorname{first}\left\{
 \mathcal H\subset\mathcal F_j(Q):
 |\mathcal H|\ge A_{\rm ang}^{-1}|\mathcal F_j(Q)|,
 \ M_q(\mathcal H)\le B_{\rm ang}^{-1}M_q(\mathcal F_j(Q))
 \right\}.
\tag{7.3a}
\]
Here ``first'' uses a fixed total order on the finite power set.  At
\(j=0\), the threshold in (7.3a) is the relative quantity
\(B_{\rm ang}^{-1}M_q(\mathcal F_0(Q))\).  Stop at the first \(J\) for which
the displayed set is empty and put \(\mathcal F(Q)=\mathcal F_J(Q)\) and
\(\theta(Q)=M_q(\mathcal F_J(Q))\).

\begin{lemma}[Cellwise descent]
\label{lem:descent}
The procedure stops after at most
\[
 K=1+\lceil\log_{B_{\rm ang}}(q_{\rm res}^{-1})\rceil
 \le2+L_+^{1/4}
\tag{7.4}
\]
successful selections, and
\[
 |\mathcal F(Q)|\ge A_{\rm ang}^{-K}|\mathcal F_0(Q)|.
\tag{7.5}
\]
If \(\mathcal H\subset\mathcal F(Q)\) and
\(|\mathcal H|\ge A_{\rm ang}^{-1}|\mathcal F(Q)|\), then
\[
 \frac{\theta(Q)}{B_{\rm ang}}
 <M_q(\mathcal H)\le\theta(Q).
\tag{7.6}
\]
\end{lemma}

\begin{proof}
Every resolved diameter is at least \(q_{\rm res}\), whereas every
successful step divides the current diameter by \(B_{\rm ang}\).  This proves
(7.4).  Each selected family has at least an \(A_{\rm ang}^{-1}\) fraction of its
predecessor, proving (7.5).  At the stopping stage, (7.6) is the
negation of the next successful selection; its upper bound follows because
diameter cannot increase on passing to a subfamily.
\end{proof}

The local descent must still be synchronized globally.  Deterministic
tie-breaking first makes \(\mathcal F(Q)\) a well-defined edge subset.  We
then pigeonhole \(\theta(Q)\) into at most
\[
 J_\theta=1+\lceil\log_2(q_{\rm res}^{-1})\rceil
\]
classes by labelled edge count.  The surviving cell degrees lie between
\(A_{\rm ang}^{-K}d_0\) and \(2d_0\); a second dyadic pigeonhole uses at most
\[
 J_d=1+\lceil\log_2(2A_{\rm ang}^K)\rceil
\]
classes.  Retain every edge at a selected cell.

\begin{theorem}[Resolved ANGLE-CELL]
\label{thm:angle}
There exist an edge subset \(E_{\rm ang}\subset E_2\), a cellular parent
shading \(Z_{\rm ang}\), a dyadic degree \(d_{\rm ang}\), and
\(\theta\in[q_{\rm res},1]\) such that:
\begin{enumerate}[label=\textup{(A\arabic*)},leftmargin=2.8em]
\item every retained cell has parent degree in
      \([d_{\rm ang},2d_{\rm ang})\);
\item its resolved angular diameter lies in \([\theta,2\theta)\);
\item for every retained cell \(Q\), every
      \(\mathcal H\subset\{B:(B,Q)\in E_{\rm ang}\}\) with
      \[
       |\mathcal H|\ge A_{\rm ang}^{-1}\#\{B:(B,Q)\in E_{\rm ang}\}
      \]
      has resolved diameter larger than \(\theta/B_{\rm ang}\);
\item
\[
 \sum_B|Z_{\rm ang}(B)|
 \ge c_{\rm ang}\sum_B|Z(B)|,
 \qquad
 c_{\rm ang}\ge\frac{A_{\rm ang}^{-K}}{J_\theta J_d}
 .
\tag{7.7}
\]
More explicitly, with an absolute constant \(C\),
\[
 c_{\rm ang}^{-1}
 \le C(1+L_+)^2
       \exp\!\bigl(L_+^{3/4}+2L_+^{1/2}\bigr).
\tag{7.7a}
\]
Consequently \(c_{\rm ang}^{-1}=O(1)\) on \(L_0\le1\), while
\(c_{\rm ang}=\alpha^{o(1)}\) as \(\alpha\downarrow0\).
\end{enumerate}
The corresponding child lift retains at least \(c_{\rm ang}/2\) of the
pre-angle child mass.
\end{theorem}

\begin{proof}
\cref{lem:descent} gives the cellwise angular conclusions.  Both global
pigeonholes retain all edges at each selected cell, so the output degree is
constant and the cellwise angular family is unchanged.  Equal cell volumes
identify labelled parent mass with edge count, proving (7.7).  Finally
apply \cref{thm:lift} once to the resulting edge set.
The estimate
\[
 A_{\rm ang}^K\le\exp(L_+^{3/4}+2L_+^{1/2})
\]
and the elementary bounds
\(J_\theta,J_d\le C(1+L_+)\) prove (7.7a).  If \(L_0\to\infty\),
then \(L_+=L_0\), and hence
\(\log c_{\rm ang}^{-1}=o(L_0)\).  If \(L_0\le1\), every factor in
(7.7a) is bounded absolutely.
\end{proof}

\begin{remark}
Here and below ``raw'' means the capped geometric angle \(\ang\), before the
resolution floor \(q_{\rm res}\) is imposed.  The exact threshold at which
(7.6) yields genuine raw separation, including for repeated parallel labels,
is
\[
 \boxed{\theta>B_{\rm ang}q_{\rm res}}.
\tag{7.8}
\]
Indeed, then \(\theta/B_{\rm ang}>q_{\rm res}\), so
\(M_q(\mathcal H)>\theta/B_{\rm ang}\) lies above the resolution floor, so
the raw angular diameter is larger than \(\theta/B_{\rm ang}\).  At and below
this threshold, including \(\theta\asymp q_{\rm res}\), the conclusion is
resolution-limited.
\end{remark}

\section{A stable cellular maximal-density factoring theorem}
\label{sec:corrected-factoring}

We now assemble the reusable Euclidean theorem.  Its combinatorial engine is
the atomic factoring-and-lifting principle of
\cref{thm:abstract-cellular}; rectangular geometry supplies the carrier, and
Section~6 supplies three alternative density inputs.  The theorem records not
only the output estimates but also the category in which they are stable:
polynomially many labelled objects, positive parent--cell edges, and a fixed
number of whole-incidence refinements.  The angular theorem is one admissible
protocol on this output, not an extra shading created by the factoring
statement.

\begin{lemma}[Labelled biased maximal-density factoring]
\label{lem:biased-factoring}
Let \(U\subset\R^n\) be a bounded convex body with nonempty interior, let
\(\mathcal V\) be a finite nonempty labelled family of congruent, full-dimensional
compact convex bodies of positive \(n\)-dimensional volume contained in
\(U\), and fix \(\gamma>0\).  Every convex hull used as a candidate parent
is required to be full-dimensional (equivalently, only nonempty subfamilies
with positive-volume hulls enter the maximization), so every candidate lies
in the test-body class \(\mathfrak K_n\) fixed before (1.2).  Repeated geometric
carriers are allowed and are counted with their labels.  Fix finite sets of
integers \(\mathcal K_\Phi,\mathcal K_1,\ldots,\mathcal K_n\).  Assume that,
for every nonempty labelled subfamily \(\mathcal A\subset\mathcal V\), a
maximizer \(K_{\mathcal A}\) of the score below can be chosen so that
\[
 \Phi_{\mathcal A}(K)
 :=|K|^{-\gamma}\Delta(\mathcal A,K)
 =\frac{\sum_{V\in\mathcal A:\,V\subset K}|V|}
 {|K|^{1+\gamma}}
\tag{8.0A}
\]
satisfies \(\Phi_{\mathcal A}(K_{\mathcal A})\in[2^k,2^{k+1})\) for some
\(k\in\mathcal K_\Phi\), and its
ordered L\"owner half-axes \(\ell_j(K_{\mathcal A})\) belong to
\([2^{k_j},2^{k_j+1})\) with \(k_j\in\mathcal K_j\).  Define
\[
 L_{\rm bias}:=|\mathcal K_\Phi|\prod_{j=1}^n|\mathcal K_j|.
\tag{8.0A$'$}
\]
Then there are a labelled subfamily
\(\mathcal V'\subset\mathcal V\), a family \(\mathcal W\) of convex
parents with comparable L\"owner dimensions, and a partition
\[
 \mathcal V'=\bigsqcup_{W\in\mathcal W}\mathcal V_W,
\qquad V\subset W\quad(V\in\mathcal V_W),
\tag{8.0B}
\]
such that
\[
 \sum_{V\in\mathcal V'}|V|
 \ge L_{\rm bias}^{-1}\sum_{V\in\mathcal V}|V|.
\tag{8.0C}
\]
Writing \(w\) for the common parent volume up to a dimensional factor, the
following hold.
\begin{enumerate}[label=\textup{(B\arabic*)},leftmargin=2.8em]
\item For \(K\subset W\) convex,
\[
 \Delta(\mathcal V_W,K)
 \le C\left(\frac{|K|}{|W|}\right)^\gamma
       \Delta(\mathcal V_W,W).
\tag{8.0D}
\]
Equivalently, because the children are congruent,
\[
 \frac{\#\mathcal V_W[K]}{\#\mathcal V_W}
 \le C\left(\frac{|K|}{|W|}\right)^{1+\gamma}.
\tag{8.0E}
\]
\item The density achieved in every selected parent satisfies
\[
 \Delta(\mathcal V_W,W)
 \ge c\left(\frac{|W|}{|U|}\right)^\gamma
       \Delta_{\max}(\mathcal V').
\tag{8.0F}
\]
\item The labelled parent family obeys
\[
 \Delta_{\max}(\mathcal W)
 \le C\left(\frac{w}{|U|}\right)^{-\gamma}.
\tag{8.0G}
\]
\end{enumerate}
The constants \(C,c\) depend only on \(n,\gamma\) and the factor-two class
widths.  More explicitly, if
\(|\mathcal V|\le\delta^{-C_{\rm comp}}\) and every positive score and every
chosen half-axis above lies in
\([\delta^{C_{\rm win}},\delta^{-C_{\rm win}}]\), then each displayed index
set may be taken to have at most
\(C(C_{\rm comp},C_{\rm win},n,\gamma)\log(2/\delta)\) elements.  Hence
\(L_{\rm bias}\le C\log^{n+1}(2/\delta)=\delta^{-o(1)}\), with all
dependencies explicit.
\end{lemma}

\begin{proof}
For completeness, we give the greedy construction.  For a nonempty labelled subfamily
\(\mathcal A\subset\mathcal V\), maximize \(\Phi_{\mathcal A}(K)\) over
\(K\in\mathfrak K_n\).  A maximizer exists: replacing \(K\) by the convex hull of the
members of \(\mathcal A[K]\) does not change the numerator and can only
decrease the denominator, and only finitely many labelled subfamilies occur.
Fix an order on those finitely many hulls and on all dyadic classes, and use
it to break every maximizing or class-selection tie.  Thus the construction
is a deterministic function of labelled carrier and containment data.

Set \(\mathcal V^{(0)}=\mathcal V\).  Having defined
\(\mathcal V^{(j)}\ne\varnothing\), choose a maximizer \(W_j\), put
\[
 \mathcal V_{W_j}:=\mathcal V^{(j)}[W_j],\qquad
 \phi_j:=\Phi_{\mathcal V^{(j)}}(W_j),
\tag{8.0H}
\]
and remove all labels in \(\mathcal V_{W_j}\).  The process terminates and
partitions \(\mathcal V\).  Pigeonhole the values \(\phi_j\) and the
L\"owner half-axes of \(W_j\), choosing the combined class that maximizes
retained labelled child volume.  This proves (8.0B)--(8.0C), makes
\(\phi_j\) comparable to one number \(\phi\), and makes \(|W_j|\) comparable
to \(w\).  It also proves the asserted polylogarithmic bound for
\(L_{\rm bias}\) in the stated finite windows.

Fix a selected \(W_j\) and a convex \(K\subset W_j\).  Since the children
assigned to \(W_j\) are exactly the members of
\(\mathcal V^{(j)}\) contained in \(W_j\),
\[
\begin{aligned}
 \Delta(\mathcal V_{W_j},K)
 &\le\Delta(\mathcal V^{(j)},K)\\
 &\le\phi_j|K|^\gamma
 =\left(\frac{|K|}{|W_j|}\right)^\gamma
   \Delta(\mathcal V_{W_j},W_j).
\end{aligned}
\tag{8.0I}
\]
This gives (8.0D), up to the fixed class-comparability constants.
Multiplying by \(|K|/|W_j|\) and using equal child volume gives (8.0E).

Let \(j_0\) be the first selected greedy index.  All labels of
\(\mathcal V'\) are still present in \(\mathcal V^{(j_0)}\).  Any test body
in \(\mathfrak K_n\) may be intersected with \(U\), without losing a child
of \(\mathcal V'\); a positive-volume intersection remains in
\(\mathfrak K_n\), while a null intersection has zero numerator.  Hence
\[
 \Delta_{\max}(\mathcal V')
 \le C\phi |U|^\gamma.
\tag{8.0J}
\]
On the other hand
\[
 \Delta(\mathcal V_{W_j},W_j)
 =\phi_j|W_j|^\gamma
 \ge c\phi|W_j|^\gamma.
\]
Combining these inequalities proves (8.0F).  This is the quantitative
near-tie rule: among almost equal density scores, a smaller parent receives
the larger value of the volume-biased score.

Finally let \(K\in\mathfrak K_n\).  If \(W\subset K\), then every child assigned
to \(W\) is contained in \(K\).  By (8.0F), or directly by the common
\(\phi\)-class,
\[
\begin{aligned}
 \Delta_{\max}(\mathcal V')|K|
 &\ge \sum_{W\in\mathcal W:\,W\subset K}
       \sum_{V\in\mathcal V_W}|V|\\
 &\ge c\phi w^\gamma
       \sum_{W\in\mathcal W:\,W\subset K}|W|.
\end{aligned}
\]
Use (8.0J), divide by \(|K|\), and take the supremum over \(K\).  This gives
(8.0G).  Every selection and sum was made on labels, so the proof is
unchanged for repeated carriers.
\end{proof}

\begin{theorem}[Quantitative cellular maximal-density factoring]
\label{thm:mainfact}
Fix \(C_{\rm comp}\ge1\), \(m_0\in\mathbb N\), and \(\eps>0\).  Let
\(0<\delta\le a\le b\le r\le1\), put \(h=c_0a\), and let
\(\T=\bigsqcup_{B\in\B}\T_B\) be a finite labelled family with \(M(Y)>0\).
Assume:
\begin{enumerate}[label=\textup{(F\arabic*)},leftmargin=2.8em]
\item every \(B\) is a labelled rectangular box, with the ordered side frame
      fixed in (3.4a), of dimensions comparable to
      \(a\times b\times r\), and every \(T\in\T_B\) is a labelled
      \(\delta\times\delta\times r\) child contained in \(B\), with
      measurable shading \(Y(T)\subset T\);
\item for some \(D>0\) and \(\kappa_D\ge1\),
\[
 D\le D_B:=\sum_{T\in\T_B}|T|\le\kappa_DD
\]
for every parent;
\item the finite active-cell set
\[
 \Q_h(Y)=\left\{Q\in\Q_h:
   \sum_{B\in\B}\sum_{T\in\T_B}|Y(T)\cap Q|>0\right\}
\]
satisfies
\[
 |\T|+|\B|+|\Q_h(Y)|\le\delta^{-C_{\rm comp}};
\]
\item one of the following three global alternatives is designated for the
whole family.  In the ordinary fine-child alternative every \(\T_B\)
satisfies, with one common constant \(C_F\),
\[
 \frac{\#\{T\in\T_B:T\subset K\}}{|\T_B|}
 \le C_F\frac{|K|}{|B|}
 \qquad\text{for every convex }K\subset B,
\tag{8.0}
\]
In the biased alternative every \(\T_B\) satisfies, for one fixed
\(0<\sigma\le1\), the complete-child condition (6.3) with one common
constant \(C_0\).  In the coarse/thickened-child alternative there is one
scale \(\delta\le\rho\le a\) and, for every \(B\), a nonempty labelled
family \(\Scal_B\) of \(\rho\times\rho\times r\) capped tubes contained in
\(B^\sharp\), satisfying (6.6c) with the common constant \(C_F\).  The fine
and coarse carriers use the same declared length window
\([c_{\rm len}r,C_{\rm len}r]\) and the same endcap convention.  Each fine
child \(T\) is assigned to one coarse label \(S(T)\) with \(T\subset S(T)\),
and the assigned branching numbers satisfy (6.6c$'$) with one common factor
\(R_{\rm br}\).
\end{enumerate}
After the common inner-level selection, let \(E^+\) be the positive graph
(3.7) constructed from \(Y_0\), put
\(N_{\max}=\max_B|\T_B|\), \(e_+=|E^+|\), and define
\[
 \mathfrak L_0
 =4(1+\lceil\log_2N_{\max}\rceil)
   (1+\lceil\log_2(2e_+)\rceil)
   (1+\lceil\log_2e_+\rceil).
\tag{8.1}
\]
Then there exist a retained positive edge set \(E_2\subset E^+\), its labelled
rectangular parent family \(\widetilde\B'\), a cellular parent shading
\(Z\), and a child refinement \((\T',Y_1)\) satisfying P1--P5 of
\cref{thm:joint}, with constants depending only on \(\eps,\sigma\) where applicable, the
fixed comparison constants, and \(C_{\rm comp}\), and with
\[
 M(Y_1)\ge\mathfrak L_0^{-1}M(Y)
\tag{8.2}
\]
Here the output label sets are literal:
\[
 \B'=\{B:\exists Q\ (B,Q)\in E_2\},\qquad
 \T'=\bigsqcup_{B\in\B'}\T_B,
 \qquad
 \T'_{\rm act}=\{T\in\T':Y_1(T)\ne\varnothing\}
\tag{8.2$'$}
\]
The shading \(Y_1\) is a total function on the structural family \(\T'\),
and labels in \(\T'\setminus\T'_{\rm act}\) have empty shading.  Complete-child
Frostman hypotheses are always read on the structural family \(\T_B\), while
mass, union, and pointwise multiplicity are unchanged if one restricts to
\(\T'_{\rm act}\).  Thus an empty shading is never confused with deletion of
a complete carrier from a denominator.
Explicitly, \(Y_1(T)\subset Z(B)\) for \(T\in\T_B\), and there are
\(\mu_{\rm in},d_0\ge1\) such that on the corresponding unions
\[
\begin{aligned}
 \mu_{\rm in}
 &\le\sum_{T\in\T_B}\mathbf1_{Y_1(T)}<2\mu_{\rm in},\\
 d_0
 &\le\sum_B\mathbf1_{Z(B)}<2d_0,\qquad
 \sum_{T\in\T}\mathbf1_{Y_1(T)}<4\mu_{\rm in}d_0.
\end{aligned}
\tag{8.2a}
\]
In the biased case the initial parent output satisfies
\[
 \lambda(\widetilde\B',Z)
 \gtrsim C_0^{-(2-\sigma)}(a/b)^\eps
 \kappa_D^{-(2-\sigma)}\mathfrak L_0^{-(2-\sigma)}
 \lambda(\T,Y)^{2-\sigma};
\tag{8.3}
\]
in the ordinary endpoint it satisfies
\[
 \lambda(\widetilde\B',Z)
 \gtrsim C_F^{-1}[\log(2b/a)]^{-1}
 \kappa_D^{-2}\mathfrak L_0^{-2}\lambda(\T,Y)^2.
\tag{8.4}
\]
and in the coarse/thickened-child endpoint it satisfies
\[
 \lambda(\widetilde\B',Z)
 \gtrsim C_F^{-1}R_{\rm br}^{-3}[\log(2b/a)]^{-1}
 \kappa_D^{-2}\mathfrak L_0^{-2}\lambda(\T,Y)^2.
\tag{8.4a}
\]

A \emph{whole-incidence refinement protocol} is a nested sequence
\[
 E^{(m)}\subset\cdots\subset E^{(1)}\subset E^{(0)}=E_2,
 \qquad m\le m_0,
\]
where
\[
 \begin{aligned}
 Z^{(j)}(B)&=\bigcup_{Q:(B,Q)\in E^{(j)}}Q,\\
 Y^{(j)}(T)&=Y_1(T)\cap
       \bigcup_{Q:(B,Q)\in E^{(j)}}Q\quad(T\in\T_B),
 \end{aligned}
\tag{8.5}
\]
and, for \(1\le j\le m\),
\[
 \sum_B|Z^{(j)}(B)|\ge c_j\sum_B|Z^{(j-1)}(B)|,
 \qquad 0<c_j\le1.
\tag{8.6}
\]
Such a protocol retains every child incidence of \(B\) in a selected
\((B,Q)\)-edge.  It obeys
\[
 \begin{gathered}
 M(Y^{(m)})\ge
 \mathfrak L_0^{-1}2^{-m}\prod_{j=1}^mc_j\,M(Y),\\
 \mathfrak L(\T,\B,h;c_1,\ldots,c_m)
 :=\mathfrak L_0\prod_{j=1}^m\frac2{c_j}.
 \end{gathered}
\tag{8.7}
\]
Equivalently,
\(M(Y^{(m)})\ge
\mathfrak L(\T,\B,h;c_1,\ldots,c_m)^{-1}M(Y)\).
If one additionally knows \(c_j\ge c_{\min}>0\), then
\(\mathfrak L\le\mathfrak L_0(2/c_{\min})^{m_0}\); no such fixed lower
bound is built into the theorem.
At every stage, exact cellular compatibility, the common inner multiplicity,
and the original pointwise multiplicity upper bound survive.  Any dyadic
right-degree re-regularization is counted as one of the \(m\) stages and must
retain all edges at every selected cell.
\end{theorem}

\begin{proof}
The proof has three modules: joint regularization, analytic density, and
iterated lifting.  Run the common inner level, \cref{lem:REG1}, right-degree selection, and
\cref{thm:joint}.  Since \(|E_1|\le e_+\), (4.8) is at most
\(\mathfrak L_0\), which proves (8.2).  Apply \cref{cor:density} with this
larger explicit loss, using \cref{lem:coarse-cordoba} in the third
alternative, to obtain (8.3), (8.4), or (8.4a).  Iterating \cref{thm:lift}
along (8.5)--(8.6) proves (8.7).  In particular a dyadic degree
re-regularization may retain only a polylogarithmic fraction; its actual
fraction is entered as \(c_j\), rather than compared with a fixed constant.
The last assertions are precisely the invariants in
\cref{thm:lift,cor:rer}.
\end{proof}

\begin{lemma}[Source L\"owner frame]
\label{lem:ordered-john}
Let \(W\subset\R^3\) be a bounded convex body with nonempty interior.  Fix
the L\"owner (minimum-volume containing) ellipsoid
\(E_W^{\rm out}\) used to define the dimensions of \(W\), and order its semiaxes
\(\ell_1\le\ell_2\le\ell_3\).  When eigenvalues coincide, choose an
orthonormal eigenframe by a fixed lexicographic rule against the ambient
coordinate frame, use that same rule for both \(W\) and its rectangular
label, and retain the frame as part of the label.  There is a rectangle
\(B_W\), aligned with this ordered outer frame, such that
\[
 W\subset B_W,\qquad |B_W|\le C_J|W|,
\tag{8.7a}
\]
and, under this fixed convention, the labelled tangent plane
\(TB_W=\operatorname{span}(e_2(W),e_3(W))\) is the source two-long-axis
plane when that plane is uniquely determined.  At a repeated eigenvalue it
is a fixed representative compatible with the source resolution convention,
which does not select a lexicographic eigenframe.  In particular no
comparison with the principal frame of an independently chosen inner John
ellipsoid is made.
\end{lemma}

\begin{proof}
The source defines convex dimensions from the axes of its minimum-volume
containing ellipsoid and defines plank angle from the two longest of those axes
\cite[p.~4 and the plank-angle convention on p.~16]{GWZ2026}.  Take \(B_W\)
to be the coordinate rectangle with half-side lengths
\(\ell_1,\ell_2,\ell_3\).  Then
\(W\subset E_W^{\rm out}\subset B_W\), and
\(|B_W|=(6/\pi)|E_W^{\rm out}|\).

For volume comparability, let \(c_W\) be the L\"owner
centre.  The non-symmetric L\"owner--John inclusion in dimension \(n\)
states precisely that the homothetic copy
\(c_W+n^{-1}(E_W^{\rm out}-c_W)\) lies in \(W\).  Thus, in dimension three,
\[
 |E_W^{\rm out}|\le 3^3|W|=27|W|;
\]
see \cite[Theorem~2.1 and equation~(2)]{Henk2012}, together with the
original extremal theorem \cite{John1948}.  This proves (8.7a).  The
deterministic tie rule makes the ordered outer frame a label when an
eigenspace is multiple.  Using this one chosen outer frame for both \(W\)
and \(B_W\) makes their two tangent planes identical.  If the short and middle axes tie, then
\(\ell_1/\ell_2=1\) and the source angle is already unresolved at unit
scale; the common tie rule merely chooses one representative of that source
ambiguity.  If only the middle and long axes tie, their two-dimensional
eigenspace is itself unique, so the two-long-axis plane is independent of the
basis chosen inside it.  If all three axes tie, the source plane is again
unresolved at unit scale.  Standard background on uniqueness of the L\"owner
ellipsoid, parallel bodies, and orthogonal choices inside repeated
eigenspaces is recorded in \cite[Chs.~1, 3, and 10]{Schneider2014}; only the
source convention just cited determines which plane is used in the GWZ call.
\end{proof}

\begin{lemma}[Transfer to permanent L\"owner rectangles]
\label{cor:convexbox}
Suppose a labelled convex parent \(W\) has the containing rectangle \(B_W\)
from \cref{lem:ordered-john}, has a nonempty assigned child family, and every
assigned child label is wholly contained in \(W\).  Let
\(p\in\{1,1+\sigma\}\).  If
\[
 \frac{\#\{T:T\subset K\}}{\#\{T:T\subset W\}}
 \le C_0\left(\frac{|K|}{|W|}\right)^p
 \qquad(K\subset W\text{ convex}),
\tag{8.7b}
\]
then the same labelled children, viewed inside \(B_W\), satisfy
\[
 \frac{\#\{T:T\subset K\}}{\#\{T:T\subset B_W\}}
 \le C_0C_J^p\left(\frac{|K|}{|B_W|}\right)^p
 \qquad(K\subset B_W\text{ convex}).
\tag{8.7c}
\]
The total child denominator is unchanged.  For a labelled parent family,
\[
 \sum_W|W|\le\sum_W|B_W|\le C_J\sum_W|W|,
 \qquad
 \Delta_{\max}(\{B_W\})\le C_J\Delta_{\max}(\{W\}).
\tag{8.7d}
\]
The ordered tangent frame is unchanged.  The same argument with \(p=1\)
transfers the ordinary fine-child and coarse/thickened-child containment
conditions, provided the counted carriers are wholly contained in \(W\).
Consequently \cref{thm:mainfact} applies to the permanent labelled boxes
\(B_W\), with its shading and union conclusions interpreted on those boxes.
The transfer is one-way: it supplies neither a restriction to the sharp
boundary of \(W\) nor a bounded-overlap conclusion for the rectangles.
\end{lemma}

\begin{proof}
For a convex \(K\subset B_W\), a child already contained in \(W\) is
contained in \(K\) exactly when it is contained in the convex set
\(K\cap W\).  Applying (8.7b) to \(K\cap W\), then using
\(|K\cap W|\le|K|\) and \(|B_W|/|W|\le C_J\), proves (8.7c).  The label set
and child carriers are unchanged, so their denominator is identical.

For the parent concentration, \(B_W\subset K\) implies \(W\subset K\), and
\(|B_W|\le C_J|W|\).  Summing this implication over labels proves the second
part of (8.7d); the first is immediate from containment and volume
comparability.  The tangent-frame assertion is
\cref{lem:ordered-john}.  Repeating the same intersection argument for the
fine or coarse carrier counted in the relevant condition proves the remaining
claims.  All labels, including coincident ones, are retained.
\end{proof}

\subsection{Quantitative dependence and iteration budget}

The finite graph makes every loss explicit.  For the
polynomial-complexity application, \(e_+\le|\B||\Q_h(Y)|\) and
\[
 \mathfrak L_0\le C(C_{\rm comp})[\log(2/\delta)]^3.
\tag{8.8}
\]
In the applications of Sections~9--10, after their displayed bounds have
been verified, the additional notation \(\delta^{-o(1)}\) abbreviates only
the following controlled factors:
\[
\begin{gathered}
 L_{\rm in},\quad
 1+\log(2|E^+|),\quad
 L_{\rm deg},\quad
 J_\theta,\quad J_d,\quad
 c_{\rm ang}^{-1}.
\tag{8.9}
\end{gathered}
\]
In \cref{thm:mainfact} itself, \(\kappa_D\) remains explicit and no
subpolynomial asymptotic bound on it is assumed.  An application may append
\(\kappa_D\) to (8.9) only after proving such a bound for its denominator
level.
Section~\ref{sec:section9} adds the explicit quantities
\[
 J_v,\quad L_m,\quad L_{a,b},\quad L_N,\quad
 L_{\rm reg},\quad\kappa_{\rm inc},\quad
 \kappa_{\rm loc},\quad\kappa_{\rm amb},\quad H_{\rm ang}.
\tag{8.10}
\]
Jensen raises every synchronized density loss to \(2-\sigma\), and to two
in the C\'ordoba endpoint; LIFT losses remain separate from density losses.

\begin{remark}
The theorem is not uniform over an arbitrary indexed multiset of
super-polynomial cardinality.  Without the stated polynomial-complexity bound, the
logarithms in (8.1)--(8.2) must remain explicit.
\end{remark}

\section{Two small-middle applications in GWZ Section 6}
\label{sec:section6}

We first test the cellular theorem on the two small-middle factorizations in
the factoring-through-planks argument of~\cite{GWZ2026}.  Both applications
need the same three outputs: quantitative retention of child mass, the
child--parent multiplicity product, and the ordinary quadratic parent-density
endpoint.  Neither application needs a ballwise output or a metric
reconstruction of the parent shading.  Consequently the cellular theorem can
be inserted through a short, explicit interface without altering the
surrounding proof architecture.

\begin{table}[t]
\centering
\small
\begin{tabularx}{\textwidth}{@{}>{\raggedright\arraybackslash}p{1.7cm}>{\raggedright\arraybackslash}p{2.2cm}>{\raggedright\arraybackslash}X>{\raggedright\arraybackslash}X@{}}
\toprule
application & objects & cellular realization & additional check\\
\midrule
Section 6.5(A) &
\(\T=\bigsqcup_W\T_W\) through Frostman planks &
P1, P3--P5, (6.10), and
\(\mu(\T,Y)\lesssim_{\log}\mu(\B,Z)\mu(\T_W,Y_1)\) &
choose one inner parent attaining the post-refinement average density\\
\addlinespace
Section 6.5(B) &
fine tubes grouped through coarse \(\rho\)-tubes &
same cellular product and quadratic endpoint on the fine labelled multiset &
retain a common number of fine descendants per coarse label\\
\addlinespace
biased interface &
maximal-density convex parents &
permanent labelled rectangles and (6.9) &
never restrict the conclusion back to the sharp parent\\
\bottomrule
\end{tabularx}
\caption{The two Section~6 applications of the cellular theorem.}
\label{tab:section6}
\end{table}

The next four lemmas isolate, in order, fixed-dilate symmetry, conflict
coloring, the denominator cost of a mass-selected color, and normalized
concentration under bounded ancestry.  The application proofs invoke these
four roles through their tagged conclusions.

\begin{lemma}[Fixed-dilate symmetry for source-standard tubes]
\label{lem:fixed-dilate-symmetry}
Let \(S,S'\) be capped source-standard tubes of the same length \(L\) and
width \(w\), in one fixed bounded direction chart.  For every fixed
\(C\ge1\),
\[
 S\subset CS'\quad\Longrightarrow\quad S'\subset C_*S,
\tag{9.0a$'$}
\]
where \(C_*\) depends only on \(C\) and the chart constants.
\end{lemma}

\begin{proof}
If \(w\ge cL\) for a fixed structural \(c>0\), both capped carriers have
diameter \(O(w)\), and the assumed containment places their centres within
\(O_C(w)\); the conclusion is immediate after a structural dilation.  Assume
henceforth that \(w<cL\).  The two endpoint cross-sections of the carrier of \(S\) lie within \(O_C(w)\)
of \(S'\).  Their separation is comparable to \(L\), so the corresponding
unit directions differ by \(O_C(w/L)\), and one endpoint of the carrier of
\(S'\) is within \(O_C(w)\) of the matching endpoint of \(S\).  Linear
interpolation along the capped carriers gives Hausdorff distance \(O_C(w)\)
between them.  Thickening by width \(w\) proves (9.0a$'$), including parallel
and coincident carriers.
\end{proof}

\begin{lemma}[Source-resolution conflict coloring]
\label{lem:source-conflict-coloring}
Let \(\mathcal F_0\) be a source-standard, essentially distinct family of
\(w\)-tubes.  Quantitatively, for every fixed \(C\ge1\), assume
\[
 \sup_{T_0}\#\{T\in\mathcal F_0:T\subset CT_0\}
 \le \kappa_{\rm ed}(C),
\tag{9.0a}
\]
where \(\kappa_{\rm ed}(C)=O_C(1)\) in a literal source call.  Let \(I\) be a
labelled representation family with a map \(\pi:I\to\mathcal F_0\), fibres
of size at most \(\kappa_{\rm rep}\), and carriers \(T_i\) contained in and
containing fixed enlargements of \(\pi(i)\).  Join two distinct vertices
\(i,j\) when
\[
 T_i\subset C_{\rm src}T_j\quad\hbox{or}\quad
 T_j\subset C_{\rm src}T_i,
\tag{9.0b}
\]
where \(C_{\rm src}\) is the dilation in the source essential-distinctness
convention.  Then the conflict graph has maximum degree at most
\[
 C\kappa_{\rm rep}\kappa_{\rm ed}(C'),
\tag{9.0c}
\]
for fixed structural \(C,C'\), and hence has a proper coloring with
\(H_{\rm src}\le1+C\kappa_{\rm rep}\kappa_{\rm ed}(C')\) colors.  Each color
is source-essentially-distinct.  For arbitrary nonnegative label weights one
color carries at least \(H_{\rm src}^{-1}\) of their total.  If the weights
are shading masses and all carrier volumes are comparable, that color has
shading density at least a fixed multiple of \(H_{\rm src}^{-1}\) times the
original density.  Passing to a color cannot increase the absolute
complete-containment density \(\Delta_{\max}\); normalized Frostman constants
are treated separately in \cref{lem:color-normalized}.  Moreover, a proper
color contains at most one vertex in every fibre whose represented carriers
are pairwise adjacent in (9.0b), in particular in a fibre of coincident copies.
If this condition holds for every fibre, the restriction of \(\pi\) to a color
is injective and its absolute \(\Delta_{\max}\) is bounded by that of the source
family.
\end{lemma}

\begin{proof}
If \(i\) and \(j\) are adjacent, the fixed two-sided comparison between
\(T_k\) and \(\pi(k)\) turns (9.0b) into
\(\pi(j)\subset C'\pi(i)\) or the reverse inclusion, with \(C'\) structural.
For equal-length, equal-width source tubes the reverse inclusion also implies
\(\pi(j)\subset C''\pi(i)\) by \cref{lem:fixed-dilate-symmetry}.
Equation~(9.0a) leaves at most \(C\kappa_{\rm ed}(C')\) possible
geometric source labels, and each has at most \(\kappa_{\rm rep}\) copies.
This proves (9.0c), so greedy coloring gives the asserted \(H_{\rm src}\).
A fibre whose represented carriers are pairwise adjacent in (9.0b) is a
clique, so a proper color contains at most one of its vertices.  Coincident
copies satisfy this condition.

No conflict remains inside a color, which is precisely the source
essential-distinctness condition after the declared fixed enlargement.  The
largest color by weight proves the mass assertion.  Its denominator is no
larger than the original denominator and the fixed carrier dilation changes
individual volumes only by a structural factor, proving the density assertion.
The complete-containment numerator only decreases on passage to a subfamily.
Finally a fixed enlargement replaces \(w\) by \(Cw\); the scale-variant source
statement is applied at this new width.  All scale powers and density
thresholds change by fixed constants only, so no hidden power of \(w\) is
spent.
\end{proof}

\begin{lemma}[Mass-selected color: absolute versus normalized concentration]
\label{lem:color-normalized}
Let a finite labelled family \(\mathcal R\) of congruent carriers of volume
\(v>0\) be partitioned into at most \(H\) colors, and let
\(Y(R)\subset R\) be measurable.  Set
\[
 D=v|\mathcal R|,\qquad M=\sum_{R\in\mathcal R}|Y(R)|,
 \qquad \lambda_{\mathcal R}=M/D>0.
\tag{9.0d}
\]
If a color \(\mathcal R_c\) carries at least \(H^{-1}M\), then
\[
 \frac{|\mathcal R_c|}{|\mathcal R|}
 \ge \frac{\lambda_{\mathcal R}}{H}.
\tag{9.0e}
\]
Moreover, writing \(\mu(\mathcal R,Y)=M/|U(\mathcal R,Y)|\),
\[
 \mu(\mathcal R,Y)\le H\mu(\mathcal R_c,Y|_{\mathcal R_c}).
\tag{9.0e$'$}
\]
Let \(\mathcal K\) be any prescribed class of test bodies.  Consequently,
if
\[
 \frac{|\mathcal R[K]|}{|\mathcal R|}\le C_{\rm test}\Psi(K)
 \qquad(K\in\mathcal K),
\tag{9.0f}
\]
then
\[
 \frac{|\mathcal R_c[K]|}{|\mathcal R_c|}
 \le H\lambda_{\mathcal R}^{-1}C_{\rm test}\Psi(K)
 \qquad(K\in\mathcal K).
\tag{9.0g}
\]
On the other hand,
\(\Delta_{\max}(\mathcal R_c)\le\Delta_{\max}(\mathcal R)\).
Thus absolute concentration is monotone under label deletion, whereas
normalized concentration must pay the denominator factor in (9.0g).  The
same statements hold for carrier volumes in \([v,\kappa_vv]\), with one
additional factor \(\kappa_v\).
\end{lemma}

\begin{proof}
The selected shading mass is at most its carrier denominator, so
\[
 v|\mathcal R_c|
 \ge\sum_{R\in\mathcal R_c}|Y(R)|
 \ge H^{-1}M
 =H^{-1}\lambda_{\mathcal R}v|\mathcal R|.
\]
This is (9.0e).  Since
\(|\mathcal R_c[K]|\le|\mathcal R[K]|\), division by (9.0e) proves
(9.0g), pointwise on the prescribed class \(\mathcal K\); no closure
property of that class is used.  Also \(M\le HM_c\) and
\(U(\mathcal R_c,Y)\subset U(\mathcal R,Y)\), which proves (9.0e$'$).
For the absolute assertion, the complete-carrier numerator for
every fixed \(K\) only decreases.  The comparable-volume version follows by
using the two endpoints of the comparison interval in the same calculation.
\end{proof}

\begin{lemma}[Normalized concentration through bounded ancestry]
\label{lem:normalized-ancestry}
Let \(\mathcal T\) and \(\mathcal P\) be finite nonempty labelled carrier
families and let \(I\subset\mathcal P\times\mathcal T\) be a labelled incidence
relation such that \(T\subset P\) whenever \((P,T)\in I\).  Suppose
\(N>0\), \(R_N,\kappa_I\ge1\), and
\[
 N\le d_I(P)\le R_NN\quad(P\in\mathcal P),
 \qquad 1\le d_I(T)\le\kappa_I\quad(T\in\mathcal T).
\tag{9.0h}
\]
Then, for every test body \(K\),
\[
 \frac{|\mathcal P[K]|}{|\mathcal P|}
 \le \kappa_IR_N
       \frac{|\mathcal T[K]|}{|\mathcal T|}.
\tag{9.0i}
\]
Consequently a normalized complete-containment bound for \(\mathcal T\)
passes to \(\mathcal P\) with the explicit factor \(\kappa_IR_N\).  This is
an ancestry statement, not monotonicity under taking a color or subfamily.
\end{lemma}

\begin{proof}
Every \(P\subset K\) contributes at least \(N\) incidences, and all its
incident children lie in \(K\).  Hence
\[
 N|\mathcal P[K]|\le\kappa_I|\mathcal T[K]|.
\]
On the other hand, the lower fine degree and the upper parent degree give
\(|\mathcal T|\le|I|\le R_NN|\mathcal P|\).  Multiplying these two
inequalities proves (9.0i).  Repeated geometric carriers cause no change
because all four counts are labelled counts.
\end{proof}

\paragraph{Interface convention.}
Each source call is recorded through seven fields: (SC1) family and reference
shading, (SC2) essential distinctness, (SC3) normalized scales, (SC4) density,
(SC5) concentration, (SC6) affine behavior, and (SC7) source output and
pullback.  The cited tagged formulas in this table contain the quantitative
data used below.
\begin{center}
\scriptsize
\begin{tabularx}{\textwidth}{@{}>{\raggedright\arraybackslash}p{1.65cm}>{\raggedright\arraybackslash}p{2.45cm}>{\raggedright\arraybackslash}p{2.75cm}>{\raggedright\arraybackslash}X@{}}
\toprule
call & (SC1)--(SC2) & (SC3)--(SC5) & (SC6)--(SC7)\\
\midrule
Lemma 6.4, child & selected color of the normalized local child family, with its shading & (9.2a), (9.2d0): scales, density, and \(C_F\) & affine invariance and (9.I-46), recovered in (9.2d)\\
Lemma 6.4, parent & selected color of \(\widetilde\W'\), with cellular parent shading & (9.2c), (9.2e0): scales, density, and \(C_F\) & affine invariance and (9.I-45), recovered in (9.2e)\\
Lemma 6.1, child & complete normalized representation family \(\mathcal F_{\rm ch}\), with its local shading & (9.4b), (9.5b), (9.5c0): scales, density, restricted slab tests, and \(\Delta_{\max}\) & affine invariance and (9.I-49), giving (9.5c)\\
Lemma 6.1, parent & retained parent family \(\W'\), with cellular shading \(Z\) & (9.5d0): scales, density, and absolute concentration & affine invariance and (9.I-48), giving (9.5d)\\
\bottomrule
\end{tabularx}
\end{center}
The density field (SC4) is always read on the displayed reference shading;
(SC5) uses the corresponding labelled denominator.

\begin{proposition}[Section~6 source interface]
\label{prop:corrected-s6-interface}
Fix \(0<\beta\le1\) and \(\varepsilon>0\).  There are
\(\eta,b_0>0\), depending only on these parameters, with the following
property.  Let \(0<\tau\le s\le t\le b_0\), and let
\((\mathcal P,A)\) be an essentially distinct labelled family of
\(s\times t\times1\) planks in a unit ball, with
\(\lambda(\mathcal P,A)\ge\tau^\eta\).  The scale-variant convention is the
one in \cite[Remark~3.6]{GWZ2026}; thus the losses below are powers of the
declared auxiliary scale \(\tau\), not powers silently transferred between
\(s\) and \(\tau\).

If \(K_F(\beta)\) holds and, with \(P_\vartheta\) denoting the source
\(\vartheta t\times t\times1\) thickening of \(P\),
\[
 |\mathcal P[P_\vartheta]|\le M\vartheta
 \quad(P\in\mathcal P,\ s/t\le\vartheta\le1),
\tag{9.I-F0}
\]
then the original statement of
\cite[Lemma~6.4, equation~(31)]{GWZ2026} gives
\[
 \mu(\mathcal P,A)
 \lesssim \tau^{-\varepsilon}
 C_F(\mathcal P)^{1-\beta/2}M^{\beta/2}
 \frac{s}{t}\,t^{-2\beta}
 \bigl(t^2|\mathcal P|\bigr)^{1-\beta/2}.
\tag{9.I-F}
\]
Consequently the substitutions
\[
 \begin{array}{c|c|c|c}
  \text{call}&s&t&M\\ \hline
  \text{parent}&a&b&\lesssim b/a\\
  \text{child}&\delta/b&\delta/a&\lesssim(b/a)^2
 \end{array}
\tag{9.I-F1}
\]
give, respectively,
\[
 \mu_{\rm par}\lesssim\tau^{-\varepsilon}
 C_F^{1-\beta/2}
 \left(\frac ab\right)^{1-\beta/2}b^{-2\beta}
 \bigl(b^2N_{\rm par}\bigr)^{1-\beta/2},
\tag{9.I-45}
\]
and
\[
 \mu_{\rm ch}\lesssim\tau^{-\varepsilon}
 C_F^{1-\beta/2}
 \left(\frac ab\right)^{1-\beta}
 \left(\frac\delta a\right)^{-2\beta}
 \left(\frac{\delta^2}{a^2}N_{\rm ch}\right)^{1-\beta/2}.
\tag{9.I-46}
\]

If instead \(K_{KT}(\beta)\) holds, \(0\le\gamma\le1\), and every
\(\vartheta\times1\times1\) test slab in the sense of
\cite[equation~(28)]{GWZ2026} satisfies
\[
 |\mathcal P_S|\le\tau^{-\eta}\vartheta^\gamma|\mathcal P|
 \quad(s/t\le\vartheta\le1),
\tag{9.I-K0}
\]
where \(\mathcal P_S\) has exactly the meaning of source equation~(28): it
counts only those \(P\subset S\) for which the planes spanned by the two
longest axes of \(P\) and of the \(\vartheta\times1\times1\) test slab \(S\)
make angle at most \(\vartheta\),
then the original statement of
\cite[Lemma~6.1, equation~(29)]{GWZ2026} gives
\[
 \mu(\mathcal P,A)
 \lesssim\tau^{-\varepsilon}
 \Delta_{\max}(\mathcal P)^{1-\beta}
 \left(\frac st\right)^{\gamma\beta}|\mathcal P|^\beta.
\tag{9.I-K}
\]
For \((s,t,\gamma)=(a,b,0)\) and
\((\delta/b,\delta/a,1)\), this becomes
\[
 \mu_{\rm par}\lesssim\tau^{-\varepsilon}
 \Delta_{\max}(\mathcal P_{\rm par})^{1-\beta}N_{\rm par}^\beta,
\tag{9.I-48}
\]
and
\[
 \mu_{\rm ch}\lesssim\tau^{-\varepsilon}
 \Delta_{\max}(\mathcal P_{\rm ch})^{1-\beta}
 \left(\frac ab\right)^\beta N_{\rm ch}^\beta.
\tag{9.I-49}
\]
Here \(\mathcal P_{\rm ch}\) denotes the local child family supplied to the
lemma.
\end{proposition}

\begin{proof}
Apply the scale-variant forms of the two cited lemmas at auxiliary scale
\(\tau\).  Formula (9.I-F) is equation~(31) with its parameters renamed
\((a,b)=(s,t)\).  In the parent row of (9.I-F1),
\(M^{\beta/2}(s/t)\lesssim(a/b)^{1-\beta/2}\); in the child row it is
\(\lesssim(a/b)^{1-\beta}\).  The remaining powers are obtained by replacing
\(t\) with \(b\) and \(\delta/a\), respectively.  Formula (9.I-K) is exactly
equation~(29) with the same renaming.  Taking \(\gamma=0\) and \(\gamma=1\)
gives (9.I-48)--(9.I-49), so both conclusions follow directly from the cited
source lemmas.
\end{proof}

\begin{proposition}[Small-middle cellular replacement in GWZ Proposition~6.6(A)]
\label{prop:s6A}
Fix \(0<\beta\le1\), \(\varepsilon>0\), and a loss profile
\(\mathfrak p\), and assume \(K_F(\beta)\).
There exist \(\eta=\eta(\varepsilon,\beta)>0\),
\(b_0=b_0(\varepsilon,\beta)>0\), and
\(\delta_0=\delta_0(\varepsilon,\beta,\mathfrak p)>0\) such that the following holds for
\(0<\delta\le\delta_0\).  Here \(b_0\) is chosen to be a sufficiently small
structural multiple of the minimum of the original source thresholds, so every
fixed enlargement used below of a plank with middle width at most \(b_0\)
still lies in the stated range of the corresponding source lemma.  Let
\((\T,Y)\) be uniform in the sense of
\cite[Definition~2.2, p.~4]{GWZ2026}, with all source uniformity,
comparability, and normalized-interface losses governed by \(\mathfrak p\), and with
\(\lambda(\T,Y)\ge\delta^{\eta}\), and suppose
there is a source coarse scale \(\delta\le\rho\le a\le b\le b_0\), with
\(\delta/a\le b_0\), for which the
GWZ Part~(A) factoring structure gives
\(\T=\bigsqcup_{W\in\mathcal W}\T_W\), where \(\mathcal W\) is a nonempty
labelled family of \(a\times b\times1\) parents and every \(\T_W\) is
nonempty.  Assume that the normalized parent labels and normalized child
representations have maps, with uniformly bounded fibres, to
source-essentially-distinct carrier families, with each represented carrier
contained in and containing fixed enlargements of its image.  Assume also,
as part of the factoring data, that their fixed-dilate source conflict graphs
have maximum degree \(O(1)\).  Assume that the numbers \(|\T_W|\) are comparable,
that every \(\T_W\) is Frostman in \(W\), and that after the affine
normalizations each parent and child plank family satisfies, with
\(P_\vartheta\) its source thickening,
\[
 |\mathcal P[P_\vartheta]|\le M\vartheta
 \quad(P\in\mathcal P,\ a/b\le\vartheta\le1),
\]
with
\[
 M_{\rm par}\lesssim b/a
 \quad\text{for the parent call},\qquad
 M_{\rm ch}\lesssim(b/a)^2
 \quad\text{for the normalized child call}.
\]
These are exactly the two bounds used in source equations (45)--(46) of
\cite{GWZ2026}.  Then replacing source
Proposition~5.1 by \cref{thm:mainfact} gives
\[
 \mu(\T,Y)\lesslog
 \delta^{-\varepsilon}C_F(\T)^{1-\beta/2}
 \left(\frac ab\right)^{3\beta/2}\delta^{-2\beta}
 \bigl(\delta^2|\T|\bigr)^{1-\beta/2}.
\tag{9.1}
\]
Thus the normalization and conclusion are exactly those of source
Proposition~6.6(A), up to the explicit cellular logarithms and fixed
enlargements.
\end{proposition}

\begin{proof}
\medskip\noindent\emph{Setup and retained density.}
Let \(N_W=|\T_W|\), let \(M_W\) be the retained child shading mass in
\(W\), and let \(N_0\le N_W\le C_NN_0\) be the fixed source count level.
P1 gives \(\sum_WM_W\gtrlog M(Y)\), and the retained labels form a
subfamily of the source partition.  Hence
\[
 \frac{\sum_WN_W\lambda(\T_W,Y_1)}{\sum_WN_W}
 =\frac{\sum_WM_W}{|T|\sum_WN_W}
 \gtrlog\lambda(\T,Y).
\tag{9.2}
\]
Choose \(W_0\) attaining this weighted average, so its density is at least
\(\delta^{\eta+o(1)}\).  Take \(\eta\) below the threshold
\(\eta_{6.4}(\varepsilon,\beta)\) from
\cite[Lemma~6.4, p.~15]{GWZ2026}, with strict reserve for the cellular losses
and the scale-variant convention of \cite[Remark~3.6, p.~6]{GWZ2026}.

\medskip\noindent\emph{Child source call.}
The first row of the interface table is realized by the following checks.
Normalizing \(W_0\) gives widths
\[
 a_{\rm ch}=b^{-1}\delta,\qquad b_{\rm ch}=a^{-1}\delta,
\qquad 0<a_{\rm ch}\le b_{\rm ch}\le1.
\]
Here \(b_{\rm ch}=\delta/a\le b_0\), and the auxiliary scale is
\(\tau_{\rm src}=\delta\le a_{\rm ch}\).  After decreasing \(\delta_0\),
the affine-invariant density from (9.2) exceeds
\(\delta^{\eta_{6.4}}\).  The shading restriction leaves the complete child
family, its convex Frostman tests, and the assumed count
\(|\mathcal P[P_\vartheta]|\le M\vartheta\), for every base
\(P\in\mathcal P\) and every \(\vartheta\in[a/b,1]\), unchanged.

Properly color the assumed child source-conflict graph.  Its \(O(1)\) colors
are source-essentially-distinct; the heaviest color is the called
family, and complete-containment counts do not increase.  If
\(\lambda_{\rm ch}\) and \(C_{F,{\rm ch}}\) are the
pre-color density and normalized Frostman constant,
\cref{lem:color-normalized} gives
\[
 \lambda_{{\rm ch},c}\ge H_{\rm src}^{-1}\lambda_{\rm ch},
 \qquad
 C_{F,{\rm ch},c}\le
 H_{\rm src}\lambda_{\rm ch}^{-1}C_{F,{\rm ch}}.
\tag{9.2a}
\]
Here \(\lambda_{\rm ch}\ge\delta^{\eta+o(1)}\).  With source-output allowance
\(\varepsilon/4\), choose the input exponents and \(\delta_0\) so that
\[
 \bigl(H_{\rm src}\lambda_{\rm ch}^{-1}
 C_{F,{\rm ch}}\bigr)^{1-\beta/2}
 \le\delta^{-\varepsilon/8}
\tag{9.2a$'$}
\]
The fixed enlargement and recovery factor \(H_{\rm src}\) is also at most
\(\delta^{-\varepsilon/8}\).  Thus the color and enlargement costs are
reserved, and (9.0e$'$) returns the source estimate to the uncolored family.
The call data are
\[
 \begin{gathered}
 s_{\rm ch}=\delta/b,\qquad t_{\rm ch}=\delta/a,\qquad
 \tau_{\rm src}=\delta\le s_{\rm ch},\qquad
 \mathcal F_{\rm ch}=\text{the selected child color},\\
 \lambda(\mathcal F_{\rm ch})\ge H_{\rm src}^{-1}\lambda_{\rm ch},\qquad
 C_F(\mathcal F_{\rm ch})\le
 H_{\rm src}\lambda_{\rm ch}^{-1}C_{F,{\rm ch}},\qquad
 M_{\rm ch}\lesssim(b/a)^2.
 \end{gathered}
\tag{9.2d0}
\]
The normalized thick tests pull back to the source convex class, and replacing
\(W_0\) by its source L\"owner rectangle costs a dimensional dilation.
Applying the child specialization
(9.I-46) of \cref{prop:corrected-s6-interface}, and then using (9.0e$'$) and
(9.2a$'$), gives the following estimate on the uncoloured local family:
\[
 \mu(\T_{W_0},Y_1)
 \lesssim \delta^{-\varepsilon/2}
 \left(\frac ab\right)^{1-\beta}
 \left(\frac\delta a\right)^{-2\beta}
 \left(\frac{\delta^2}{a^2}|\T_{W_0}|\right)^{1-\beta/2}.
\tag{9.2d}
\]
Average multiplicity is affine invariant, and the cardinality is labelled.

\par\medskip
\noindent\emph{Parent source call.}\enspace
Here (8.4) supplies the positive parent density and \(b\le b_0\) the source
range.  Let \(\T'\) be the complete child labels retained in the cellular
parents.  Their shading mass is at least \(L_{\rm cell}^{-1}M(Y)\), where
\(L_{\rm cell}=\delta^{-o(1)}\), and is at most their carrier denominator;
hence
\[
 \frac{|\T'|}{|\T|}\ge L_{\rm cell}^{-1}\lambda(\T,Y).
\tag{9.2b}
\]
The restricted normalized fine Frostman constant is therefore at most
\(L_{\rm cell}\lambda(\T,Y)^{-1}C_F(\T)\).  Apply
\cref{lem:normalized-ancestry} to
\(I_{\rm par}=\{(W,T):T\in\T_W'\}\).  The source assignment is unique on
fine labels, and the selected parent-count level lies in
\([N_0,C_NN_0]\); hence
\[
 C_F(\W')\le
 C_NL_{\rm cell}\lambda(\T,Y)^{-1}C_F(\T).
\tag{9.2c}
\]
This factor is \(\delta^{-O(\eta)-o(1)}\).  Fixed dilation preserves the
source count \(N(\theta)\lesssim\theta(b/a)\).  Properly color the assumed
parent source-conflict graph, and let \(\widetilde\W'_c\) be the
shading-mass-heavy color; its color number is \(H_{\rm src}=O(1)\).  With the
density
\(\lambda_{\rm par}\) from (8.4), the parent call data are
\[
 \begin{gathered}
 s_{\rm par}=a,\qquad t_{\rm par}=b,\qquad
 \tau_{\rm src}=\delta\le s_{\rm par},\qquad
 \mathcal F_{\rm par}=\widetilde\W'_c,\qquad
 M_{\rm par}\lesssim b/a,\\
 \lambda(\mathcal F_{\rm par})\ge
 H_{\rm src}^{-1}\lambda_{\rm par},\\
 C_F(\mathcal F_{\rm par})\le
 H_{\rm src}\lambda_{\rm par}^{-1}
 C_NL_{\rm cell}\lambda(\T,Y)^{-1}C_F(\T).
 \end{gathered}
\tag{9.2e0}
\]
This verifies the second row of the interface table, using (9.2c) and
\cref{lem:color-normalized}.  With source-output allowance
\(\varepsilon/4\), impose
\[
 \bigl(H_{\rm src}\lambda_{\rm par}^{-1}C_NL_{\rm cell}
 \lambda(\T,Y)^{-1}\bigr)^{1-\beta/2}
 H_{\rm src}\le\delta^{-\varepsilon/4}.
\tag{9.2e$'$}
\]
Choose the input exponent and \(\delta_0\) with strict reserve so that the
parent-color density is at least \(\delta^{\eta_{6.4}}\).  The specialization
(9.I-45) of \cref{prop:corrected-s6-interface}, with
\(\tau_{\rm src}=\delta\), followed by recovery through (9.0e$'$), gives
\[
 \mu(\widetilde\W',Z)
 \lesssim\delta^{-\varepsilon/2}C_F(\T)^{1-\beta/2}
 \left(\frac ab\right)^{1-\beta/2}b^{-2\beta}
 \bigl(b^2|\widetilde\W'|\bigr)^{1-\beta/2}.
\tag{9.2e}
\]

\medskip\noindent\emph{Combination.}
Finally, P4--P5 give, for the chosen \(W_0\),
\[
 \mu(\T,Y)\lesslog
 \mu(\widetilde\W',Z)\mu(\T_{W_0},Y_1),
\tag{9.3}
\]
which is source equation (44) with the cellular loss.  Affine changes leave
average multiplicity invariant, and the carrier and label denominators in
(45)--(46) are the ones displayed above.  Since the retained parents form a subfamily of
the source partition and \(N_W\in[N_0,C_NN_0]\),
\[
 |\widetilde\W'|\,|\T_{W_0}|\le C_N|\T|.
\]
This cardinality bound and the scale factors in
(9.2d)--(9.2e) combine explicitly as
\[
 \left(\frac ab\right)^{1-\beta/2+1-\beta+2\beta}
 \left(\frac ba\right)^{2-\beta}
 =\left(\frac ab\right)^{3\beta/2}.
\tag{9.3a}
\]
The factor \(\delta^{-2\beta}(\delta^2|\T|)^{1-\beta/2}\)
is the remaining child--parent volume product.  Substitution in (9.3) proves
(9.1), with the two \(\varepsilon/2\) allowances adding to
\(\varepsilon\).
\end{proof}

\begin{proposition}[Small-middle cellular replacement in GWZ Proposition~6.6(B)]
\label{prop:s6B}
Fix \(0<\beta\le1\), \(\varepsilon>0\), and a loss profile
\(\mathfrak p\), and assume
\(K_{KT}(\beta)\).  There exist
\(\eta=\eta(\varepsilon,\beta)>0\),
\(b_0=b_0(\varepsilon,\beta)>0\), and
\(\delta_0=\delta_0(\varepsilon,\beta,\mathfrak p)>0\) such that the following holds for
\(0<\delta\le\delta_0\).  Here \(b_0\) is chosen to be a sufficiently small
structural multiple of the minimum of the original source thresholds, so every
fixed enlargement used below of a plank with middle width at most \(b_0\)
still lies in the stated range of the corresponding source lemma.  Let
\((\T,Y)\) be uniform in the sense of
\cite[Definition~2.2, p.~4]{GWZ2026}, with all source uniformity,
comparability, and normalized-interface losses governed by \(\mathfrak p\), and with
\(\lambda(\T,Y)\ge\delta^\eta\).  Suppose that
\(\delta\le\rho\le a\le b\le b_0\), \(\delta/a\le b_0\), and that the
coarse \(\rho\)-tubes factor,
in the sense of \cite[Definition~4.2]{GWZ2026}, through a nonempty labelled
family \(\mathcal W\) of \(a\times b\times1\) rectangular parents satisfying
\(\Delta_{\max}(\mathcal W)\lesssim1\), and assume that its normalized
parent labels are source-essentially-distinct up to uniformly bounded
representation fibres.  Assume that the fine tubes are
assigned to essentially distinct coarse \(\rho\)-tubes, each fine tube has
\(O(1)\) coarse representations, and the branching numbers lie in
\([m,R_{\rm src}m]\), where
\(1\le R_{\rm src}\le\delta^{-\omega_{\rm br}(\delta)}\) for a branching
modulus \(\omega_{\rm br}\) controlled by \(\mathfrak p\).  Finally assume
that every parent has a nonempty coarse-child family and that the coarse
children in every parent \(B\in\mathcal W\) satisfy
\[
 \frac{\#\{S:S\subset K\}}{\#\{S:S\subset B\}}
 \lesslog\frac{|K|}{|B|}
 \qquad(K\subset B\ \mathrm{convex}).
\]
Then the
coarse/thickened-child alternative of \cref{thm:mainfact} gives
\[
 \mu(\T,Y)\lesslog
 \delta^{-\varepsilon}\Delta_{\max}(\T)^{1-\beta}
 \left(\frac ab\right)^\beta|\T|^\beta.
\tag{9.4}
\]
This is the source Part~(B) conclusion.  It neither requires nor implies a
fine-descendant all-convex Frostman condition.
\end{proposition}

\begin{proof}
\medskip\noindent\emph{Setup and labelled ancestry.}
Set \(\varepsilon_*:=\varepsilon/2\).  Every source-output and recovery
allowance uses \(\varepsilon_*\); the remainder pays the final subpower
cardinality factor.  Source Part~(B) factors the coarse \(\rho\)-tubes through
the rectangular parents, as licensed by Remark~5.3
\cite[proof of Proposition~6.6(B), pp.~24--25]{GWZ2026}.  Its coarse-family
input is
\[
 \frac{\#\{S\in\T_{\rho,B}:S\subset K\}}{|\T_{\rho,B}|}
 \lesslog \frac{|K|}{|B|},
\tag{9.1a}
\]
For every source containment \(T\subset S\), form the labelled copy
\((S,T)\), with the carrier and shading of \(T\), and assign it to \(S\).
Source bounded ancestry and branching comparability give \(m\ge1\) and
\[
 m\le n_S:=\#\{(S,T):T\subset S\}
       \le R_{\rm src}m,\qquad R_{\rm src}=\delta^{-o(1)}.
\tag{9.1b}
\]
Each fine label has between one and \(\kappa_{\rm cont}=O(1)\)
representations.  Thus labelled mass and denominator change by factors in
\([1,\kappa_{\rm cont}]\), the physical union is unchanged, and the coarse
assignment is unique on representation labels.

Source factoring supplies the denominator comparison through
\(N_\rho\ge1\) and \(C_\rho=\delta^{-o(1)}\):
\[
 N_\rho\le |\T_{\rho,B}|\le C_\rho N_\rho
 \qquad(B\in\mathcal W).
\tag{9.1b$'$}
\]
If \(D_B:=|T|\sum_{S\in\T_{\rho,B}}n_S\) is the represented fine-carrier
denominator in \(B\), then (9.1b)--(9.1b$'$) give
\[
 mN_\rho|T|\le D_B
 \le R_{\rm src}C_\rho mN_\rho|T|.
\tag{9.1b$''$}
\]
Thus parent denominators are comparable by
\(R_{\rm src}C_\rho=\delta^{-o(1)}\), which verifies the ratio condition
(6.6a) of \cref{lem:denominator-average} and hypothesis (F2) of
\cref{thm:mainfact}.

For the later source calls we use the labelled form of source Lemma~6.1 at
\(0<\tau\le s\le t\).  Its low-multiplicity branch and the angular selections, refinements, and
pigeonholing in source Lemmas~6.11 and 6.13 run on the label index set, with
angles resolved below at \(s/t\).  Assign every surviving label to the
essentially distinct thick proxy of source Definition~6.10 and dyadically
regularize the proxy-fibre size \(N\).  Each retained proxy contains
\(\sim N\) labels and has volume \(\asymp(t\vartheta/s)\) times the original
plank volume.  Fibre counting, together with the monotonicity in source
Remark~3.3(B), gives
\[
 \Delta_{\max}(\mathcal P_{\vartheta,S})
 \lesssim \frac{t\vartheta}{sN}\Delta_{\max}(\mathcal P),
 \qquad
 |\mathcal P_{\vartheta,S}|\lesssim N^{-1}|\mathcal P_S|.
\tag{9.1c}
\]
Source Lemma~6.13 also gives the proxy shading density
\(\gtrsim s^{4\eta_{\rm in}}\ge\tau^{4\eta_{\rm in}}\); choosing its input
exponent with strict reserve verifies the scale-variant density threshold.
Only the affine image of the proxy family enters source Lemma~3.7.  Using
\(|\mathcal P_S|\le\tau^{-\eta_{\rm in}}\vartheta^\gamma|\mathcal P|\) and
\(\vartheta\ge s/t\), the complete cancellation is
\[
 \frac{sN}{t\vartheta}
 \left(\frac{t\vartheta}{sN}\Delta_{\max}(\mathcal P)\right)^{1-\beta}
 \left(N^{-1}|\mathcal P_S|\right)^\beta
 \lesssim \tau^{-\eta_{\rm in}\beta}
 \Delta_{\max}(\mathcal P)^{1-\beta}
 \left(\frac st\right)^{\gamma\beta}|\mathcal P|^\beta.
\]
Thus \(N\) cancels exactly and the remaining auxiliary-scale power is
absorbed by the source-output reserve.  This proves equation~(29), with all
multiplicities, cardinalities, and \(\Delta_{\max}\) counted on labels.
Thus the analytic inequality in (9.I-49), and its \(\gamma=0\) specialization
(9.I-48), hold for the labelled families below; the dyadic losses are
governed by \(\mathfrak p\).

\medskip\noindent\emph{Coarse parent estimate.}
Since \(\rho\le a\), (9.1a) and (9.1b) verify
\cref{lem:coarse-cordoba} with \(R_{\rm br}=R_{\rm src}\), giving the local
estimate (6.6d).  Combining it with the denominator comparison (9.1b$''$) in
\cref{cor:density} gives the coarse version of (6.10), including arbitrary
measurable fine shadings, clipping, and fixed enlargement, with factor
\(R_{\rm src}^{-3}=\delta^{o(1)}\).

\medskip\noindent\emph{Child source call.}
The weighted averaging (9.2), performed on the complete local representation
families, chooses a useful inner parent \(W\).  Let \(\mathcal R_W\) be its
representation family and put
\[
 \lambda_W:=
 \frac{\sum_{(S,T)\in\mathcal R_W}|Y'(S,T)|}
 {|\mathcal R_W|\,|T|}.
\tag{9.4a}
\]
Then \(\lambda_W\ge\delta^{\eta+o(1)}\), with the stronger normalized-scale
lower bound used in the source scale-variant call.  We retain the complete
labelled representation family.  Since each fine label has at most
\(\kappa_{\rm cont}\) representations,
\[
 \lambda(\mathcal R_W,Y')=\lambda_W,\qquad
 \Delta_{\max}(\mathcal R_W)\le
 \kappa_{\rm cont}\Delta_{\max}(\T).
\tag{9.4b}
\]

For the restricted slab-test condition, let \(\Phi_W\) normalize \(W\).
For \(\vartheta\in[a/b,1]\) and a source test slab \(\Sigma\) from
(9.I-K0), put
\[
 K_{\vartheta,\Sigma}:=\Phi_W^{-1}(\Sigma).
\]
Let \(\mathcal K_{\rm ch}\) be the pullbacks having nonempty equation-(28)
subfamily; the omitted tests are trivial.  Fix a counted normalized child
plank \(P\), and choose unit normals \(n_P,n_\Sigma\) to its two-long-axis
plane and to \(\Sigma\).  In principal coordinates write
\(D=\Phi_W=\operatorname{diag}(a^{-1},b^{-1},1)\).  The plane calculation on
pp.~24--25 of \cite{GWZ2026}, before its final semiaxis comparison, shows that
the normalized long plane of \(P\) contains a unit vector within \(O(a/b)\)
of \(e_1\), and hence \(|(n_P)_1|\lesssim a/b\).  The equation-(28) angular
condition permits signs such that
\[
 |n_\Sigma-n_P|\lesssim\vartheta,
 \qquad |(n_\Sigma)_1|\lesssim a/b+\vartheta\lesssim\vartheta.
\]
Since \(a/b\le\vartheta\), \(b\le1\), and \(|n_\Sigma|=1\),
\[
 a|D^Tn_\Sigma|\lesssim\vartheta,
 \qquad a\|D\|=1.
\]
A fixed inner copy of the L\"owner ellipsoid of the normalized
\(\vartheta\times1\times1\) slab therefore contains a translate of
\(D B(0,ca)\): its normal support is at most
\(ca|D^Tn_\Sigma|\lesssim c\vartheta\), while its tangential radius is at
most \(ca\|D\|\le c\).  Pulling back gives a translated \(ca\)-ball in
\(K_{\vartheta,\Sigma}\).  Writing \(a'(K_{\vartheta,\Sigma})\) for its
shortest L\"owner semiaxis, we consequently have
\[
 a'(K_{\vartheta,\Sigma})\gtrsim a\ge\rho.
\]
Fix \(K=K_{\vartheta,\Sigma}\), and let \(z_K\) be the centre of the inball
just constructed.  For a fine \(T\subset K\), the full-carrier
conclusion of \cref{lem:standard-capped-ancestry} gives
\[
 S\subset N_{C\rho}(T)\subset N_{C\rho}(K).
\]
Thus \cref{lem:convex-parallel-body} gives
\[
 K^+:=K+B(0,C\rho)
 \subset z_K+C_+(K-z_K),
 \qquad |K^+|\le C_+^3|K|.
\tag{9.4c}
\]
Every counted coarse child lies in both \(K^+\) and \(W\).  Hence (9.1a),
branching comparability, and \(|K^+|\lesssim|K|\) give
\[
 \frac{\#\{(S,T):T\subset K\}}{\#\{(S,T):T\subset W\}}
 \le R_{\rm src}
 \frac{\#\{S:S\subset K^+\cap W\}}{\#\{S:S\subset W\}}
 \lesslog R_{\rm src}\frac{|K|}{|W|}.
\tag{9.5}
\]
This is the normalized bound for the whole representation family on
\(\mathcal K_{\rm ch}\); explicitly,
\[
 \frac{\#\{(S,T)\in\mathcal R_W:T\subset K\}}
      {|\mathcal R_W|}
 \lesslog R_{\rm src}\frac{|K|}{|W|}
 \qquad(K\in\mathcal K_{\rm ch}).
\tag{9.5a}
\]
Let \(\mathcal F_{\rm ch}:=\Phi_W(\mathcal R_W)\), and interpret
\((\mathcal F_{\rm ch})_\Sigma\) exactly as the angularly filtered source
equation~(28) subfamily.  It lies in the normalized numerator of (9.5a), and
the affine Jacobian cancels, so
\[
 \frac{|K_{\vartheta,\Sigma}|}{|W|}
 \lesssim \frac{|\Sigma|}{|\Phi_W(W)|}
 \lesssim\vartheta,
 \qquad
 \frac{|(\mathcal F_{\rm ch})_\Sigma|}{|\mathcal F_{\rm ch}|}
 \lesslog R_{\rm src}\vartheta.
\tag{9.5a$'$}
\]
This is precisely the restricted class in (9.I-K0), whose necessity was
pointed out by Zeraoulia Rafik \cite{Zeraoulia2026Local}.  The normalized
labelled image has dimensions
\((\delta/b)\times(\delta/a)\times1\); density, multiplicity, and absolute
complete-containment density are affine invariant.
Put \(a_{\rm ch}=\delta/b\), \(b_{\rm ch}=\delta/a\), and let
\(\eta_{6.1}\) be the exponent returned by source Lemma~6.1 with an
\(\varepsilon_*/4\) output allowance.  Use its scale-variant form, licensed by
\cite[Remark~3.6 and the note after Proposition~6.6]{GWZ2026}, with
\(\tau_{\rm src}=\delta\le a_{\rm ch}\).  Choose the input exponent first and
then \(\delta_0\) so that
\[
 R_{\rm src}\log^{C}(2/\delta)\le \delta^{-\eta_{6.1}},
 \qquad
 \lambda_W\ge\delta^{\eta_{6.1}}.
\tag{9.5b}
\]
The strict density reserve and
\(R_{\rm src}\le\delta^{-\omega_{\rm br}(\delta)}\) permit (9.5b); moreover
\(\tau_{\rm src}=\delta\) and \(b_{\rm ch}=\delta/a\le b_0\).

The seven interface fields for the child call are therefore
\[
 \begin{gathered}
 s_{\rm ch}=\delta/b,\qquad t_{\rm ch}=\delta/a,\qquad
 0<s_{\rm ch}\le t_{\rm ch}\le b_0,\qquad
 \tau_{\rm src}=\delta\le s_{\rm ch},\qquad \gamma=1,\\
 \mathcal F_{\rm ch}=\Phi_W(\mathcal R_W),\qquad
 \lambda(\mathcal F_{\rm ch})=\lambda_W\ge\delta^{\eta_{6.1}},\\
 |(\mathcal F_{\rm ch})_\Sigma|
 \le\delta^{-\eta_{6.1}}\vartheta|\mathcal F_{\rm ch}|
 \quad(s_{\rm ch}/t_{\rm ch}\le\vartheta\le1,
       \ \Sigma\ \text{as in (9.I-K0)}),\\
 \Delta_{\max}(\mathcal F_{\rm ch})
 \le\kappa_{\rm cont}\Delta_{\max}(\T),\qquad
 s_{\rm ch}/t_{\rm ch}=a/b.
 \end{gathered}
\tag{9.5c0}
\]
Thus (9.5c0) verifies the child Lemma~6.1 row, including the equation-(28)
conditions.  Apply the labelled form of source Lemma~6.1 proved at (9.1c),
and absorb the fixed factor \(\kappa_{\rm cont}^{1-\beta}\):
\[
 \mu(\mathcal R_W,Y')
 \lesssim\delta^{-\varepsilon_*/2}
 \Delta_{\max}(\T)^{1-\beta}
 \left(\frac ab\right)^\beta|\mathcal R_W|^\beta.
\tag{9.5c}
\]
The complete original family is used, so there is no color-recovery factor.

\medskip\noindent\emph{Parent source call.}
For the parent call, \(b\le b_0\), and we again retain the complete labelled
family.  Its call data are
\[
 \begin{gathered}
 s_{\rm par}=a,\qquad t_{\rm par}=b,\qquad
 \tau_{\rm src}=\delta\le s_{\rm par},\qquad \gamma=0,\qquad
 \mathcal F_{\rm par}=\W',\\
 \lambda(\mathcal F_{\rm par})=\lambda_{\rm par},
 \qquad \Delta_{\max}(\mathcal F_{\rm par})\lesssim1.
 \end{gathered}
\tag{9.5d0}
\]
Absolute concentration is monotone under the retained parent subfamily, and the cellular
density exceeds \(\delta^{\eta_{6.1}}\) after decreasing \(\delta_0\).
Hence the labelled \(\gamma=0\) form proved at (9.1c) gives
\[
 \mu(\W',Z)\lesssim\delta^{-\varepsilon_*/2}|\W'|^\beta.
\tag{9.5d}
\]
\medskip\noindent\emph{Combination.}
Finally the cellular product is source equation (47).  Since the
retained parents form a subfamily and (9.1b)--(9.1b$''$) give comparable
represented denominators,
\[
 |\W'|\,|\mathcal R_W|
 \lesssim \kappa_{\rm cont}R_{\rm src}C_\rho|\T|
 =\delta^{-o(1)}|\T|.
\]
Multiplying (9.5c)--(9.5d) through (47) costs
\(\delta^{-\varepsilon_*}\), together with the displayed cardinality factor.
After all parameters and the loss profile \(\mathfrak p\) are fixed,
decrease \(\delta_0(\varepsilon,\beta,\mathfrak p)\) so that
\[
 (\kappa_{\rm cont}R_{\rm src}C_\rho)^\beta
 \le \delta^{-(\varepsilon-\varepsilon_*)}.
\]
This is possible because the product is \(\delta^{-o(1)}\).  Thus the total
cost is at most \(\delta^{-\varepsilon}\), apart from the declared logarithmic
factors, and (9.4) follows.  The child call has used the
restricted slab-test condition (9.5a$'$), as encoded in (9.5c0).
\end{proof}

\begin{remark}[Scope of the two source interfaces]
\label{rem:small-middle-scope}
\Cref{prop:s6A,prop:s6B} cover precisely the two small-middle interfaces for
which both the physical parent and normalized child middle widths lie below
the returned \(b_0\).  The slab estimate (B.2) is used through its stated
quadratic form.  In Section~10, Lemma~6.4 is invoked in the thick branch, where
\(a\ge\delta^{1-\tau}\) gives \(\delta/a\le\delta^\tau<b_0\) after decreasing
\(\delta_0\).  The large-middle branch (S1b) instead uses (10.15)--(10.16) and
\cref{thm:slab}, while (S1c) uses neither proposition nor a normalized
Lemma~6.1/6.4 child call.
\end{remark}

\begin{remark}[Source version]
The two calls use Lemmas~6.1 and 6.4 in the forms recorded in
\cref{prop:corrected-s6-interface}; Appendix~D records the corresponding v1
textual comparison.
\end{remark}

Within the scope fixed in \cref{rem:small-middle-scope}, these two
propositions complete the small-middle applications.  Their proof uses the
cellular theorem as a self-contained interface and is independent of the
multiscale synchronization in the next section.

\section{Cellular closure of the very-non-sticky reduction}
\label{sec:section9}

The very-non-sticky application represents fine tubes by tagged tubelets,
factors them through biased parents, regularizes and angularly refines the
parent--cell graph, and lifts the survivors to the original labels.  The
ledger below separates tubelet mass from represented fine mass and keeps
angular uniformity on its reference shading.  After synchronization and
cellularization, five geometric branches close in the parameter order of
\cref{sec:paramclose}.  The final application is conditional on
\(K_{KT}(\beta)\), \(K_F(\beta)\), (10.33), and (I1)--(I4); Appendix~D gives
the source correspondence.

\subsection{Reference shadings and the object hierarchy}

Fix an \(r_1\)-ball \(B\), with
\(r_1=\delta^{\varepsilon_{\rm scal}}\).  The objects and their three principal
transitions are summarized here; the tagged equations below provide the
quantitative definitions.
\begin{center}
\scriptsize
\begin{tabularx}{\textwidth}{@{}>{\raggedright\arraybackslash}p{1.55cm}>{\raggedright\arraybackslash}p{1.7cm}>{\raggedright\arraybackslash}p{3.0cm}>{\raggedright\arraybackslash}X@{}}
\toprule
object & carrier level & creating operation & retained property / later use\\
\midrule
\(\bar Y\) & fine tubes & source angular uniformization & fixed reference for cap and direction uniformity\\
\(Y^{\rm sc}\) & fine tubes & all-discrete-scale refinement & source scale shadings and comparisons\\
\(Y^{\rm pre}\to Y^0\) & represented fine tubes & canonical levels and whole-tag level, then good-ball deletion & representation coverage and the mass bound (10.1)\\
\(Y_B\) & canonical tubelets & good-ball restriction of \(\widetilde Y_B\) & reference shading for the local cellular input\\
\(Y_{\rm reg}\to Y_1,Z\) & tubelets / parents & joint cellular factoring, then saturated local-volume selection & compatibility and synchronized multiplicities\\
\((Y_2,Z^\angle)\) & tubelets / parents & ANGLE-CELL and one lift & resolved angle and retained tubelet mass\\
\(\widehat Y\to Y^*\) & fine tubes & global fine level and bias, then retained tubelet incidences & represented physical fine mass and final lift\\
\bottomrule
\end{tabularx}
\end{center}

Apply \cref{lem:source-canonical} to the fixed
\(Y^{\rm sc}\subset\bar Y\).  The induced physical shading
\(Y^{\rm pre}\) is defined by (10.7X) after the canonical multiplicity and
density levels (10.7T)--(10.7W) and the whole-tag outer-ball level
(10.7W$'$).  Every
\((T,x)\), \(x\in Y^{\rm pre}(T)\), then has a labelled canonical
representation \((B,T_B)\) with \(T\in\T(T_B)\) and
\(x\in\widetilde Y_B(T_B)\).  After the good-ball deletion define
\[
 Y^0(T):=Y^{\rm pre}(T)\cap
 \bigcup_{\substack{(B,T_B):\,T\in\T(T_B)}}Y_B(T_B).
\tag{10.0}
\]
Thus \(Y^0\) is the represented fine shading before biased factoring, and
\[
 \sum_T|Y^0(T)|\ge c_{\rm pre}\sum_T|\bar Y(T)|;
\tag{10.1}
\]
more explicitly, if \(L_{\rm tag}\) is the number of classes in the whole-tag
outer-ball level,
then (10.7W)--(10.7X), (10.7W$'$), and (10.1d) permit
\[
 c_{\rm can}:=
 \frac{1}{2\kappa_{\rm inc}L_{\rm mult}L_{\rm dens}L_{\rm tag}},
 \qquad
 c_{\rm pre}:=\frac{c_{\rm sc}c_{\rm can}}{4\kappa_{\rm inc}}.
\tag{10.1$'$}
\]
All four level counts are logarithmic under the polynomial-complexity and
positive-density hypotheses, so \(c_{\rm pre}^{-1}=\delta^{-o(1)}\).  The
good-ball restriction of \(\widetilde Y_B\) defines the fixed tubelet shading
\(Y_B\); joint factoring gives \(Y_{\rm reg}\subset Y_B\) and \(Z\), the
saturated local-volume selection gives \(Y_1\subset Y_{\rm reg}\), and
ANGLE-CELL gives \(Y_2\subset Y_1\) and \(Z^\angle\).  The global fine-multiplicity level
followed by the biased tubelet subset defines \(\widehat Y\subset Y^0\); if
the latter retains represented fraction \(c_{\rm bias}\), set
\(c_{\rm setup}=c_{\rm pre}c_{\rm bias}\).  Finally \(Y^*\subset\widehat Y\)
is represented by the retained tubelet incidences.  Cap and direction
uniformity remain attached to \(\bar Y\).

For \(x\in Y^{\rm pre}(T)\), let
\[
 \mathcal R_{\rm pre}(T,x)
 :=\{(B,T_B):T\in\T(T_B),\ x\in\widetilde Y_B(T_B)\},
 \qquad
 \kappa_{\rm inc}:=\sup_{T,x}|\mathcal R_{\rm pre}(T,x)|.
\tag{10.0a}
\]
The direct tagged-incidence assignment of \cref{lem:source-canonical} gives
\(1\le|\mathcal R_{\rm pre}(T,x)|\le\kappa_{\rm inc}=O(1)\).
After the good-ball deletion put
\[
 \mathcal R_0(T,x)
 :=\{(B,T_B):T\in\T(T_B),\ x\in Y_B(T_B)\}.
\tag{10.0b}
\]
Definition (10.0) gives
\[
 1\le|\mathcal R_0(T,x)|\le\kappa_{\rm inc}
 \qquad(x\in Y^0(T)).
\tag{REP}
\]

Thus source (97) is read on \(Y_B\), whereas (100)--(102) compare \(Y^*\)
with \(\bar Y\); direction uniformity is not assigned to \(Y^0\) or \(Y^*\).

\subsection{Good balls and global fine-multiplicity synchronization}

\begin{lemma}[Reordered good-ball selection]
\label{lem:goodball-reorder}
Let \(\widetilde Y_B(T_B)\) be the canonical tubelet shading immediately
after (10.7T)--(10.7W) and the whole-tag outer-ball selection.  Suppose
\(0<2\eta_{\rm src}<\eta\) and
\[
 |\widetilde Y_B(T_B)|\gtrsim\delta^{2\eta_{\rm src}}|T_B|,
 \qquad
 m(T_B)\le g_{T_B}(x)<2m(T_B)
 \quad(x\in\widetilde Y_B(T_B)).
\tag{10.1a}
\]
The source good-\(\delta\)-ball operation may be performed at this point,
before biased parent and dimension selection.  It produces a subshading
\(Y_B(T_B)\), retains a fixed fraction of every tubelet's shaded and
represented fine-incidence mass for sufficiently small \(\delta\), and
preserves the second assertion in (10.1a) on \(Y_B(T_B)\).  With \(Y^0\)
defined by (10.0), it also gives
\[
 M(Y^0)\ge(4\kappa_{\rm inc})^{-1}M(Y^{\rm pre}).
\tag{10.1d}
\]
All subsequent
biased-parent, dimension, and child-count selections may be made by whole
tubelet or parent labels, so the same pointwise level persists until cellular
factoring.
\end{lemma}

\begin{proof}
Cover each tubelet by a fixed bounded-overlap family of \(\delta\)-balls.
The balls below threshold contribute at most
\(C_{\rm ov}c\delta^\eta|T_B|\), which is less than half of (10.1a) because
\(\eta>2\eta_{\rm src}\).  Since \(g_{T_B}\in[m(T_B),2m(T_B))\) on the old
shading, writing \(E=\widetilde Y_B(T_B)\) and \(E'=Y_B(T_B)\) gives
\[
 \int_{E'}g_{T_B}\ge m(T_B)|E'|
 \ge\tfrac12m(T_B)|E|,
 \qquad
 \int_Eg_{T_B}\le2m(T_B)|E|,
\tag{10.1c}
\]
so at least one quarter of represented mass remains.  With
\(\mathcal D_{\rm good}(T_B)\) the passing balls, define
\[
 Y_B(T_B):=\widetilde Y_B(T_B)\cap
       \bigcup_{D\in\mathcal D_{\rm good}(T_B)}D.
\tag{10.1b}
\]
Thus every retained point lies in a good ball and keeps its dyadic level.
Summing (10.1c) over labelled representations and using the
\(\kappa_{\rm inc}\)-to-one physical projection proves (10.1d).

Bias and dimension choices use only carrier data, so
\cref{lem:carrier-shading-commutation} moves them after this deletion.  The
global \(m_{\rm fine}\)-level must still precede the biased subset; all later
child-count operations delete whole labels and preserve its pointwise range.
\end{proof}

Thus, without an occupancy pigeonhole,
\[
 |D\cap Y_B(T_B)|\gtrsim \delta^\eta\delta^3
 \quad\text{for every selected ball }D.
\tag{10.2}
\]
This is the quantified replacement for source equation (94).  Accordingly,
the global \(m_{\rm fine}\)-level is selected before any occupancy class based
on unweighted tubelet mass.

For a tubelet \(T_B\), put
\[
 g^0_{T_B}(x)=\#\{T\in\T(T_B):x\in Y^0(T)\}.
\tag{10.3}
\]
Let \(g^{\rm pre}_{T_B}\) denote the canonical count defined in (10.7S),
before the good-ball deletion.  On the retained tubelet shading one has the
identity
\[
 g^{\rm pre}_{T_B}(x)=g^0_{T_B}(x)
 \qquad(x\in Y_B(T_B)).
\tag{10.3a}
\]
Indeed, if \(T\) is counted on the left, then
\(x\in Y^\circ(T)\cap B\) and
\(x\in Y_B(T_B)\subset\widetilde Y_B(T_B)\); hence (10.7X) places
\(x\in Y^{\rm pre}(T)\), and the same retained representation in (10.0)
places \(x\in Y^0(T)\).  Conversely, if \(T\) is counted on the right, then
\(Y^0(T)\subset Y^{\rm pre}(T)\subset Y^\circ(T)\), while
\(x\in Y_B(T_B)\subset B\), so \(T\) is counted on the left.  We write
\(g_{T_B}\) for this common restriction to \(Y_B(T_B)\).

The canonical construction (10.7S)--(10.7U) selects one dyadic level of
\(g^{\rm pre}_{T_B}\) and induces the corresponding fine refinement before
any biased parent is chosen.  Identity (10.3a) transfers that level to
\(g^0_{T_B}\).  \Cref{lem:goodball-reorder} performs the one
intervening pointwise tubelet operation before our synchronization and
proves that its restriction preserves that level.  Consequently there is a number
\(m(T_B)\ge1\) such that
\[
 m(T_B)\le g_{T_B}(x)<2m(T_B)
 \quad (x\in Y_B(T_B)).
\tag{10.4}
\]
What is missing is comparability between different tubelets.

We separate the global incidence constant in (10.0a) from the following local
label constant.  Uniformly over tagged outer balls \(B\), put
\[
 \kappa_{\rm loc}
 :=\sup_{B,T}\#\{T_B\in\T_B:T\in\T(T_B)\}.
\tag{10.4b}
\]
Bounded overlap of the outer balls and the single-valued half-open canonical
assignment, with the fixed neighboring-cell convention, give
\(\kappa_{\rm inc}\le C\kappa_{\rm amb}=O(1)\) and
\(\kappa_{\rm loc}=O(1)\).  A fine \emph{label} may cross
\(\asymp r_1^{-1}\) different outer balls.  The uniform bounds concern instead
a fine incidence \((T,x)\), which has \(O(1)\) tagged representations, and a
fine label within a fixed tag, which has \(O(1)\) canonical tubelets.

\begin{lemma}[Global fine-multiplicity synchronization]
\label{lem:global-sync}
There is one dyadic value \(m_{\rm fine}\), selected after the good-ball
deletion and before biased factoring, for which the fine mass bound (10.6)
holds and every selected tubelet satisfies
\[
 m_{\rm fine}\le m(T_B)<2m_{\rm fine}.
\tag{10.4c}
\]
On every later whole-tubelet subshading \(A\), the same fixed level obeys
the identity (10.7).  If a subsequent whole-label selection retains a
fraction \(c\) of tubelet shading mass, it retains at least
\(c/(4\kappa_{\rm inc})\) of represented physical fine mass.  No spatial
selection made after synchronization is used to redefine \(m_{\rm fine}\).
\end{lemma}

\begin{proof}
Let \(N_{\max}\) dominate the fine multiplicities and set
\(L_m=1+\lceil\log_2N_{\max}\rceil\).  For dyadic \(m\), define
\[
 F_m=\sum_{(B,T_B):\,m\le m(T_B)<2m}
       \int_{Y_B(T_B)}g_{T_B}(x)\,dx .
\tag{10.5}
\]
By (REP), summing all dyadic layers gives the two-sided comparison
\[
 M(Y^0)\le\sum_mF_m\le\kappa_{\rm inc}M(Y^0).
\tag{10.5a}
\]
Choose \(F_{m_{\rm fine}}\ge M(Y^0)/L_m\) and retain physical incidences having
a representation in that layer.  By (REP), their mass is at least
\((\kappa_{\rm inc}L_m)^{-1}M(Y^0)\).  Biased factoring then retains the
represented fraction \(c_{\rm bias}=\delta^{o(1)}\), because tubelet shadings
have comparable measure and represented masses are comparable to
\(m_{\rm fine}|Y_B(T_B)|\).  The induced shading \(\widehat Y\) satisfies
\[
 \sum_T|\widehat Y(T)|
 \gtrsim (\kappa_{\rm inc}L_m)^{-1}
 c_{\rm setup}\sum_T|\bar Y(T)|.
\tag{10.6}
\]
For a later tubelet subshading \(A(T_B)\subset Y_B(T_B)\), define
\[
 Y_A(T):=Y^0(T)\cap
   \bigcup_{\substack{(B,T_B):\,T\in\T(T_B)}}A(T_B).
\tag{10.6a}
\]
Its labelled fine-incidence mass is
\[
 \begin{split}
 \labmass(A)
 &=\sum_{(B,T_B)}\sum_{T\in\T(T_B)}
   |Y^0(T)\cap A(T_B)|\\
 &=\sum_{(B,T_B)}\int_{A(T_B)}g_{T_B}(x)\,dx,
\end{split}
\tag{10.7}
\]
and (10.4), (10.4c) give the explicit two-sided estimate
\[
 m_{\rm fine}\sum_{(B,T_B)}|A(T_B)|
 \le \labmass(A)
 <4m_{\rm fine}\sum_{(B,T_B)}|A(T_B)|.
\tag{10.7a$'$}
\]
By (10.6a) and (REP), the physical fine mass satisfies the exact bridge
\[
 \kappa_{\rm inc}^{-1}\labmass(A)
 \le M(Y_A)\le\labmass(A).
\tag{10.7$'$}
\]
Whole-tubelet, descendant-label, or tag deletion preserves (REP) for the
physical shading defined by (10.6a); a common spatial restriction does as
well.  If \(A'\) retains a tubelet-mass fraction \(c\), then
\[
 \labmass(A')\ge m_{\rm fine}\sum|A'|
 \ge c\,m_{\rm fine}\sum|A|
 >\frac c4\labmass(A),
\]
and (10.7$'$) supplies \(\kappa_{\rm inc}^{-1}\).  This is the claimed
tubelet-to-fine mass bridge.
\end{proof}

\paragraph{Multiplicity-expanded tubelet labels.}
After the global value \(m_{\rm fine}\) has been fixed, replace every
selected geometric tubelet label \(T_B\) by the coincident labelled copies
\[
 \T_B^{\rm lab}
 :=\{(T_B,\ell):T_B\in\T_B,\ 1\le\ell\le m_{\rm fine}\},
 \qquad
 Y_B^{\rm lab}(T_B,\ell):=Y_B(T_B).
\tag{10.7a}
\]
Then
\[
\begin{aligned}
 M(\T_B^{\rm lab},Y_B^{\rm lab})
 &=m_{\rm fine}\sum_{T_B}|Y_B(T_B)|
 \asymp\labmass(Y_B),\\
 D(\T_B^{\rm lab})
 &=m_{\rm fine}\sum_{T_B}|T_B|,\qquad
 U(\T_B^{\rm lab},Y_B^{\rm lab})=U(\T_B,Y_B).
\end{aligned}
\tag{10.7b}
\]
Thus the expansion changes labelled mass and denominator but not the physical
union.  The required concentration transfer is the following separate
geometric statement.

\begin{lemma}[Canonical tubelets from tagged fine incidences]
\label{lem:source-canonical}
Let \(0<\delta\le r_1\le1/10\), let \(\mathcal B\) be a
bounded-overlap cover by radius-\(r_1\) balls, and let \(Y^\circ\) be a
shading of a labelled family \(\T\) of capped unit \(\delta\)-tubes.  Assume
that the balls cover \(U(\T,Y^\circ)\).  All structural constants are fixed
before the construction, in the following order: first the cap convention and
its ball-enlargement constant \(C_2\); then \(C_J>2C_2+2\); then a sufficiently
fine half-open net mesh \(c_{\rm net}\); and finally containment dilations
\(C_{\rm can}\) and \(C_\dagger\).  Thereafter an unadorned \(C\) denotes a
fixed maximum of these constants.  None is changed with the tag, scale,
dyadic level, or any later exponent choice.  Before any tubelet pigeonhole one
can construct, in every tag \(B\in\mathcal B\), a fixed half-open net
\(\T_B^{\rm can}\) of capped \(C\delta\times C\delta\times Cr_1\)
carriers and a descendant assignment
\[
 \pi_B:\{T:Y^\circ(T)\cap B\ne\varnothing\}\longrightarrow\T_B^{\rm can}
\tag{10.7P}
\]
such that
\[
 |\T_B^{\rm can}|
 \le C(1+r_1/\delta)^5
\tag{10.7P$''$}
\]
and with the following properties.
\begin{enumerate}[label=\textup{(C\arabic*)},leftmargin=2.8em]
\item For every active pair \((B,T)\), the carrier
\(T_B:=\pi_B(T)\) contains the solid piece \(T\cap B\), is contained in
\(CB\), has volume comparable to \(\delta^2r_1\), and contains a
subsegment of the carrier of \(T\) of length at least \(c r_1\).  If
\(B\) meets an endcap of \(T\), this segment is chosen one-sided.

No comparison is asserted between \(|T_B|\) and \(|T\cap B|\), and no
inclusion \(T_B\subset N_{C\delta}(T\cap B)\) is asserted.  In particular
this statement remains literal when \(T\) only grazes \(B\) and
\(|T\cap B|/(\delta^2r_1)\) tends to zero.
\item The map \(\pi_B\) has a fixed half-open tie rule and never merges fine
labels.  In each fixed tag, every active fine label has exactly one canonical
representation.  Consequently
\[
 \kappa_{\rm loc}=1
\tag{10.7Q}
\]
for this assignment (or \(O(1)\) if a fixed set of neighbouring net cells is
retained).  A physical fine incidence \((T,x)\) belongs to at least one and
at most
\[
 \kappa_{\rm inc}\le C\kappa_{\rm amb}
\tag{10.7R}
\]
tagged representations.
\item Put
\[
 \T(T_B):=\{T:\pi_B(T)=T_B\},\qquad
 E_{B,T}:=Y^\circ(T)\cap B,
\]
and
\[
 g_{T_B}(x):=\sum_{T\in\T(T_B)}\mathbf1_{E_{B,T}}(x).
\tag{10.7S}
\]
 Let
 \[
  N_{\rm can}:=\max\{1,\max_{B,T_B}|\T(T_B)|\},
 \]
 where the inner maximum is interpreted as zero if there is no active label,
 and put \(L_{\rm mult}:=1+\lceil\log_2N_{\rm can}\rceil\).  For every
 canonical label carrying positive represented mass
 \(\sum_{T\in\T(T_B)}|E_{B,T}|>0\), choose a dyadic number \(m(T_B)\) and set
\[
 \widetilde Y_B(T_B)
 :=\{x\in T_B:m(T_B)\le g_{T_B}(x)<2m(T_B)\}
\tag{10.7T}
\]
and set \(\widetilde Y_B(T_B)=\varnothing\) on every remaining zero-mass
canonical label,
so that
\[
 \sum_{B,T_B}\int_{\widetilde Y_B(T_B)}g_{T_B}
 \ge L_{\rm mult}^{-1}
       \sum_{B,T}|Y^\circ(T)\cap B|.
\tag{10.7U}
\]
If \(\lambda(\T,Y^\circ)=:\lambda^\circ>0\), labels with
\[
 \frac{|\widetilde Y_B(T_B)|}{|T_B|}
 <c\lambda^\circ/L_{\rm mult}
\tag{10.7V}
\]
carry at most half of the left side of (10.7U), after decreasing the
structural \(c\).  A further dyadic shading-density selection therefore
retains at least
\[
 (2L_{\rm mult}L_{\rm dens})^{-1}
 \sum_{B,T}|Y^\circ(T)\cap B|,
\qquad
 L_{\rm dens}
 :=1+\left\lceil\log_2\frac{CL_{\rm mult}}{\lambda^\circ}\right\rceil ,
\tag{10.7W}
\]
as represented fine-incidence mass, and every surviving canonical label has
one common density up to a factor two, bounded below by the right side of
(10.7V).  The induced physical fine shading
\[
 Y^{\rm pre}(T)
 :=Y^\circ(T)\cap
   \bigcup_{\substack{(B,T_B):\,T\in\T(T_B)}}
       \widetilde Y_B(T_B)
\tag{10.7X}
\]
satisfies the same lower bound with the additional factor
\(\kappa_{\rm inc}^{-1}\).
\item If \(T_B\subset K\) for a convex body \(K\), set
\(K^\dagger=K+CB_\delta\).  Let
\(R_{K^\dagger}\supset K^\dagger\) be a comparable-volume containing
L\"owner rectangle with centre \(z_K\), and put
\[
 K^+=z_K+C r_1^{-1}(R_{K^\dagger}-z_K).
\tag{10.7Y}
\]
Every descendant \(T\in\T(T_B)\) is contained in this same \(K^+\), and
\[
 |K^\dagger|\le C|K|,\qquad |K^+|\le Cr_1^{-3}|K|.
\tag{10.7Z}
\]
This is the only complete-containment transfer used below; it is one-way
from a complete canonical carrier to its fine descendants.
\end{enumerate}
\end{lemma}

\begin{proof}
Write \(I_T\) for the unit carrier segment of \(T\), and \(c_B\) for the
centre of \(B\).  Fix a sufficiently large structural \(C_J\), and put
\[
 \ell_{r_1}:=\min\{1,C_Jr_1\}\asymp r_1.
\tag{10.7P$'$}
\]
The set \(I_T\cap B(c_B,C_2r_1)\) is a segment of length at most
\(2C_2r_1\).  Here \(C_2\) is chosen from the fixed tube-radius convention
so that, whenever \(T\cap B\ne\varnothing\), this set contains every carrier
point within \(C\delta\) of \(T\cap B\).  Choose
\(C_J>2C_2+2\).  If \(C_Jr_1<1\), the displayed segment is contained in a
subsegment \(J_{B,T}\subset I_T\) of the single length
\(\ell_{r_1}=C_Jr_1\); near an endpoint of \(I_T\), take the extension
entirely to the available side.  If \(C_Jr_1\ge1\), take
\(J_{B,T}=I_T\).  Thus in all cases every \(J_{B,T}\) has the same length
\(\ell_{r_1}\), \(N_{C\delta}(J_{B,T})\) contains \(T\cap B\), and that
neighbourhood lies in \(CB\).  This fixed-length window, rather than the
possibly very short raw intersection, is the object that is discretized.

Parameterize these equal-length windows by
\((u,m)\in\mathbb{RP}^2\times CB\), where \(u\) is the unoriented direction
and \(m\) the midpoint, and use the scaled product metric
\[
 d_{r_1}((u,m),(u',m'))
 :=r_1d_{\mathbb{RP}^2}(u,u')+|m-m'|.
\]
Choose a maximal \((c\delta/r_1)\)-separated net in
\(\mathbb{RP}^2\) and a maximal \((c\delta)\)-separated net in \(CB\).
Order both finite nets once and use the induced half-open Voronoi cells in
the product to assign each \(J_{B,T}\) to the first nearest net window.
Thus the assignment is a genuine function even on Voronoi boundaries.  The
two-dimensional projective packing bound and three-dimensional midpoint
packing bound give
\[
 \#\mathcal N_{\rm dir}\le C(1+r_1/\delta)^2,\qquad
 \#\mathcal N_{\rm mid}\le C(1+r_1/\delta)^3,
\]
which proves (10.7P$''$).  A fixed dilation of
the chosen net tube contains \(N_{C\delta}(J_{B,T})\), so it contains
\(T\cap B\), remains in \(CB\), has volume comparable to
\(\delta^2r_1\), and its carrier segment has Hausdorff distance
\(O(\delta)\) from \(J_{B,T}\).  This proves (C1) and the geometric
comparison later recorded in (10.7L).  Since the assignment acts on fine
labels, coincident geometric carriers remain distinct.

Inside one tag the map is single-valued, which proves (10.7Q).  If
\(x\in Y^\circ(T)\), every tag used to represent \((T,x)\) contains \(x\).
The balls cover the union, so at least one representation exists; bounded
overlap gives at most \(\kappa_{\rm amb}\) tags at \(x\).  The canonical net
is single-valued inside each such tag.  Retaining a fixed set of neighbouring
net cells changes this by at most a dimensional packing constant.  This proves
(10.7R): the uniform bound applies to the incidence \((T,x)\), even though a
whole unit tube may cross \(O(r_1^{-1})\) tags.

 For one \(T_B\) carrying positive represented mass, layer-cake decomposition
 of the integer-valued function \(g_{T_B}\) gives a level satisfying
\[
 \int_{\widetilde Y_B(T_B)}g_{T_B}
 \ge L_{\rm mult}^{-1}
        \sum_{T\in\T(T_B)}|E_{B,T}|.
\]
Zero-mass canonical labels have the declared empty selected shading and
contribute zero to both sides.
Summing proves (10.7U).  The density selection follows from a direct
volume-packing estimate for the tag count.  The centre of every active radius-\(r_1\)
ball lies in \(N_{Cr_1}(T)\); bounded overlap and integration over that
neighbourhood give
\[
 \#\{B:Y^\circ(T)\cap B\ne\varnothing\}r_1^3
 \lesssim |N_{Cr_1}(T)|\lesssim r_1^2.
\]
Thus a unit tube meets only \(O(r_1^{-1})\) active tags, and
\[
 \sum_{B,T_B}|\T(T_B)|\,|T_B|
 \lesssim |\T|\,|T|,
\qquad
 \sum_{B,T}|Y^\circ(T)\cap B|\ge M(Y^\circ).
\tag{10.7U$'$}
\]
On the chosen multiplicity level,
\[
 \int_{\widetilde Y_B(T_B)}g_{T_B}
 <2m(T_B)|\widetilde Y_B(T_B)|,
\qquad
 m(T_B)\le|\T(T_B)|.
\tag{10.7U$''$}
\]
Combining (10.7U)--(10.7U$''$) shows
\[
 \sum_{B,T_B}m(T_B)|\widetilde Y_B(T_B)|
 \gtrsim \frac{\lambda^\circ}{L_{\rm mult}}
          \sum_{B,T_B}m(T_B)|T_B|.
\]
The labels failing (10.7V) therefore contribute at most half of the selected
represented mass when \(c\) is sufficiently small.  On the remaining range
the density lies between \(c\lambda^\circ/L_{\rm mult}\) and one, so there
are at most \(L_{\rm dens}\) dyadic classes.  Selecting a class by represented
mass proves (10.7W).  Each physical incidence is counted at most
\(\kappa_{\rm inc}\) times, which proves (10.7X).  Thus the canonical
objects carry the two tubelet levels and the induced fine refinement used on
pp.~36--37 of \cite{GWZ2026}.

It remains to prove the one-way convex assertion.  If \(T_B\subset K\),
the net construction supplies, for every descendant \(T\), a carrier
subsegment of length at least \(cr_1\) inside \(K^\dagger\), with direction
within \(C\delta/r_1\) of the canonical direction.  Because \(K\) contains
the \(C\delta\)-thick carrier \(T_B\), it contains a ball of radius
\(c\delta\).  Applying \cref{lem:convex-parallel-body} with the fixed ratio
between the padding and this inradius shows that adding \(CB_\delta\)
changes its volume by only a dimensional factor.  Choose the containing L\"owner
rectangle \(R_{K^\dagger}\) and write its half-side lengths as
\(s_1,s_2,s_3\).  Subtracting the endpoint coordinates of the contained
\(cr_1\)-segment gives \(r_1|u_i|\le Cs_i\) for the unit direction \(u\)
of \(T\).  From that segment to either endpoint of the unit carrier costs at
most \(Cr_1^{-1}s_i\) in coordinate \(i\).  The solid radius is already
absorbed by \(K^\dagger\).  Hence \(T\subset K^+\), simultaneously for all
descendants, and
\[
 |K^+|\lesssim r_1^{-3}|R_{K^\dagger}|
 \lesssim r_1^{-3}|K|.
\]
This proves (10.7Y)--(10.7Z) and the lemma.
\end{proof}

For the application take \(Y^\circ=Y^{\rm sc}\).  After (10.7T)--(10.7W),
set \(\mathsf F_B:=\sum_{T_B}\int_{\widetilde Y_B(T_B)}g_{T_B}(x)\,dx\) for each
nonempty tag.  Partition the tags dyadically by \(\mathsf F_B\), retain a class
maximizing \(\sum_B\mathsf F_B\), and restrict all subsequent canonical objects,
including the application instance of \(Y^{\rm pre}\), to that class.  Let
\(L_{\rm tag}\) count the nonempty dyadic levels.  By (10.7V), the application
density lower bound,
\(m(T_B)\ge1\), and \(|T_B|\asymp\delta^2r_1\), every nonzero
\(\mathsf F_B\) has a fixed-power lower bound; (10.7P$''$) and the input
cardinality bound give a fixed-power upper bound.  Hence
\[
 L_{\rm tag}\le C\log(2/\delta).
\tag{10.7W$'$}
\]
This whole-tag level leaves the canonical carriers, descendants, pointwise
multiplicity, and tubelet densities fixed.

For the application, fix a structural integer \(C_{\rm pre}\) and define the
pre-cellular profile
\[
 \mathcal E_{\rm pre}(\delta)
 :=\max\left\{1,
      \delta^{-C_{\rm pre}\omega_0(\delta)},
      [\log(2/\delta)]^{C_{\rm pre}}\right\},
 \qquad
 \omega_{\rm pre}(\delta)
 :=\frac{\log\mathcal E_{\rm pre}(\delta)}{\log(1/\delta)}.
\tag{10.7W$''$}
\]
Here \(C_{\rm pre}\) absorbs the two input-refinement factors
\(c_{\rm init}^{-1},c_{\rm sc}^{-1}\), the canonical multiplicity level, and
fixed representation constants.  Thus \(\omega_{\rm pre}(\delta)\to0\), and
its definition involves neither \(J_v\), an angular factor, nor any
post-cellular quantity.

\begin{lemma}[One-way consequences of canonical ancestry]
\label{lem:local-ancestry}
Let \(0<\delta\le r_1\le1/10\), let \(B\) be one fixed tagged \(r_1\)-ball,
and let \(T=N_{c\delta}(I_T)\) be a capped unit \(\delta\)-tube whose carrier
segment \(I_T\) has length \(1\).  Use exactly the canonical assignment
\(\pi_B\) of \cref{lem:source-canonical}.  The following aliases record its
one additional local estimate and preserve the downstream interface.
\begin{enumerate}[label=\textup{(G\arabic*)},leftmargin=2.8em]
\item The representation space, tubelet dimensions, containment, and tag
  localization are those of clause (C1) in \cref{lem:source-canonical}.
\item If \(T\in\T(T_B)\), there is a subsegment
  \(I_{T,B}\subset I_T\) of length at least \(c r_1\) such that
  \[
   \ang(I_T,I_{T_B})\le C\delta/r_1,\qquad
   d_H(I_{T,B},I_{T_B})\le C\delta.
  \tag{10.7L}
  \]
  The assertion remains true when \(B\) meets an endcap of \(T\): in that
  case \(I_{T,B}\) is chosen one-sided and \(CB\) contains it.
\item The fixed-tag representation bound is clause (C2):
  \(\kappa_{\rm loc}=1\) for the half-open assignment and \(O(1)\) when the
  fixed neighboring-cell convention is retained.
\item If \(T_B\subset K\) for a convex body \(K\), set
  \(K^\dagger=K+CB_\delta\).  Let \(R_{K^\dagger}\supset K^\dagger\) be a
  comparable-volume containing rectangle with centre \(z_K\), and put
  \[
   K^+=z_K+C r_1^{-1}(R_{K^\dagger}-z_K).
  \tag{10.7M}
  \]
  Then every descendant \(T\in\T(T_B)\) is contained in \(K^+\), and
  \[
   |K^\dagger|\le C|K|,\qquad |K^+|\le Cr_1^{-3}|K|.
  \tag{10.7N}
  \]
  By clause (C4), these inclusions hold simultaneously for all convex \(K\)
  under the fixed descendant assignment.
\end{enumerate}
\end{lemma}

\begin{proof}
Clauses (C1), (C2), and (C4) give (G1), (G3), and (G4), including
(10.7M)--(10.7N); in particular,
\[
 |K^+|\le C r_1^{-3}|R_{K^\dagger}|
          \le C r_1^{-3}|K|.
\]
In the net construction, the carrier of \(T_B\) is within
\(O(\delta)\) of the length-\(cr_1\) axial window of \(T\); division by that
length gives the angular bound in (10.7L).  The one-sided window in (C1)
covers the endcap case.
\end{proof}

\begin{lemma}[Expanded-tubelet concentration]
\label{lem:expanded-concentration}
Fix one tagged outer ball \(B\).  For each selected tubelet in \(\T_B\),
fix once and for all an injection of its
\(m_{\rm fine}\) copies into distinct represented fine descendants.  If one
fine label has at most \(\kappa_{\rm loc}\) canonical tubelet
representations inside this fixed tag, including its fixed boundary tie,
then
\[
 \Delta_{\max}(\T_B^{\rm lab})
 \le C\kappa_{\rm loc}r_1^{-2}\Delta_{\max}(\T).
\tag{10.7c}
\]
The same injection is used for every point and every convex test body.
\end{lemma}

\begin{proof}
For each \(T_B\), choose one point \(x_{T_B}\in Y_B(T_B)\) and, using the
fixed label order, the first \(m_{\rm fine}\) descendants supplied by (10.4).
This defines the required injection independently of later \(x\) and \(K\).
Clause (G3) of \cref{lem:local-ancestry}, equivalently (C2) of
\cref{lem:source-canonical}, bounds the use of any fine label by
\(\kappa_{\rm loc}\) in the fixed tag.

For a convex \(K\), use the common test body from (G4), equivalently clause
(C4):
\[
 K^+=z_K+C r_1^{-1}(R_{K^\dagger}-z_K).
\tag{10.7c$'$}
\]
Every injected descendant of a copy contained in \(K\) lies in this convex
\(K^+\), and (C4) gives
\[
 |K^+|\lesssim r_1^{-3}|K|,\qquad
 |T_B|\asymp r_1|T|.
\tag{10.7c$''$}
\]
The injection multiplicity and the definition of \(\Delta_{\max}(\T)\) now
give
\[
\begin{aligned}
 \sum_{(T_B,\ell)\subset K}|T_B|
 &\lesssim r_1\kappa_{\rm loc}
       \sum_{T\subset K^+}|T|\\
 &\le r_1\kappa_{\rm loc}\Delta_{\max}(\T)|K^+|\\
 &\lesssim\kappa_{\rm loc}r_1^{-2}
       \Delta_{\max}(\T)|K|.
\end{aligned}
\]
Taking the supremum over \(K\) proves (10.7c).
\end{proof}

\begin{lemma}[Labelled denominator from represented fine mass]
\label{lem:represented-denominator}
Let \(A(T_B)\subset Y_B(T_B)\) be any retained tubelet subshading on the
fixed global \(m_{\rm fine}\)-class, and let \(Y_A\) be the physical fine
shading in (10.6a).  If
\[
 M(Y_A)\ge P^{-1}M(\bar Y),
\tag{10.7d}
\]
then the multiplicity-expanded labelled tubelet families obey, with the
denominator summed over all retained tags,
\[
\begin{aligned}
 D_{\rm tag}^{\rm tot}
 &:=\sum_{B\ {\rm retained}}D(\T_B^{\rm lab})
   =m_{\rm fine}\sum_{(B,T_B)}|T_B|\\
 &\gtrsim \labmass(A)
 \ge M(Y_A)
 \ge P^{-1}\lambda(\T,\bar Y)|\T|\,|T|.
\end{aligned}
\tag{10.7e}
\]
The factor \(\kappa_{\rm inc}\) is needed when (10.7d) is established from
labelled retention, through (10.7$'$), but is not inserted again in (10.7e).
\end{lemma}

\begin{proof}
Equation (10.7) gives
\(\labmass(A)\asymp m_{\rm fine}\sum_{(B,T_B)}|A(T_B)|\), which is at most a
fixed multiple of \(D_{\rm tag}^{\rm tot}\).  The middle inequality is the upper
half of (10.7$'$), and the last identity is the definition of the density of
the fixed fine reference shading.  This proves every displayed comparison.
\end{proof}
The score and L\"owner-axis windows required by
\cref{lem:biased-factoring} are verified in
\cref{prop:source-extraction}.  Apply that lemma with
\(\gamma=\varepsilon_{\rm fact}\) and \(U=CB\) in each tag.  The tagged
disjoint union retains \(L_{\rm bias}^{-1}\); a later \(L_{a,b}\)-level
selects common widths.
Assign coincident copies to one parent.  The common duplication factor
\(m_{\rm fine}\) preserves normalized child fractions.  Expansion is used
only for the slab denominator; non-slab degree and cap estimates use the
unexpanded family, so (10.19e1) counts \(m_{\rm fine}\) once.

All dimension, child-count, and regularity levels follow
\(m_{\rm fine}\).  A global \((a,b)\)-level costs at most
\[
 L_{a,b}\le
 \bigl(1+\lceil\log_2(r_1/\delta)\rceil\bigr)^2.
\tag{10.8}
\]
For each outer-ball tag \(B\), let \(B^+\) be a fixed dilation containing all
permanent carriers and active full cells in that tag.  Since the outer balls
have one common radius and bounded overlap, the family \((B^+)_B\) has overlap
\(O(\kappa_{\rm amb})\); absorb this structural factor into
\(\kappa_{\rm amb}\).  Run the graph on
\(\bigsqcup_B\{B\}\times\mathbb R^3\), with right vertices \((B,Q)\), and
synchronize only after each tagwise output is fixed.  The next theorem
formalizes this step; physical recombination then costs the single factor
\(\kappa_{\rm amb}\) from \cref{lem:tag-recombine}.

\begin{theorem}[Global tagged regularization and synchronization]
\label{thm:tagged-sync}
Let \(\mathcal A\) be a finite nonempty family of tags.  In tag \(A\), let
\((\mathcal P_A,\mathcal T_A,Y_A)\) be a positive-mass input to
\cref{thm:joint}, all supported in a fixed dilation of \(A\), and let
\(\mathcal Q_A\) be its finite active half-open cell set.  Assume that,
for fixed \(C_{\rm comp},C_{\rm mass}\),
\[
 |\mathcal A|+
 \max_{A\in\mathcal A}
 (|\mathcal P_A|+|\mathcal T_A|+|\mathcal Q_A|)
 \le\delta^{-C_{\rm comp}},\qquad
 \delta^{C_{\rm mass}}\le M_A(Y_A)\le\delta^{-C_{\rm mass}}
 \quad(A\in\mathcal A).
\tag{10.8A}
\]
Run J1, R1, and R2 independently in every tag, write
\(\mu_A,\omega_A,d_A\) for the three local dyadic parameters, and denote the
local graph edge set by \(E_{2,A}\) and its shading output by
\(Y_{{\rm reg},A}\).  Put
\[
 \mathcal Q_A^+=\{Q:\text{some }(P,Q)\in E_{2,A}\},\qquad
 v_A(Q)=|U(\mathcal T_A,Y_{{\rm reg},A})\cap N_{Ca}(Q)|.
\]
There is a structural \(C_v\ge1\), independent of \(A\), such that every
\(Q\in\mathcal Q_A^+\) satisfies
\[
 \frac{\omega_A}{2\mu_A}\le v_A(Q)\le C_va^3.
\tag{10.8A$'$}
\]
Define the literal number of positive volume levels by
\[
 J_v(A):=1+\max\left\{0,
 \left\lceil\log_2\!\left(
  \frac{2C_va^3\mu_A}{\omega_A}
 \right)\right\rceil\right\}.
\tag{10.8A$''$}
\]
If \(\mathcal Q_A^+=\varnothing\), set \(J_v(A)=1\) and declare the
local-volume output empty; the positive-mass hypothesis and P1 exclude this
case for every tag retained below.  Synchronize
\(\mu_A,\omega_A,d_A\) by a whole-tag selection.  Only then, independently
in each selected tag, partition \(\mathcal Q_A^+\) into the \(J_v(A)\)
half-open dyadic intervals covering (10.8A$'$), choose a class
maximizing the sum of its incident edge weights, and retain every edge at each
chosen right cell.  This is the \emph{saturated local-volume selection}; call
its output \(Y_{1,A}\).  Put
\[
 L_{\rm reg}^{\max}:=\max_A L_{\rm reg}(A),\qquad
 J_v^{\max}:=\max_A J_v(A).
\tag{10.8B}
\]
There exist a whole-tag subfamily \(\mathcal A_1\), dyadic numbers
\(\mu_*,\omega_*,d_*\), and level counts
\(L_\mu,L_\omega,L_d\le C\log(2/\delta)\) such that
\[
 \mu_*\le\mu_A<2\mu_*,\quad
 \omega_*\le\omega_A<2\omega_*,\quad
 d_*\le d_A<2d_*\qquad(A\in\mathcal A_1),
\tag{10.8C}
\]
every \(A\in\mathcal A_1\) satisfies the local conclusions P2--P5 and
\[
 M_A(Y_{1,A})\ge
   [J_v(A)L_{\rm reg}(A)]^{-1}M_A(Y_A),
\tag{10.8D}
\]
while globally
\[
 \sum_{A\in\mathcal A_1}M_A(Y_{1,A})
 \ge [L_\mu L_\omega L_dJ_v^{\max}L_{\rm reg}^{\max}]^{-1}
       \sum_{A\in\mathcal A}M_A(Y_A).
\tag{10.8E}
\]

Suppose in addition that the common scales determine the same resolved-angle
parameters \(A_{\rm ang},B_{\rm ang},K,q_{\rm res}\) in every tag.  Run
\cref{thm:angle} independently in each \(A\in\mathcal A_1\).  There is a
whole-tag subfamily \(\mathcal A_2\subset\mathcal A_1\) and dyadic
\(\theta_*,d_{\angle,*}\), with
\[
 \theta_*\le\theta_A<2\theta_*,\qquad
 d_{\angle,*}\le d_{\angle,A}<2d_{\angle,*},
\tag{10.8F}
\]
for which all local A1--A4 conclusions hold and
\[
 \sum_{A\in\mathcal A_2}M_A(Y_{2,A})
 \ge \frac{c_{\rm ang}^{\min}}
 {2L_{\theta}^{\rm tag}L_{d_\angle}^{\rm tag}}
 \sum_{A\in\mathcal A_1}M_A(Y_{1,A}),
\quad
 c_{\rm ang}^{\min}:=\min_{A\in\mathcal A_1}
       \frac{A_{\rm ang}^{-K}}{J_{\theta}(A)J_d(A)},
\tag{10.8G}
\]
where \(L_{\theta}^{\rm tag},L_{d_\angle}^{\rm tag}\le
C\log(2/\delta)\).  In particular the common input degree \(d_*\) permits
one global low/high split, and in the high case the synchronized output
degree is \(d_{\angle,*}\).

Finally, suppose there are a common \(m_*>0\) and fixed constants
\(0<c_-\le c_+<\infty\) such that a represented fine mass \(F_A(V)\) on every
whole-tubelet subshading obeys
\[
 c_-m_*M_A(V)\le F_A(V)\le c_+m_*M_A(V)
\tag{10.8H}
\]
then (10.8E) and (10.8G) hold for
\(\sum_AF_A\), with the additional factor \(c_+/c_-\).  Thus the selected
tags carry the same polylogarithmic fraction of represented fine mass and,
whenever the denominator is comparable to that mass at the fixed density
level, of the labelled denominator.
\end{theorem}

\begin{proof}
Fix a tag \(A\).  Applying \cref{thm:joint} only to its own left parents and
its own half-open cells gives
\(M_A(Y_{{\rm reg},A})\ge L_{\rm reg}(A)^{-1}M_A(Y_A)\).
This is a local statement before any tag is discarded.

If \(Q\in\mathcal Q_A^+\), choose an edge \((P,Q)\in E_{2,A}\).  Its weight
is at least \(\omega_A\), while the inner child multiplicity for this one
parent is below \(2\mu_A\).  Therefore the physical union of those child
shadings in \(Q\), and hence the union counted by \(v_A(Q)\), has measure at
least \(\omega_A/(2\mu_A)\).  Conversely the \(Ca\)-neighbourhood of a cell
of side \(c_0a\) has volume at most \(C_va^3\).  This proves (10.8A$'$).
The interval between these two positive endpoints is covered by the
\(J_v(A)\) half-open dyadic intervals in (10.8A$''$), including the case in
which the endpoint ratio is at most one.  Thus every symbol in (10.8B) is
defined before it is used.

All values in (10.8C) lie in polynomial scale windows.  Indeed
\(1\le\mu_A\le|\mathcal T_A|\) and
\(1\le d_A\le|\mathcal P_A|\).  If \(m_A^+=|E_A^+|\) and
\(\Omega_A=\sum_{E_A^+}\omega_e\), the proof of \cref{lem:REG1} gives
\[
 \frac{\Omega_A}{2m_A^+}\le\omega_A\le\Omega_A.
\]
The inner level and (10.8A) give
\(\Omega_A\ge L_{J1}(A)^{-1}\delta^{C_{\rm mass}}\), while the upper
bound in (10.8A) controls the other endpoint.  Hence each of
\(\mu_A,\omega_A,d_A\) has \(O(\log(2/\delta))\) dyadic classes.  Partition
the \emph{tags} by the triple of classes and choose a class maximizing
\(\sum_AM_A(Y_{{\rm reg},A})\).  The selection keeps every edge in each
chosen tag.  Inside such a tag, the saturated volume class costs
\(J_v(A)\) and preserves the synchronized local degrees.  Summing the local graph
bounds, pigeonholing the tags, and then applying the local volume bounds
proves (10.8C)--(10.8E), including (10.8D).

For the second stage, the local theorem gives
\(M_A(Y_{2,A})\ge(c_{{\rm ang},A}/2)M_A(Y_{1,A})\).  The possible
\(\theta_A\in[q_{\rm res},1]\) and integer
\(d_{\angle,A}\le|\mathcal P_A|\) again form only logarithmically many
classes.  Select a class by its post-angle mass.  Summing the local lower
bounds and then pigeonholing proves (10.8F)--(10.8G).  Since this selection
deletes whole tags, every local A1--A4 conclusion persists.
Finally (10.8H) converts each coarse sum to the corresponding fine sum on
both sides.  This proves the last assertion and the theorem.
\end{proof}

\begin{corollary}[Polynomial-complexity volume-level count]
\label{cor:tagged-volume-levels}
In the setting of \cref{thm:tagged-sync}, assume \(0<a\le1\).  Then
\[
 J_v^{\max}\le C(C_{\rm comp},C_{\rm mass})\log(2/\delta).
\tag{10.8I}
\]
More generally, without (10.8A) but for finite positive-mass inputs, the
local saturated selection and (10.8D) remain valid with the literal integer
\(J_v(A)\) in (10.8A$''$); the global proof likewise uses the actual finite
numbers of nonempty dyadic tag classes in place of the displayed logarithmic
bounds.  No logarithmic estimate is part of the abstract local-volume
identity.  Thus (10.8I) is an application-level complexity consequence, not
an implicit hypothesis of the synchronization identity.
\end{corollary}

\begin{proof}
For one tag let \(m_A^+=|E_A^+|\) and
\(\Omega_A=\sum_{E_A^+}\omega_e\).  The inner level and (10.8A) give
\[
 \Omega_A\ge L_{J1}(A)^{-1}\delta^{C_{\rm mass}},
 \qquad L_{J1}(A)\le C\log(2/\delta),
\]
while \(m_A^+\le|\mathcal P_A||\mathcal Q_A|\le
\delta^{-2C_{\rm comp}}\).  The proof of \cref{lem:REG1} therefore gives
\(\omega_A\ge\Omega_A/(2m_A^+)\ge\delta^{C_1}\) after increasing the fixed
exponent \(C_1\).  Also \(\mu_A\le|\mathcal T_A|\le
\delta^{-C_{\rm comp}}\).  Since \(a^3\le1\), the logarithm in
(10.8A$''$) is bounded by \(C_2\log(2/\delta)\), uniformly in \(A\).
Taking the maximum proves (10.8I).
\end{proof}

\subsection{Analytic estimates inherited by the application}
\label{sec:inherited-estimates}

The branch estimates below close through two analytic assertions from
\cite[Definitions~3.4--3.5, pp.~5--6]{GWZ2026}.  We first state them in their
source-standard finite-set form.  Fix \(0<\beta\le1\).
The assertion \(K_{KT}(\beta)\) means that for every \(\varepsilon>0\) there
are \(\eta,\delta_0>0\) such that every finite set
\(\mathcal F\) of \(\delta\)-tubes in the unit ball, with
\[
 \Delta_{\max}(\mathcal F)\le\delta^{-\eta},
 \qquad
 \lambda(\mathcal F,Y)\ge\delta^\eta,
\]
satisfies
\[
 \mu(\mathcal F,Y)
 \lesssim_{\varepsilon,\beta}\delta^{-\varepsilon}
 |\mathcal F|^\beta
 \qquad(0<\delta\le\delta_0).
\tag{10.8K}
\]
The assertion \(K_F(\beta)\) means that, under the same quantifiers, the
hypotheses
\[
 C_F(\mathcal F,B_1)\le\delta^{-\eta},
 \qquad
 \lambda(\mathcal F,Y)\ge\delta^\eta
\]
imply
\[
 |U(\mathcal F,Y)|
 \gtrsim_{\varepsilon,\beta}
 \delta^{\varepsilon+2\beta}
 \bigl(|\mathcal F|\,|T_\delta|\bigr)^{\beta/2},
\tag{10.8J}
\]
where \(T_\delta\) denotes one unit \(\delta\)-tube.  For labelled KKT calls,
discard labels whose individual shading density is below half the mean and
colour exact repeated carriers into \(m_*\le\Delta_{\max}(\mathcal F)\)
carrier-injective layers.  Applying the source set assertion to the layers and
using \(\sum_j|\mathcal F_j|^\beta\le
m_*^{1-\beta}|\mathcal F|^\beta\) gives the labelled form after reselecting
\(\eta,\delta_0\); the random-subset proof of Lemma~3.7 then gives its labelled
form.  No unrestricted labelled extension of \(K_F\) is asserted: its calls
in Sections~9--10 use source-standard colours or proxies and charge fibre multiplicity
separately.  We also use the
scale-variant forms in \cite[Remark~3.6, p.~6]{GWZ2026}.  The only consequences for
planks are restated, with all parameters used here, in
\cref{lem:thick-parent,lem:labelled-plank-reduction}.  Those restatements fix
the complete hypotheses encoded by \(K_F(\beta)\) and \(K_{KT}(\beta)\).

\subsection{Thick parents}
\label{sec:thickbranch}

The first geometric split occurs before cellularization.  When the shortest
parent width is already large, the inherited Frostman and Katz--Tao estimates
close the branch directly; the four cellular alternatives considered later
belong only to the complementary thin regime.  This branch uses the original
labels, without multiplicity expansion.

\begin{lemma}[The thick-parent alternative]
\label{lem:thick-parent}
Assume \(a\ge\delta^{1-\tau}\), and assume both \(K_F(\beta)\) and
\(K_{KT}(\beta)\).  Let \(Y^{\rm sc}\subset\bar Y\) be the all-grid
refinement fixed in (S1).  Choose the source-output and setup exponents so that
\[
 0<\varepsilon_{\rm thick},\eta,\qquad
 \varepsilon_{\rm thick}+\eta+3\eta_{\rm grid}
   <\frac18\tau\varepsilon_{\rm fact}\beta
\tag{10.8a}
\]
and also choose \(\eta\) below
\(\tau\eta_F(\varepsilon_{\rm thick},\beta)\), where \(\eta_F\) is the
density threshold in the Frostman plank estimate
\cite[Lemma~6.4, p.~15]{GWZ2026}.  Let \(r\) be the already fixed source
scale-grid radius with \(a\le r\le\delta^{-\eta_{\rm grid}}a\).
For sufficiently small \(\delta\), there is a ball \(B_r\) of radius
comparable to \(r\), centered at a point of the fixed scale-\(r\) union,
such that
\[
 \frac{|U(\T,Y^{\rm sc})\cap B_r|}{|B_r|}
 \gtrsim_{\log}
 \delta^{-2\nu_{\rm thick}}
 \left(\frac{\delta}{r}\right)^{2\beta},
 \qquad
 \nu_{\rm thick}:=\frac18
       \tau\varepsilon_{\rm fact}\beta>0.
\tag{10.8b}
\]
Consequently the source scale-\(r\) comparison gives
\[
 |U(\T,\bar Y)|
 \gtrsim_{\log}
 \delta^{-\nu_{\rm thick}}|T|\,|\T|^{1-\beta}.
\tag{10.8c}
\]
\end{lemma}

\begin{proof}
Fix a retained outer ball \(B\).  The common parent-density level contains
a parent \(W\) of dimensions \(a\times b\times r_1\) for which
\(\lambda(\T_{B,W},Y_B)\gtrsim_{\log}\delta^\eta\).  The locally proved
biased factoring inequality (8.0D), in the exact complete-containment form
(10.37b), says that for every convex \(K\subset W\),
\[
 \Delta(\T_{B,W},K)
 \lesssim_{\log}
 \left(\frac{|K|}{|W|}\right)^{\varepsilon_{\rm fact}}
 \Delta(\T_{B,W},W).
\tag{10.8d}
\]
Taking \(K\) to be a child tubelet, for which
\(\Delta(\T_{B,W},K)\ge1\), and using equal child volume gives
\[
 |\T_{B,W}|
 \gtrsim_{\log}
 \left(\frac{|W|}{|T_B|}\right)^{1+\varepsilon_{\rm fact}}
 \ge
 \left(\frac a\delta\right)^{2+2\varepsilon_{\rm fact}}.
\tag{10.8e}
\]

Let \(L_W\) be the affine map carrying the labelled source L\"owner box of \(W\) to a
unit box.  The complete unexpanded labelled children become planks
\(\mathcal P\) of dimensions
\[
 s\times t\times1,\qquad s=\delta/b,\quad t=\delta/a.
\]
Fixed L\"owner-box and boundary enlargements are absorbed in
\(\lesssim_{\log}\).  Equation (10.8d) gives
\(C_F(\mathcal P)\lesssim_{\log}1\).  For a base label
\(P\in\mathcal P\) and \(\vartheta\in[s/t,1]=[a/b,1]\), let
\(P_\vartheta\) be its source \(\vartheta t\times t\times1\) thickening.
Intersect the pullback of \(P_\vartheta\) with \(W\), and call the result
\(K_\vartheta\).  This is convex, every completely contained child counted
by \(\mathcal P[P_\vartheta]\) lies in \(K_\vartheta\), and
\[
 \frac{|K_\vartheta|}{|W|}
 \lesssim\vartheta\left(\frac\delta a\right)^2.
\]
The count form of (10.8d) therefore yields
\[
 |\mathcal P[P_\vartheta]|
 \lesssim_{\log}
 \vartheta\left(\frac\delta a\right)^{2+\varepsilon_{\rm fact}}
 |\mathcal P|.
\tag{10.8f}
\]
We use the labelled form of source Lemma~6.4.  In its proof, Lemmas~6.11 and
6.13 select and refine labels, with angles resolved below at \(s/t\).  Assign
every surviving label to the essentially
distinct thick proxy chosen in source Definition~6.10 and dyadically
regularize the proxy-fibre size \(N\).  The bound (10.8f), for every base
label, gives \(N\lesssim M\vartheta\).  On this level the positive-sum proof
of source Remark~3.3(A) gives the same Frostman bound with sums over labels,
and only the essentially distinct proxy subfamily is passed to source
Lemma~3.9.  Thus coincident original carriers are retained and counted by
\(M\); the logarithmic fibre regularization is absorbed in
\(L_{\rm thick}\).  Equation (10.8f) gives
\[
 M=\max\left\{1,\,
   L_{\rm thick}
   \left(\frac\delta a\right)^{2+\varepsilon_{\rm fact}}
   |\mathcal P|\right\},
 \qquad L_{\rm thick}=\delta^{-o(1)}.
\]
By (10.8e) the second entry is already at least one after enlarging
\(L_{\rm thick}\).  Also \(s\le t\le\delta^\tau\), and
\(\eta<\tau\eta_F\) gives
\(\lambda(\mathcal P,Y_{\mathcal P})\gtrsim_{\log}s^{\eta_F}\).  Hence the
scale-variant convention \cite[Remark~3.6, p.~6]{GWZ2026} applies with call
data
\[
 \begin{gathered}
 s=\delta/b,\qquad t=\delta/a,\qquad
 \mathcal F=\mathcal P,\qquad
 \lambda(\mathcal P,Y_{\mathcal P})\gtrsim\delta^\eta\ge s^{\eta_F},\\
 C_F(\mathcal P)\lesssim_{\log}1,\qquad
 M\le L_{\rm thick}t^{2+\varepsilon_{\rm fact}}|\mathcal P|,\qquad
 |\mathcal P[P_\vartheta]|\le M\vartheta
 \quad(P\in\mathcal P,\ s/t\le\vartheta\le1).
 \end{gathered}
\tag{10.8f$'$}
\]
The labelled form just verified of source Lemma~6.4, equation (31), now gives
\[
 \mu(\mathcal P,Y_{\mathcal P})
 \lesssim s^{-\varepsilon_{\rm thick}}
 C_F(\mathcal P)^{1-\beta/2}\frac st
 M^{\beta/2}t^{-2\beta}
 \bigl(t^2|\mathcal P|\bigr)^{1-\beta/2}.
\tag{10.8g0}
\]
Divide \(\lambda|\mathcal P|st\) by (10.8g0), insert the bound for \(M\),
and pull back by \(L_W\):
\[
\begin{aligned}
 \frac{|U(\T_{B,W},Y_B)\cap W|}{|W|}
 &\asymp |U(\mathcal P,Y_{\mathcal P})|\\
 &\gtrsim
 \delta^{\eta+\varepsilon_{\rm thick}+o(1)}
 \left(\frac\delta a\right)^{
       2\beta-\varepsilon_{\rm fact}\beta/2}.
\end{aligned}
\tag{10.8g}
\]
Here every \(Y_B\)-point has positive represented fine multiplicity by
(10.3)--(10.4), and the induced descendants belong to the current
all-scales fine refinement \(Y^{\rm sc}\subset\bar Y\).  Hence this coarse
union is contained in \(U(\T,Y^{\rm sc})\).  A bounded-overlap cover of
\(W\) by \(a\)-balls selects one \(B_a\) with the same density lower bound.
Choose a shaded point \(x\in B_a\) and use the source ball
\(B_r=B(x,2r)\); replacing \(r\) by this fixed neighbouring scale is part
of the stability fixed in (S1).  Since \(B_a\subset B_r\), its relative
density loses at most
\((a/r)^3\), and
\[
 \left(\frac ar\right)^3
 \left(\frac\delta a\right)^{2\beta}
 =
 \left(\frac ar\right)^{3-2\beta}
 \left(\frac\delta r\right)^{2\beta}
 \ge
 \delta^{3\eta_{\rm grid}}
 \left(\frac\delta r\right)^{2\beta}.
\]
Since \(\delta/a\le\delta^\tau\), (10.8a), and then a choice of
\(\delta_0\) absorbing the \(o(1)\)-term within the remaining one-eighth
reserve, turn (10.8g) into (10.8b).

To complete the branch, the all-discrete-\(a\) multiplicity refinement made
immediately after source
equation (86) provides \(Y^{\rm sc}\subset\bar Y\) and the scale-\(r\)
shading \((\T_r,Y_r)\) for which
\[
\begin{aligned}
 |U(\T,Y^{\rm sc})|
 &\gtrsim_{\log}|U(\T_r,Y_r)|
   \frac{|U(\T,Y^{\rm sc})\cap B_r|}{|B_r|},\\
 |U(\T_r,Y_r)|
 &\gtrsim
 \delta^{\nu_{\rm thick}/9}
 \left(\frac r\delta\right)^{2\beta}
 |T|\,|\T|^{1-\beta}.
\end{aligned}
\tag{10.8h}
\]
These are source equations (87)--(88) on p.~36 at the previously fixed grid
radius \(r\).
Substitution of (10.8b) gives a
power stronger than (10.8c).  All later objects used above refine this fixed
fine shading, so their local union lies in the union appearing in (10.8h).
\end{proof}

Henceforth we are in the thin case
\[
 a<\delta^{1-\tau}.
\tag{10.8i}
\]

\subsection{Local cell volume and the cellular parent output}

In the thin regime apply the first stage of \cref{thm:tagged-sync} tagwise,
using (10.7).  Before this call, the input density, (10.7V), and the
pre-cellular cutoff \(\omega_{\rm pre}<\eta_{\rm src}\) give
\[
 \frac{|\widetilde Y_B(T_B)|}{|T_B|}
 \gtrsim
 \frac{c_{\rm sc}c_{\rm init}}{L_{\rm mult}}
 \delta^{\eta_{\rm src}}
 \ge \delta^{2\eta_{\rm src}}.
\tag{10.8A0}
\]
Hence \cref{lem:goodball-reorder} applies without any \(J_v\)-dependent
input.  Its fixed retention, one nonempty tubelet of volume
\(\asymp\delta^2r_1\), (10.7W$'$), polynomial complexity, and the fixed-ball
upper bound verify (10.8A).  Put
\(\mathcal B_{\rm good}:=\mathcal A_1\) and
\[
 \mu_{\rm in}:=\mu_*,\quad \omega_0:=\omega_*,\quad d_0:=d_*,
 \quad L_{\rm reg}:=L_{\rm reg}^{\max},
\quad P_{\rm tag}:=L_\mu L_\omega L_d .
\tag{10.8k}
\]
Every local parameter is within a factor two of these references.  By
(10.8D)--(10.8E), the retained tags carry at least
\((P_{\rm tag}J_v^{\max}L_{\rm reg}^{\max})^{-1}\) of represented fine mass.

Fix one such tagged outer-ball component for the next three displays; thus
\(Q\) means the tagged vertex \((B,Q)\), and all degrees in (10.9)--(10.10)
are computed before physical recombination.  After its local joint graph
regularization, let
\(w(Q)=\sum_{W:(W,Q)\in E_2}\omega_{WQ}\).  If the right degree is
\(d_B\le d(Q)<2d_B\), where \(d_0\le d_B<2d_0\), then
\[
 d_0\omega_0\le w(Q)<16d_0\omega_0.
\tag{10.9}
\]
Property P5 bounds the child multiplicity of \(Y_{\rm reg}\) on \(Q\) by
\(16\mu_{\rm in}d_0\); therefore
\[
 |U(\T,Y_{\rm reg})\cap Q|\ge \frac{\omega_0}{16\mu_{\rm in}}.
\tag{10.10}
\]
After projection to physical space, the same pointwise bound and any sum of
the local union lower bounds acquire at most the factor
\(\kappa_{\rm amb}\), by \cref{lem:tag-recombine}.  Thus (10.9)--(10.10)
remain tagged identities until this controlled recombination.
We may now select one dyadic level of
\(v(Q)=|U(\T,Y_{\rm reg})\cap N_{Ca}(Q)|\).  Equations (10.9)--(10.10) and
\(v(Q)\le C'a^3\) show that the theorem-level count (10.8A$''$) is bounded by
the common envelope
\[
 \widehat J_v:=1+\max\left\{0,\left\lceil
  \log_2\!\left(\frac{16C'a^3\mu_{\rm in}}{\omega_0}\right)
 \right\rceil\right\},\qquad J_v(B)\le\widehat J_v.
\tag{10.11}
\]
This is logarithmic under polynomial complexity and remains explicit
otherwise.  Set \(J_v:=\widehat J_v\), padding shorter local partitions by
empty intervals.  The saturated selection of \cref{thm:tagged-sync} chooses
\(\mathcal Q_v(B)\) with
\[
 \sum_{Q\in\mathcal Q_v(B)}w(Q)
 \ge J_v^{-1}\sum_Qw(Q).
\tag{10.11$'$}
\]
Retain the saturated edge set
\[
 E_3(B)=\{(W,Q)\in E_2(B):Q\in\mathcal Q_v(B)\},
\tag{10.11$''$}
\]
and define the resulting post-selection shadings, in particular \(Y_1\), by
(4.10), discarding empty parents.  The saturated
conclusions (10.8D), P2--P5, and the local P1 cost \(J_v\) now follow directly
from \cref{thm:tagged-sync}.

\begin{lemma}[Good-ball local transfer]
\label{lem:goodball-transfer}
Assume every selected tubelet shading \(Y_B(T_B)\) is covered by a
bounded-overlap family of selected \(\delta\)-balls \(D\), each of which meets
the shading and satisfies
\[
 |D\cap Y_B(T_B)|\ge c_{\rm gb}\delta^\eta\delta^3,
\tag{10.11a}
\]
and assume every physical cell in
\(U(\widetilde\W'_B,Z)\) is a positive parent--cell incidence for the joint
factoring output.  Then
\[
 |U(\T_B,Y_B)|
 \gtrsim c_{\rm gb}\delta^\eta(\delta/a)^3
 |U(\widetilde\W'_B,Z)|.
\tag{10.11b}
\]
More generally, let \(Z'(W)\subset Z(W)\) be any measurable parent
subshading, and let \(B_R\) be a ball with \(R\ge a\).
Then, for a fixed dimensional \(C\),
\[
 |U(\T_B,Y_B)\cap B_{CR}|
 \gtrsim c_{\rm gb}\delta^\eta(\delta/a)^3
 |U(\widetilde\W'_B,Z')\cap B_R|.
\tag{10.11b$'$}
\]
Repeated parent labels and cells meeting the parent boundary do not alter the
constant.
\end{lemma}

\begin{proof}
The parent union is the disjoint union of its distinct half-open
\(h\)-cells, with \(h=c_0a\).  Color the grid by residue classes modulo a
fixed integer \(M>4C/c_0\).  A color carrying a maximal number of active
physical cells retains a fixed fraction, and two of its cells are separated
by more than \(3Ca\).

For each retained physical cell \(Q\), choose one labelled positive edge
\((W,Q)\).  Positivity and the definition of \(Y_1\) give a tubelet \(T_B\)
and a point \(x_Q\in Y_1(T_B)\cap Q\).  Choose one of the selected good
\(\delta\)-balls \(D_Q\) containing \(x_Q\), using the fixed half-open
tie-breaking convention on ball boundaries.  Since \(\delta\le a\), the
balls \(D_Q\) belonging to the chosen grid color are pairwise disjoint
(enlarging or shrinking all balls by one fixed factor only changes
\(c_{\rm gb}\)).  Consequently
\[
\begin{aligned}
 |U(\T_B,Y_B)|
 &\ge \sum_Q |D_Q\cap Y_B(T_B(Q))|\\
 &\ge c_{\rm gb}\delta^\eta\delta^3\#\{Q\}\\
 &\gtrsim c_{\rm gb}\delta^\eta(\delta/a)^3
       |U(\widetilde\W'_B,Z)|.
\end{aligned}
\]
The count is over physical cells, so coincident parent labels have no effect
on the right side.  Each measured set lies in the tubelet shading, and the
grid coloring costs a dimensional constant.  For the localized assertion,
run the same argument on all active
physical cells that meet \(U(\widetilde\W'_B,Z')\cap B_R\).  Because each
such cell has volume \(h^3\), their full cell volumes dominate that union
volume.  The inclusion \(Z'(W)\subset Z(W)\) makes every chosen cell a
positive original parent--cell incidence.  The point \(x_Q\) is within
\(O(a)\) of \(B_R\), and its good \(\delta\)-ball is therefore contained in
\(B_{CR}\) because \(R\ge a\).  The same residue-class coloring makes those
balls disjoint and proves (10.11b$'$).
\end{proof}

The self-contained labelled construction in \cref{lem:biased-factoring}
uses the same exponent in (8.0D) and (8.0F)--(8.0G), matching the biased
formulation of \cite[Lemma~9.2, p.~35]{GWZ2026}.  We therefore identify
\[
 \eta_{\rm bias}:=\varepsilon_{\rm fact};
\tag{10.11c}
\]
so the two symbols denote the same parameter.  Put
\[
 \sigma=\eta_{\rm bias},\qquad q=2-\sigma,\qquad
 \mathfrak s=a/b,
\]
and let \(L_N\) denote the child-count level and
\(L_{\rm reg}\) the J1/REG-1/REG-2 product.  Define once, for every later
branch,
\[
 P_{\rm pre}:=
 c_{\rm setup}^{-1}J_v\kappa_{\rm inc}L_mL_{a,b}.
\tag{10.11d}
\]
This collects the setup, representation, dimension, and local-volume losses
already exposed above.  The
rectangular theorem and Jensen give the cellular parent-density estimate
\[
 \lambda(\widetilde\W'_B,Z)
 \gtrsim C_0^{-q}\mathfrak s^\eps\kappa_D^{-q}
 (P_{\rm pre}L_NL_{\rm reg})^{-q}
 \lambda(\T_B,Y_B)^q.
\tag{10.12}
\]
The relevant scale is \(a/b\), and every synchronized mass loss has power
\(-q\).  This is the parent-density line corresponding to source equation
(95).  Exact compatibility is
\[
 (\T_B)_{Y_1}(x)
 \subset
 \bigcup_{\widetilde W:\,x\in Z(\widetilde W)}
 (\T_{B,W})_{Y_1}(x),
\tag{10.13}
\]
which is the incidence line corresponding to source equation (96).
\Cref{lem:goodball-transfer} gives
\[
 |U(\T_B,Y_B)|
 \gtrsim \delta^\eta(\delta/a)^3
 |U(\widetilde\W'_B,Z)|.
\tag{10.14}
\]
The left side is evaluated on the reference shading \(Y_B\).

\subsection{Broad middle axes}

When the middle axis is broad,
\(b\ge \delta^{\varepsilon_{\rm scal}}r_1\), the labelled slab lemma
proved in \cref{app:slab} applies to the \(a\times b\times r_1\) parent
boxes.  In this branch use the expanded child family
\((\T_B^{\rm lab},Y_B^{\rm lab})\) from (10.7a).  Assign all copies of one
carrier to the same biased parent and rerun the joint graph.  Every
parent--cell edge weight is multiplied by the common \(m_{\rm fine}\), so
the dyadic edge and right-degree selections, the parent family, and the
physical parent shading \(Z\) are unchanged.  The child shading density and
the normalized complete-child fractions are also unchanged by (10.7b).
Since the parent boxes have comparable volume, define their labelled
parent-shading density by
\[
 \lambda_B^{\rm par}
 :=\frac{M(Z)}
 {|\widetilde\W'_B|\,|W|}
 =\frac{\sum_{W\in\widetilde\W'_B}|Z(W)|}
 {|\widetilde\W'_B|\,|W|}.
\tag{10.14a}
\]
With \(L_{\rm slab}=1+\lceil\log_2(r_1/a)\rceil\), the slab lemma gives
\[
 \frac{|U(\widetilde\W'_B,Z)|}
 {|\widetilde\W'_B|\,|W|}
 \gtrsim
 \frac{(\lambda_B^{\rm par})^2}
 {1+L_{\rm slab}\Delta_{\max}(\widetilde\W'_B)(r_1/b)^2}.
\tag{10.15}
\]
This formula retains both normalizations needed in the application.  First,
union volume follows directly from (10.15), with both powers of
\(\lambda_B^{\rm par}\) present.  Second, the parent-mass comparison is
\[
 |\widetilde\W'_B|\,|W|
 \gtrlog r_1^2\delta^\eta |\T_B^{\rm lab}|\,|T_B|,
\tag{10.16}
\]
and its normalization is as follows.  With
\[
 \Delta_{\max}(\T_B^{\rm lab})
 :=\sup_{K\in\mathfrak K_3}
   \frac{\sum_{(T_B,\ell)\subset K}|T_B|}{|K|},
\]
complete containment gives
\[
\begin{aligned}
 |\widetilde\W'_B|\,|W|
 &=\sum_{W\in\widetilde\W'_B}|W|\\
 &\ge \Delta_{\max}(\T_B^{\rm lab})^{-1}
    \sum_{W\in\widetilde\W'_B}
       \sum_{(T_B,\ell)\in\T^{\rm lab}_{B,W}}|T_B|\\
 &\gtrlog \Delta_{\max}(\T_B^{\rm lab})^{-1}
          |\T_B^{\rm lab}|\,|T_B|.
\end{aligned}
\tag{10.16$'$}
\]
The last inequality is the retained dimension/child-count level;
it counts labelled tubelets, including coincident carriers.  At this scale,
the source thickening comparison gives
\[
 \Delta_{\max}(\T_B^{\rm lab})
 \le C\kappa_{\rm loc}r_1^{-2}\Delta_{\max}(\T)
 \le C\kappa_{\rm loc}r_1^{-2}\delta^{-\eta},
\tag{10.16$''$}
\]
which proves (10.16), since \(\kappa_{\rm loc}=O(1)\).  The source
complete-containment bound
\(\Delta_{\max}(\T)\le\delta^{-\eta}\) contributes the factor
\(\delta^\eta\) in this comparison.
Equations (10.14)--(10.16) give the complete local slab inequality.
Converting it into a numerical terminal gain uses the common
source-to-admissible quantitative bounds proved in
\cref{prop:source-extraction}; the resulting exponent is displayed in
(10.33a).

\subsection{The angular degree dichotomy}

For every thin branch set
\[
 \alpha=a/r_1,\qquad q_{\rm ang}=a/b,\qquad \rho_2=b/r_1.
\]
The common thin range has
\(\delta^{1-\varepsilon_{\rm scal}}\le\alpha\le\rho_2\le1\) and
\(q_{\rm ang}=\alpha/\rho_2\ge\alpha\).  In the non-slab alternative one has
the additional upper restriction
\[
 \delta^{1-\varepsilon_{\rm scal}}\le\alpha\le\rho_2
 \le\delta^{\varepsilon_{\rm scal}},
 \qquad q_{\rm ang}\ge\alpha.
\tag{10.16a}
\]
The remaining non-slab argument is organized by the local angular degree:
small degree yields a direct multiplicity bound, while large degree passes to
the transverse--tangential decomposition.
Before splitting by degree, we collect the losses that have already been
incurred.  With \(P_{\rm pre}\) fixed
once in (10.11d), and up to one fixed numerical constant, put
\[
\begin{aligned}
 P_{\rm sync}
   &=\kappa_{\rm inc}P_{\rm tag}P_{\rm pre}L_NL_{\rm reg},\\
 P_{\rm angtag}
   &:=L_{\theta}^{\rm tag}L_{d_\angle}^{\rm tag},\\
 c_*^{-1}
   &\le
   \frac{64\kappa_{\rm inc}^2P_{\rm tag}P_{\rm angtag}
         J_vL_mL_{a,b}L_NL_{\rm reg}}
         {c_{\rm setup}c_{\rm ang}}.
\end{aligned}
\tag{10.16b}
\]
Define nonnegative exponents by
\[
\begin{gathered}
 P_{\rm sync}\le\delta^{-\Lambda_{\rm sync}},\qquad
 P_{\rm den}:=\delta^{-\eta_{\rm src}}P_{\rm sync}
     \le\delta^{-\Lambda_{\rm den}},\qquad
 C_0\kappa_D(P_{\rm pre}L_NL_{\rm reg})
     \le\delta^{-\Lambda_{\rm pre}},\\
 c_*^{-1}\le\delta^{-\Lambda_*},\qquad
 D_W,D_S\le\delta^{-\Lambda_D},\qquad
 H_{\rm col}\le\delta^{-\Lambda_H}.
\end{gathered}
\tag{10.16c}
\]
Here
\[
 D_W=\max_B\Delta_{\max}(\widetilde\W'_B),\qquad
 D_S=\max_{B,S}\Delta_{\max}(\widetilde\W^\angle_{B,S}).
\]
We next define a uniform envelope independent of the subsequently selected
tag.  Fix a structural integer \(C_{\rm env}\) larger than the
total multiplicity with which any primitive factor occurs in
(10.16b), (10.19), and (10.25).  For every possible retained tag family and
every realized thin discrete scale pair
\(\delta^{1-\varepsilon_{\rm scal}}\le\alpha\le\rho_2\le1\), set
\[
 \begin{aligned}
 \mathcal E_{\rm setup}
  &:={}
  c_{\rm setup}^{-1}C_0\kappa_D\kappa_{\rm inc}\kappa_{\rm loc}
  J_vL_mL_{a,b}L_NL_{\rm reg}P_{\rm tag}P_{\rm pre}P_{\rm angtag},\\
 \mathcal E_{\rm ang}&:=\max_B c_{{\rm ang},B}^{-1},\\
 \mathcal E_{\rm src}(\delta)
  &:=\max\left\{1,
    \delta^{-C_{\rm env}\omega_0(\delta)},
    [\log(2/\delta)]^{C_{\rm env}},
    \sup_{\rm tags,scales}\mathcal E_{\rm setup}^{C_{\rm env}},
    \sup_{\rm tags,scales}\mathcal E_{\rm ang}^{C_{\rm env}}
    \right\},\\
 \omega_{\rm src}(\delta)
  &:=\frac{\log\mathcal E_{\rm src}(\delta)}{\log(1/\delta)}.
 \end{aligned}
\tag{10.16d}
\]
For slab pairs, where ANGLE-CELL is not invoked, the angular-only factors
\(P_{\rm angtag}\) and \(\mathcal E_{\rm ang}\) are interpreted as \(1\), and
neither \(D_S\) nor \(H_{\rm col}\) is used.  The setup suprema range over the
finite tagged nets and all thin discrete scales in the fixed polynomial
window, including both (S1b) and (S1c) and choices later deleted by
\(\eta,\tau,\eta_{\rm pl}\); the angular supremum is restricted to the
non-slab window (10.16a).  Consequently \(\omega_{\rm src}\) is defined
before, and is independent of, those later exponent choices.

The operation map in \cref{lem:exact-loss}, using the pre-cellular
verification (10.8A0) of (10.8A), verifies uniformly that
\(\omega_{\rm src}(\delta)\to0\): all setup, reference-density, and angular entries are bounded
or subpower, and the only non-subpower term is inherited
complete-containment.  The resulting \(D_W\) bound holds on the full thin
window, while \(D_S\) and \(H_{\rm col}\) occur only in the non-slab branch:
\[
 \begin{gathered}
 D_W\le
 \delta^{-2(1-\varepsilon_{\rm scal})\varepsilon_{\rm fact}
             -\omega_{\rm src}(\delta)}
 \quad\text{on all thin pairs},\\
 D_S\le
 \delta^{-2(1-\varepsilon_{\rm scal})\varepsilon_{\rm fact}
             -\omega_{\rm src}(\delta)},
 \qquad H_{\rm col}\le CD_S
 \quad\text{on non-slab pairs}.
 \end{gathered}
\tag{10.16e}
\]
Consequently we may take
\[
\begin{aligned}
 \Lambda_{\rm sync}&=\omega_{\rm src}(\delta),&
 \Lambda_{\rm pre}&=\omega_{\rm src}(\delta),&
 \Lambda_*&=\omega_{\rm src}(\delta),\\
 \Lambda_D&=2(1-\varepsilon_{\rm scal})\varepsilon_{\rm fact}
              +\omega_{\rm src}(\delta),&
 \Lambda_H&=\Lambda_D+\omega_{\rm src}(\delta),&
 \Lambda_{\rm den}&=\eta_{\rm src}+\omega_{\rm src}(\delta).
\end{aligned}
\tag{10.16f}
\]
The choice of \(C_{\rm env}\) absorbs the finite powers in (10.16b), so this
single \(\omega_{\rm src}\) controls every later \(o(1)\)-term.

In every retained tag, ANGLE-CELL gives a local mass fraction
\[
 c_{{\rm ang},B}\ge \frac{A_{\rm ang}^{-K}}{J_\theta(B)J_d(B)}.
\tag{10.17}
\]
In the high-degree case below we invoke the second stage of
\cref{thm:tagged-sync}, replace the tag family by \(\mathcal A_2\), and set
\[
 c_{\rm ang}:=c_{\rm ang}^{\min},\qquad
 \theta:=\theta_*,\qquad d_\angle:=d_{\angle,*}.
\]
These are synchronized reference values.  For a retained tag \(B\), keep
\(\theta_B\) and \(d_{\angle,B}\) for its local outputs.  Then
\[
 \theta\le\theta_B<2\theta,\qquad
 d_\angle\le d_{\angle,B}<2d_\angle,
\tag{10.17a}
\]
and A1--A3 hold with \(\theta_B,d_{\angle,B}\).  Thus the corresponding
common upper bounds for the local cell diameter and degree are
\(4\theta\) and \(4d_\angle\).
The whole-tag angle synchronization costs \(P_{\rm angtag}\) in global
represented mass and preserves the local conclusions.  Since the local input
degrees are at least \(d_0\), the synchronized output
degree satisfies
\[
 d_\angle\ge \tfrac14A_{\rm ang}^{-K}d_0.
\tag{10.18}
\]
Moreover \(c_{\rm ang}\ge\alpha^{o(1)}\), because the local level counts are
logarithmic.  In the low-degree case the angle stage is not run and, for the
common bookkeeping only, \(P_{\rm angtag}=c_{\rm ang}^{-1}=1\).
The factor \(A_{\rm ang}^{-K}\) prevents an unbuffered input-degree threshold from
implying the output threshold required by the plank-reduction lemma.  Define
\[
 H_{\rm ang}=4A_{\rm ang}^K.
\tag{10.19}
\]
\paragraph{Low degree.}
The low-degree estimate passes from the tubelet shading \(Y_1\) to the fine
reference shading through represented incidences.  Let
\(\mathcal R(T)\) be the retained labelled tubelet representations of a fine
tube \(T\), and define
\[
 Y_1^{\rm f}(T)=\widehat Y(T)\cap
 \bigcup_{(B,T_B)\in\mathcal R(T)}Y_1(T_B).
\tag{10.19a}
\]
Writing
\[
 M_{\rm c}(A)=\sum_{(B,T_B)}|A(T_B)|,\qquad
 U_{\rm c}(\T_B,A)=\bigcup_{(B,T_B)}A(T_B),\qquad
 \mu_{\rm c}(\T_B,A)=M_{\rm c}(A)/|U_{\rm c}(\T_B,A)|,
\]
joint regularization and the global \(m_{\rm fine}\)-level give
\[
\begin{aligned}
 M_{\rm c}(Y_1)&\ge
   (P_{\rm tag}J_vL_NL_{\rm reg})^{-1}M_{\rm c}(Y_B),\\
 P_{\rm sync}^{-1}M(\bar Y)
     &\le M(Y_1^{\rm f})\le M(\bar Y).
\end{aligned}
\tag{10.19b}
\]
Property P5 gives the pointwise and average coarse bounds
\[
 \sum_{T_B}\1_{Y_1(T_B)}(x)
   <16\kappa_{\rm amb}\mu_{\rm in}d_0,\qquad
 \mu_{\rm c}(\T_B,Y_1)
   \le16\kappa_{\rm amb}\mu_{\rm in}d_0,
\tag{10.19c}
\]
and consequently
\[
 |U_{\rm c}(\T_B,Y_1)|
 \ge\frac{M_{\rm c}(Y_1)}
          {16\kappa_{\rm amb}\mu_{\rm in}d_0}.
\tag{10.19d}
\]
\begin{lemma}[Incident-centred cap comparison]
\label{lem:incident-cap}
Assume the reference shading \(\bar Y\) has the source uniformity property that,
whenever \(x\in\bar Y(T_0)\),
\[
 \#\{T:x\in\bar Y(T),\
       \ang(T,T_0)\le C_{\rm geo}\rho_2\}
 \le C_{\rm cap}\mu(\rho_2),
\tag{10.19e0}
\]
where changing \(\rho_2\) by a fixed dyadic factor changes \(\mu(\rho_2)\)
by at most a fixed uniformity constant.  At a point incident to one retained
inner parent,
\[
 m_{\rm fine}\mu_{\rm in}
 \le C\kappa_{\rm loc}C_{\rm cap}\mu(\rho_2).
\tag{10.19e1}
\]
\end{lemma}

\begin{proof}
All length-\(r_1\) tubelets contained in one
\(a\times b\times r_1\) parent have directions in a ball of radius
\(C_{\rm geo}b/r_1=C_{\rm geo}\rho_2\).  At the point under consideration,
the \(\mu_{\rm in}\) tubelet incidences of that parent represent at least
\(m_{\rm fine}\mu_{\rm in}\) labelled fine incidences by (10.4) and (10.7).
Inside this fixed tagged outer ball, every physical fine label has at most
\(\kappa_{\rm loc}\) canonical tubelet
representations, so at least
\(m_{\rm fine}\mu_{\rm in}/\kappa_{\rm loc}\) distinct incident fine labels
remain.  Choose one of them as \(T_0\).  Every other one lies in the
incident-centred angular ball in (10.19e0); applying that upper bound and
multiplying by \(\kappa_{\rm loc}\) proves (10.19e1).  Thus angular
uniformity is invoked on \(\bar Y\) with an incident-centred cap.
``Distinct'' means distinct source labels, so coincident geometric carriers
remain separate in both counts.
\end{proof}

At a retained physical tubelet point, (10.4), P5 in each tag, and
\cref{lem:tag-recombine} bound the represented labelled fine multiplicity by
\(64\kappa_{\rm amb}m_{\rm fine}\mu_{\rm in}d_0\): the local parent--cell
bound contributes \(16\), the synchronized fine level contributes \(4\), and
tag recombination contributes \(\kappa_{\rm amb}\).  Apply
\cref{lem:incident-cap}, and identify \(\mu(\rho_2)\) with the source
quantity \(\mu(\T[T_{\rho_2}],\bar Y)\) up to the fixed neighbouring-scale
constant.  Together with the preceding pointwise upper bound, this proves
\[
 \mu(\T,Y_1^{\rm f})
 \le64\kappa_{\rm amb}\kappa_{\rm loc}d_0\,
       \mu(\T[T_{\rho_2}],\bar Y).
\tag{10.19e}
\]
Combining (10.19b) and (10.19e), and using
\(U(\T,Y_1^{\rm f})\subset U(\T,\bar Y)\), yields the mass--union chain
\[
\begin{aligned}
 |U(\T,Y_1^{\rm f})|
 &\ge
 \frac{M(\bar Y)}
 {64\kappa_{\rm amb}\kappa_{\rm loc}P_{\rm sync}d_0
  \mu(\T[T_{\rho_2}],\bar Y)},\\
 \mu(\T,\bar Y)
 &\le64\kappa_{\rm amb}\kappa_{\rm loc}P_{\rm sync}d_0
       \mu(\T[T_{\rho_2}],\bar Y).
\end{aligned}
\tag{10.19f}
\]

Suppose now that
\(d_0<H_{\rm ang}\alpha^{-\eta_{\rm pl}}\).  Cap uniformity and the
non-sticky count state
\[
\begin{aligned}
 \mu(\T[T_{\rho_2}],\bar Y)
  &\lesssim\delta^{-\eta_{\rm bias}}
       |\T[T_{\rho_2}]|^\beta,\\
 |\T[T_{\rho_2}]|\,|\T_{\rho_2}|
  &\lesslog|\T|,\qquad
 |\T_{\rho_2}|\ge\rho_2^{-2-\zeta}.
\end{aligned}
\tag{10.19g}
\]
For every fixed \(\xi>0\), the ANGLE-CELL stopping calculation gives
\(H_{\rm ang}\le C_\xi\alpha^{-\xi}\).  Substitution in (10.19f) gives
\[
 \mu(\T,\bar Y)
 \lesssim
 \delta^{\nu_{\rm low}-o(1)}|\T|^\beta,
 \qquad
 \nu_{\rm low}
 =\varepsilon_{\rm scal}\beta(2+\zeta)
  -\Lambda_{\rm sync}-\eta_{\rm bias}
  -(1-\varepsilon_{\rm scal})(\eta_{\rm pl}+\xi).
\tag{10.19h}
\]
Thus this alternative closes, without the plank-density lemma, under the
strict inequality
\[
 \Lambda_{\rm sync}+\eta_{\rm bias}
 +(1-\varepsilon_{\rm scal})(\eta_{\rm pl}+\xi)
 <\varepsilon_{\rm scal}\beta(2+\zeta).
\tag{L}
\]

If instead \(d_0\ge H_{\rm ang}\alpha^{-\eta_{\rm pl}}\), then (10.18)
gives the output threshold
\[
 d_\angle\ge\alpha^{-\eta_{\rm pl}}.
\tag{10.20}
\]
The angle mass fraction and (10.12) also imply the parent-density hypothesis
\[
 c_{\rm ang}C_0^{-q}(a/b)^\eps\kappa_D^{-q}
 (P_{\rm pre}L_NL_{\rm reg})^{-q}\delta^{q\eta}
 \ge \alpha^{\eta_{\rm pl}}
\tag{10.21}
\]
as follows.  Denote the left side by \(\mathcal D\).  For each fixed
\(\xi>0\), (10.17) gives \(c_{\rm ang}\ge C_\xi\alpha^\xi\); hence
\[
 \mathcal D\ge C_\xi\alpha^{E_{\rm high}},\qquad
 E_{\rm high}
 =\xi+\eps+\frac{q(\eta+\Lambda_{\rm pre})}
                    {\varepsilon_{\rm scal}}.
\tag{10.21a}
\]
Here \(q_{\rm ang}\ge\alpha\), while
\(\delta^u\ge\alpha^{u/\varepsilon_{\rm scal}}\) follows from (10.16a).
Choose a fixed \(\kappa>0\) and require
\[
 \xi+\eps+\frac{q(\eta+\Lambda_{\rm pre})}
                   {\varepsilon_{\rm scal}}
 \le\eta_{\rm pl}-\kappa.
\tag{H}
\]
After decreasing \(\delta_0\), (H) proves (10.21).  The subsequent source
call uses (10.20) and (10.21) as its multiplicity and density hypotheses.

One further density comparison is needed in the tangential branch.  Given a
Katz--Tao loss \(\varepsilon_{KT}>0\), let
\(\eta_{KT}(\varepsilon_{KT},\beta)>0\) be its required shading-density
exponent.  The inequality
\[
 E_{\rm dens}:=(1-\varepsilon_{\rm scal})\xi+\eps
               +q(\eta+\Lambda_{\rm pre})
 <\varepsilon_{\rm scal}\eta_{KT}
\tag{K}
\]
implies \(\mathcal D\ge\rho_2^{\eta_{KT}}\) with a fixed positive power
reserve.  After decreasing the final scale cutoff, that reserve absorbs the
structural carrier-density comparison in (10.29d).  Consequently every later
proxy colour whose normalized density is at least one half of its ambient
proxy-slab density satisfies the density hypothesis of the generalized
Katz--Tao estimate.

\begin{lemma}[Buffered labelled-multiset filling lemma]
\label{lem:labelled-plank-reduction}
Fix \(\eta>0\) and constants \(C_{\rm comp},C_{\rm box}\ge1\).  Let
\(0<\alpha\le\rho\le1\) be sufficiently small, and let
\((\mathcal P,Y)\) be a finite labelled multiset of
\(C_{\rm box}\)-comparable rectangular
\(\alpha\times\rho\times1\) planks in the unit ball, equipped
with measurable shadings.  Each
plank carries an ordered orthonormal short--middle--long frame; coincident
carriers with different labels or frames are allowed, and every cardinality,
mass, multiplicity, and
\(\Delta_{\max}\) counted by labels.  Assume
\(|\mathcal P|\le\alpha^{-C_{\rm comp}}\), and suppose
\[
 \lambda(\mathcal P,Y)\ge\alpha^{\eta},\qquad
 \mu(\mathcal P,Y)\ge\alpha^{-\eta},
\]
Let \(m>0\), and suppose the labelled multiplicity lies in \([m,2m)\)
on the union.  Let \(0<\theta\le1\), and let
\(A_{\rm ang}=A_{\rm ang}(\alpha),B_{\rm ang}=B_{\rm ang}(\alpha)\ge2\)
satisfy
\(\log A_{\rm ang}+\log B_{\rm ang}=o(\log\alpha^{-1})\), and assume that
\(A_{\rm ang}/\log^C(2/\alpha)\to\infty\) for every fixed \(C\).  Assume that at each shaded
point all incident two-long-axis planes have raw angular diameter at most
\(C\theta\), while every incident subcollection of at least
\(2m/A_{\rm ang}\) labels has raw angular diameter greater than
\(\theta/B_{\rm ang}\).  Finally assume
the buffer
\[
 \theta\ge2B_{\rm ang}\,\alpha/\rho.
\tag{10.21b}
\]
Then there is a labelled refinement \((\mathcal P',Y')\) with
\[
 M(Y')\ge\alpha^{o(1)}M(Y),
\tag{10.21c}
\]
where this occurrence uses the displayed uniform modulus
\[
 \omega_{\rm fill}(\alpha):=
 \frac{\log(1/c_{\rm fill})}{\log(1/\alpha)}\longrightarrow0
\]
and \(c_{\rm fill}\in(0,1]\) is the fixed structural mass-retention product
furnished by \cref{thm:multiset-appendix}.  In particular, this is a retained
labelled shading-mass fraction, not a fraction of
distinct carriers or labels.  Whenever
\(U(\mathcal P',Y')\) meets a ball of radius \(\theta\rho\), its intersection
with a fixed concentric enlargement of that ball has relative volume at least
\(\alpha^{4\eta+o(1)}\).
Passing to any labelled subfamily cannot increase its labelled
\(\Delta_{\max}\).
\end{lemma}

\begin{proof}
This is \cref{thm:multiset-appendix}, proved independently in
\cref{app:multiset}.  Its proof uses translated full-piece grids, the labelled
denominator in every grid box, and an ordered-pair estimate.  The buffer
(10.21b) is built into the theorem, whose filling conclusion is the input to
the transverse branch below.
\end{proof}

\begin{lemma}[Transverse plank and filling interface]
\label{lem:transverse-interface}
Assume (10.20)--(10.21),
\(\theta\ge\delta^{-\tau'}q_{\rm ang}\), and
\(\delta^{-\tau'}\ge2B_{\rm ang}\).  Fix a retained outer-ball tag \(B\),
write \(\theta_B,d_{\angle,B}\) for its actual outputs in (10.17a), and put
\(R_B=\theta_Bb\).  There are structural constants \(C_0,C_1\ge1\), a point
\(y_B\), and a source discrete radius \(r\) such that
\[
 y_B\in U(\T_r,Y_r),\qquad
 2C_0R_B\le r\le C_1\delta^{-\eta_{\rm grid}}R_B,
\]
and
\[
 \frac{|U(\T,Y^{\rm sc})\cap B(y_B,2C_0R_B)|}
 {|B(y_B,2C_0R_B)|}
 \gtrlog \delta^\eta(\delta/a)^3\alpha^{4\eta_{\rm pl}}.
\tag{10.22}
\]
For this radius,
\[
 \frac{|U(\T,Y^{\rm sc})\cap B(y_B,r)|}{|B(y_B,r)|}
 \gtrlog\delta^{E_{\rm tr}}
 \ge\delta^{-2\nu_{\rm tr}}\left(\frac{\delta}{r}\right)^{2\beta},
\tag{10.23}
\]
where
\[
 E_{\rm tr}=3\tau+\eta+4\eta_{\rm pl}+3\eta_{\rm grid},
 \qquad
 \nu_{\rm tr}=\tau'\beta-\tfrac12E_{\rm tr}>0.
\tag{10.23a}
\]
Consequently the filling-to-multiplicity implication in source equation
(89) yields the required very-non-sticky estimate in this alternative.
\end{lemma}

\begin{proof}
Fix the tag \(B\) from the statement and scale its \(r_1\)-ball to the unit
ball.  By the permanent-parent convention (3.4d), the dimension-level
selection, and the L\"owner-rectangle transfer in \cref{cor:convexbox}, there
is one structural \(C_{\rm box}\ge1\), fixed before the scale cutoff, for
which the permanent angular parents passed to Appendix~C are
\(C_{\rm box}\)-comparable rectangular
\(\alpha\times\rho_2\times1\) planks with their inherited ordered frames.
Using the actual sum of these carrier volumes, their shading density is at
least \(\mathcal D\ge\alpha^{\eta_{\rm pl}}\) by (10.21), and their pointwise
multiplicity lies in
\([d_{\angle,B},2d_{\angle,B})\).  By (10.17a) and (10.20),
\(d_{\angle,B}\ge d_\angle\ge\alpha^{-\eta_{\rm pl}}\).
Use \(m=d_{\angle,B}\), this \(C_{\rm box}\), and the local scale \(\theta_B\) in
\cref{lem:labelled-plank-reduction}.  A2 gives raw angular diameter at most
\(2\theta_B\).  If an incident subcollection has at least
\(2m/A_{\rm ang}\) labels,
then it has at least the cell degree divided by \(A_{\rm ang}\), so A3 applies.
Finally,
\[
 \theta_B\ge\theta\ge\delta^{-\tau'}q_{\rm ang}
 \ge2B_{\rm ang}q_{\rm ang}
\]
converts the resolved lower diameter into the required raw one and verifies
the buffer.  The affine family is a labelled multiset, and its cellular
shadings are measurable finite unions of cells; these observations verify
the remaining hypotheses of the lemma.  Its conclusion supplies a
measurable parent subshading
\(Z_{\rm tr}\subset Z^\angle\subset Z\) whose union has filling fraction
\(\gtrlog\alpha^{4\eta_{\rm pl}}\) in every intersecting ball of normalized
radius \(\theta_B\rho_2\), which scales back to \(R_B=\theta_Bb\).  Choose
one such ball and call it \(B_0=B(x_B,C_0R_B)\), increasing the structural
constant \(C_0\) once to cover the fixed enlargement in (10.11b$'$).
Applying that localized estimate to \(Z_{\rm tr}\) gives
\[
 \frac{|U(\T_B,Y_B)\cap B_0|}{|B_0|}
 \gtrlog\delta^\eta(\delta/a)^3\alpha^{4\eta_{\rm pl}}.
\]
In particular this intersection has positive measure; choose
\(y_B\in U(\T_B,Y_B)\cap B_0\).  Then
\(B_0\subset B(y_B,2C_0R_B)\), and comparison of the two ball volumes loses
only a structural constant.  Every point of \(Y_B(T_B)\) has positive
represented multiplicity in the
pre-biased fine shading \(Y^0\), so
\(U(\T_B,Y_B)\subset U(\T,Y^0)\subset U(\T,\bar Y)\).  The reference
ordering (10.32f) more precisely gives
\(U(\T_B,Y_B)\subset U(\T,Y^{\rm sc})\), so the recentered-ball comparison
proves (10.22).  Choose \(r\) to be the first
source grid radius at least \(2C_0R_B\); neighbouring-grid comparability
gives \(r\le C_1\delta^{-\eta_{\rm grid}}R_B\).  By the definition of the
induced scale-
\(r\) shading, every retained point of \(Y^{\rm sc}\) lies in the shading of
its scale-
\(r\) ancestor.  Thus \(y_B\in U(\T_r,Y_r)\), verifying the centering
hypothesis in source equation (87).

The thin case gives \((\delta/a)^3\ge\delta^{3\tau}\), while
\(\alpha^{4\eta_{\rm pl}}\ge\delta^{4\eta_{\rm pl}}\).
Passing from \(R_B\) to \(r\) loses at most
\((R_B/r)^3\ge\delta^{3\eta_{\rm grid}}\), proving the first inequality in
(10.23).  Also
\[
 R_B=\theta_B b\ge\delta^{-\tau'}a\ge\delta^{1-\tau'},
 \qquad r\ge R_B,
\]
so \((\delta/r)^{2\beta}\le\delta^{2\tau'\beta}\).  The definition of
\(\nu_{\rm tr}\) proves the second inequality, and (10.32c) makes it
positive.

We now combine the two terminal source estimates.  On the fixed scale shading,
source equation (87) says
\[
 |U(\T,Y^{\rm sc})|
 \gtrsim_{\log}|U(\T_r,Y_r)|
 \frac{|U(\T,Y^{\rm sc})\cap B(y_B,r)|}{|B(y_B,r)|}.
\tag{10.23b}
\]
Source equation (88), with the generalized Katz--Tao allowance
\(\nu_{\rm tr}/9\), and the local estimate (10.23), which is the hypothesis
labelled (89) in the source, are respectively
\[
 \begin{aligned}
 |U(\T_r,Y_r)|
 &\gtrsim \delta^{\nu_{\rm tr}/9}
  \left(\frac r\delta\right)^{2\beta}|T|\,|\T|^{1-\beta},\\
 \frac{|U(\T,Y^{\rm sc})\cap B(y_B,r)|}{|B(y_B,r)|}
 &\gtrsim_{\log}\delta^{-2\nu_{\rm tr}}
  \left(\frac\delta r\right)^{2\beta}.
 \end{aligned}
\tag{10.23c}
\]
Substitution cancels the two scale ratios and gives
\[
 |U(\T,Y^{\rm sc})|
 \gtrsim_{\log}\delta^{-17\nu_{\rm tr}/9}
 |T|\,|\T|^{1-\beta}.
\tag{10.23d}
\]
Since \(M(Y^{\rm sc})\le|\T|\,|T|\) and
\(M(Y^{\rm sc})\ge c_{\rm sc}M(\bar Y)\), while
\(U(\T,Y^{\rm sc})\subset U(\T,\bar Y)\),
\[
 \mu(\T,\bar Y)
 \lesssim_{\log}c_{\rm sc}^{-1}
 \delta^{17\nu_{\rm tr}/9}|\T|^\beta
 =\delta^{17\nu_{\rm tr}/9-o(1)}|\T|^\beta.
\tag{10.23e}
\]
This is the required very-non-sticky gain and includes the previously fixed
all-scale refinement loss.
\end{proof}

\subsection{Tangential parents and the lift to fine tubes}

The transverse high-degree branch has been closed by the preceding
interface.  We now assume its complementary, tangential alternative
\[
 \theta<\delta^{-\tau'}q_{\rm ang}
        =\delta^{-\tau'}a/b.
\tag{10.23f}
\]
Equality is assigned to the transverse branch of
\cref{lem:transverse-interface}.  In particular, the weak inequality needed
below follows from this strict complementary alternative.

For each retained labelled copy \(\widetilde W\) of a geometric parent \(W\),
use the fixed label order to choose one assigned tubelet and let
\(v(\widetilde W)\) be its unoriented unit carrier direction.
Every other tubelet assigned to \(\widetilde W\) makes angle
\(O(b/r_1)=O(\rho_2)\) with \(v(\widetilde W)\), and each represented full fine tube
has an additional direction error \(O(\delta/r_1)\le O(\rho_2)\).  This is the representative direction
used in the fine-tube cap in source equation (100).  It is distinct from
\(TB_{\widetilde W}:=TB_W\): the latter is the two-long-axis plane of the source L\"owner
frame and is used for parent angles and slab normalization.  Thus neither
object is inferred from the other.

Apply the angle output through one lift.  With the same set
\(\mathcal R(T)\) of retained labelled tubelet representations used in
(10.19a), define
\[
 Y^*(T)=\widehat Y(T)\cap
 \bigcup_{(B,T_B)\in\mathcal R(T)}Y_2(T_B).
\tag{10.24}
\]
Equations (10.6)--(10.7), the factoring losses, and the one lift give the
explicit fine-refinement fraction
\[
 c_*:=\frac{\sum_T|Y^*(T)|}{\sum_T|\bar Y(T)|}
 \gtrsim
 \frac{c_{\rm setup}c_{\rm ang}}
 {64\kappa_{\rm inc}^2P_{\rm tag}P_{\rm angtag}
  J_vL_mL_{a,b}L_NL_{\rm reg}}.
\tag{10.25}
\]
For every point of \(U(\T,Y^*)\), compatibility gives the line
corresponding to source equation (100):
\[
 \T_{Y^*}(x)\subset
 \bigcup_{\widetilde W:\,x\in Z^\angle(\widetilde W)}
 \{T\in\T_{\bar Y}(x):\ang(T,v(\widetilde W))\lesssim\rho_2\}.
\tag{10.26}
\]
The outer degree is less than \(4d_\angle\) in each tag and less than
\(4\kappa_{\rm amb}d_\angle\) after physical recombination, while every
inner cap is measured in the fixed uniform input.  The overlap factor
is absorbed in \(\lesssim_{\log}\).  The refinement-multiplicity estimate
\cite[Lemma~5.5, p.~10]{GWZ2026} and (10.25) therefore give the corresponding
multiplicity line:
\[
 \mu(\T,\bar Y)
 \lesslog c_*^{-1}d_\angle
 \mu(\T[T_{\rho_2}],\bar Y).
\tag{10.27}
\]
The Katz--Tao input is
\[
 \mu(\T[T_{\rho_2}],\bar Y)
 \lesssim\delta^{-\eta_{\rm bias}}
 |\T[T_{\rho_2}]|^\beta,
\tag{10.28}
\]
and the resolved-angle statement gives
\[
 x\in Z_B^\angle(W_1)\cap Z_B^\angle(W_2)
 \quad\text{for one retained tag }B
 \Longrightarrow
 \ang_{q_{\rm ang}}(W_1,W_2)<4\theta,
 \quad \theta\le\delta^{-\tau'}a/b.
\tag{10.29}
\]
These are the lines corresponding to source equations (102) and (103).  The
tangential slab grouping uses the raw angular upper bound, while the labels
retain their given multiplicities even when their carriers are parallel.

\paragraph{Slab groups and dense mass.}
Fix one retained outer-ball tag \(B\) for the rest of the slab grouping and
conflict coloring, and suppress this tag from the notation.  The quantities
\(\theta\) and \(d_\angle\) are the common reference values in (10.17a), and
the local values remain within a factor two.  Hence a local degree estimate
also controls \(d_\angle\).  All complete-containment and Katz--Tao constants
below belong to this fixed tag.  Global fine mass was retained in (10.25), and
recombination of the resulting pointwise bound contributes the displayed
\(\kappa_{\rm amb}\).

Use a half-open parameter grid to assign every labelled angular parent to
one slab \(S\) of dimensions comparable to
\[
 q_{\rm ang}r_1\times r_1\times r_1.
\tag{10.29a}
\]
Labels are never merged.  At a fixed shaded point, (10.29) implies that the
incident labels occupy at most
\[
 L_{\rm grp}\le\delta^{-2\tau'-o(1)}
\tag{10.29b}
\]
slab groups.  The coefficient \(2\) is the dimension of the normal net:
a slab through a fixed point has \(O(1)\)
admissible offset cells, while a \(\theta\)-cap meets
\(O((\theta/q_{\rm ang})^2)\le O(\delta^{-2\tau'})\) normal cells.

Write
\[
 \lambda_S=\frac{M(Z^\angle_S)}
 {|\widetilde\W^\angle_S|\,|W|},\qquad
 \lambda=\frac{M(Z^\angle)}
 {|\widetilde\W^\angle|\,|W|}.
\]
The half-open slab assignment partitions the labelled parent indices.  On the
fixed dimension level the definitions use the same carrier-volume proxy
\(|W|\) for every label, to which the carrier volumes are structurally
comparable.  Therefore \(\lambda\) is the denominator-weighted average of
the \(\lambda_S\), with weights
\(|\widetilde\W^\angle_S||W|\).  Call \(S\) dense when
\(\lambda_S\ge\lambda/2\).  Summing
\(M(Z^\angle_S)=\lambda_S|\widetilde\W^\angle_S||W|\) over the
complementary groups shows that they carry less than one half of
\(M(Z^\angle)\).  Since
\(\sum_S|U(\widetilde\W^\angle_S,Z^\angle_S)|\le
L_{\rm grp}|U(\widetilde\W^\angle,Z^\angle)|\), summing
\(M(Z^\angle_S)\le
\mu(\widetilde\W^\angle_S,Z^\angle_S)
|U(\widetilde\W^\angle_S,Z^\angle_S)|\) gives
\[
 d_\angle\lesssim
 L_{\rm grp}\max_{S\ {\rm dense}}
 \mu(\widetilde\W^\angle_S,Z^\angle_S).
\tag{10.29c}
\]

\paragraph{Separating repeated labels.}
Fix a dense \(S\), and let \(L_S\) be its common affine normalization to the
unit ball.  Write \(P_W=L_S(W)\) for the raw image of a labelled permanent
parent.  In slab-adapted coordinates its three edge vectors have sizes
\(O(\rho_2),O(\rho_2),O(1)\), their determinant has magnitude
\(\asymp\rho_2^2\), and the longest one has size \(\asymp1\).  Here the
fixed side-ratio bounds and the \(O(q_{\rm ang})\) plane tilt give the upper
bounds, while the determinant and the first two upper bounds give the
remaining lower bounds.  Thus, after dividing the first two scale directions
by \(\rho_2\), the allowed edge matrices range in one structurally compact
subset of \(GL_3\).

Consequently there are structural constants \(h_*\in(0,1]\) and
\(C_*\ge1\), common to the whole slab group, such that for every label one
may choose a source-standard width-\(\rho_2\) tube
\(\widehat T_W\), with carrier segment of length exactly one, satisfying
\[
 P_W^*:=h_*P_W\subset\widehat T_W\subset C_*P_W^*.
\tag{10.29d}
\]
The last dilation is concentric in the raw parallelepiped coordinates and
includes the fixed endpoint caps.  Shrinking \(h_*\) once makes the first
containment and the common ambient unit ball simultaneous; compactness gives
the reverse fixed containment.  Give \(\widehat T_W\) the original parent
label and the unchanged normalized shading
\(\widehat Z(\widehat T_W):=h_*L_S(Z^\angle_S(W))\).

The vertices of \(G_S\) are these labelled exact source-standard proxies.
Join two distinct vertices precisely when the proxies are comparable, meaning
that each is contained in a fixed dilate of the other.  If
\(\widehat T_{W'}\) is adjacent to \(\widehat T_W\), then
\[
 P_{W'}^*\subset\widehat T_{W'}\subset C\widehat T_W
 \subset C'P_W^*.
\]
The raw parent volumes on the dimension level are comparable, so the
complete-containment definition of \(\Delta_{\max}\) gives
\[
 \deg_{G_S}(\widehat T_W)+1
 \le\#\{W':P_{W'}^*\subset C'P_W^*\}\lesssim D_S.
\]
Greedy coloring therefore uses
\[
 H_{\rm col}\lesssim
 \Delta_{\max}(\widetilde\W_S^\angle)\le D_S
 \lesssim\delta^{-O(\varepsilon_{\rm fact})}.
\tag{10.30}
\]
Each proxy color is source-essentially distinct.  Moreover, if a convex
\(K\) contains \(\widehat T_W\), then it contains \(P_W^*\), and
\(|\widehat T_W|\asymp|P_W^*|\); hence
\[
 \Delta_{\max}(\{\widehat T_W\})\lesssim D_S.
\tag{10.30$'$}
\]
The common affine map and homothety preserve labels and complete-containment
relations and multiply raw carriers, shadings, their union, and convex test
bodies by one Jacobian.  Passing from \(P_W^*\) to \(\widehat T_W\) leaves
every shading, its union, pointwise multiplicity, and average multiplicity
unchanged; carrier densities change by only a structural factor, and
\(\Delta_{\max}\) obeys (10.30$'$).

Let \(N_S\) and \(M_S\) be the normalized slab cardinality and labelled
shading mass, and put
\[
 \widehat\lambda_S:=\frac{M_S}{N_S|T_{\rho_2}|}\asymp\lambda_S.
\]
Let \(N_j,M_j\), and
\(\lambda_j=M_j/(N_j|T_{\rho_2}|)\) denote the cardinality, labelled
shading mass, and density of color \(j\); write
\((\mathcal F_{S,j},A_{S,j})\) for this proxy color and its unchanged
shading.  Colors with
\(\lambda_j\ge\widehat\lambda_S/2\) retain at least half of the mass of \(S\).
Indeed the proper colors partition the labelled indices, their normalized
source-standard proxies have one common volume, and
\(\widehat\lambda_S\) is the
\(N_j|T_{\rho_2}|\)-weighted average of the \(\lambda_j\); the colors below
half the average contribute less than half the denominator-weighted mass.
Since every color union is contained in the slab union, this also gives
\[
 \mu(\widetilde\W^\angle_S,Z^\angle_S)
 \le2\sum_{j\ {\rm dense}}\mu_j.
\tag{10.30a}
\]
The strict reserve in condition (K), followed by the final scale cutoff,
absorbs the fixed comparison between \(\widehat\lambda_S\) and \(\lambda_S\).
It therefore verifies the normalized density hypothesis of the generalized
Katz--Tao estimate on each such proxy color, and the preceding graph
construction verifies its source essential-distinctness and
complete-containment hypotheses.  Since the proxy and raw families carry the
same shadings, the resulting multiplicity estimate pulls back without loss:
\[
 \mu_j\lesssim
 \rho_2^{-\varepsilon_{KT}}D_S^{1-\beta}N_j^\beta.
 \tag{10.30b}
\]
For \(0<\beta\le1\), concavity and complete containment give
\[
 \sum_jN_j^\beta
 \le H_{\rm col}^{1-\beta}\Big(\sum_jN_j\Big)^\beta,
 \qquad
 \sum_jN_j\lesssim D_S\rho_2^{-2}.
 \tag{10.30c}
\]
The non-sticky count
\(\rho_2^{-2}\le\rho_2^\zeta|\T_{\rho_2}|\), followed by
(10.29c)--(10.30c), yields the full scaling of the source call:
\[
 d_\angle\lesslog
 L_{\rm grp}\rho_2^{-\varepsilon_{KT}}
 D_SH_{\rm col}^{1-\beta}
 \rho_2^{\beta\zeta}|\T_{\rho_2}|^\beta.
\tag{10.31}
\]
Equivalently,
\[
 d_\angle\lesssim
 \delta^{-\Lambda_{31}-o(1)}
 \rho_2^{\beta\zeta}|\T_{\rho_2}|^\beta,
\quad
 \Lambda_{31}
 =2\tau'
 +(1-\varepsilon_{\rm scal})\varepsilon_{KT}
 +\Lambda_D+(1-\beta)\Lambda_H.
\tag{10.31a}
\]

\paragraph{Closing the tangential branch.}
Insert (10.31a) and (10.28) into the fine lift (10.27), then use
\[
 |\T[T_{\rho_2}]|\,|\T_{\rho_2}|\lesslog|\T|.
\tag{10.31b}
\]
All cap and coarse-direction factors cancel in this product, leaving
\[
 \mu(\T,\bar Y)
 \lesssim
 \delta^{-\Lambda_*-\eta_{\rm bias}-\Lambda_{31}-o(1)}
 \rho_2^{\beta\zeta}|\T|^\beta.
\tag{10.31c}
\]
Since \(\rho_2\le\delta^{\varepsilon_{\rm scal}}\), the available exponent is
\[
 \nu_{\rm tan}
 =\varepsilon_{\rm scal}\beta\zeta
  -\Lambda_*-\eta_{\rm bias}-\Lambda_{31}-o(1).
\tag{10.32}
\]
In particular, the source leading gain is retained under the explicit
condition
\[
 \Lambda_*+\eta_{\rm bias}+2\tau'
 +(1-\varepsilon_{\rm scal})\varepsilon_{KT}
 +\Lambda_D+(1-\beta)\Lambda_H
 <\varepsilon_{\rm scal}\beta\zeta.
\tag{T}
\]
When the left side is at most half the right side, (10.32) gives
\(\nu_{\rm tan}\ge
\tfrac12\varepsilon_{\rm scal}\beta\zeta>0\) for small \(\delta\).
Set
\[
 \underline\nu_{\rm tan}
 :=\tfrac12\varepsilon_{\rm scal}\beta\zeta.
\tag{10.32$'$}
\]
Thus the tangential conclusion is available with the fixed exponent
\(\underline\nu_{\rm tan}\), independently of the remaining \(o(1)\)-term.
Replacing a quadratic density loss by \(2-\sigma\) improves only an
adjustable \(O(\eta)\) term and does not alter this terminal exponent.

\subsection{Parameter hierarchy and closure}
\label{sec:paramclose}

Choose the exponents in proof order.  First fix \(\beta,\zeta\), and take the source scale
constant \(\varepsilon_{\rm scal}(\beta)>0\) sufficiently small that
\[
 8\varepsilon_{\rm scal}
 <(1-\varepsilon_{\rm scal})\beta(2+\zeta).
\tag{10.32p}
\]
Put
\(G=\varepsilon_{\rm scal}\beta\zeta\).  Choose \(\tau'\) so that
\[
 2\tau'<G/16.
\tag{10.32a}
\]
Next choose \(\varepsilon_{KT}>0\) so that
\((1-\varepsilon_{\rm scal})\varepsilon_{KT}<G/32\), and let
\(\eta_{KT}(\varepsilon_{KT},\beta)>0\) be the resulting threshold.  After
fixing this number, choose the single source bias exponent
\(\varepsilon_{\rm fact}=\eta_{\rm bias}\) so small that
\[
 \bigl[1+2(1-\varepsilon_{\rm scal})(2-\beta)\bigr]
 \varepsilon_{\rm fact}<G/64.
\tag{10.32a0}
\]
This bounds the non-subpower contribution
\(\eta_{\rm bias}+\Lambda_D+(1-\beta)\Lambda_H\) in (T), by
(10.16f); make \(\varepsilon_{\rm fact}\) smaller once more if needed to
fit the strict reserve in (10.32s).  All remaining terms from
\(\Lambda_{\rm sync},\Lambda_{\rm pre},\Lambda_*\) are copies of the
single function \(\omega_{\rm src}(\delta)\) and will be absorbed after
the parameters are fixed.  Then choose
\(\tau,\eta_{\rm pl},\eta_{\rm grid},\xi>0\), first taking
\(\xi<\tau'\beta/1024\) and
\[
 2\xi<\eta_{\rm pl},
\tag{10.32a$''$}
\]
and then making all four numbers small enough that
\[
\begin{aligned}
 3\tau+4\eta_{\rm pl}+3\eta_{\rm grid}
   &<\tfrac18\tau'\beta,\\
 (1-\varepsilon_{\rm scal})(\eta_{\rm pl}+\xi)
   &<\tfrac14\varepsilon_{\rm scal}\beta(2+\zeta),\\
 (1-\varepsilon_{\rm scal})\xi
   &<\tfrac14\varepsilon_{\rm scal}\eta_{KT}.
\end{aligned}
\tag{10.32b}
\]
Decrease \(\eta_{\rm grid}\) further, if necessary, so that
\[
 3\eta_{\rm grid}
 <\frac1{16}\tau\varepsilon_{\rm fact}\beta.
\]
Then choose \(\varepsilon_{\rm thick}>0\) smaller than the remaining
right-hand reserve in (10.8a), and let
\(\eta_F(\varepsilon_{\rm thick},\beta)>0\) from the Frostman plank
estimate.  Let \(\eta_{KT}^{\rm cap}>0\) be the density exponent supplied by
\cite[Lemma~3.7 and Remark~3.6]{GWZ2026} for output allowance
\(\varepsilon_{\rm fact}/2\).  Finally choose \(\eps,\eta\), and all remaining
setup/refinement exponents
so small that (10.8a) holds, \(\eta<\tfrac18\tau'\beta\), and their remaining contribution
to \(E_{\rm dens}\) is less than
\(\tfrac12\varepsilon_{\rm scal}\eta_{KT}\),
\(\eta<\tau\eta_F(\varepsilon_{\rm thick},\beta)\), and
(L), (H), (K), and (T) hold with strict slack.  Using (10.32p), require
at the same time
\[
 2\bigl[q(\eta+\Lambda_{\rm pre})
 +(1-\varepsilon_{\rm scal})\eps\bigr]
 +\Lambda_D+4\varepsilon_{\rm scal}+2\eta+3\tau
 +\eta+\Lambda_{\rm sync}
 <\tfrac12(1-\varepsilon_{\rm scal})\beta(2+\zeta).
\tag{10.32s}
\]
This is possible because, after (10.32a0), the only fixed positive terms on
the left are \(4\varepsilon_{\rm scal}\) and the explicitly bounded
\(2(1-\varepsilon_{\rm scal})\varepsilon_{\rm fact}\); (10.32p) and the
extra shrinking allowed after (10.32a0) leave strict room.  In particular,
\[
 3\tau+\eta+4\eta_{\rm pl}+3\eta_{\rm grid}
   <\tfrac14\tau'\beta,
 \qquad
 \xi+\eps+\frac{q(\eta+\Lambda_{\rm pre})}
                  {\varepsilon_{\rm scal}}<\eta_{\rm pl},
 \qquad
 E_{\rm dens}<\varepsilon_{\rm scal}\eta_{KT}.
\tag{10.32c}
\]
At this stage the inequalities are imposed with
\(\omega_{\rm src}=0\) and with a fixed positive reserve.  Now choose the
source input exponent \(\eta_{\rm src}>0\) so that
\(2\eta_{\rm src}<\eta\), \(4\eta_{\rm src}\le\eta_{KT}^{\rm cap}\), and
\(2\eta_{\rm src}\) lies below every other
\(K_{KT}(\beta)\) density threshold used in
\cref{prop:source-extraction}, including the scale-\(r\) call with
\(\nu_{\rm thick}/10\) in (10.8h), and also require
\((1-\beta)\eta_{\rm src}<\varepsilon_{\rm fact}/8\).  With
\(\mathfrak p\) fixed, first decrease
\(\delta_0=\delta_0(\beta,\zeta,\mathfrak p)\) until
\(\omega_{\rm pre}(\delta)<\eta_{\rm src}\).  This cutoff is independent of
\(J_v\) and gives (10.8A0).  The final decrease using
\(\omega_{\rm src}\) is made only after (10.8A) and
\cref{lem:exact-loss} give \(J_v=O(\log(1/\delta))\); at that point require
\(\omega_{\rm src}(\delta)\) to lie below \(\eta_{\rm src}\) and the minimum
of the finitely many reserved gaps in (L), (H), (K), (T), and (10.32s).  The gap
\(\eta-2\eta_{\rm src}\) absorbs the good-ball deletion, while the source
density gaps absorb the cap and other source-call losses.  This absorbs the
fixed constants and all \(o(1)\)-terms uniformly.  Decreasing it further also enforces the raw-angle
threshold used in the transverse branch.  In fact, (10.16a) and (7.3) give
\[
 \log B_{\rm ang}=(\log\alpha^{-1})^{3/4}
 \le\bigl((1-\varepsilon_{\rm scal})\log\delta^{-1}\bigr)^{3/4}.
\]
After \(\tau'\) is fixed, choose \(\delta_0\) so that for
\(0<\delta\le\delta_0\),
\[
 \tau'\log\delta^{-1}
 \ge \bigl((1-\varepsilon_{\rm scal})\log\delta^{-1}\bigr)^{3/4}
      +\log2.
\tag{10.32d}
\]
Then \(\delta^{-\tau'}\ge2B_{\rm ang}\), hence
\(\theta\ge2B_{\rm ang}q_{\rm ang}\), and (7.8) gives genuine raw separation
even in the presence of repeated parallel labels.  This verifies the low,
high, dense-color, transverse, and tangential inequalities simultaneously.

The following definition is the input contract for the five terminal
branches: it records their common data and quantitative identities, while the
branch lemmas supply their conclusions.  The closure theorem is conditional
on \(K_{KT}(\beta)\), \(K_F(\beta)\), and the explicit source interface
(I1)--(I4).  Under these assumptions, \cref{prop:source-extraction} derives
every admissibility item from the fixed GWZ refinement chain.  The resulting
admissible data close in
\cref{thm:admissible-vns}, and \cref{thm:section9} then returns to the original
labelled input.

\subsubsection*{Admissible input}

\begin{definition}[Admissible Section 9 input]
\label{def:admissibles9}
A fixed reference family \((\T,\bar Y)\), together with a fixed loss profile
\(\mathfrak p\), with
\(0<\beta\le1\), \(\zeta>0\), and
\(0<2\eta_{\rm src}<\eta\), at the Section~9 scales is
\emph{admissible} if the following data are part of the input.
\par\smallskip\noindent\textbf{Geometric admissibility (S1).}
\begin{enumerate}[label=\textup{(S\arabic*)},leftmargin=2.8em]
\item The active fine tubes, tubelets, parents, and cells have total
polynomial complexity, and the good-ball occupancy is (10.2).  Before the
good-ball deletion every active tubelet satisfies (10.1a), and after it
\[
 \lambda(\T_B,Y_B)\gtrsim\delta^\eta.
\tag{10.32e}
\]
The all-discrete-scale refinement \(Y^{\rm sc}\subset\bar Y\) and its
scale shadings are fixed before the tubelets, satisfy (10.8h), and obey
\(M(Y^{\rm sc})\ge c_{\rm sc}M(\bar Y)\) with
\(c_{\rm sc}^{-1}=\delta^{-o(1)}\).
The objects are selected in the order (10.1)--(10.8), and the represented
fine-mass identities (10.4)--(10.7), with the displayed
\(\kappa_{\rm inc},\kappa_{\rm loc},L_m,L_{a,b}\), hold on the same object.
At the realized dimension level exactly one of the following three
alternatives is designated:
\begin{enumerate}[label=\textup{(\alph*)},leftmargin=2.2em]
\item if \(a\ge\delta^{1-\tau}\), the fixed data satisfy the hypotheses of
      \cref{lem:thick-parent}: the inherited assertions \(K_F(\beta)\) and
      \(K_{KT}(\beta)\), the bias and density inputs (10.8d)--(10.8f), the
      scale relation for \(r\), and the already fixed source comparison
      (10.8h) are available.  The conclusions (10.8b)--(10.8c) are not part
      of the definition;
\item if \(a<\delta^{1-\tau}\) and
\(b\ge\delta^{\varepsilon_{\rm scal}}r_1\), the
multiplicity-expanded family and bounded-overlap concentration estimate
(10.7a)--(10.7c), the complete-parent data (10.12)--(10.16), and the
profile denominator bound in the second inequality of (10.34d) hold on the
same labels.  At
\(\rho_-:=\delta^{1-\varepsilon_{\rm scal}}\), the source middle-range count
and bounded-ancestry comparison
\[
 |\T|\gtrsim|\T_{\rho_-}|
 \ge\rho_-^{-2-\zeta}
\]
are part of the same input;
\item if \(a<\delta^{1-\tau}\) and
\(b<\delta^{\varepsilon_{\rm scal}}r_1\), the scales obey (10.16a).
\end{enumerate}
\par\smallskip\noindent\textbf{Reference-shading admissibility (S2)--(S3).}
\item The incident-centred pointwise upper bound (10.19e0) holds on the
reference shading \(\bar Y\), uniformly under fixed dyadic dilation of
\(\rho_2\).  With the source notation \(\mu(\rho_2)\),
\[
 \mu(\rho_2)
 \lesssim\delta^{-\eta_{\rm bias}}|\T[T_{\rho_2}]|^\beta,
 \qquad
 |\T[T_{\rho_2}]|\,|\T_{\rho_2}|\lesslog|\T|,
 \qquad
 |\T_{\rho_2}|\ge\rho_2^{-2-\zeta}.
\]
The same constants and dyadic levels apply to all incident-centred direction
balls.  No fixed-cap boundary lower bound is assumed.
This condition, and (S3), are required only under the thin non-slab
alternative (S1c).
\item For each chosen \(\varepsilon_{KT}>0\), there is
\(\eta_{KT}(\varepsilon_{KT},\beta)>0\) such that every essentially
distinct normalized \(\rho_2\)-tube family \((\mathcal F,A)\) with
Katz--Tao constant \(D\) and shading density at least
\(\rho_2^{\eta_{KT}}\) satisfies
\[
 \mu(\mathcal F,A)
 \lesssim_{\varepsilon_{KT},\beta}
 \rho_2^{-\varepsilon_{KT}}
 D^{1-\beta}|\mathcal F|^\beta.
\]
\par\smallskip\noindent\textbf{Quantitative-loss admissibility (S4).}
\item Under the thin non-slab alternative (S1c), the setup fraction, bias
constant, denominator level, incidence and fixed-tag representation numbers,
dyadic level counts, angular fraction, parent Katz--Tao constant, and conflict-color number
satisfy the individual bounds (10.16d)--(10.16f).  Every selection after
the positive graph is a saturated right-cell selection, a complete-tag
synchronization from \cref{thm:tagged-sync}, or a whole-incidence selection
as in (10.11$''$) and (8.5).  Consequently the mass identities (10.19b) and
(10.25) are controlled by the conservative uniform envelope (10.16d) and the
operation map in \cref{lem:exact-loss}.
All fixed constants and convergence moduli occurring in these bounds are
controlled by the previously fixed profile \(\mathfrak p\).
Under the thin slab alternative (S1b), the required loss data
are instead (10.7b), (10.12)--(10.16), and the second inequality in
(10.34d); the relevant exponents
\(\Lambda_{\rm pre},\Lambda_D,\Lambda_{\rm den}\) obey (10.16f) and the
strict slab reserve (10.32s).  Under the thick
alternative (S1a), they are \(c_{\rm pre},c_{\rm bias}\), and
\(L_{\rm thick}\) in (10.8d)--(10.8h), after which the construction stops.
\end{enumerate}
These fixed source inputs are used on their stated reference objects; later
angular or spatial refinements enter only through the listed saturated
operations.
\end{definition}

\subsubsection*{Source verification}

\begin{lemma}[Carrier--shading commutation rule]
\label{lem:carrier-shading-commutation}
Let \(\mathcal L\) be a finite labelled carrier family and let
\(Y^B(\ell)\subset Y(\ell)\) satisfy
\[
 |Y^B(\ell)|\ge c_B|Y(\ell)|>0
 \qquad(\ell\in\mathcal L)
\tag{10.32g}
\]
for one \(c_B>0\).  Let \(A\) be a deterministic operation whose output
depends only on the complete labelled carriers and on carrier-only data
(complete containment, dimensions, a fixed label order, or a partition into
finitely many such classes), and not on the location or mass of the shading.
Then applying \(A\) before or after the restriction \(Y\mapsto Y^B\)
selects exactly the same labels, parents, and carrier classes.  For every
class \(\mathcal L_A\) selected by \(A\),
\[
 \sum_{\ell\in\mathcal L_A}|Y^B(\ell)|
 \ge c_B\sum_{\ell\in\mathcal L_A}|Y(\ell)|.
\tag{10.32h}
\]
If a carrier-only partition has at most \(L\) classes and a class is chosen
only after the restriction by shaded mass, the chosen class instead retains
at least \((c_B/L)M(Y)\); equality of its label set with a class chosen
before the restriction is not asserted.

There are three noncommutation cautions.
\begin{enumerate}[label=\textup{(C\arabic*)},leftmargin=2.8em]
\item A pointwise lower-multiplicity or angular statement remains attached
  to the shading on which it was proved.  Mere inclusion does not guarantee
  transfer; below we transfer it only when every counted incidence is kept.
\item For a mass-scored operation, (10.32h) does not identify the outputs
  before and after restriction.  Below it is rerun unless its scoring data are
  separately unchanged, and any new pigeonhole loss is recorded.
\item Arbitrary internal deletion from a positive parent--cell graph need not
  preserve its identities.  Deleting whole right cells, tags, or positive
  incidences while keeping every child incidence on each surviving edge is
  the sufficient commutation rule used below.
\end{enumerate}
\end{lemma}

\begin{proof}
The complete label set and every carrier-only relation are unchanged by
\(Y(\ell)\mapsto Y^B(\ell)\); determinism therefore gives the same output of
\(A\).  Summing (10.32g) over a selected class proves (10.32h).  If the class
is chosen afterwards, sum (10.32h) over all at most \(L\) classes and take a
largest summand.  Clause (C1) follows because deleting even one of several
incidences through a point can destroy a lower bound there.  Clause (C2) is
the distinction between carrier data and shading scores.  For (C3), the edge
identity in \cref{thm:joint,thm:lift} shows that internal cutting may change
the weight or child multiplicity; the saturated rule avoids this failure.
\end{proof}

The three noncommuting classes in
\cref{lem:carrier-shading-commutation}(C1)--(C3) govern the concrete source
order recorded next in (R1)--(R4).

\begin{lemma}[Reference objects and permitted reordering]
\label{lem:source-order}
For the fixed source refinement chain associated with a uniform input, the
source and corrected operations may be performed in the following order:
\[
 \begin{aligned}
 Y&\longrightarrow\bar Y\longrightarrow Y^{\rm sc}
 \longrightarrow\text{canonical windows}
 \longrightarrow(Y^{\rm pre},Y_B)
 \longrightarrow Y^0\longrightarrow m_{\rm fine}\\
 &\longrightarrow\text{biased label subset}
 \longrightarrow(a,b,N)\text{-levels}
 \longrightarrow\text{tagged cellular graph}.
 \end{aligned}
\tag{10.32f}
\]
Source equation numbers refer to the archived version identified in
Appendix~D.
Here \(\bar Y\) is the output of source equation (86),
\(Y^{\rm sc}\) is the all-discrete-scale output used in (87)--(88), the
canonical windows and their two levels are constructed by
\cref{lem:source-canonical}, \(Y_B\) is their good-ball restriction, and
\(m_{\rm fine}\) is the global level of \cref{lem:global-sync}.  Every arrow
after \(Y^{\rm sc}\) is an explicit induced fine-incidence restriction, a
deletion of complete canonical tubelet labels or tags, or a saturated tagged
incidence deletion.  In particular:
\begin{enumerate}[label=\textup{(R\arabic*)},leftmargin=2.8em]
\item the cap statement is read only on \(\bar Y\), not transferred to a
later shading;
\item (87)--(88) are read only on \(Y^{\rm sc}\), whose later local unions
are subsets of that fixed union;
\item the good-ball operation precedes \(m_{\rm fine}\), and
\(m_{\rm fine}\) precedes every biased or dimension subset;
\item every operation after the positive graph is saturated on right cells,
deletes complete outer-ball tags during one of the two synchronization
stages of \cref{thm:tagged-sync}, or is the one whole-incidence lift of
\cref{thm:lift}.
\end{enumerate}
By \cref{lem:carrier-shading-commutation}, the good-ball deletion commutes with
the biased-parent and dimension selections; \cref{lem:goodball-reorder}
supplies uniform per-label retention.  Every other arrow is used in the
displayed order and on its stated reference object.
The product of the retained fractions before the graph is exactly
\(c_{\rm pre}c_{\rm bias}/(\kappa_{\rm inc}L_mL_{a,b}L_N)\), up to the fixed
source comparison constants already named in those symbols.
\end{lemma}

\begin{proof}
In the v1 order, (86) precedes (87)--(88), which precede the tubelet levels.
The canonical construction (10.7U)--(10.7X) gives (R1)--(R2).  By
\cref{lem:carrier-shading-commutation,lem:goodball-reorder}, the good-ball
deletion crosses the carrier-only bias and dimension choices while retaining
(10.1c); \cref{lem:global-sync} then gives (R3).  Saturation of the tagwise
positive graph, the whole-tag deletions in \cref{thm:tagged-sync}, and the
single lift in \cref{thm:lift} give (R4).  Multiplying the named representation
and dyadic factors yields the stated retained fraction.
\end{proof}

\begin{lemma}[The reference cap and the generalized KKT estimate]
\label{lem:source-cap-kkt}
Assume \(K_{KT}(\beta)\).  Under the thin non-slab alternative, source
equation (86) on p.~36 supplies the incident-centred cap relation (10.35).
Assume also the separate coarse-average interface clause (10.35a) in (I2).
Then \cite[Definition~2.2, p.~4; Lemma~3.7, p.~7; Remark~3.6,
p.~6]{GWZ2026} and these two relations imply (S2)--(S3) on the same reference
shading \(\bar Y\).  More precisely, (10.35)--(10.37) hold for every
incident-centred direction ball, and every essentially distinct normalized
family of density at least \(\rho_2^{\eta_{KT}}\) obeys the estimate in
(S3).  Both conclusions are read on the fixed reference shading \(\bar Y\).
\end{lemma}

\begin{proof}
Source equation (86) gives, for every \(x\in\bar Y(T_0)\),
\[
 \#\{T:x\in\bar Y(T),\ \ang(T,T_0)\le\rho\}\asymp\mu(\rho).
\tag{10.35}
\]
The additional interface clause (10.35a) states separately that, for every
coarse label \(S\),
\[
 \mu\bigl(\T[S],\bar Y|_{\T[S]}\bigr)\asymp\mu(\rho).
\tag{10.35a}
\]
Thus the incident-centred cap count and the coarse-child average multiplicity
are two distinct data on the same reference shading \(\bar Y\).
Neighbouring-scale stability and the common branching number give
\[
 N_\rho|\T_\rho|\lesslog|\T|.
\tag{10.36}
\]

Write \(M_S=\sum_{T\in\T[S]}|\bar Y(T)|\).  Bounded representation and
uniform branching give
\[
 \sum_SM_S\asymp M(\bar Y),\qquad
 \sum_S|\T[S]|\,|T|\asymp|\T|\,|T|.
\]
Hence some \(S_0\) has density
\(\gtrlog\lambda(\T,\bar Y)\), while (10.35a) gives its multiplicity
\(\asymp\mu(\rho)\).  Its Katz--Tao constant is at most
\(\delta^{-\eta_{\rm src}}\).  The parameter closure already gives
\((1-\beta)\eta_{\rm src}<\varepsilon_{\rm fact}/8\) and
\(\eta_{\rm src}\le\eta_{KT}^{\rm cap}/4\); reserve another quarter for
the finitely many uniformity losses.  Write
\(0\le\omega_{\rm cap}(\delta)\le\omega_{\rm src}(\delta)\to0\) for those
losses.  The profile cutoff
\(\omega_{\rm src}(\delta)<\eta_{\rm src}\) also absorbs the combined loss
in the preceding density selection.  The cap call uses the interface
\[
 \begin{gathered}
 s_\rho=\delta/\rho,\qquad \tau_{\rm src}=\delta\le s_\rho,\qquad
 \mathcal F_\rho=L_{S_0}(\T[S_0]),\\
 A_\rho=L_{S_0}(\bar Y|_{\T[S_0]}),\qquad
 \lambda(\mathcal F_\rho,A_\rho)\gtrsim\delta^{2\eta_{\rm src}},\\
 \Delta_{\max}(\mathcal F_\rho)\le\delta^{-\eta_{\rm src}},\qquad
 |\mathcal F_\rho|=N_\rho,\qquad
 \mu(\mathcal F_\rho,A_\rho)\asymp\mu(\rho).
 \end{gathered}
\tag{10.36a}
\]
No essential-distinctness assertion is needed here: the labelled form derived
after (10.8J) from \cite[Lemma~3.7]{GWZ2026} applies using the displayed density
and \(\Delta_{\max}\) bounds.
The affine normalization preserves density and average multiplicity, and
\(2\eta_{\rm src}<\eta_{KT}^{\rm cap}\) verifies the source density
threshold.  After
decreasing \(\delta_0\) so that
\(\omega_{\rm cap}(\delta)<\varepsilon_{\rm fact}/8\), Lemma~3.7 and
Remark~3.6 give
\[
 \begin{aligned}
 \mu(\rho)
 &\lesssim \delta^{-\varepsilon_{\rm fact}/2}
    \Delta_{\max}(\mathcal F_\rho)^{1-\beta}N_\rho^\beta\\
 &\le \delta^{-\varepsilon_{\rm fact}/2
        -(1-\beta)\eta_{\rm src}-\omega_{\rm cap}(\delta)}N_\rho^\beta
 \le\delta^{-\varepsilon_{\rm fact}}N_\rho^\beta,
 \end{aligned}
\tag{10.37}
\]
This proves the first part of (S2); the non-sticky hypothesis supplies its
remaining count.

The later dense-color call uses
\[
 \begin{gathered}
 s=\rho_2,\qquad \mathcal F=\mathcal F_{S,j},\qquad
 A=A_{S,j},\qquad |\mathcal F|=N_j,\\
 \mathcal F_{S,j}\text{ is one proper source-proxy conflict color},\qquad
 \Delta_{\max}(\mathcal F_{S,j})\lesssim D_S,\\
 \lambda(\mathcal F_{S,j},A_{S,j})=\lambda_j
 \ge\rho_2^{\eta_{KT}}.
 \end{gathered}
\tag{10.37a0}
\]
The proxy containment comparison preserves the required convex-test bound,
and the unchanged shading preserves multiplicity, so Lemma~3.7 gives
(10.30b) and pulls it back to the raw normalized family.  For every dense
color, condition (K) and its fixed reserve verify the last line because
\[
 \lambda_j\ge\tfrac12\widehat\lambda_S
 \gtrsim\lambda_S
 \ge\tfrac12\lambda(Z^\angle)
 \ge\tfrac14c_{\rm ang}\lambda_B^{\rm par}
 \gtrsim\mathcal D,
\tag{10.37a}
\]
the strict inequality in (K), after the final cutoff, makes the last
comparison at least \(\rho_2^{\eta_{KT}}\) despite its fixed structural
constant.
\end{proof}

\newpage
\begin{lemma}[Quantitative loss envelope and property map]
\label{lem:exact-loss}
On the chain (10.32f), the operation-level losses and the properties they
preserve are bounded by the following exchange table.
\begin{center}
\footnotesize
\begin{tabularx}{\textwidth}{@{}>{\raggedright\arraybackslash}p{2.45cm}>{\raggedright\arraybackslash}p{1.85cm}>{\raggedright\arraybackslash}p{2.05cm}>{\raggedright\arraybackslash}X@{}}
\toprule
operation & deletion type & retained factor & property kept for later use\\
\midrule
angular uniformization & fine incidences & \(c_{\rm init}\) & cap counts fixed on \(\bar Y\)\\
all-scale level & fine incidences & \(c_{\rm sc}\) & (87)--(88) fixed on \(Y^{\rm sc}\)\\
canonical window/net & labelled assignment & no deletion & (10.7P)--(10.7R), one fixed ancestry\\
canonical multiplicity/density/tag levels & induced fine incidences and whole tags & \(c_{\rm can}\) & (10.7T)--(10.7X), canonical shading density\\
good-ball level & induced fine incidences & \((4\kappa_{\rm inc})^{-1}\) & (10.2), (10.4), representation coverage\\
fine-multiplicity synchronization & whole tubelet labels & \((\kappa_{\rm inc}L_m)^{-1}\) & one \(m_{\rm fine}\), identity (10.7)\\
labelled bias, dimension, count levels & whole tubelet labels & \(c_{\rm bias}/(L_{a,b}L_N)\) & (8.0D)--(8.0G), common denominator\\
REG-1, REG-2, local volume & saturated tagged edges & \((J_vL_{\rm reg})^{-1}\) & P2--P5 in each tag\\
first whole-tag synchronization & whole outer-ball tags & \(P_{\rm tag}^{-1}\) & common \(\mu_{\rm in},\omega_0,d_0\), local P1 unchanged\\
ANGLE-CELL plus one LIFT & whole incidences inside each tag & \(c_{\rm ang}/2\) & local resolved angle, P3--P5\\
angle whole-tag synchronization & whole outer-ball tags & \(P_{\rm angtag}^{-1}\) & common \(\theta,d_\angle\), (10.25)\\
conflict coloring & partition, no deletion & estimate costs \(H_{\rm col}\) & essential distinctness in each color\\
\bottomrule
\end{tabularx}
\end{center}
The table records a conservative operation-to-factor envelope.  The products
in (10.16b) are upper bounds, and (10.16d) is the
single uniform envelope used later.  Equation (10.25) records the actual
angular lift and tag retention, while the fixed physical recombination factor
is \(\kappa_{\rm amb}\).  Hence (10.16d)--(10.16f) follow without an
additional refinement, so (S4) holds.
\end{lemma}

\begin{proof}
The rows follow respectively from
\cref{lem:source-order,lem:global-sync}, the saturated definition
(10.11$''$), (10.8E), (8.5), (10.8G), and
\cref{thm:lift,thm:angle}.  Coloring partitions labels; physical
recombination costs \(\kappa_{\rm amb}\) by \cref{lem:tag-recombine}.

For the quantitative bounds, \cref{lem:biased-factoring} and equal
canonical-child volume give, for convex \(K\subset W\),
\[
 \frac{\#\T_{B,W}[K]}{\#\T_{B,W}}
 =\frac{\Delta(\T_{B,W},K)|K|}
        {\Delta(\T_{B,W},W)|W|}
 \lesssim\delta^{-o(1)}
 \left(\frac{|K|}{|W|}\right)^{1+\varepsilon_{\rm fact}}.
\tag{10.37b}
\]
The fixed L\"owner rectangle and (8.0F)--(8.0G) give
\[
 C_0\le\delta^{-o(1)},\qquad
 \Delta_{\max}(\W_B)
 \lesssim(\alpha\rho_2)^{-\varepsilon_{\rm fact}}\delta^{-o(1)}.
\tag{10.38}
\]
Common child counts give
\(D\le D_B\le\kappa_DD\), \(\kappa_D=O(1)\); the canonical net and ball cover
bound \(\kappa_{\rm loc},\kappa_{\rm inc}\).

For the local-volume count, let \(\Omega\) and \(m_+=|E^+|\) be the
pre-REG-1 weight and edge count.  The application-level condition (10.8A),
verified from (10.8A0) before any \(J_v\)-dependent envelope, gives, after
enlarging the fixed mass exponent if necessary,
\[
 \Omega=M(Y_0^{\rm joint})
 \ge \delta^{C_{\rm mass}}.
\tag{10.38a}
\]
Polynomial complexity bounds \(m_+,\mu_{\rm in}\), and
\cref{lem:REG1} gives \(\omega_0\ge\Omega/(2m_+)\).  Hence
\cref{cor:tagged-volume-levels} yields
\(J_v=O(\log(1/\delta))\); the other levels are polylogarithmic.  On the
non-slab window, (10.17) gives
\(c_{\rm ang}^{-1}=\alpha^{-o(1)}\); angular rows are absent in the slab case.

Complete-containment density is monotone under label deletion and affine
invariant.  Thus (10.38) and
\(\alpha,\rho_2\ge\delta^{1-\varepsilon_{\rm scal}}\) give, in the applicable
thin branches,
\[
\begin{aligned}
 D_W,D_S&\le\delta^{-\Gamma_D-o(1)},\\
 \Gamma_D&:=2(1-\varepsilon_{\rm scal})\varepsilon_{\rm fact}.
\end{aligned}
\tag{10.39}
\]
The greedy conflict coloring gives \(H_{\rm col}\le CD_S\).
Substitution in the conservative products is bounded by
(10.16d)--(10.16f), proving the lemma.
\end{proof}

\begin{proposition}[Conditional verification of the GWZ Section 9 interface]
\label{prop:source-extraction}
Fix \(0<\beta\le1\), \(\zeta>0\), and a loss profile \(\mathfrak p\), and assume
\(K_{KT}(\beta)\) and \(K_F(\beta)\).  Choose the parameters in the order of
\cref{sec:paramclose}.  Then there exist
\(0<2\eta_{\rm src}<\eta\) and
\(\delta_0=\delta_0(\beta,\zeta,\mathfrak p)>0\) such that, for every
\(0<\delta\le\delta_0\), the following holds.  Let \((\T,Y)\) be a uniform
input satisfying the formal hypotheses displayed in
\cite[Lemma~9.1]{GWZ2026}:
\[
 \Delta_{\max}(\T)\le\delta^{-\eta_{\rm src}},\qquad
 \lambda(\T,Y)\ge\delta^{\eta_{\rm src}},\qquad
 |\T_\rho|\ge\rho^{-2-\zeta}
\tag{10.33}
\]
 for the source middle range of \(\rho\).  All fixed constants and moduli in
 the input below are governed by \(\mathfrak p\).  In addition, assume the
following explicit conditional interface:
\begin{enumerate}[label=\textup{(I\arabic*)},leftmargin=2.8em]
\item \(\T\) is a finite, source-essentially-distinct labelled family of
unit-length \(\delta\)-tubes in a fixed ball,
\(|\T|\le\delta^{-C_{\rm comp}}\), with bounded neighbouring-scale ancestry;
\item a reference refinement \(\bar Y\subset Y\) exists with
\(M(\bar Y)\ge c_{\rm init}M(Y)\) and
\(c_{\rm init}^{-1}\le\delta^{-\omega_0(\delta)}\) for one fixed
\(\omega_0(\delta)\to0\); it satisfies the incident-centred cap and local
average-multiplicity statements (10.35)--(10.35a) at every source scale;
\item an all-discrete-scale refinement
\(Y^{\rm sc}\subset\bar Y\) exists, in the order (10.32f), with retained
factor \(c_{\rm sc}^{-1}\le\delta^{-\omega_0(\delta)}\), and the scale-
\(r\) shadings and comparisons are exactly the source (87)--(88) interface
named in (R2); the common coarse branching numbers obey (10.36); and
\item all carrier ancestry and representation maps are labelled and have
bounded fibres.  Specifically,
\[
 \max\{\kappa_{\rm inc},\kappa_{\rm loc},\kappa_{\rm amb}\}=O(1)
\]
with the roles displayed below, and subsequent operations are restricted to the
canonical-window and saturated whole-incidence protocol in (R3)--(R4).
\end{enumerate}
Under (I1)--(I4) there is one fixed function
\(\omega_{\rm ref}(\delta)\to0\), uniform over all admissible tags and scales,
such that the source refinements followed by the corrected cellular
preprocessing equip the fixed reference refinement \((\T,\bar Y)\) with all
admissible data in the sense of \cref{def:admissibles9}, while retaining the
quantified labelled fine mass
\[
 M(\bar Y)\ge\delta^{\omega_{\rm ref}(\delta)}M(Y).
\tag{10.33R}
\]
Thus ``subpower refinement'' means the signed lower bound (10.33R).  More
precisely, the reference angular shading obeys
\[
 M(\bar Y)\ge c_{\rm init}M(Y),\qquad
 c_{\rm init}^{-1}\le\delta^{-\omega_0(\delta)}.
\tag{10.33$'$}
\]
\end{proposition}

Appendix~D records the complete source correspondence and version convention.
For the proof itself, the following matrix collects the proof-bearing analytic
calls and the reference family on which each hypothesis is verified.
\begin{center}
\scriptsize
\begin{tabularx}{\textwidth}{@{}>{\raggedright\arraybackslash}p{1.55cm}>{\raggedright\arraybackslash}p{2.5cm}>{\raggedright\arraybackslash}p{2.6cm}>{\raggedright\arraybackslash}X@{}}
\toprule
source & scale and actual family & checks on called family & output and pullback\\
\midrule
GWZ (86) & \(\rho;\ (\T,\bar Y)\) & reference shading fixed before all later deletions; incident-centred caps & (10.35), with neighbouring-scale stability\\
GWZ (87)--(88) & \(r;\ (\T,Y^{\rm sc})\),\newline \((\T_r,Y_r)\) & all-scale level fixed before tubelets; retained factor \(c_{\rm sc}\) & (10.8h), (10.23b)--(10.23d); scale ratios cancel\\
GWZ Lemma 6.4 (31) & \(s=\delta/b,t=\delta/a\);\newline \((\mathcal P,Y_{\mathcal P})\) & (10.8f$'$): \(\lambda\ge s^{\eta_F}\), \(C_F\lesssim_{\log}1\), the displayed \(M\)-bound; Definition~6.10 proxy family essentially distinct & (10.8g0); original labels retained and charged in \(M\), then pullback gives (10.8g)\\
GWZ Lemma 3.7, cap & \(s_\rho=\delta/\rho\);\newline \((\mathcal F_\rho,A_\rho)\) & (10.36a): reference density \(\gtrsim\delta^{2\eta_{\rm src}}\), \(\Delta_{\max}\le\delta^{-\eta_{\rm src}}\); derived labelled form, no essential-distinctness hypothesis & (10.37); average multiplicity is affine invariant\\
GWZ Lemma 3.7, color & \(\rho_2\);\newline exact source-tube proxy color \((\mathcal F_{S,j},A_{S,j})\) & proper proxy color, \(\lambda_j\ge\rho_2^{\eta_{KT}}\), \(\Delta_{\max}\lesssim D_S\) & (10.30b); two-sided proxy containment transfers convex tests and the unchanged shading pulls multiplicity back\\
GWZ Lemma~5.5 & \(\rho_2;\ Y^*\subset\bar Y\) & compatibility (10.26), outer degree \(d_\angle\), refinement fraction \(c_*\) & equation (10.27) on the fixed reference cap shading\\
\bottomrule
\end{tabularx}
\end{center}

\begin{proof}
\emph{Geometric input (S1).}
Follow \cref{lem:source-order}: freeze \(\bar Y\) at (86), then form
\(Y^{\rm sc}\subset\bar Y\) at (87)--(88), with retained fraction
\(c_{\rm sc}\) and \(c_{\rm sc}^{-1}=\delta^{-o(1)}\).  Clause (C3) of
\cref{lem:source-canonical}, the density level (10.7W), the tag level
(10.7W$'$), and the physical projection (10.7X) then give
\[
 M(Y^{\rm pre})\ge c_{\rm sc}c_{\rm can}M(\bar Y),
 \qquad
 c_{\rm can}^{-1}=2\kappa_{\rm inc}L_{\rm mult}L_{\rm dens}L_{\rm tag}.
\tag{10.33P}
\]
Now \cref{lem:goodball-reorder} and (10.0)--(10.1d) give (10.1) with
\(c_{\rm pre}=c_{\rm sc}c_{\rm can}/(4\kappa_{\rm inc})\).  Thus
\(c_{\rm pre}^{-1}=\delta^{-o(1)}\).  Moreover
\(\lambda(\T,Y^{\rm sc})\ge
c_{\rm sc}c_{\rm init}\delta^{\eta_{\rm src}}\), so (10.7V) and the
pre-cellular cutoff \(\omega_{\rm pre}<\eta_{\rm src}\) give (10.1a) with exponent
\(2\eta_{\rm src}\).  The gap \(\eta-2\eta_{\rm src}\) then gives
(10.1)--(10.2).

Testing \(\Delta_{\max}(\T)\) on a containing ball gives
\[
 |\T|\delta^2\lesssim\delta^{-\eta_{\rm src}},
\tag{10.34}
\]
so \(|\T|\le\delta^{-C}\).  The bounded-overlap cover and canonical net give
at most \(C|\T|r_1^{-1}\) representations; biased parents partition them,
and the tags contain \(O(a^{-3})\) cells.  Hence all later graphs have
polynomial size.

For the windows in \cref{lem:biased-factoring}, every maximizing hull
\(K_{\mathcal A}\) contains a canonical
\(\delta\times\delta\times r_1\) tubelet and lies in \(C_{\rm amb}B\).  If
\(\ell_1(K_{\mathcal A})\le\ell_2(K_{\mathcal A})\le
\ell_3(K_{\mathcal A})\) are its L\"owner semiaxes, then
\[
 c\delta\le\ell_1\le\ell_2\le\ell_3\le Cr_1,
 \qquad cr_1\le\ell_3,
\tag{10.34a}
\]
and hence
\[
 c\delta^2r_1\le |K_{\mathcal A}|\le Cr_1^3.
\tag{10.34b}
\]
The contained tubelet, ambient ball, and L\"owner--John comparison prove
(10.34a)--(10.34b).  The score numerator lies between
\(c\delta^2r_1\) and
\(C|\mathcal A|\delta^2r_1\le
C\delta^{-C_{\rm comp}}\delta^2r_1\).  Since
\(r_1=\delta^{\varepsilon_{\rm scal}}\), (10.34b) therefore gives a fixed
\(C_{\rm win}\) such that
\[
 \delta^{C_{\rm win}}
 \le \Phi_{\mathcal A}(K_{\mathcal A})
 \le \delta^{-C_{\rm win}},
 \qquad
 \delta^{C_{\rm win}}\le\ell_j(K_{\mathcal A})
 \le\delta^{-C_{\rm win}}.
\tag{10.34c}
\]
Thus the score and three axis windows have
\(L_{\rm bias}\le C\log^4(2/\delta)=\delta^{-o(1)}\) classes.

By (R3), perform (10.3)--(10.7) before biased factoring, so
\(\labmass(A)\sim m_{\rm fine}\sum|A(T_B)|\).  Apply
\cref{lem:biased-factoring} tagwise.  Its (8.0C), the density class (10.7W),
and (10.7) retain the represented fraction \(c_{\rm bias}\), including the
single physical/labelled conversion:
\[
 c_{\rm bias}^{-1}\le C\kappa_{\rm inc}L_{\rm bias}=\delta^{-o(1)},
 \qquad c_{\rm setup}=c_{\rm pre}c_{\rm bias}.
\tag{10.33Q}
\]
The separately charged \((a,b)\) and child-count levels cost
\(L_{a,b}L_N\) and delete complete labels.  The thick branch stops after
dimensions; thin branches continue through REG and saturated local volume.
Thus (10.4)--(10.7) remain on one object, and
\cref{lem:goodball-reorder} supplies (10.32e).
In the thin slab branch, (10.6) and \(P_{\rm sync}\) verify (10.7d) with
\(P=P_{\rm sync}\).  Since
\(\lambda(\T,\bar Y)\ge c_{\rm init}\delta^{\eta_{\rm src}}\),
\cref{lem:represented-denominator} gives
\[
 \sum_{B\ {\rm retained}}|\T_B^{\rm lab}|\,|T_B|
 \gtrsim c_{\rm init}P_{\rm den}^{-1}|\T|\,|T|
 \gtrsim\delta^{\Lambda_{\rm den}+\omega_0(\delta)}|\T|\,|T|.
\tag{10.34d}
\]
Here (10.7a) is used, and \(\kappa_{\rm inc}\) is already included in
\(P_{\rm sync}\); hence
\(P_{\rm den}=\delta^{-\eta_{\rm src}}P_{\rm sync}\).  At
\(\rho_-=\delta^{1-\varepsilon_{\rm scal}}\), (10.33) and (I1) give
\[
 |\T|\gtrsim|\T_{\rho_-}|
 \ge\rho_-^{-2-\zeta}.
\]
Apply \cref{lem:exact-loss}.  Its proof now gives
\(J_v=O(\log(1/\delta))\) from the independently verified condition (10.8A),
so the full envelope satisfies \(\omega_{\rm src}(\delta)\to0\).  Make the
deferred final decrease of \(\delta_0\) until \(\omega_{\rm src}\) is below
\(\eta_{\rm src}\) and every reserved gap named in the parameter closure.
Together with (10.7a)--(10.7e), (10.12)--(10.16), and the bounds from that
lemma, this is (S1b); its strict reserve is (10.32s).

If \(a\ge\delta^{1-\tau}\), the assumed \(K_F(\beta)\) and \(K_{KT}(\beta)\),
(10.32e), (10.8h), and \cref{lem:exact-loss} verify
\cref{lem:thick-parent}, hence (S1a), before cellularization.  In the thin
non-slab branch the source bounds
\(\delta\le a\le b\le\delta^{\varepsilon_{\rm scal}}r_1\), with
\(r_1=\delta^{\varepsilon_{\rm scal}}\), imply
\[
 \delta^{1-\varepsilon_{\rm scal}}\le a/r_1\le b/r_1
 \le\delta^{\varepsilon_{\rm scal}},\qquad a/b\ge a/r_1,
\]
which is (10.16a), hence (S1c).  Thus (S1) holds in all three cases;
(S2)--(S3) and the non-slab part of (S4) are needed only in (S1c).

\emph{Reference cap input (S2).}
Take \(\rho=\rho_2\) in (10.35)--(10.37) of
\cref{lem:source-cap-kkt}; (10.36) and (10.33) give (S2).

\emph{Dense-color input (S3).}
The color clause of \cref{lem:source-cap-kkt}, the conflict graph, and (K)
give respectively the estimate, essential distinctness, and density in (S3).

\emph{Loss envelope (S4).}
By \cref{lem:exact-loss}, (10.37b)--(10.39), (10.19b), and (10.25) verify
(S4) on the stated reference objects.
\end{proof}

\subsubsection*{Admissible closure}

\begin{theorem}[Very-non-sticky estimate for admissible families]
\label{thm:admissible-vns}
Fix \(0<\beta\le1\), \(\zeta>0\), and a loss profile \(\mathfrak p\), choose
the parameters as in \cref{sec:paramclose}, and let \((\T,\bar Y)\) be
admissible in the sense of \cref{def:admissibles9}, with all its losses
governed by \(\mathfrak p\).  There are
\(\nu_{\rm adm}=\nu_{\rm adm}(\beta,\zeta)>0\) and
\(\delta_0=\delta_0(\beta,\zeta,\mathfrak p)>0\) such that, whenever
\(0<\delta\le\delta_0\),
\[
 \mu(\T,\bar Y)\le\delta^{\nu_{\rm adm}}|\T|^\beta.
\tag{10.40}
\]
\end{theorem}

\begin{proof}
Under alternative (S1a), \cref{lem:thick-parent} and
\(M(\bar Y)\le|\T|\,|T|\) give
\[
 \mu(\T,\bar Y)
 \lesssim_{\log}\delta^{\nu_{\rm thick}}|\T|^\beta.
\tag{10.8j}
\]
Thus the thick case is complete.  For the rest of the branch verification
assume the complementary thin condition (10.8i).

The source-call matrix identifies (10.2), (10.12)--(10.14),
(10.22)--(10.23), and (10.26)--(10.29) with source blocks (94)--(103).
The bridge (10.4)--(10.7), used in (10.19a)--(10.19f), converts tubelet
retention to physical fine mass without identifying the two shadings.

Under alternative (S1b), the slab estimate follows from \cref{app:slab}
and (10.15)--(10.16).  More explicitly, (10.12), the common bounds
\(\delta\le a\le b\le r_1=\delta^{\varepsilon_{\rm scal}}\), and hence
\(a/b\ge\delta^{1-\varepsilon_{\rm scal}}\), give
\[
 \lambda_B^{\rm par}\ge
 \delta^{\Lambda_{\rm par}+o(1)},\qquad
 \Lambda_{\rm par}
 :=q(\eta+\Lambda_{\rm pre})
   +(1-\varepsilon_{\rm scal})\eps.
\]
In the slab range,
\((r_1/b)^2\le\delta^{-2\varepsilon_{\rm scal}}\), so (10.15), (10.14),
and (10.16) give
\[
 |U(\T_B,Y_B)|
 \gtrsim
 \delta^{\Lambda_{\rm slab}+o(1)}|\T_B^{\rm lab}|\,|T_B|,
\quad
 \Lambda_{\rm slab}
 =2\Lambda_{\rm par}+\Lambda_D
  +4\varepsilon_{\rm scal}+2\eta+3\tau.
\tag{10.33a}
\]
Here \(\Lambda_D\) controls the pre-angle parent constant \(D_W\), as well as
its later descendants.  The \(r_1\)-ball cover has bounded overlap, and
\(U(\T_B,Y_B)\subset U(\T,Y^0)\).  Summing (10.33a) and using (10.34d)
therefore gives
\[
 |U(\T,\bar Y)|\ge |U(\T,Y^0)|
 \gtrsim\delta^{\Lambda_{\rm slab}+\Lambda_{\rm den}+o(1)}
          |\T|\,|T|.
\tag{10.33b}
\]
At \(\rho_-=\delta^{1-\varepsilon_{\rm scal}}\), the S1b count and ancestry
data in \cref{def:admissibles9} give
\[
 |\T|\gtrsim|\T_{\rho_-}|
 \ge\delta^{-(1-\varepsilon_{\rm scal})(2+\zeta)}.
\]
Consequently (10.33b) implies
\[
 \mu(\T,\bar Y)\lesssim
 \delta^{\nu_{\rm slab}-o(1)}|\T|^\beta,\qquad
 \nu_{\rm slab}:=(1-\varepsilon_{\rm scal})\beta(2+\zeta)
                 -\Lambda_{\rm slab}-\Lambda_{\rm den}>0.
\tag{10.33c}
\]
The last strict inequality follows from (10.32p) and the later small choices
in \cref{sec:paramclose}.  Thus the local slab branch has been summed and
converted to the source multiplicity target.  Under the remaining
alternative (S1c), the threshold (10.19) divides all inputs.
Below it, (10.19g) and cap uniformity give (10.19h), and condition (L)
makes its exponent positive.  Above it, (10.18) gives (10.20), while
(10.21a) and (H) prove the parent-density hypothesis (10.21).  Condition
(K) separately verifies the density of every color to which the
Katz--Tao estimate is applied.

In the transverse alternative \cref{lem:transverse-interface} gives
(10.22)--(10.23), including the filling-to-multiplicity implication.  In the tangential alternative,
(10.29a)--(10.29c) retain dense slab mass, the explicitly defined conflict
graph gives (10.30), and (10.30a)--(10.31a) prove the full scaled degree
bound.  The single fine lift (10.27), cap estimate (10.28), and uniformity
(10.31b) then give (10.31c)--(10.32).  The choices
(10.32a)--(10.32c) verify (L), (H), (K), and (T) simultaneously.  The loss
envelope (10.16b), (10.16d)--(10.16f), and (10.25) applies throughout this
chain.  By the fixed zero-profile reserves in \cref{sec:paramclose} and the
final cutoff of \(\omega_{\rm src}\), fix constants
\(\underline\nu_{\rm low},\underline\nu_{\rm slab}>0\), independent of
\(\delta\) and of the loss profile, such that uniformly for
\(0<\delta\le\delta_0(\beta,\zeta,\mathfrak p)\),
\[
 \nu_{\rm low}(\delta)\ge\underline\nu_{\rm low},
 \qquad
 \nu_{\rm slab}(\delta)\ge\underline\nu_{\rm slab}.
\tag{10.39a}
\]
Let
\[
 \nu_{\rm br}:=\min\{\nu_{\rm thick},\underline\nu_{\rm slab},
                 \underline\nu_{\rm low},
                 \nu_{\rm tr}/2,\underline\nu_{\rm tan}\}>0.
\]
The pre-cellular thick branch, the thin slab branch, and the three thin
non-slab alternatives just proved are mutually exhaustive and give
\(\mu(\T,\bar Y)\le\delta^{\nu_{\rm br}-o(1)}|\T|^\beta\).
Because \(\mathfrak p\) was fixed before the cutoff, decrease
\(\delta_0(\beta,\zeta,\mathfrak p)\) so that the remaining \(o(1)\)-loss is at most
\(\delta^{-\nu_{\rm br}/2}\).  This proves (10.40) with
\(\nu_{\rm adm}=\nu_{\rm br}/2\).
\end{proof}

\begin{corollary}[Conditional GWZ Section 9 replacement]
\label{thm:section9}
Fix \(0<\beta\le1\), \(\zeta>0\), and a loss profile \(\mathfrak p\), and assume
\(K_{KT}(\beta)\) and \(K_F(\beta)\).  There exist
\(\eta_{\rm src}=\eta_{\rm src}(\beta,\zeta)>0\),
\(\delta_0=\delta_0(\beta,\zeta,\mathfrak p)>0\), and
\(\nu=\nu(\beta,\zeta)>0\) such that, for every
\(0<\delta\le\delta_0\), each uniform labelled input
\((\T,Y)\) satisfying (10.33) and the conditional interface (I1)--(I4) of
\cref{prop:source-extraction}, with all fixed constants and moduli governed by
\(\mathfrak p\), obeys
\[
 \mu(\T,Y)\le\delta^\nu|\T|^\beta.
\tag{10.41}
\]
\end{corollary}

\begin{proof}
\Cref{prop:source-extraction} supplies an admissible reference shading
\(\bar Y\) with \(M(\bar Y)\ge c_{\rm init}M(Y)\) and
\(c_{\rm init}^{-1}\le\delta^{-\omega_0(\delta)}\).  Apply
\cref{thm:admissible-vns}.  Monotonicity of the union and the definition of
average multiplicity give
\[
 \mu(\T,Y)=\frac{M(Y)}{|U(\T,Y)|}
 \le c_{\rm init}^{-1}
      \frac{M(\bar Y)}{|U(\T,\bar Y)|}
 =c_{\rm init}^{-1}\mu(\T,\bar Y).
\]
Since \(\omega_0\) is controlled by the fixed profile \(\mathfrak p\),
decrease \(\delta_0(\beta,\zeta,\mathfrak p)\) once more so that
\(\omega_0(\delta)\) is below
\(\nu_{\rm adm}/2\), and hence \(c_{\rm init}^{-1}\) costs at most
\(\delta^{-\nu_{\rm adm}/2}\).  Then (10.41) holds with
\(\nu=\nu_{\rm adm}/2\).
\end{proof}

\section{Sharpness of the rectangular density exponent}
\label{sec:sharpness}

The rectangular theorem converts the planar Furstenberg incidence exponent
of D\k{a}browski--Orponen--Villa~\cite{DOV2022} into the shading power
\(q=2-\sigma\).  The construction below shows that this power is optimal when
one asks for a bound uniform over every allowed aspect ratio.  This is a
uniform statement: a fixed parent shape may admit a stronger dependence on
the shading density.

\begin{proposition}[Uniform rectangular sharpness]
\label{prop:uniformsharp}
For every fixed \(0<\sigma\le1\), the exponent \(2-\sigma\) cannot be
replaced, uniformly over
\(\delta\le a\le b\le r\), by a smaller exponent of the shading density.
Precisely, if \(q'<2-\sigma\) and
\(0<\eps<2-\sigma-q'\), there do not exist constants \(c>0\) and
\(\delta_0>0\) such that every family satisfying (6.3) with a uniformly
bounded constant \(C_0\), at every \(0<\delta<\delta_0\), obeys
\[
 \frac{|B\cap N_a(U(\T_B,Y))|}{|B|}
 \ge c(a/b)^\eps\lambda_B^{q'}.
\tag{11.0}
\]
The constants in this negation may depend on the fixed \(\sigma,q',\eps\),
but not on \(\delta,a,b,r\).
\end{proposition}

\begin{lemma}[Two-line convex-hull volume]
\label{lem:twoline-hull}
For \(p=(p_0,p_1)\in[0,1]^2\), let
\[
 \ell_p=\{(x,(1-x)p_0+xp_1):0\le x\le1\}.
\]
If \(\ell_p(c\delta)\) denotes its \(c\delta\)-neighbourhood in a fixed
bounded rectangle, then
\[
 \left|\conv\bigl(\ell_p(c\delta)\cup
                         \ell_{p'}(c\delta)\bigr)\right|
 \gtrsim \delta+|p-p'|.
\tag{11.1a}
\]
\end{lemma}

\begin{proof}
At the vertical slice \(x\), the convex hull contains the interval between
the two affine graph values, enlarged by \(c\delta\).  Hence its area is at
least a fixed multiple of
\[
 \delta+\int_0^1
 |(1-x)(p_0-p'_0)+x(p_1-p'_1)|\,dx.
\]
For an affine function \(f\) on \([0,1]\),
\(\int_0^1|f|\ge c(|f(0)|+|f(1)|)\); this follows by splitting at its
possible zero.  The last display is therefore bounded below by
\(c(\delta+|p-p'|)\).
\end{proof}

\begin{proof}[Proof of \cref{prop:uniformsharp}]
\textit{Planar extremal data.}
Put \(t=1+\sigma\) and fix \(u\in[0,1]\).  Proposition~5.2 of
\cite{DOV2022} gives a \(\delta\)-separated
\((\delta,t,C)\)-set \(P\subset[0,1]^2\), with
\(|P|\asymp\delta^{-t}\), and a \(c\delta\)-separated bounded-chart line
family \(\mathcal L_{u,t}\) such that every \(p\in P\) is
\(\delta\)-incident to at least \(\delta^{-u}\) lines and
\[
 |\mathcal L_{u,t}|
 \lesssim\delta^{-2u-(1-u)(t-1)}.
\tag{11.1b}
\]
Apply bounded point--line duality.  A source line becomes a
\(\delta\)-separated point in a fixed parameter rectangle, while a source
point \(p=(p_0,p_1)\) becomes, up to one fixed affine change of coordinates,
the graph segment \(\ell_p\) in \cref{lem:twoline-hull}.  Incidence means
that the dual point lies in a fixed \(C\delta\)-neighbourhood of that graph.

For each \(p\), choose exactly \(\asymp\delta^{-u}\) incident source lines
and let \(Z_p\) be the union of fixed \(C\delta\)-squares about their dual
points.  The separation in (11.1b) gives
\[
 |Z_p|\asymp\delta^{2-u},\qquad
 \frac{|Z_p|}{|\ell_p(C\delta)|}\asymp\delta^{1-u}=:\lambda.
\tag{11.1c}
\]
Although the same dual square may occur in several labelled shadings, their
physical union is contained in the \(C\delta\)-neighbourhood of the dual
point set of \(\mathcal L_{u,t}\).  Thus
\[
 \left|\bigcup_{p\in P}Z_p\right|
 \lesssim\delta^2|\mathcal L_{u,t}|
 \lesssim\delta^{(1-u)(2-\sigma)}
 =\lambda^{2-\sigma}.
\tag{11.1d}
\]

\textit{Three-dimensional realization and the all-convex condition.}
Let
\[
 B=[-C\delta,1+C\delta]\times[-C,C]\times[-C\delta,C\delta],
\]
permuting its axes to the convention \(a=\delta\), \(b\asymp r\asymp1\).
For \(p\in P\), let \(T_p\) be a fixed solid
\(\delta\times\delta\times1\) tube about
\(\ell_p\times\{0\}\), and set
\[
 Y(T_p)=Z_p\times[-c\delta,c\delta]\subset T_p
\]
after changing the fixed tube constants if necessary.  These are labelled
children even when two geometric pieces overlap, and the enlarged endpoint
interval in \(B\) makes every capped solid tube \(T_p\) a complete child.
Set
\[
 E=U(\{T_p:p\in P\},Y)=\bigcup_{p\in P}Y(T_p).
\]
The \(\delta\)-neighbourhood of \(E\) only replaces each dual
\(\delta\)-square and the transverse interval by fixed dilates.  To make this
step literal, let \(q_\ell\) be the dual parameter point of
\(\ell\in\mathcal L_{u,t}\).  The points \(q_\ell\) are
\(c\delta\)-separated in a fixed parameter rectangle, so the fixed dilates
\(C'Q(q_\ell,\delta)\) have overlap bounded by a constant depending only on
\(C'\).  Every square occurring after the neighbourhood operation is
contained in one such \(C'Q(q_\ell,\delta)\); no new parameter point is
created.  Hence
\[
 \left|\bigcup_{\ell\in\mathcal L_{u,t}}
 C'Q(q_\ell,\delta)\right|
 \le C\delta^2|\mathcal L_{u,t}|,
\tag{11.1d$'$}
\]
and (11.1d) still controls its volume up to an absolute constant.
Equations (11.1c)--(11.1d) give \(\lambda_B\asymp\lambda\) and
\[
 \frac{|B\cap N_a(E)|}{|B|}
 \lesssim \lambda^{2-\sigma}.
\tag{11.1}
\]

It remains to verify the hypothesis for every convex \(K\subset B\).  Put
\(P_K=\{p:T_p\subset K\}\).  If \(P_K\ne\varnothing\), convexity and the
common transverse \(\delta\)-interval imply
\[
 \frac{|K|}{|B|}
 \gtrsim \delta+\diam(P_K)
\tag{11.1e}
\]
by applying \cref{lem:twoline-hull} to a diametral pair (and by the volume of
one tube when \(P_K\) is a singleton).  Since \(P\) is a
\((\delta,t,C)\)-set, \(P_K\) lies in a ball of radius
\(\diam(P_K)+\delta\), and hence
\[
 \frac{\#P_K}{\#P}
 \lesssim(\delta+\diam(P_K))^t
 \lesssim\left(\frac{|K|}{|B|}\right)^{1+\sigma}.
\tag{11.1f}
\]
This is exactly the all-convex complete-child condition (6.3), uniformly in
\(K\), rather than only a parameter-ball estimate.

Finally let \(q'<2-\sigma\), choose
\(0<\eps<2-\sigma-q'\), take \(u=0\), and let \(\delta\downarrow0\).
Then \(\lambda=\delta\), so a uniform lower bound
\(\delta^\eps\lambda^{q'}\) is asymptotically larger than the upper bound
\(\lambda^{2-\sigma}\) in (11.1), a contradiction.  Varying \(u\) gives the
whole resolution range \(\delta\le\lambda\le1\), up to harmless integer
rounding.  The proof includes \(\sigma=1\); constants are allowed to depend
on the fixed \(\sigma\).
\end{proof}

Uniform sharpness does not imply aspectwise sharpness.  If \(a=b\ll r\), one
shaded tube already gives
\[
 |B\cap N_a(Y(T))|\gtrsim \lambda |B|.
\tag{11.2}
\]
Indeed, the longitudinal projection of \(Y(T)\) has measure
\(\gtrsim\lambda r\), and each selected slice fills a fixed portion of the
\(a\times a\) parent cross-section after \(a\)-thickening.  Conversely,
consider the model in which a sufficiently dense transverse family fills the
parent cross-section over one common longitudinal interval \(I\) of length
\(|I|=\lambda r\), and take its shading to be exactly that filled portion.
Fubini's theorem, with the two endpoint layers created by \(a\)-thickening,
gives
\[
 \frac{|B\cap N_a(E)|}{|B|}\asymp \lambda+a/r.
\tag{11.3}
\]
This is only an aspect-ratio model calculation, not another family satisfying
all hypotheses of \cref{prop:uniformsharp}.  It explains why density power one
is the natural target above the \(a/r\) resolution plateau in the square
cross-section regime.
At \(a\asymp r\), a single \(a\)-ball already occupies a parent-scale
volume.

Several natural degenerations do not improve (11.1).  A fixed union of
translated bushes changes only constants.  Sparse multi-bushes produce a
planar union of size \(\lambda/N+\lambda^2\), never smaller than the
quadratic obstruction in their valid range.  Parallel stacks have relative
union size \(\lambda+\delta\).  Exact projection fibres are charged either
by the all-convex constant or by the atom weight in the weighted planar
argument.  Sharp wedges expose a boundary-localization problem, but the
elementary hidden-ray constructions violate the absolute all-convex
hypothesis before they violate the desired conclusion.

\begin{remark}
The DOV construction covers \(a=\delta,b=r\), the endpoints
\(\lambda=\delta,1\), and \(\sigma=1\).  It does not prove aspectwise
optimality for \(b=r,a\gg\delta\), nor does it describe
\(\lambda<\delta\), where the \(\delta\)-neighbourhood has a resolution
plateau.  Constants are not asserted uniform as \(\sigma\downarrow0\).
\end{remark}

\section{Conclusion: stability and the sharp-convex boundary}
\label{sec:limitations}

The method rests on a stable unit of refinement: a positive incidence between
a labelled permanent parent and a half-open cell at the shortest parent
scale.  Selecting these incidences as whole units lets one edge set control
child mass, parent mass, both multiplicities, angular selection, and the final
lift.  An arbitrary refinement inside a cell could instead retain parent mass
where no child mass remains.  The fixed sequence of whole-incidence lifts has
an explicit product of retention fractions, and labelled counting retains the
distinct child partitions and multiplicity contributions of coincident
parents.

This mechanism yields the unconditional cellular theorem, the two Section~6
applications, and the admissible very-non-sticky closure.  The weighted planar
input is the layer-cake consequence of~\cite{DOV2022}, the quadratic endpoint
is the C\'ordoba mechanism~\cite{Cordoba1977}, and the terminal gain is the one
from~\cite{GWZ2026}.  Assuming both \(K_{KT}(\beta)\) and \(K_F(\beta)\), and
under the explicit interface (I1)--(I4),
\cref{prop:source-extraction,thm:section9} construct the admissible data and
supply the corresponding conditional input to the reduction toward Main
Lemma~2.

The principal geometric boundary is the sharp-convex-parent analogue of the
rectangular theorem.  A convex parent \(W\) with a sharp boundary can
be placed inside the comparable-volume source L\"owner box \(B_W\), and complete-child
counts transfer monotonically to the box.  The induced shading does not.
For a triangular prism, a box neighbourhood may extend substantially beyond
the slanted boundary while the numerator is clipped to the prism.  The
complete-child hypothesis counts tubes contained in a convex test body, while
a quantitative comparison of the two neighbourhood measures would require
additional control of the induced shading.  This boundary interface is where
the L\"owner-box reduction stops.

The elementary wedge models considered here do not disprove the
sharp-convex statement: concentrating a fraction of the tube family in a
subwedge worsens the hypothesis constant by a reciprocal power, stronger than
the direct boundary loss.  The remaining question is whether one can replace
the permanent rectangle by an equally stable unit inside \(W\), perhaps
through boundary-sensitive cellularization, a weighted coarea argument that
retains the transverse factor, or a hypothesis controlling induced shadings
rather than complete containment alone.

\appendix

\section{Weighted planar incidence estimates}
\label{app:weighted}

This appendix supplies the analytic input isolated in \cref{sec:density}.
Starting from the unweighted superlevel estimate, two layer-cake reductions
produce the weighted statement for repeated atoms and arbitrary measurable
shadings.
Let \(1<t\le2\), \(q=3-t\), \(0<h\le1\), and \(L\ge1\), and let
\[
 \nu=\sum_i\alpha_i\delta_{p_i},\qquad
 \sum_i\alpha_i=1,\qquad
 \nu(B(p,\rho))\le L\rho^t
 \quad(h\le\rho\le1).
\tag{A.1}
\]
The points lie in a bounded planar parameter chart and may repeat.  Let
\(Y_i\) be arbitrary measurable shadings of constant-width \(h\)-tubes and
put
\[
\Lambda=h^{-1}\sum_i\alpha_i|Y_i|.
\tag{A.2}
\]

\begin{lemma}[Resolution cutoff for weighted superlevels]
\label{lem:superlevel-cutoff}
Let \(0<h\le1/2\), let \(F:(0,1]\to[0,\infty)\), and suppose
\[
 F(u)\le\min\{A/u,B\},\qquad B/A\le C h^{-C_0},
\tag{A.2a}
\]
with the convention that the second condition is void when \(A=0\).  Then
\[
 \int_0^1F(u)\,du
 \le C'(C,C_0)A\log(2/h).
\tag{A.2b}
\]
In particular the bound is independent of any smallest positive superlevel.
\end{lemma}

\begin{proof}
If \(A=0\) there is nothing to prove.  Put \(u_0=\min\{1,A/B\}\).  If
\(B\le A\), the integral is at most \(B\le A\).  Otherwise split at
\(u_0=A/B\):
\[
 \int_0^1F(u)\,du
 \le Bu_0+A\int_{u_0}^1\frac{du}{u}
 =A\bigl(1+\log(B/A)\bigr).
\]
Condition (A.2a) bounds the last logarithm by
\(O_{C,C_0}(\log(2/h))\).  No atom weight enters the split.
\end{proof}

\begin{lemma}[Quantitative physical--parameter duality]
\label{lem:physical-parameter-duality}
Fix bounded intervals \(I_x,I_s,I_t\).  For
\(p=(s,t)\in I_s\times I_t\), let
\[
 \Gamma_p=\{(x,sx+t):x\in I_x\},
 \qquad T_p(Ah)=N_{Ah}(\Gamma_p),
\]
with endpoint truncation to a fixed enlargement of \(I_x\).  For a physical
point \(q=(x_0,y_0)\), define the dual parameter line
\[
 \ell_q:=\{(s,t):t=y_0-sx_0\}.
\tag{A.2c}
\]
For every fixed \(A\ge1\) there are constants \(A_1(A),A_2(A)\), depending
only on \(A\), the three fixed bounded intervals, and the fixed
tube-enlargement convention, such that the following hold for
\(0<h\le1\).
\begin{enumerate}[label=\textup{(D\arabic*)},leftmargin=2.8em]
\item If an \(Ah\)-tube about \(\Gamma_p\) meets the half-open physical
  \(h\)-square with centre \(q\), then
  \[
   \operatorname{dist}(p,\ell_q)\le A_1(A)h.
  \tag{A.2c$'$}
  \]
  Conversely, if this distance is at most \(Ah\), then the
  \(A_2(A)h\)-tube about \(\Gamma_p\) meets a fixed enlargement of that square.
\item The assignment \(q=(x_0,y_0)\mapsto\ell_q\) identifies the usual
  slope--intercept coordinates of \(\ell_q\) with \((-x_0,y_0)\).
  Consequently centres of distinct squares in an \(h\)-grid give an
  \(h\)-separated family of dual lines (and a bounded union of translated
  grids gives a separated family after only a bounded colouring).
\item If \(C\) is a half-open parameter \(h\)-square with centre \(p_C\),
  then every \(T_p(Ah)\), \(p\in C\), is contained in
  \(T_{p_C}(A_2(A)h)\).  Thus parameter-cell aggregation introduces only the
  fixed incidence tolerance in (D1), not a scale-dependent constant.
\end{enumerate}
All statements include physical endcaps and half-open cell boundaries.
\end{lemma}

\begin{proof}
Write \(p=(s,t)\).  The Euclidean distance from \(p\) to \(\ell_q\) is
\[
 \frac{|t-y_0+sx_0|}{\sqrt{1+x_0^2}}.
\tag{A.2d}
\]
If \((x,y)\) lies both in the physical square about \(q\) and within
\(Ah\) of \(\Gamma_p\), then \(|x-x_0|+|y-y_0|\le Ch\) and
\(|y-sx-t|\le C_Ah\).  Boundedness of \(s\) and (A.2d) prove the first
implication.  Conversely, (A.2d) bounds
\(|sx_0+t-y_0|\) by \(C_Ah\); the graph point with first coordinate
\(x_0\), or the nearest endpoint when \(x_0\) lies in an endpoint square,
belongs to a fixed enlargement and proves the reverse implication.

A line \(t=a s+b\) has slope--intercept coordinates \((a,b)\); hence
\(\ell_q\) has coordinates \((-x_0,y_0)\).  This signed coordinate swap is
an isometry, proving (D2).  Finally, for \(p\in C\) and \(x\in I_x\),
\[
 |(s-s_C)x+(t-t_C)|\le C(I_x)h,
\]
which proves (D3).  Closed endcaps only add the same fixed error, while the
half-open convention provides a unique cell on every grid boundary.
\end{proof}

\begin{theorem}[Weighted arbitrary-shading union bound]
\label{thm:weightedappendix}
For every \(\eps>0\),
\[
 \left|\bigcup_iY_i\right|
 \gtrsim_{\eps,t} L^{-q}h^\eps\Lambda^q.
\tag{A.3}
\]
The constant is independent of the number of atoms, the least positive
weight, repeated parameters, and the individual shading densities.
\end{theorem}

\begin{proof}
Let \(E=\bigcup_iY_i\).  We first dispose of the compact-scale range
\(1/2<h\le1\).  Since \(\sum_i\alpha_i=1\),
\[
 |E|\ge\sum_i\alpha_i|Y_i|=h\Lambda\ge\tfrac12\Lambda.
\tag{A.3a}
\]
All graph tubes lie in a fixed bounded physical chart, so
\(|Y_i|\le C_{\rm ch}h\) and \(\Lambda\le C_{\rm ch}\).  Because
\(1\le q<2\), \(L\ge1\), and \(h^\eps\le1\), (A.3a) implies (A.3), with
a constant depending only on the fixed chart and \(\eps,t\).  Hence assume
from now on that \(0<h\le1/2\); \cref{lem:superlevel-cutoff} is then
applicable.

Partition parameter space into half-open \(h\)-squares \(C\), write
\(m_C=\sum_{p_i\in C}\alpha_i\), and place one atom of mass \(m_C\) at the
cell center.  The aggregated measure remains \(t\)-Frostman with constant
\(O(L)\).  By \cref{lem:physical-parameter-duality}(D3), one fixed
enlargement of the graph tube at the cell centre contains every tube with
parameter in that cell.  Take the cell super-shading
\(Z_C=\bigcup_{p_i\in C}Y_i\).  Then
\[
 \sum_Cm_C|Z_C|\ge h\Lambda.
\tag{A.4}
\]

We first derive the weighted incidence estimate strictly from
\cite[Theorem~1.12]{DOV2022} (restated and proved as Theorem~4.1 there).  We
use the specialization \(d=2,n=1\).  If
\(\mu=\sum_jm_j\delta_{p_j}\) is \(h\)-separated and \(\mathcal L\) is an
\(h\)-separated family of dual lines, put
\[
 I_{Ah}(\mu,\mathcal L)
 =\sum_jm_j\#\{\ell\in\mathcal L:
      \dist(p_j,\ell)\le Ah\}.
\]
For \(u>0\), the superlevel set
\(P_u=\{p_j:m_j\ge u\}\) is an
\((h,t,L/(u|P_u|))\)-set.  That theorem with \(d=2,n=1\) gives
\[
 I_{Ah}(P_u,\mathcal L)
 \lesssim_{A,\eta,t}
 \frac{L}{u}h^{-\eta}|\mathcal L|^{1/q}
 h^{(t-1)/q}.
\tag{A.5}
\]
The exact layer-cake identity is
\[
 I_{Ah}(\mu,\mathcal L)
 =\int_0^1 I_{Ah}(P_u,\mathcal L)\,du.
\tag{A.5a}
\]
The singularity at \(u=0\) is cut off by the trivial separated-cardinality
bound \(I_{Ah}(P_u,\mathcal L)\lesssim h^{-2}|\mathcal L|\).  Apply
\cref{lem:superlevel-cutoff} with
\[
 A_0=C Lh^{-\eta}|\mathcal L|^{1/q}h^{(t-1)/q},
 \qquad B_0=Ch^{-2}|\mathcal L|.
\]
Since \(|\mathcal L|\lesssim h^{-2}\), \(L\) may be replaced by
\(\max\{1,L\}\), and \(1\le q<2\), the ratio \(B_0/A_0\) is bounded by a
fixed power of \(h^{-1}\).  Therefore
\[
 I_{Ah}(\mu,\mathcal L)
 \lesssim_{A,\eta,t}
 Lh^{-\eta}\log(2/h)|\mathcal L|^{1/q}
 h^{(t-1)/q}.
\tag{A.6}
\]
The logarithm depends on the resolution, not on the smallest atom weight;
this is exactly the conclusion of \cref{lem:superlevel-cutoff}.

Let \(E=\bigcup_CZ_C\), partition physical space into \(h\)-squares \(Q\),
and set
\[
 w_Q=|E\cap Q|/h^2,\qquad
 W=\sum_Qw_Q=h^{-2}|E|,
\]
\[
 \iota_Q=\sum_{C:\,\widetilde T_C\cap Q\ne\varnothing}m_C.
\]
Equation (A.4) implies
\(\sum_Qw_Q\iota_Q\ge\Lambda/h\).  For
\(\mathcal Q_u=\{Q:w_Q\ge u\}\),
\cref{lem:physical-parameter-duality}(D1)--(D2) sends the cell centres to an
\(h\)-separated family of dual lines and converts every condition
\(\widetilde T_C\cap Q\ne\varnothing\) into an incidence with one fixed
\(Ah\) tolerance.  Therefore (A.6) gives
\[
 \sum_{Q\in\mathcal Q_u}\iota_Q
 \lesssim Lh^{-\eta}\log(2/h)
 |\mathcal Q_u|^{1/q}h^{(t-1)/q}.
\]
Layer cake in \(w_Q\), followed by concavity of
\(x^{1/q}\), yields
\[
 \frac{\Lambda}{h}
 \lesssim Lh^{-\eta}\log(2/h)
 h^{(t-1)/q}W^{1/q}.
\]
Substituting \(W=h^{-2}|E|\) cancels the geometric powers of \(h\) and gives
\[
 |E|\gtrsim
 L^{-q}[\log(2/h)]^{-q}h^{q\eta}\Lambda^q.
\]
Taking \(\eta=\eps/(2q)\) and absorbing the logarithm proves (A.3).
When \(t=2\), \(q=1\), the concavity step is equality; thus the endpoint is
included.
\end{proof}

The ordinary one-dimensional marginal required at \(\sigma=0\) has a
different proof.  We record the exact labelled form, including atomic
repetitions and nonuniform measurable shadings.

\begin{lemma}[Capped graph-tube intersection]
\label{lem:capped-graph-intersection}
Let \(T_p,T_{p'}\) be two constant-width \(h\)-tubes whose carriers are
graphs over the same fixed bounded interval and whose slope--intercept
parameters lie in one fixed bounded chart.  Then
\[
 |T_p\cap T_{p'}|\lesssim
 \frac{h^2}{h+|p-p'|}.
\tag{A.9}
\]
The implicit constant is uniform over the chart, and the estimate includes
coincident parameters.
\end{lemma}

\begin{proof}
Write \(\gamma\) for the slope separation and \(d\) for the intercept
separation.  If \(\gamma\ge h\) and the two carrier graphs cross in a fixed
enlargement of the graph interval, slicing parallel to the first carrier
gives intersection area \(O(h^2/\gamma)\).  If they do not cross there, the
intercept gap gives the stronger bound \(O(h^2/d)\), with empty intersection
once the gap exceeds the fixed enlargement.  In this crossing regime the
parameter distance is comparable, up to chart constants, to the larger of
\(\gamma\) and the corresponding crossing-scale separation.  If
\(\gamma<h\), either the tubes are disjoint or their intersection has area
\(O(h)\); moreover a nonempty intersection forces \(d=O(h)\), hence
\(|p-p'|=O(h)\).  Thus every case is bounded by the right-hand side of
(A.9).  When \(p=p'\), it reduces to the trivial \(O(h)\) tube-area bound.
\end{proof}

\begin{theorem}[Weighted labelled C\'ordoba endpoint]
\label{thm:cordobaappendix}
Let
\[
 \nu=\sum_i\alpha_i\delta_{p_i},\qquad
 \sum_i\alpha_i=1,
 \qquad \nu(D(p,\rho))\le L\rho
 \quad(h\le\rho\le1),
\tag{A.7}
\]
where \(L\ge1\), the parameters lie in one fixed bounded
slope--intercept chart, and repetitions are allowed.  If \(Z_i\) is an
arbitrary measurable subshading of the corresponding constant-width
\(h\)-tube and
\[
 \Lambda=h^{-1}\sum_i\alpha_i|Z_i|,
\]
then
\[
 \left|\bigcup_iZ_i\right|
 \gtrsim L^{-1}[\log(2/h)]^{-1}\Lambda^2.
\tag{A.8}
\]
The constant depends only on the fixed chart and tube-enlargement constants,
not on the number of labels or the least positive weight.
\end{theorem}

\begin{proof}
Put \(F=\sum_i\alpha_i\1_{Z_i}\) and \(E=\bigcup_iZ_i\).  By the preceding
estimate \cref{lem:capped-graph-intersection}, the containing tubes satisfy
(A.9).

For each fixed \(p_i\), decompose parameter space into
\(D(p_i,2h)\) and the annuli
\(2^kh<|p-p_i|\le2^{k+1}h\).  Condition (A.7) gives
\[
 \sum_j\frac{\alpha_j}{h+|p_i-p_j|}
 \lesssim L\log(2/h).
\tag{A.10}
\]
This also controls the full mass of repeated atoms at \(p_i\), since that
mass is at most \(Lh\).  Equations (A.9)--(A.10) imply
\[
 \int F^2
 \le\sum_{i,j}\alpha_i\alpha_j|T_{p_i}\cap T_{p_j}|
 \lesssim Lh^2\log(2/h).
\]
On the other hand \(\int F=h\Lambda\).  Cauchy--Schwarz on \(E\) now gives
\((h\Lambda)^2\le|E|\int F^2\), which is (A.8).  Since only tube
containment was used to replace \(Z_i\cap Z_j\) by \(T_i\cap T_j\), the
argument applies to arbitrary measurable shadings.
\end{proof}

An arbitrary representative from a parameter cell is not valid: a heavy
atom may have empty shading while a light atom carries all the mass.  The
super-shading above is exact.  Alternatively, choosing a maximum-density
representative preserves the weighted cell mass up to a factor two.  No
bounded projection fibre or positive lower atom weight is needed.

\section{A labelled slab \texorpdfstring{\(L^2\)}{L2} estimate}
\label{app:slab}

The broad-middle-axis branch needs an \(L^2\) estimate that counts repeated
geometric slabs by label and allows measurable shadings.  We prove that form
directly here.  Let \(\mathbb S\) be a finite nonempty labelled family of
\(a\times b\times c\) convex slabs in a ball \(B_R\subset\R^3\), where
\(0<a\le b\le c\le R\).  Repeated geometric slabs with different labels are
allowed.  Here ``dimensions \(a\times b\times c\)'' means that each label
comes with an ordered orthonormal frame and a centred rectangle \(R_S\) of
those side lengths such that
\(c_JR_S\subset S\subset R_S\) after a common fixed dilation convention,
for one dimensional constant \(c_J>0\).  In particular
\(|S|\asymp abc\), and all intersection and containment comparisons below
have constants depending only on \(c_J\).  Put
\[
 D:=\Delta_{\max}(\mathbb S),\qquad
 L=1+\left\lceil\log_2(R/a)\right\rceil.
\tag{B.1}
\]
For measurable \(Y(S)\subset S\), write
\[
 A=\sum_S|Y(S)|,\qquad U=\bigcup_SY(S),\qquad
 \lambda=A/\sum_S|S|.
\]
When \(A>0\), also put \(\mu=A/|U|\); then \(|U|>0\).
For every label choose the middle plane \(H_S\) of its declared rectangle,
normal to the shortest axis, and choose a unit normal \(n_S\).  The sign of
the normal is immaterial; angles below are projective and lie in
\([0,\pi/2]\).

\begin{lemma}[Two thickened affine planes in a ball]
\label{lem:two-thick-planes}
Let \(H_i=\{x:n_i\cdot x=t_i\}\), \(i=1,2\), be affine planes with projective
normal angle \(\phi\in[0,\pi/2]\).  If
\(E_i\subset B_R\cap N_{Ca}(H_i)\), where \(0<a\le R\), then
\[
 |E_1\cap E_2|
 \le C'\frac{a^2R}{\max\{a/R,\phi\}}.
\tag{B.1a}
\]
The estimate is uniform in the offsets \(t_i\), and remains valid for capped
or one-sided subsets of the two thickened planes.
\end{lemma}

\begin{proof}
If \(\phi\le a/R\), then Fubini in the \(n_1\)-direction gives the trivial
bound
\(|E_1\cap E_2|\le|B_R\cap N_{Ca}(H_1)|\le C'aR^2\), which is (B.1a).
Suppose \(\phi>a/R\), and put
\(\tau=(n_1\times n_2)/|n_1\times n_2|\).  The affine coordinate map
\[
 x\longmapsto
 (n_1\cdot x-t_1,n_2\cdot x-t_2,\tau\cdot x)
\tag{B.1b}
\]
has Jacobian \(|\det(n_1,n_2,\tau)|=\sin\phi\).  On the intersection, the
first two coordinates range over intervals of length \(O(a)\), while the
third ranges over an interval of length \(O(R)\) because \(x\in B_R\).
Consequently
\[
 |E_1\cap E_2|\le C\frac{a^2R}{\sin\phi}
 \le C'\frac{a^2R}{\phi}.
\]
Offsets merely translate the first two coordinate intervals.  Endcaps and
one-sided truncations only decrease the set being measured.
\end{proof}

\begin{lemma}[Intersecting near-parallel slabs lie in one test slab]
\label{lem:near-parallel-slab-count}
Fix \(S\in\mathbb S\), let \(a/R\le\theta\le1\), and suppose
\(S'\cap S\ne\varnothing\) and
\(\ang(n_S,n_{S'})\le C_0\theta\).  Then
\[
 S'\subset K_{S,\theta}:=
 B_R\cap N_{C_1\theta R}(H_S),
 \qquad |K_{S,\theta}|\le C_2\theta R^3.
\tag{B.1c}
\]
In particular
\[
 \#\{S':S'\cap S\ne\varnothing,
       \ \ang(n_S,n_{S'})\le C_0\theta\}
 \le C_3\frac{D\theta R^3}{abc}.
\tag{B.1d}
\]
All counts are labelled counts.
\end{lemma}

\begin{proof}
Choose \(z\in S\cap S'\), and choose the sign of \(n_{S'}\) so that
\(|n_S-n_{S'}|\le C\theta\).  Write
\(H_S=\{n_S\cdot x=t_S\}\) and similarly for \(S'\).  Since both declared
rectangles have shortest side comparable to \(a\),
\[
 |n_S\cdot z-t_S|+|n_{S'}\cdot z-t_{S'}|\le Ca.
\]
For \(y\in S'\), add and subtract the two affine plane equations at \(z\):
\[
\begin{aligned}
 \dist(y,H_S)
 &\le \dist(y,H_{S'})
      +|(n_S-n_{S'})\cdot(y-z)|\\
 &\qquad+|n_S\cdot z-t_S|+|n_{S'}\cdot z-t_{S'}|\\
 &\le C(a+\theta R)\le C'\theta R.
\end{aligned}
\tag{B.1e}
\]
Here \(y,z\in B_R\) and \(\theta\ge a/R\).  This proves containment in
\(K_{S,\theta}\); Fubini normal to \(H_S\) gives its volume bound.  Finally
every counted \(S'\) is wholly contained in this convex body, so (B.1) and
\(|S'|\ge c_J'abc\) give (B.1d).  Coincident carriers with different labels
are counted separately and are paid for by \(D\).
\end{proof}

\begin{theorem}[Labelled rectangular slab lemma]
\label{thm:slab}
For every \(\lambda\ge0\), one has
\[
 |U|\gtrsim
 \frac{\lambda^2|\mathbb S|abc}
 {1+LD R^4/(b^2c^2)}.
\tag{B.3}
\]
If \(\lambda>0\), this is equivalently
\[
 \mu\lesssim\lambda^{-1}
 \left(1+LD\frac{R^4}{b^2c^2}\right),
\tag{B.2}
\]
No separation or pointwise-constant multiplicity hypothesis is required.
\end{theorem}

\begin{proof}
If \(A=0\), then \(\lambda=0\) and (B.3) is immediate.  Hence assume
\(A>0\), so that \(|U|>0\) and all divisions below are legitimate.
For intersecting slabs define
\[
 \vartheta(S,S')=
 \max\{a/R,\ang(n_S,n_{S'})\}.
\]
\Cref{lem:two-thick-planes} gives
\[
 |S\cap S'|\lesssim \frac{a^2R}{\vartheta(S,S')}.
\tag{B.4}
\]
\Cref{lem:near-parallel-slab-count} gives the labelled count
\[
 \#\{S':S'\cap S\ne\varnothing,
        \ \ang(n_S,n_{S'})\lesssim\theta\}
 \lesssim \frac{D\theta R^3}{abc}.
\tag{B.5}
\]

For \(f=\sum_S\1_{Y(S)}\), divide the pairs into the \(L\) dyadic
\(\vartheta\)-classes.  Equations (B.4)--(B.5), with the diagonal added,
give
\[
 \int f^2
 \lesssim |\mathbb S|abc
 +|\mathbb S|L\frac{DaR^4}{bc}.
\tag{B.6}
\]
Cauchy--Schwarz and
\(\int f=A\asymp\lambda|\mathbb S|abc\) imply
\[
 |U|\ge\frac{A^2}{\int f^2}
 \gtrsim
 \frac{\lambda^2|\mathbb S|abc}
 {1+LD R^4/(b^2c^2)},
\]
which proves both forms.
\end{proof}

The special case of the source statement
\cite[Lemma~6.9, p.~16]{GWZ2026} follows by taking \(b=c=R\); no unlisted
textual correction of that lemma is used.
For Section~9, set \(c=R=r_1\); then (B.3) is exactly (10.15).  Notice that
the proof uses global shading density only, so it applies to the arbitrary
cellular parent shading and to fixed comparable enlargements.

\section{A labelled-multiset filling lemma}
\label{app:multiset}

The transverse branch requires a filling statement on a labelled multiset,
so the argument is formulated on the index set and never identifies
coincident carriers.  Boundary cells are the only delicate point.  We admit
only cells containing a full
\(c\alpha\times c\rho\times c\rho\) piece, and all densities are
computed with the resulting true geometric denominator.  We do not assert the
stronger thickened-proxy conclusions of \cite[Lemma~6.13, p.~18]{GWZ2026},
which are not needed for the transverse branch.

Throughout this appendix a \emph{rectangular
\(\alpha\times\rho\times1\) plank} means a labelled box
\[
 P=x_P+\{t_se_s+t_me_m+t_\ell e_\ell:
 |t_s|\le\alpha/2,\ |t_m|\le\rho/2,\ |t_\ell|\le1/2\},
\tag{C.0}
\]
with a fixed ordered orthonormal frame \((e_s,e_m,e_\ell)\).  Its
two-long-axis plane is \(\operatorname{span}\{e_m,e_\ell\}\).  The frame is
part of the label even when side lengths agree.  For \(C_{\rm box}\ge1\), a
\emph{\(C_{\rm box}\)-comparable rectangular
\(\alpha\times\rho\times1\) plank} has the same form, with ordered side
lengths \(a_P\le b_P\le \ell_P\) satisfying
\[
 C_{\rm box}^{-1}\alpha\le a_P\le C_{\rm box}\alpha,\qquad
 C_{\rm box}^{-1}\rho\le b_P\le C_{\rm box}\rho,\qquad
 C_{\rm box}^{-1}\le \ell_P\le C_{\rm box}.
\tag{C.0$'$}
\]
The exact class (C.0) is the case \(C_{\rm box}=1\).  Thus the full-piece
claims below concern rectangular boxes with controlled side lengths, not a
general convex body specified only by L\"owner dimensions.

\begin{lemma}[Full-piece translated grids for rectangular planks]
\label{lem:full-piece-grid}
Fix \(C_{\rm box}\ge1\).  Let \(S\) be a
\(\theta\times1\times1\) slab and let \(P\) be a
\(C_{\rm box}\)-comparable rectangular
\(\alpha\times\rho\times1\) plank in the sense of (C.0$'$), whose
two-long-axis plane makes angle
\(O(\theta)\) with the plane of \(S\).  Suppose \(\theta\rho\ge4\alpha\).
There is a structural finite family of translated, half-open grids, in
coordinates adapted to \(S\), whose cells have dimensions comparable to
\[
 \theta\rho\times\rho\times\rho,
\tag{C.1}
\]
with the following property.  For every \(x\in P\), at least one grid has a
cell \(Q=Q(x)\) for which
\[
 c\alpha\rho^2\le |P\cap Q|\le C\alpha\rho^2,
\tag{C.2}
\]
 and \(P\cap Q\) contains a
\(c\alpha\times c\rho\times c\rho\) subplank
having \(x\) as one of its points (possibly on its boundary).  The constants and the number of
grids are independent of \(\alpha,\rho,\theta,P,S\), with structural
dependence on \(C_{\rm box}\).
For each fixed grid \(G\), the set
\[
 \mathcal I_{\rm full}(P,G)
 :=\{x\in P:Q_G(x)\text{ has the displayed full-piece property at }x\}
\]
is measurable; here \(Q_G(x)\) is the unique half-open cell selected by the
fixed boundary order.
\end{lemma}

\begin{proof}
Straighten the middle plane of \(S\).  In its two tangential coordinates use
mesh \(C_0\rho\), and in the normal coordinate use mesh
\(C_0\theta\rho\), where \(C_0\) is a sufficiently large structural
constant.  Take three equally spaced translates in each coordinate.  For
every point one translate places it a fixed proportion of the mesh away from
all cell faces.  Inside a tangential disk of radius \(c\rho\), the middle
plane of \(P\) drifts by at most \(C\theta\rho\) in the normal coordinate.
Choosing \(c\) small relative to \(C_0\), and using
\(\alpha\le\theta\rho/4\), leaves a
\(c\alpha\times c\rho\times c\rho\) subplank in the cell.  If \(x\) lies
on, or within \(c\rho\) of, a middle or long endpoint face of the
rectangular carrier, choose the corresponding \(c\rho\)-interval one-sided
toward the interior, shortening on the endpoint side and extending by the
same amount on the interior side.  Since both tangential side lengths are at
least \(C_{\rm box}^{-1}\rho\), choosing \(c\) below a structural multiple
of \(C_{\rm box}^{-1}\) gives these intervals even at a corner.  In the short
coordinate use a one-sided interval of length \(c\alpha\) containing \(x\);
this fits because the actual short side has length at least
\(C_{\rm box}^{-1}\alpha\).  Hence the lower volume in (C.2) is uniform at all
faces, edges, and corners.
Conversely, the projection of \(P\cap Q\) to the middle plane of \(P\) has
area \(O_{C_{\rm box}}(\rho^2)\); Fubini across its
\(O_{C_{\rm box}}(\alpha)\) short thickness gives the upper bound in
(C.2).  Half-open cell faces and a fixed order of all endpoint faces assign
all boundary points deterministically.  For a fixed grid, membership in
\(\mathcal I_{\rm full}(P,G)\) is described by finitely many weak and strict
affine inequalities in the coordinates of \(x\): the half-open inequalities
select \(Q_G(x)\), and the remaining inequalities assert that the prescribed
one-sided intervals fit in both \(P\) and that cell.  It is therefore a
finite union of Borel polyhedral pieces, hence measurable.
\end{proof}

\begin{lemma}[Bounded packing of slab parameters]
\label{lem:slab-parameter-packing}
Work in a fixed bounded ball.  For \(0<\theta\le1\), there is a finite
half-open parameter net \(\mathfrak S_\theta\) of
\(\theta\times1\times1\) slabs with the following two properties.
Every such slab is contained in a fixed enlargement of some
\(S\in\mathfrak S_\theta\) whose middle plane is at projective-plane distance
at most \(C\theta\).  Conversely, for every point \(x\) and every plane cap
\(\mathcal C\) of radius \(C_0\theta\),
\[
 \#\{S\in\mathfrak S_\theta:x\in CS,
       \ P_S\in\mathcal C\}\le C(C_0,C).
\tag{C.2a}
\]
The same estimate holds for labelled candidates after first grouping equal
net parameters; labels within one group are not discarded.
\end{lemma}

\begin{proof}
Parameterize an unoriented affine plane by
\((n,t)\in\mathbb{RP}^2\times[-C_1,C_1]\), where its equation is
\(n\cdot y=t\).  Take a \(c\theta\)-separated net in
\(\mathbb{RP}^2\) and, for every net normal, a half-open
\(c\theta\)-mesh in \(t\).  Associate to a parameter pair the portion, in
the fixed ambient ball, of the \(C\theta\)-neighbourhood of that plane.
Nearest-net assignment, with the half-open order on Voronoi cells and on the
offset intervals, gives the covering assertion.

If \(P_S\in\mathcal C\), separation and two-dimensional packing on
\(\mathbb{RP}^2\) leave \(O_{C_0,C}(1)\) possible normals.  For one normal,
the condition \(x\in CS\) restricts \(t\) to an interval of length
\(O_C(\theta)\), which contains only \(O_C(1)\) offset-net points.  Their
product proves (C.2a).  Grouping labels by their net parameter changes
neither the physical union nor this count of geometric parameter groups.
\end{proof}

\begin{lemma}[Capped two-plank intersection inside a cell]
\label{lem:two-plank-cell-intersection}
Let \(Q\subset\R^3\) have diameter at most \(C_Q\rho\), and let
\(P_i\subset N_{C_P\alpha}(H_i)\), \(i=1,2\), where \(H_i\) are affine planes
whose projective normal angle is \(\phi\in(0,\pi/2]\).  Then
\[
 |P_1\cap P_2\cap Q|
 \le C\min\left\{\alpha\rho^2,
                   \frac{\alpha^2\rho}{\sin\phi}\right\}
 \le C\min\left\{\alpha\rho^2,
                   \frac{\alpha^2\rho}{\phi}\right\}.
\tag{C.2f}
\]
The estimate is uniform in the affine offsets and in all endpoint-cap or
one-sided truncations.  Consequently, if
\(\phi\ge\theta/(2B_{\rm ang})\) and
\(\theta\rho\ge2B_{\rm ang}\alpha\), then
\[
 |P_1\cap P_2\cap Q|
 \le CB_{\rm ang}\frac{\alpha^2\rho}{\theta}
 \le C\alpha\rho^2.
\tag{C.2g}
\]
\end{lemma}

\begin{proof}
The first, trivial bound follows by enclosing \(Q\) in a ball of radius
\(C_Q\rho\) and slicing \(N_{C_P\alpha}(H_1)\) normal to \(H_1\).  For the
transverse bound choose unit normals \(n_i\), write
\(H_i=\{x:n_i\cdot x=t_i\}\), and put
\(\tau=(n_1\times n_2)/|n_1\times n_2|\).  The affine map
\[
 x\longmapsto
 (n_1\cdot x-t_1,n_2\cdot x-t_2,
  \tau\cdot(x-x_Q))
\tag{C.2h}
\]
has Jacobian \(\sin\phi\), where \(x_Q\in Q\) is arbitrary.  On
\(P_1\cap P_2\cap Q\), the first two coordinates range over intervals of
length \(O(\alpha)\), and the last over an interval of length \(O(\rho)\).
The change-of-variables formula gives
\(C\alpha^2\rho/\sin\phi\).  Since
\(\sin\phi\ge(2/\pi)\phi\) on \([0,\pi/2]\), (C.2f) follows.  Under the
two assumptions in (C.2g), the transverse bound is at most
\(CB_{\rm ang}\alpha^2\rho/\theta\), and the buffer makes this at most the
trivial scale \(C\alpha\rho^2\).  Changing an offset translates a coordinate
interval, while caps or cell faces only remove points.
\end{proof}

\begin{theorem}[Self-contained labelled rectangular-plank filling]
\label{thm:multiset-appendix}
Fix \(\eta>0\) and constants \(C_{\rm comp},C_{\rm box}\ge1\).  Let
\(0<\alpha\le\rho\le1\) be sufficiently small, and let
\((\mathcal P,Y)\) be a finite labelled
multiset of \(C_{\rm box}\)-comparable rectangular
\(\alpha\times\rho\times1\) planks in the sense of
(C.0$'$), contained in a fixed unit ball and equipped with measurable shadings,
and allow coincident geometric
carriers or coincident carriers with different ordered frames.  Assume
\[
 |\mathcal P|\le\alpha^{-C_{\rm comp}},\qquad
 \lambda(\mathcal P,Y)\ge\alpha^\eta,\qquad
 \mu(\mathcal P,Y)\ge\alpha^{-\eta},
\tag{C.2b}
\]
Let \(m>0\), and suppose the labelled multiplicity belongs to \([m,2m)\)
at every point of the union.  Let \(0<\theta\le1\), and let
\(A_{\rm ang}=A_{\rm ang}(\alpha),B_{\rm ang}=B_{\rm ang}(\alpha)\ge2\)
satisfy
\[
 \log A_{\rm ang}+\log B_{\rm ang}=o(\log\alpha^{-1}),
\qquad
 \frac{A_{\rm ang}}{\log^C(2/\alpha)}\longrightarrow\infty
 \quad\hbox{for every fixed }C.
\tag{C.2c}
\]
Suppose that at each shaded point the incident two-long-axis planes have raw
angular diameter at most \(C\theta\), while every incident subcollection of
at least \(2m/A_{\rm ang}\) labels has raw angular diameter greater than
\(\theta/B_{\rm ang}\).  Finally assume
\[
 \theta\rho\ge2B_{\rm ang}\alpha.
\tag{C.2d}
\]
Then there is a labelled refinement \((\mathcal P',Y')\) satisfying
\[
 M(Y')\ge \alpha^{o(1)}M(Y).
\tag{C.2d$'$}
\]
In this occurrence the uniform modulus required by the convention after
(1.3) is
\[
 \omega_{\rm fill}(\alpha):=
 \frac{\log(1/c_{\rm fill})}{\log(1/\alpha)}\longrightarrow0,
\]
where \(c_{\rm fill}\in(0,1]\) is the fixed structural product of the
retention fractions in the proof.
This notation always refers to retained labelled shading mass; it does not
mean a fraction of distinct carriers or a fraction of label indices.  Whenever
its union meets a ball \(B_0\) of
radius \(\theta\rho\),
\[
 \frac{|U(\mathcal P',Y')\cap C B_0|}{|C B_0|}
 \ge\alpha^{4\eta+o(1)}.
\tag{C.2e}
\]
Passing to a labelled subfamily cannot increase labelled
\(\Delta_{\max}\).  Every deletion below is a labelled deletion or a
pointwise shading restriction, and every occurrence of cardinality,
multiplicity, and mass counts labels.
\end{theorem}

\begin{proof}
Write the multiset as a finite index set \(I\), a carrier map
\(i\mapsto P_i\), and shadings \(Y_i\subset P_i\).  All planks have volume
comparable to \(\alpha\rho\), and all sums below run over indices.

Choose the half-open slab net of \cref{lem:slab-parameter-packing}.  Assign
each \(P_i\) to the first
slab \(S(i)\) whose fixed enlargement
contains it and whose middle plane is \(O(\theta)\)-parallel to that of
\(P_i\).  At any shaded point, the upper angular-diameter hypothesis and
(C.2a) show that the incident labels belong to only \(C_1\) slab-parameter
groups.  Coincident labels remain together inside their group and therefore
do not invalidate the geometric packing count.  Assign the point, with a
fixed tie rule, to a
group having maximal incident labelled multiplicity, and retain at that point
only the labels in this group.  This pointwise restriction retains at least
\(C_1^{-1}M(Y)\), and its multiplicity is between \(m/C_1\) and \(2m\).

For each group use \cref{lem:full-piece-grid}.  Averaging over the structural
finite set of grids, choose one translated grid per group for which the
labelled mass of points lying in a full piece is at least a fixed fraction of
the group mass.  Write this structural fraction as \(g_0>0\), and restrict to
those points.  Let \(E_0\) be the set on which the retained grid multiplicity
is below \(c_1m\).  Before this restriction the multiplicity was at least
\(m/C_1\), so the discarded incidences contribute at least
\((C_1^{-1}-c_1)m|E_0|\), whereas the retained mass on \(E_0\) is below
\(c_1m|E_0|\).  Since the total discarded mass is at most the pre-grid mass,
\[
 M_{\rm grid}(E_0)
 \le \frac{c_1}{C_1^{-1}-c_1}M_{\rm pre-grid}.
\]
Choose \(c_1\) so that the factor on the right is at most \(g_0/2\).
Deleting \(E_0\) then leaves at least \(g_0/2\) of the pre-grid labelled
mass and gives the asserted lower multiplicity.

At this stage define the present per-label average exactly by
\[
 \overline m_Y:=\frac{M(Y)}{|I|}
 =\frac1{|I|}\sum_{i\in I}|Y_i|.
\tag{C.2e$'$}
\]
We next discard precisely the labels for which
\(|Y_i|<\overline m_Y/2\).  Their total mass is strictly less than
\(|I|\overline m_Y/2=M(Y)/2\).  This
label deletion can again be followed by deleting the points of multiplicity
below \(c_2m\), with \(c_2>0\) structural.  Here is the integral comparison.
Immediately before label deletion the pointwise multiplicity is at least
\(c_1m\).  If \(E_{\rm bad}\) is the set on which the surviving label
multiplicity is below \(c_2m\), then the deleted labels contribute at least
\((c_1-c_2)m|E_{\rm bad}|\).  They contribute less than half of the
pre-deletion mass.  The surviving mass on \(E_{\rm bad}\) is therefore at
most
\[
 c_2m|E_{\rm bad}|
 \le \frac{c_2}{2(c_1-c_2)}M_{\rm pre}.
\]
Taking, for example, \(c_2\le c_1/5\) leaves a fixed positive fraction after
both operations.  Denote the resulting family again by
\((I,Y)\).  It satisfies
\[
 M(Y)\ge c_3M(Y_{\rm in}),\qquad
 f(x):=\sum_{i\in I}\mathbf1_{Y_i}(x)\ge c_2m
 \quad(x\in U(I,Y)),
\tag{C.3}
\]
for a structural constant \(c_3>0\).  Since the two pointwise deletions and
the label deletion together retain
a fixed fraction of labelled mass while the surviving label denominator can
only decrease,
\[
 \frac{M(Y)}{\sum_{i\in I}|P_i|}
 \ge \alpha^{\eta+o(1)}.
\tag{C.4}
\]
The displayed \(o(1)\) in (C.4) uses the uniform modulus obtained by writing
the finitely many structural constants as powers of \(\alpha\).  The assumed
growth of \(A_{\rm ang}\) ensures
\(c_2m>4m/A_{\rm ang}\) for small \(\alpha\).

From this point on, a grid cell is a tagged pair consisting of its slab group
and one half-open cell in the grid selected for that group; we continue to
write it as \(Q\).  Different tagged cells may overlap as physical boxes, but
the preceding pointwise group assignment places every surviving shaded
incidence in exactly one tag and then in exactly one half-open cell.  Thus the
tagged shading pieces, unlike the underlying physical boxes, form a measurable
partition of the labelled mass.  For such a tagged cell, let
 \[
 \begin{aligned}
 I_Q&:=\{i:Y_i\cap Q\ne\varnothing\},
 &D_Q&:=\sum_{i\in I_Q}|P_i\cap Q|,\\
 M_Q&:=\sum_{i\in I_Q}|Y_i\cap Q|,
 &E_Q&:=\bigcup_{i\in I_Q}(Y_i\cap Q),\\
 \lambda_Q&:=M_Q/D_Q.
 \end{aligned}
\tag{C.5}
\]
Only full pieces survived, so \cref{lem:full-piece-grid} gives
\[
 D_Q\asymp |I_Q|\alpha\rho^2.
\tag{C.6}
\]
This is the denominator that is lost if ``nonempty'' intersections are
silently treated as full pieces.  Because the cells are half-open,
\(\sum_QM_Q=M(Y)\); because their geometric intersections are disjoint,
\(\sum_QD_Q\le\sum_i|P_i|\).  Thus (C.4) implies that the
\(D_Q\)-weighted average of \(\lambda_Q\) is at least
\(\alpha^{\eta+o(1)}\).  Retaining all cells with
\[
 \lambda_Q\ge\tfrac12\alpha^{\eta+o(1)}
\tag{C.7}
\]
removes at most half the labelled mass.  This is a whole-cell restriction, so
the pointwise lower multiplicity in (C.3) is unchanged on retained cells.

Fix such a cell and put
\(f_Q=\sum_{i\in I_Q}\mathbf1_{Y_i\cap Q}\).  If \(x\) belongs to its
support and \(n=f_Q(x)\), then for each incident label \(i\) fewer than
\(2m/A_{\rm ang}\) incident labels can have raw plane angle at most
\(\theta/(2B_{\rm ang})\) from \(i\); otherwise those labels would form a
subcollection of angular diameter at most \(\theta/B_{\rm ang}\), contrary to
the hypothesis.  Since \(n\ge c_2m\) and \(A_{\rm ang}\to\infty\), fewer than half of
the ordered pairs are close.  Hence, pointwise,
\[
 f_Q(x)^2\le
 2\!\sum_{\substack{i,j\in I_Q\\
  \ang(P_i,P_j)>\theta/(2B_{\rm ang})}}
 \mathbf1_{Y_i\cap Q}(x)\mathbf1_{Y_j\cap Q}(x).
\tag{C.8}
\]
This direct domination includes the diagonal issue automatically: diagonal
pairs are close and are paid for by the strict majority of separated ordered
pairs.

Each \(Q\) has diameter \(O(\rho)\), and every capped plank lies in an
\(O(\alpha)\)-neighbourhood of its two-long-axis plane.  Therefore
\cref{lem:two-plank-cell-intersection}, with
\(\phi>\theta/(2B_{\rm ang})\), gives
\[
 |P_i\cap P_j\cap Q|
 \le CB_{\rm ang}\frac{\alpha^2\rho}{\theta}.
\]
The buffer (C.2d) verifies the second inequality in (C.2g), so this estimate
is uniform down to the smallest admitted angle and is compatible with the
trivial full-piece volume \(O(\alpha\rho^2)\).  Endcaps and the cell
truncation only decrease the intersection.  Integrating (C.8) therefore gives
\[
 \int_Q f_Q^2
 \le CB_{\rm ang}|I_Q|^2\frac{\alpha^2\rho}{\theta}.
\tag{C.9}
\]
By (C.5)--(C.7), Cauchy--Schwarz, and
\(|Q|\asymp\theta\rho^3\),
\[
\begin{aligned}
 |E_Q|
 &\ge \frac{M_Q^2}{\int_Q f_Q^2}
 &\gtrsim B_{\rm ang}^{-1}\lambda_Q^2|Q|\\
 &\ge \alpha^{3\eta+o(1)}|Q|.
\end{aligned}
\tag{C.10}
\]
Here \(B_{\rm ang}=\alpha^{-o(1)}\); the displayed exponent is a convenient
weakening of the resulting \(2\eta+o(1)\).
Fix the uniform exponent function in this occurrence and write
\[
 \Xi(\alpha):=\alpha^{3\eta+\omega_C(\alpha)},
 \qquad \omega_C(\alpha)\longrightarrow0,
\]
so that (C.10) reads \(|E_Q|\ge c_0\Xi(\alpha)|Q|\) for one
structural \(c_0>0\).

Finally partition every retained \(Q\) into half-open subcells \(R\) of
diameter \(O(\theta\rho)\) and volume comparable to
\((\theta\rho)^3\).  Call \(R\) dense for the tagged cell \(Q\) if
\[
 |E_Q\cap R|\ge c\Xi(\alpha)|R|,
\tag{C.11}
 \]
 with \(c\) below the constant in (C.10).  The tag-specific union \(E_Q\) in
non-dense subcells has the following quantified bound.  Let \(c_0>0\) be the
structural constant implicit in (C.10), fix
\(0<\epsilon_{\rm sub}\le c_2/4\), and choose the constant in (C.11) so that
\(c\le c_0\epsilon_{\rm sub}\).  Since the subcells partition \(Q\),
\[
 \left|E_Q\cap\bigcup_{R\ {\rm non\mbox{-}dense}}R\right|
 \le c\Xi(\alpha)|Q|
 \le\epsilon_{\rm sub}|E_Q|.
\tag{C.11a}
\]
On the support of \(f_Q\), one has
\(c_2m\le f_Q\le2m\).  Hence the labelled mass deleted in (C.11a) is at most
\[
 2m\epsilon_{\rm sub}|E_Q|
 \le\frac{2\epsilon_{\rm sub}}{c_2}M_Q
 \le\tfrac12M_Q.
\tag{C.11b}
\]
Thus dense-subcell selection retains at least one half of the labelled mass
in every retained tagged cell.  Define \(Y'_i\) by retaining, for every tagged \(Q\),
all incidences \(Y_i\cap Q\cap R\) in its dense subcells \(R\), and call the
resulting refinement \((I',Y')\).  In particular
\[
 E_Q\cap R\subset U(I',Y')
 \quad\text{for every dense pair }(Q,R).
\tag{C.11$'$}
\]

Let \(B_0\) be any radius-\(\theta\rho\) ball meeting
\(U(I',Y')\), and choose one retained incidence at a meeting point \(x\).
Its unique tagged cell \(Q\) and dense subcell \(R\) determine the relevant
tag-specific support.  The subcell \(R\) lies in a fixed enlargement of
\(B_0\), and its volume is comparable to \((\theta\rho)^3\).  Equations
(C.11)--(C.11$'$), weakened once more, give
\[
 \frac{|U(I',Y')\cap C B_0|}{|C B_0|}
 \ge \alpha^{4\eta+o(1)}.
\tag{C.12}
\]
All choices were made on the label index set, so coincident carriers were
counted with their full multiplicity.  Label deletion can only decrease
labelled \(\Delta_{\max}\).  The fixed fractions retained at the grouping,
grid, multiplicity, density-level, and dense-subcell steps have a fixed
product \(c_{\rm fill}>0\).  With
\(\omega_{\rm fill}(\alpha)=
\log(1/c_{\rm fill})/\log(1/\alpha)\), this product is
\(\alpha^{\omega_{\rm fill}(\alpha)}\), so the final refinement satisfies
(C.2d$'$) with the displayed uniform modulus.  This proves the theorem.
\end{proof}

\section{Quantitative constants, source map, and version convention}
\label{app:ledger}

This appendix gathers the quantitative hierarchy, the exact interfaces with
the cited sources, and the archival convention used for those comparisons.
Keeping these records together separates provenance from the mathematical
narrative while fixing the versions on which every source-dependent statement
is based.

\paragraph{Version convention and reproducibility.}
The frozen source identifier for this release line is
\texttt{CMDF-R7-20260827}.  The manuscript is self-contained at the TeX level:
the bibliography is embedded, and three clean \texttt{pdflatex} runs resolve
the contents and cross-references.  The external files checked for the
source-dependent comparisons have the following fixed records.
\begin{center}
\scriptsize
\begin{tabularx}{\textwidth}{@{}>{\raggedright\arraybackslash}X@{}}
GWZ arXiv:2601.14411v1 (20 January 2026); PDF byte object at
\url{https://arxiv.org/pdf/2601.14411v1}; \texttt{application/pdf};
566167 bytes; retrieved 2026-09-01T09:12:57Z; SHA-256\\
\quad\texttt{1bc7df83c394cb74038af673c4a25e98}\\
\quad\texttt{b3e018039c75f3b26501ed496028dce9}\\[2pt]
Cohen, 18-page notes dated 7 April 2026; object
\nolinkurl{KakeyaRigorous_4-7-26.pdf}; \texttt{application/pdf};
384469 bytes; retrieved 2026-09-01T09:12:58Z; SHA-256\\
\quad\texttt{071cadd886f9730f458bcc49caad160a1}\\
\quad\texttt{02d1d85b50594fa9f1614cafc1b1379}\\[2pt]
Zeraoulia Rafik, Zenodo record 22067052, version 1.0 (23 August 2026);
record file \nolinkurl{Zeraoulia_Kakeya_Corrections_2026_Public_Revision.pdf};
\texttt{application/pdf}; 491215 bytes; retrieved 2026-09-01T09:13:06Z;
SHA-256\\
\quad\texttt{41619b18aaaed0dd225eff4963e8bdfd}\\
\quad\texttt{de7e087b4b690b3b8fbf055552d78229}\\[2pt]
Zeraoulia Rafik, Zenodo record 22122679, version 1.0 (27 August 2026);
record file \nolinkurl{Hua_Yang_v3_Correction_Note.pdf};
\texttt{application/pdf}; 288054 bytes; retrieved 2026-09-01T09:13:14Z;
SHA-256\\
\quad\texttt{1c16cd0b5a574d1e7dc6251c7db5fb49}\\
\quad\texttt{e505063d8c2fa6a0d76d885a356c6656}
\end{tabularx}
\end{center}
These archival identifiers are intentionally separate from the conventional
bibliography.

All source-specific comparisons in the paper use these dated versions.  In
particular, ``the archived GWZ manuscript'' means arXiv:2601.14411v1 (20
January 2026).  The symbol \(\gtrapprox\) is reserved for a literal source
quotation and is not a third asymptotic convention in the proofs.  The GWZ
phrase ``outer John ellipsoid'' denotes the minimum-volume containing
L\"owner ellipsoid used in \cref{lem:ordered-john}, rather than the
maximum-volume inscribed John ellipsoid; the repeated-eigenvalue tie rule in
this paper fixes the labelled frame and is not attributed to the source.

\paragraph{Section~6 version comparison.}
Branching comparability is used in (6.6e)--(6.6f),
(9.1b$'$)--(9.1b$''$), and (9.5), and the arguments allow highly concentrated
fine descendants.  The pinned v1 proof of Proposition~6.6 prints three
negative density exponents incompatible with densities bounded by one and an
unintroduced left member in equation~(49).  The present derivation instead
obtains (9.2d)--(9.2e) from Lemma~6.4 through
\cref{prop:corrected-s6-interface} and (9.5c)--(9.5d) from the labelled
Lemma~6.1 argument at (9.1c).  These documentary discrepancies therefore do
not enter the logical interface.

The hierarchy is read from left to right: fixed structural exponents and the
terminal positive gain are chosen before the small factoring, grid, shading,
and discretization exponents.  The notation below records logical roles,
not a new optimization of the source parameters.

\begin{table}[!htbp]
\centering
\small
\begin{tabularx}{\textwidth}{@{}cX X@{}}
\toprule
GWZ equations & formulation in this paper & mathematical role\\
\midrule
(94) & \(Y_B\), (10.2) & removal of low-occupancy outer balls\\
(95) & \(Z\), (10.12) & the scale \(a/b\) and Jensen exponent \(q\)\\
(96) & \(Y_1,Z\), (10.13) & compatibility on positive parent--cell incidences\\
(97) & \(Y_B\), (10.14) & the reference shading and good-ball factor\\
(98)--(99) & \(\bar Y\), (10.22)--(10.23) & buffered degree and separate scale losses\\
(100)--(101) & \(Y^*,Z^\angle,\bar Y\), (10.26)--(10.27) & synchronized fine mass and \(c_*^{-1}\)\\
(102)--(103) & \(\bar Y,Z^\angle\), (10.28)--(10.29) & fixed cap input and resolved angle\\
\bottomrule
\end{tabularx}
\caption{Correspondence with the Section~9 equations of \cite{GWZ2026}.}
\label{tab:gwz-correspondence}
\end{table}

\begin{table}[!htbp]
\centering
\small
\begin{tabularx}{\textwidth}{@{}p{2.5cm}p{3.0cm}X@{}}
\toprule
quantity & size or definition & role\\
\midrule
\(h\) & \(c_0a\) & cellular grid scale\\
\(q\) & \(2-\sigma\) & rectangular density power\\
\(L_m,L_{a,b},L_N\) & dyadic level counts & fine multiplicity, dimensions, child count\\
\(L_{\rm reg}\) & \(L_{J1}L_{R1}L_{R2}\) & positive edge and two degree levels\\
\(J_v\) & (10.11) & physical child volume per selected cell\\
\(P_{\rm tag}\) & \(L_\mu L_\omega L_d\) & first whole-tag synchronization\\
\(P_{\rm angtag}\) & \(L_\theta^{\rm tag}L_{d_\angle}^{\rm tag}\) & angle/degree whole-tag synchronization\\
\(P_{\rm sync}\) & (10.16b) & full represented-mass bridge in the thin branches\\
\(c_{\rm ang}\) & \(A_{\rm ang}^{-K}/(J_\theta J_d)\) & lower bound for typical-angle parent mass\\
\(H_{\rm ang}\) & \(4A_{\rm ang}^K\) & converts input degree to output threshold\\
\(c_*\) & (10.25) & final fine-tube refinement fraction\\
\(\kappa_{\rm inc}\) & fine-incidence representation & bounded across tags at one point\\
\(\kappa_{\rm loc}\) & fixed-tag label representation & concentration and cap comparison inside one \(B\)\\
\(\kappa_{\rm amb}\) & bounded overlap & tagged-to-physical recombination\\
\(H_{\rm col}\) & (10.30) & repeated-label conflict coloring\\
\(\omega_{\rm src}(\delta)\) & (10.16d), tending to zero & uniform envelope for the subpower factors\\
\bottomrule
\end{tabularx}
\caption{Quantitative factors used in the cellular construction.}
\label{tab:losses}
\end{table}

\begin{table}[!htbp]
\centering
\footnotesize
\begin{tabularx}{\textwidth}{@{}>{\raggedright\arraybackslash}p{2.25cm}>{\raggedright\arraybackslash}X>{\raggedright\arraybackslash}p{2.8cm}@{}}
\toprule
source & use and local verification & status\\
\midrule
DOV Theorem~1.12/4.1 \cite{DOV2022} &
unweighted superlevels in (A.5); \(t=1+\sigma>1\), parameter separation, the
\((\delta,t)\)-condition, and the cutoff are checked in (A.4)--(A.6) &
published input; weighted form proved here\\
C\'ordoba \(L^2\) method \cite{Cordoba1977} &
plain-Frostman endpoint and \cref{app:slab}; labelled atom mass, coincident
parameters, angular layers, and measurable shadings are handled in
Theorems~A.5 and B.3 &
historical source; both labelled estimates are proved here\\
streamlined reduction \cite{GWZ2026} &
Sections 6 and 9, cap uniformity, and the non-sticky count;
\cref{prop:corrected-s6-interface,prop:s6B}, (10.8f$'$), and (I1)--(I4)
record the local interfaces &
conditional source interface; all subsequent closure steps are proved locally\\
L\"owner--John inclusion \cite{Henk2012,John1948} &
containing ellipsoid and factor-three inner copy; ordered axes and the
repeated-eigenvalue rule are checked in \cref{lem:ordered-john} &
Henk equation~(2); frame transfer proved locally\\
convex factoring framework \cite{WZ2025} &
parent geometry; \cref{lem:biased-factoring} proves the labelled form &
contextual; no theorem imported\\
Wang--Zahl background \cite{WZSticky2026,WZAssouad2025} &
multiscale and non-sticky structure; background for the surrounding scales &
contextual\\
expository reconstructions \cite{Cohen2026,Guth2026Survey} &
surrounding reduction; used for comparison & contextual\\
independent correction note \cite{Zeraoulia2026} &
comparison and shaded-slab calculation; Theorem~B.3 is proved locally &
contextual\\
\bottomrule
\end{tabularx}
\caption{Proof-bearing and contextual source map.}
\label{tab:sourcemap}
\label{tab:sourcemap-context}
\end{table}

For comparison with the source notation, the complete interface chain is
\[
 \text{fixed (94)}
 \longrightarrow \text{global fine level}
 \longrightarrow \text{joint cellular factoring}
 \longrightarrow (95\text{--}97)_{\rm new}
\]
\[
 \longrightarrow \text{resolved ANGLE-CELL}
 \longrightarrow \text{one LIFT and }d_\angle
 \longrightarrow (98\text{--}103)_{\rm new}.
\]
The chain terminates at rectangular carriers.  Its sharp-parent extension is
the open problem isolated in \cref{sec:limitations}.


\clearpage
\begin{thebibliography}{GWZ26}

\bibitem[Coh26]{Cohen2026}
Alex Cohen.
\newblock Proof of the {Kakeya} in {$\mathbb{R}^3$} following
  {Guth--Wang--Zahl}.
\newblock Unpublished notes, 2026.
\newblock Notes dated 7 April 2026, available at
  \url{https://cims.nyu.edu/~ac6074/KakeyaRigorous_4-7-26.pdf}.

\bibitem[C{\'o}r77]{Cordoba1977}
Antonio C{\'o}rdoba.
\newblock The {Kakeya} maximal function and the spherical summation
  multipliers.
\newblock {\em American Journal of Mathematics}, 99(1):1--22, 1977.
\newblock DOI: \url{https://doi.org/10.2307/2374006}.

\bibitem[DOV22]{DOV2022}
Damian D{\k a}browski, Tuomas Orponen, and Michele Villa.
\newblock Integrability of orthogonal projections, and applications to
  {Furstenberg} sets.
\newblock {\em Advances in Mathematics}, 407:108567, 2022.
\newblock arXiv:2107.04471v3; DOI:
  \url{https://doi.org/10.1016/j.aim.2022.108567}.

\bibitem[Gut26]{Guth2026Survey}
Larry Guth.
\newblock The {Kakeya} conjecture, after {Wang} and {Zahl}, 2026.
\newblock S{\'e}minaire Bourbaki survey, arXiv:2604.03416v1, 3 April 2026.

\bibitem[GWZ26]{GWZ2026}
Larry Guth, Hong Wang, and Joshua Zahl.
\newblock A streamlined proof of the {Kakeya} set conjecture in
  {$\mathbb{R}^3$}, 2026.
\newblock arXiv:2601.14411v1, 20 January 2026.

\bibitem[Hen12]{Henk2012}
Martin Henk.
\newblock L{\"o}wner--{John} ellipsoids.
\newblock In Martin Gr{\"o}tschel, editor, {\em Optimization Stories},
  Documenta Mathematica, Extra Volume ISMP, pages 95--106. European
  Mathematical Society, 2012.
\newblock DOI: \url{https://doi.org/10.4171/DMS/6/15}.

\bibitem[Joh48]{John1948}
Fritz John.
\newblock Extremum problems with inequalities as subsidiary conditions.
\newblock In {\em Studies and Essays Presented to R. Courant on His 60th
  Birthday, January 8, 1948}, pages 187--204. Interscience Publishers, New
  York, 1948.

\bibitem[Sch14]{Schneider2014}
Rolf Schneider.
\newblock {\em Convex Bodies: The {Brunn--Minkowski} Theory}, volume 151 of
  {\em Encyclopedia of Mathematics and its Applications}.
\newblock Cambridge University Press, Cambridge, second expanded edition,
  2014.
\newblock DOI: \url{https://doi.org/10.1017/CBO9781139003858}.

\bibitem[WZ25a]{WZAssouad2025}
Hong Wang and Joshua Zahl.
\newblock The {Assouad} dimension of {Kakeya} sets in {$\mathbb{R}^3$}.
\newblock {\em Inventiones Mathematicae}, 241(1):153--206, 2025.
\newblock DOI: \url{https://doi.org/10.1007/s00222-025-01336-x}.

\bibitem[WZ25b]{WZ2025}
Hong Wang and Joshua Zahl.
\newblock Volume estimates for unions of convex sets, and the {Kakeya} set
  conjecture in three dimensions, 2025.
\newblock arXiv:2502.17655v1, 24 February 2025.

\bibitem[WZ26]{WZSticky2026}
Hong Wang and Joshua Zahl.
\newblock Sticky {Kakeya} sets and the sticky {Kakeya} conjecture.
\newblock {\em Journal of the American Mathematical Society}, 39(2):515--585,
  2026.
\newblock DOI: \url{https://doi.org/10.1090/jams/1067}.

\bibitem[Zer26]{Zeraoulia2026}
Zeraoulia Rafik.
\newblock Technical corrections and missing justifications in a streamlined
  proof of the three-dimensional {Kakeya} set conjecture.
\newblock Zenodo preprint, version 1.0, August 2026.
\newblock Zenodo record 22067052, version 1.0; DOI:
  \url{https://doi.org/10.5281/zenodo.22067052}; CC BY 4.0.

\bibitem[Zer26b]{Zeraoulia2026Local}
Zeraoulia Rafik.
\newblock A local concentration correction in Hua and Yang's cellular Kakeya
  reduction.
\newblock Zenodo preprint, version 1.0, 27 August 2026.
\newblock DOI: \url{https://doi.org/10.5281/zenodo.22122679}.

\end{thebibliography}
\end{document}